\documentclass[12pt, a4paper, titlepage, headsepline, twoside]{book}

\usepackage{style_PM}

\theoremstyle{theorem}
\newtheorem{Thm}{Theorem}[section]
\newtheorem{Prop}[Thm]{Proposition}

\newtheorem{Lem}[Thm]{Lemma}
\newtheorem{Cor}[Thm]{Corollary}
\newtheorem{As}[Thm]{Assumption}
\newtheorem{Claim}[Thm]{Claim}

\newtheorem{Def}[Thm]{Definition}

\newtheorem{Rem}[Thm]{Remark}

\theoremstyle{remark}
\newtheorem{Ex}[Thm]{Example}

\newcommand{\Min}{\min\limits_}
\newcommand{\Sup}{\sup\limits_}
\newcommand{\Inf}{\inf\limits_}

\newcommand{\PP}{\mathds{P}}

\newcommand{\R}{\mathbb{R}}

\newcommand{\N}{\mathbb{N}}
\newcommand{\borel}{\mathcal{B}}

\newcommand{\Llra}{\Longleftrightarrow}
\newcommand{\ra}{\rightarrow}

\newcommand{\ind}{\mathds{1}}
\newcommand{\Let}{\coloneqq}

\newcommand{\diff}{\mathrm{d}}

\newcommand{\wt}{\widetilde}

\newcommand{\opt}{^\star}
\newcommand{\ball}[2]{\mathsf{B}_{#2}(#1)}		

\newcommand{\eps}{\varepsilon}
\DeclareMathOperator{\e}{e}

\newcommand{\X}{\mathbb{X}}
\newcommand{\C}{\mathbb{C}}
\newcommand{\B}{\mathbb{B}}
\newcommand{\Y}{\mathbb{Y}}
\newcommand{\U}{\mathbb{U}}
\newcommand{\conv}{\mathsf{conv}}

\newcommand{\M}{\mathcal{M}}
\newcommand{\Lip}{\mathscr{L}}
\newcommand{\dir}[1]{\pmb{\delta}_{#1}}

\newcommand{\inner}[2]{\big \langle #1, #2 \big \rangle}
\newcommand{\KL}[2]{\ensuremath{\mathsf{D}\hspace{0.5mm}{\vphantom{T}}\!\!\left ({#1}\vphantom{\big|}\vert \hspace{-0.4mm} \vert \vphantom{\big|}{#2}\right)}}

\newcommand{\set}[1]{\mathbb{#1}}

\newcommand{\Prim}{\mathcal{P}}
\newcommand{\Dual}{\mathcal{D}}
\newcommand{\op}{\mathcal{A}}
\newcommand{\opn}{\mathcal{A}_{n}}
\newcommand{\Jp}{J}
\newcommand{\Jpn}{J_n}
\newcommand{\Jpnd}{J_n(\delta)}
\newcommand{\JpnN}{J_{n,N}}
\newcommand{\Jd}{\wt{J}}
\newcommand{\Jdn}{\wt{J}_n}
\newcommand{\Jdnd}{\wt{J}_n(\delta)}
\newcommand{\cone}{\mathbb{K}}

\newcommand{\proj}{\Pi}
\newcommand{\comp}[1]{\overline {#1}}

\newcommand{\Jdnr}{\wt{J}_{n,\eta}}

\newcommand{\Jac}{J^{\mathrm{AC}}}
\newcommand{\Jdc}{J^{\mathrm{DC}}}

\newcommand{\Jacn}{J^{\mathrm{AC}}_n}
\newcommand{\JacnN}{J^{\mathrm{AC}}_{n,N}}

\newcommand{\Jdcn}{J^{\mathrm{DC}}_n}

\newcommand{\Jnlb}{J_{n,\eta}^{\rm LB}}
\newcommand{\Jnub}{J_{n,\eta}^{\rm UB}}
\newcommand{\Jneta}{\wt{J}_{n,\eta}}

\newcommand{\Jlb}{J_n^{\rm LB}}

\newcommand{\yeta}{y_{\eta}^{\star}}

\newcommand{\yn}{y^\star_n}

\newcommand{\ext}{\mathcal{E}}
\newcommand{\Cont}{\mathcal{C}}
\newcommand{\Prob}{\mathcal{P}}
\newcommand{\Meas}{\mathcal{M}}
\newcommand{\Sigalg}{\mathcal{G}}
\newcommand{\Func}{\mathcal{F}}
\newcommand{\lip}{\mathrm{L}}
\newcommand{\wass}{\mathrm{W}}

\newcommand{\uball}{\mathsf{B}_n}

\newcommand{\ynb}{{\theta_{\mathcal{D}}}}
\newcommand{\xnb}{\theta_{\mathcal{P}}}

\newcommand{\Rnorm}{\mathfrak{R}}
\newcommand{\Uball}{\U_n}

\newcommand{\cnew}{{\mathbf{c}}}

\newcommand{\wh}{\widehat}

\newcommand{\order}{\mathcal{O}}

\newcommand{\ratio}{\varrho_n}
\newcommand{\T}{\mathds{T}}

\newcommand{\Leb}{\lambda}
\newcommand{\Ab}{\mathscr{A}}
\newcommand{\Yb}{\mathscr{Y}}

\newcommand{\cost}{\psi}
\newcommand{\NN}{\mathsf{N}}

\newcommand{\gmax}{g_{\rm max}}

\usepackage{makecell}
\usepackage{todonotes}

\usepackage{mathrsfs}

\newcommand{\Norm}[1]{\| #1 \|}

\DeclareMathOperator{\XX}{\mathbb{X}}

\newcommand{\inprod}[2]{\ensuremath{\left\langle{#1}\vphantom{\big|},\vphantom{\big|}{#2}\right\rangle}}

\newcommand{\geqc}[1]{\succeq_{#1}}

\newcommand{\norm}[1]{\left\lVert#1\right\rVert}

\newcommand{\ProbIT}[1]{\ensuremath{\mathbb{P}\!\left[\vphantom{\big|}#1\vphantom{\big|}\right]}}

\newcommand{\Borelsigalg}[1]{\ensuremath{\mathcal{B}\!\left(#1\right)}}
\usepackage{tikz}
\usetikzlibrary{fadings,shapes.arrows,shadows}   
\usetikzlibrary{chains}
\usetikzlibrary{fit}
\usepackage{pgflibraryarrows}		
\usepackage{pgflibrarysnakes}		
\usepackage{xcolor}
\usepackage{epsfig}
\usetikzlibrary{shapes.symbols,patterns} 
\usepackage{pgfplots}

\newcommand{\prlsection}[1]{{\bf{#1}.}}

\definecolor{darkgreen}{rgb}{0.0, 0.42, 0.24}

\newcommand{\K}{\mathbb{K}}

\newcommand{\MM}{\mathbb{M}}

\newcommand{\drv}{\ensuremath{\mathrm{d}}}
\newcommand{\indic}[1]{\ensuremath{\boldsymbol{1}_{#1}}}

\newcommand{\Borel}[1]{\ensuremath{\mathcal{B}\!\left(#1\right)}}

\newcommand{\Rp}{\mathbb{R}_{+}}

\newcommand{\Probpi}[1]{\ensuremath{\mathbb{P}^{\pi}_{\nu}\!\left(\vphantom{\big|}#1\vphantom{\big|}\right)}}
\newcommand{\Expecpi}[1]{\ensuremath{\mathbb{E}^{\pi}_{\nu}\!\left[\vphantom{\big|}#1\vphantom{\big|}\right]}}

\newcommand{\AAA}{\mathcal{A}}

\newcommand{\Kb}{\mathcal{K}}

\DeclareMathOperator{\subjectto}{s.t.}

\newcommand{\supp}[2]{\ensuremath{\sigma_{#1}\!\left({#2}\right)}}

\DeclarePairedDelimiter\ceil{\lceil}{\rceil}

\newcommand{\transp}{\ensuremath{^{\scriptscriptstyle{\top}}}}

\newcommand{\I}[2]{I\!\left({#1},{#2}\right)} 
\newcommand{\Hh}[1]{H\!\left({#1}\right)} 
\newcommand{\Hdiff}[1]{h\!\left({#1}\right)} 
\newcommand{\D}[2]{D\!\left({#1}\right| \!\!\left|{#2}\right)} 
\newcommand{\E}[1]{\,{\mathbb E}\!\left[#1\right]} 
\newcommand{\W}{\mathsf{W}} 
\newcommand{\Lp}[1]{\mathrm{L}^{#1}}
\def\Hb{\ensuremath{H_{\rm b}}}
\newcommand{\A}{\mathbb{A}}

\def\setC{\mathsf{c}}

\newcommand{\RQ}{\mathcal{T}}
\newcommand{\RQm}{\mathcal{T}_m}
\newcommand{\ExpecQ}[1]{\ensuremath{\mathbb{E}^{Q(\cdot|x,a)}\!\left[\vphantom{\big|}#1\vphantom{\big|}\right]}}

\newcommand{\meas}{\mathcal{P}}
\newcommand{\WW}{\mathcal{W}}

\newcommand{\bracket}[1]{[#1]}

\newcommand{\Rsp}{\ensuremath{\R_{> 0}}}

\newcommand{\mypart}[2]{
    \setcounter{part}{#1}
    \setcounter{chapter}{6}
    \part*{#2}
    \addcontentsline{toc}{part}{#2}
}

\begin{document}

\thispagestyle{empty}
	\vspace*{-2cm} 
	
	\hspace{-0.0cm}

	\begin{minipage}[t]{\textwidth}

		\begin{center}

			\hspace{-0.9cm}DISS.\ ETH NO.\ 24732


		\end{center}

   \end{minipage}\\[1.0cm]

	\begin{minipage}[h]{\textwidth}

		\begin{center}\bf\Large

			\hspace{-1.0cm} 
			Convex programming in optimal control and information theory

		\vspace{0.5cm}

		\end{center}

   \end{minipage}\\[2cm]

	\begin{minipage}[h]{\textwidth}

		\begin{center}\large

			\hspace{-1.0cm}  A thesis submitted to attain the degree of \\[0.5cm]
			
			\hspace{-1.0cm} DOCTOR OF SCIENCES of ETH ZURICH \\ [0.3cm]
			
			\vspace{0.3cm}
			
			\hspace{-1.0cm} (Dr.\ sc.\ ETH Zurich) \\ [1.9cm]

%
%
%


			\hspace{-1.0cm}    presented by\\[0.5cm]

			\hspace{-1.0cm}    \textsc  {Tobias Samuel Sutter} \\[0.5cm]

			\hspace{-1.0cm}    MSc ETH Zurich\\ 

			\hspace{-1.0cm}    born on 14.\ March 1987\\

			\hspace{-1.0cm}    citizen of St.Gallen-Straubenzell, Switzerland\\[1.0cm]
			\vspace{1cm}
			\hspace{-1.0cm}    accepted on the recommendation of\\[0.5cm]
			\vspace{0.4cm}

			\hspace{-1.0cm}    Prof.\ Dr.\ John Lygeros, examiner (ETH Zurich, Switzerland)\\
			
			\hspace{-1.0cm}    Prof.\ Dr.\ Daniel Kuhn, co-examiner (EPF Lausanne, Switzerland)\\
			
			\hspace{-1.0cm}    Prof.\ Dr.\ Sean Meyn, co-examiner (University of Florida, USA)\\

			\hspace{-1.0cm}    2017

		\end{center}

	\end{minipage}


\pagenumbering{gobble}
\thispagestyle{empty}
\newpage
\vspace*{\fill}
\copyright\ November 2017

Tobias Sutter

All Rights Reserved\\

ISBN 978-3-906916-04-0

DOI\ \,\,10.3929/ethz-b-000218720




\pagenumbering{roman}

\addcontentsline{toc}{chapter}{Acknowledgments}



\chapter*{Acknowledgements}
First and foremost, I would like to express my sincere gratitude to my advisor Professor~John Lygeros for the opportunity to do my PhD at the Automatic Control Lab.
His vast knowledge, patience and continuous guidance together with his unique open-mindedness to new territories are the main reasons why my PhD turned out to be such an exciting and truly wonderful journey.
I will never forget the amount of academic freedom, trust and support I received from him --- John, you are a fantastic supervisor thank you for everything!

It was an incredible luck for me to meet Peyman Mohajerin Esfahani, who approaches scientific problems with an unparalleled level of creativity that from the very beginning of our collaboration deeply impressed me. Better than anybody else he taught me how important (and rewarding) it is for a scientist to leave his comfort zone, while never loosing or sacrificing the mathematical rigour.  His passion and dedication to science is absolutely captivating and heavily influenced my own approach to scientific problems. I remember many nights and evenings we spent in our office filling (and erasing) one whiteboard after another.

I am especially grateful to Professor Daniel Kuhn not only for serving in my PhD exam committee, but also for our fruitful collaboration on approximate dynamic programming.~I greatly appreciate his kind willingness and precious time to work with me and support me with his immense knowledge. Moreover, my sincere thanks go to Professor Sean Meyn for being my co-examiner. Thank you very much for sharing your vast knowledge and experience as well as for your consistent interest in my work.

Thanks to Professor Manfred Morari and John, the Automatic Control Lab provided a priceless collaborative environment to work with brightest minds and genuinely friendly people. Let me start with Debasish Chatterjee, to whom I am always indebted. It was his deep mathematical knowledge combined with his outstanding ability of motivating people, that made it clear to me as a Master's student that I would like to pursue a PhD in the Automatic Control Lab. Ever since then he has been a close friend and scientific advisor. My special thanks go to Arnab Ganguly, who was not only a wonderful friend, but also a great teacher providing me with his unique perspective on research that only a mathematician could have. I also owe Federico Ramponi a large debt, for being extremely influential and motivating particularly in the beginning of my scientific path. Moreover, I would like to offer my sincere thanks to my former office mates Stefan Richter, Kostas Margellos and Andreas Milias for sharing and dealing with research and other topics and always finding time to discuss them as well as for all the valuable comments --- I have benefited considerably from all of you.

My honest and deepest thanks go to my twin brother, who has always very closely supported and guided me, on a personal as well as on an academic level. I have learned so much from him that allowed me to enter scientific territories that I would never even have dared to touch without him.
Also, my family has played an important role within my PhD; my deepest thanks go to my parents and my sister who have supported my life in the academic world unselfishly from the beginning and put trust in me even at times when my research activities seemed quite ominous to them. 
Finally, I simply want to thank Janine for being by my side during this journey. Thanks for your steady belief in me, the encouragement, and for standing by me in all tough times during these years.

\vspace{0.8cm}


\hfill{Tobias Sutter}

\hfill{Zurich, November 2017}




\chapter*{Abstract}
\addcontentsline{toc}{chapter}{Abstract}

The main theme of this thesis is the development of computational methods for classes of infinite-dimensional optimization problems arising in optimal control and information theory.\ The first part of the thesis is concerned with the optimal control of discrete-time continuous space Markov decision processes (MDP). The second part is centred around two fundamental problems in information theory that can be expressed as optimization problems: the channel capacity problem as well as the entropy maximization subject to moment constraints.

Solutions to the optimal control of MDPs traditionally are characterized by means of dynamic programming, that is broadly applicable but suffers from an exponential computational growth with the problem dimension. In addition, it is oftentimes impossible to obtain an explicit solution of such MDPs, which motivates the task of finding tractable approximations leading to explicit solutions. Such approximation schemes are the core of a methodology known as \emph{approximate dynamic programming} (ADP).
We focus on the so-called linear progamming approach to MDPs; in particular we develop an approximation bridge from the infinite-dimensional LP to tractable finite convex programs in which the performance of the approximation is quantified explicitly. The first stage of our proposed approximation scheme provides an explicit error bound from the infinite LP to its semi-infinite approximation and as a key feature introduces an additional norm constraint that effectively acts as a regularizer. For the second step, we offer two approaches that both provide explicit bounds on the approximation error. One is based on randomized algorithms and the other on fast gradient methods in the spirit of Nesterov.

The channel capacity, describing the ultimate limit of reliable communication can be represented by an optimization problem that is difficult to evaluate analytically or numerically for most channels of interest, which motivates the task of finding tractable approximations leading to explicit solutions.
We present a novel iterative method for approximately computing the capacity of discrete memoryless channels, possibly under additional constraints on the input distribution. Based on duality of convex programming, we derive explicit upper and lower bounds for the capacity. 
The presented method requires $O(M^2 N \sqrt{\log N}/\varepsilon)$ to provide an estimate of the capacity to within $\varepsilon$, where $N$ and $M$ denote the input and output alphabet size; a single iteration has a complexity $O(M N)$. We also show how to approximately compute the capacity of memoryless channels having a bounded continuous input alphabet and a countable output alphabet under some mild assumptions on the decay rate of the channel's tail.
It is shown that discrete-time Poisson channels fall into this problem class.

Finally this thesis presents a novel numerical approximation scheme to minimize the relative entropy subject to noisy moment constraints. This is a generalization of the maximum entropy estimation problem and appears in various applications ranging from Machine Learning to Physics and is a key object in first part of this thesis. Our emphasis is on the provable efficiency of the approximation scheme, and in particular we derive explicit bounds on the approximation error.

%




\chapter*{Zusammenfassung}
\addcontentsline{toc}{chapter}{Zusammenfassung}

Diese Dissertation befasst sich mit der numerischen Berechnung einer Klasse unendlich-dimensionaler Optimierungsprobleme, die sowohl in der Regelungstechnik als auch in der Informationstheorie auftauchen. Der erste Teil behandelt die optimale Steuerung von zeitdiskreten Markov-Entscheidungsproblemen (MDP) auf kontinuierlichen R\"aumen. Im zweiten Teil untersuchen wir zwei fundamentale Probleme der Informationstheorie, die als Optimierungsprobleme repr\"asentiert werden k\"onnen: das Kanalkapazit\"ats-Problem und die Berechnung einer Verteilung mit gegebenen Momenten welche die Entropie maximiert.

Traditionellerweise werden die L\"osungen solcher MDPs durch dynamische Programmierung charakterisiert, wobei wobei die Komplexit\"at exponentiell mit der Problemgr\"osse anw\"achst. Diese L\"osungen haben im Allgemeinen keine geschlossene Form und m\"ussen daher in einer m\"oglichst effizienten und exakten Art und Weise numerisch approximiert werden. Solche Approximationsmethoden begr\"unden das Gebiet der \textit{approximativen dynamischen Programmierung} (ADP). Wir befassen uns mit der Formulierung von MDPs als lineare Programme (LP).
Insbesondere entwickeln wir eine Approximation von unendlich-dimensionalen LPs hin zu endlich-dimensionalen konvexen Optimierungsproblemen, wobei der Approximationsfehler explizit quantifiziert wird. Diese Approximation besteht aus zwei Teilen: Im ersten Teil wird das unendlich-dimensionale LP durch ein semi-unendliches konvexes Programm approximiert, wobei wir einen neuen Regularisierungsterm einf\"uhren. Im zweiten Teil etwickeln wir zwei Methoden mit expliziten Fehlerschranken. Eine basiert auf randomisierten Algorithmen und die andere auf schnellen Gradienten Verfahren.

Die h\"ochste Rate mit der Information asymptotisch fehlerfrei über einen fehlerhaften Kanal \"ubertragen werden kann wird durch die Kanalkapazit\"at beschrieben. Diese wiederum kann als konvexes Optimierungsproblem ausgedr\"uckt werden.
F\"ur die meisten Kan\"ale ist diese Optimierung jedoch schwierig analytisch zu l\"osen, warum man an approximativen L\"osungen interessiert ist, welche in effizienter Weise gefunden werden k\"onnen. Wir stellen eine neue iterative Methode vor um die Kanalkapazit\"at zu berechnen, welche zus\"atzliche Bedingungen an die Eingangverteilung ber\"ucksichtigen kann. Basierend auf der Dualit\"at der konvexen Optimierung leiten wir obere und untere Schranken f\"ur die Kanalkapazit\"at her. Die Methode hat eine Komplexit\"at von $O(M^2 N \sqrt{\log N}/\varepsilon)$ Iterationen um eine $\varepsilon$-genaue Approximation der Kapazit\"at zu berechnen, wobei $N$ und $M$ die Gr\"osse des Eingangs- respektive Ausgangsalphabets sind; jede Iteration hat die Komplexit\"at $O(MN)$.
Des Weiteren zeigen wir wie man die Kapazit\"at eines ged\"achtnislosen Kanals mit einem beschr\"ankten kontinuierlichen Eingangsalphabet und einem abz\"ahlbaren Ausgansalphabet approximiert unter schwachen Annahmen an die Abfallrate des Kanalendes. Diese Annahmen sind beispielsweise erf\"ullt f\"ur den zeitdiskreten Poisson-Kanal, dessen Kapazit\"at wir berechnen.

Der letzte Teil besch\"aftigt sich mit der numerischen Minimierung der relativen Entropie unter gegebenen (m\"oglicherweise verrauschten) Moment Bedingungen. Dies ist eine Verallgemeinerung des Maximierungsproblems der Entropie unter gegebenen Momenten --- ein Problem welches zahlreiche Anwendungen im maschinellen Lernen, der Physik oder auch im ersten Teil dieser Arbeit hat. Unser Schwerpunkt liegt wiederum auf der beweisbaren Effizienz der Methode sowie auf expliziten Schranken des Approximationsfehlers.



\setlength{\parskip}{0.51mm}

\addcontentsline{toc}{chapter}{Contents}

\tableofcontents

\setlength{\parskip}{1.1mm}



\pagenumbering{arabic}
\chapter{Introduction }\label{ch.preliminaries}

Mathematical optimization problems arise in many settings, ranging from the control of heating systems to managing entire economies. An amazing variety of practical problems involving decision making (or system design, analysis, and operation) can be cast in the form of a mathematical optimization problem. In general these optimization problems, however, are surprisingly difficult to solve
and approaches to the general problem therefore involve some kind of compromise, such as very long computation time, or the possibility of not finding the solution and relying to some level of approximate solution. There are, however, some important exceptions to the general rule that most optimization problems are difficult to solve.

Two disciplines of engineering that naturally deal with mathematical optimization problems are automatic control and optimal decision making as well as information theory.
This thesis mainly addresses classes of optimization problems from these two disciplines, namely the optimal control of discrete-time continuous space Markov decision processes (MDP), the channel capacity problem as well as the entropy maximization subject to moment constraints.

The first part of the thesis is concerned with a class of discrete-time, stochastic control systems, MDPs. Solutions to these problems traditionally are characterized by means of dynamic programming (DP), that is broadly applicable but suffers from two so-called curses (that were already pointed out by Bellman): the \textit{curse of dimensionality}, i.e., the exponential computational growth with the problem dimension and the \textit{curse of modelling}, i.e., the fact that DP methods require an explicit system model. As a consequence of the first curse, it is oftentimes impossible to obtain an explicit solution of such MDPs, which motivates the task of finding tractable approximations leading to explicit solutions. Such approximation schemes are the core of a methodology known as approximate dynamic programming (ADP) that is equipped with a substantial body of literature. The second curse traditionally is addressed by using data drawn from the system model that is unknown, which refers to the widely studied fields of neural networks or reinforcement learning. Finally, the combination of the two methods has been settled with the term \emph{neuro-dynamic programming} in the 1990s and since then it remains an active field of research with a corresponding literature that is rapidly growing, additionally driven by technological advancements over the last years that have led to an unprecedented increase in the availability of computational resources. Here, we focus on the so-called linear progamming approach to MDPs and more specifically, we leverage recent breakthroughs in the area of sample-based optimization (the "scenario approach", [CC06, CG08, Cal10]) in combination with parametric approximate representation of the cost-to-go function to efficiently compute solutions to MDPs defined on continuous spaces. More details regarding this part and its contribution is provided in Section~\ref{sec:contributions:ADP}

The second part of the thesis studies optimization problems arising from an information theoretic perspective, namely the channel capacity problem and the maximum entropy estimation problem. In his famous 1948 paper "A Mathematical Theory of Communication" Claude E. Shannon derives the ultimate limit of reliable communication - the channel capacity - and gives a general expression of this limit as a function of the conditional probability distribution describing the channel. This expression, however, involves an optimization that is difficult to evaluate analytically or numerically for most channels of interest, which motivates the task of finding tractable approximations leading to explicit solutions. Another, somewhat related class of optimization problems investigated in this thesis are the entropy maximization problems, subject to finitely many moment constraints. 
The details regarding this second part can be found in Section~\ref{sec:contributions:IT}

\section{Outline and contributions}
Here we outline the organization and contributions of the thesis:

\subsection{Part~\ref{part:ADP}: Approximate Dynamic Programming} \label{sec:contributions:ADP}
Part~\ref{part:ADP} is motivated by the optimal control of a class of discrete-time MDPs on continuous compact state and action spaces. This problems are equipped with a mature theoretical framework \cite{ref:Bertsekas-78, ref:Hernandez-96, ref:Hernandez-99}, its numerical computation, however, remains challenging. The first part of this thesis contributes in this direction.

\subsection*{A. Linear programming in infinite-dimensional spaces}
Chapter~\ref{chap:infLP} explains how a wide range of optimal control problems involving MDPs can be equivalently expressed as \emph{static} optimization problems over a closed convex set of measures, more specifically, as infinite-dimensional linear programming problems \cite{Manne, ref:Hernandez-96, ref:Hernandez-99, ref:chapter:Lerma}. This LP reformulation is particularly appealing for dealing with unconventional settings involving additional constraints \cite{ref:Hernandez-03, ref:Borkar-08}, secondary costs \cite{ref:Dufour-14} and information-theoretic considerations \cite{ref:Raginsky-15}, where traditional dynamic programming techniques are not applicable. In addition, the infinite LP reformulation allows one to leverage the developments in the optimization literature, in particular convex approximation techniques, to develop approximation schemes for MDP problems.

Section~\ref{subsec:dual-pair} introduces the general setting for the infinite-dimensional linear programming problems considered in this thesis.
In Section~\ref{section:MDP:LP}, we formally define the Markov control model considered and recall some standard definitions. Then, we review and introduce the LP approach to a particular class of discrete-time MDPs.

\subsection*{B. Approximation of infinite-dimensional LPs}
Chapter~\ref{chap:approx:LP} represents the main contribution of the first part of this thesis. Under suitable assumptions, we develop an approximation bridge from the infinite-dimensional LP (as introduced in Chapter~\ref{chap:infLP}) to tractable finite convex programs in which the performance of the approximation is quantified explicitly. 
The first stage of our proposed approximation scheme provides an explicit error bound from the infinite LP to its semi-infinite approximation and basically uses perturbation analysis and duality of convex optimization. For the second step, we offer two approaches that both provide explicit bounds on the approximation error. One is based on randomized algorithms motivated by \cite{ref:MohSut-13} (Section~\ref{sec:semi-fin:rand}) and the other on fast gradient methods in the spirit of \cite{ref:Nest-05} (Section~\ref{sec:semi-fin:smoothing}). A combination of the two approaches for the second step provides an a posteriori error bound that in practice often is much smaller than the a priori error bounds (Section~\ref{sec:inf-fin}).

\subsection*{C. Approximation of MDPs}

In Chapter~\ref{chap:approx:MDP} we apply and simplify the approximation scheme developed in the previous chapter to the class of infinite-dimensional LPs that arise from discrete-time MDP as introduced in Section~\ref{section:MDP:LP}, which is the content of Sections~\ref{subsec:semi:MDP}, \ref{subsec:rand:MDP} and \ref{subsec:smooth:MDP}. The applicability of the theoretical results is demonstrated through a constrained linear quadratic optimal control problem and a fisheries management problem in Section~\ref{sec:sim}. The approximation method is generalized to the case where the transition kernel of the Markov process is unknown and has to be inferred from data, in Section~\ref{sec:RL}.

\subsection{Part~\ref{part:IT}: Information theoretic problems} \label{sec:contributions:IT}

Part~\ref{part:IT} is concerned with two convex optimization problems that are fundamental in information theory: the channel capacity problem as well as the entropy maximization problem \cite{cover}. Both problems can be cast as infinite-dimensional convex programs that in general do not admit an explicit solution and as such one is interested in finding tractable approximations leading to approximate solutions.

\subsection*{A. Channel capacity approximation}
Chapter~\ref{chap:channel:cap} presents a novel iterative method for approximately computing the capacity memoryless channels, possibly under additional constraints on the input distribution. The method is based on on duality of convex programming and recent developments in structured convex optimization \cite{ref:Nest-05} and provides explicit (prior and posterior) upper and lower bounds for the capacity. In the context of discrete memoryless channels, the presented method requires $O(M^2 N \sqrt{\log N}/\varepsilon)$ to provide an estimate of the capacity to within $\varepsilon$, where $N$ and $M$ denote the input and output alphabet size; a single iteration has a complexity $O(M N)$ and as such in particular (favourable) cases outperforms state-of the art approaches such as for example the celebrated Blahut-Arimoto Algorithm \cite{blahut72, arimoto72}.

We also show how to approximately compute the capacity of memoryless channels having a bounded continuous input alphabet and a countable output alphabet under some mild assumptions on the decay rate of the channel's tail. It is shown that discrete-time Poisson channels fall into this problem class. As an example, we compute sharp upper and lower bounds for the capacity of a discrete-time Poisson channel with a peak-power input constraint.

\subsection*{B. Entropy maximization}

Chapter~\ref{chap:entropy:max} presents a novel numerical approximation scheme to minimize the relative entropy subject to noisy moment constraints. This is a generalization of the so-called maximum entropy estimation problem and appears in various applications ranging from Machine Learning to Physics and in the context of this thesis appears as a key object in Part~\ref{part:ADP} as well as in Chapter~\ref{chap:channel:cap}. Our emphasis is on the provable efficiency of the approximation scheme, and in particular we derive explicit bounds on the approximation error; something, that to the best of our knowlege has not been available yet in the literature.

\section{Publications}

The work presented in this thesis mainly relies on previously published or submitted articles. The thesis only contains a subset of the research performed throughout my PhD studies and several projects are not featured here. The corresponding articles are listed below according to the related chapters.

\noindent {\bf Part \ref{part:ADP} (Chapters~\ref{chap:infLP} to \ref{chap:approx:MDP})}
\begin{itemize}
	\item P.\ Mohajerin Esfahani, {\bf T.\ Sutter}, D.\ Kuhn and J.\ Lygeros, ``From Infinite to Finite Programs: Explicit Error Bounds with an Application to Approximate Dynamic Programming", submitted to \textit{SIAM Journal on Optimization}, June 2017, \cite{ref:Peyman-17}.
	\item P.\ Mohajerin Esfahani, {\bf T.\ Sutter} and J.\ Lygeros, ``Performance Bounds for the Scenario Approach and an Extension to a Class of Non-convex Programs", \textit{IEEE Transactions on Automatic Control}, vol. 60, no. 1, January 2015, \cite{ref:MohSut-13}.
	\item {\bf T.\ Sutter}, A.\ Kamoutsi, P.\ Mohajerin Esfahani and J.\ Lygeros, ``Data-driven approximate dynamic programming: A linear programming approach", \textit{IEEE Conference on Decision and Control}, Melbourne, Australia, 2017,  \cite{ref:SutKam-17}.
	\item {\bf T.\ Sutter}, P.\ Mohajerin Esfahani and J.\ Lygeros, ``Approximation of Constrained Average Cost Markov Control Processes", \textit{IEEE Conference on Decision and Control}, Los Angeles, USA, 2015, \cite{ref:sutter-CDC-14}.

\end{itemize}



\noindent {\bf Part~\ref{part:IT} to (Chapter~\ref{chap:channel:cap} to \ref{chap:entropy:max})}
\begin{itemize}
	\item {\bf T.\ Sutter}, D.\ Sutter, P.\ Mohajerin Esfahani and J.\ Lygeros, ``Efficient Approximation of Channel Capacities", \textit{IEEE Transactions on Information Theory}, vol. 61, no. 4, April 2015, \cite{TobiasSutter15}.

	\item {\bf T.\ Sutter}, D.\ Sutter, P.\ Mohajerin Esfahani and J.\ Lygeros, ``Generalized maximum entropy estimation", submitted to \textit{Journal on Machine Learning Research}, September 2017, \cite{SutSut-17}.

	\item D.\ Sutter, {\bf T.\ Sutter}, P.\ Mohajerin Esfahani and R.\ Renner, ``Efficient Approximation of Quantum Channel Capacities", \textit{IEEE Transactions on Information Theory}, vol. 62, no. 1, January 2016, \cite{ref:D:Sutter-16}.

	\item {\bf T.\ Sutter}, D.\ Sutter, P.\ Mohajerin Esfahani and J.\ Lygeros, ``Capacity Approximation of Memoryless Channels with Countable Output Alphabets", \textit{IEEE International Symposium on Information Theory}, June 2014, \cite{ref:tsutter-ISIT-14}.

	\item D.\ Sutter, {\bf T.\ Sutter}, P.\ Mohajerin Esfahani and J.\ Lygeros, ``Efficient Approximation of Discrete Memoryless Channel Capacities", \textit{IEEE International Symposium on Information Theory}, June 2014, \cite{ref:dsutter-ISIT-14}.

\end{itemize}

\subsection{Other publications}

The following papers have been prepared during my doctorate but are not part of this thesis.

\subsubsection*{Journal papers}

\begin{itemize}
		\item {\bf T.\ Sutter}, A.\ Ganguly and H.\ Koeppl, ``Path Estimation and Variational Inference for Hidden Diffusion Processes", \textit{Journal on Machine Learning Research}, vol. 17, October 2016, \cite{ref:tsutter-JMLR-16}.
		\item {\bf T.\ Sutter}, D.\ Sutter and J.\ Lygeros, ``Capacity of Random Channels with Large Alphabets", \textit{Advances in Mathematics of Communications}, vol. 11, no. 4, pp. 813 - 835, November 2017, \cite{ref:random:channel}.
		\item A.\ Kamoutsi, {\bf T.\ Sutter}, P.\ Mohajerin Esfahani and J.\ Lygeros, ``On Infinite Linear Programming and the Moment Approach to Deterministic Infinite Horizon Discounted Optimal Control Problems", \textit{IEEE Control Systems Letters}, vol. 1, no. 1, June 2017, \cite{ref:Kamoutsi-17}.
		\item {\bf T.\ Sutter}, D.\ Chatterjee, F.\ Ramponi and J.\ Lygeros, ``Isospectral flows on a class of finite-dimensional Jacobi matrices", \textit{Systems and Control Letters},  vol. 62, no. 5, May 2013, \cite{ref:tsutter-SCL-13}.
		\item {\bf T.\ Sutter} and J.\ Lygeros, ``Signals and Systems II: A Flipped Classroom Experiment for Undergraduate Control Education", \textit{ASME Dynamic Systems and Control Magazine}, vol. 138, no. 6, June 2016, \cite{ref:tsutter-ASME-16}.

\end{itemize}

\subsubsection*{Conference papers}
\begin{itemize}
\item {\bf T.\ Sutter}, D.\ Sutter and J.\ Lygeros, ``Asymptotic Capacity of a Random Channel", \textit{Allerton Conference on Communication, Control, and Computing}, September 2014, \cite{ref:Allerton:random:channel}.
\item M.\ Th{\'e}ly, {\bf T.\ Sutter}, P.\ Mohajerin Esfahani and J.\ Lygeros, ``Maximum Entropy Estimation via Gauss-LP Quadratures", \textit{IFAC World Congress}, Toulouse, France, 2017, \cite{ref:Thely-17}.
\end{itemize}

\part{Approximate Dynamic Programming}\label{part:ADP}



\chapter{Linear programming in infinite-dimensional spaces} \label{chap:infLP}
In this chapter, in Section~\ref{subsec:dual-pair}, we formally introduce a class of infinite-dimensional linear programming problems, for which an approximation scheme is developed in Chapter~\ref{chap:approx:LP}, that can be seen as the main contribution in Part~\ref{part:ADP} of this thesis. A motivation for this treatment is presented (and will be carried though in Chapter~\ref{chap:approx:MDP}) in Section~\ref{section:MDP:LP}, namely the control of discrete-time MDP and their LP characterization. Using standard results in the literature we embed these MDP in the more general framework of infinite LPs.

\section{Introduction}
Linear programming (LP) problems in infinite-dimensional spaces appear in various areas as for example engineering, economics, operations research and probability theory  \cite{ref:Luenberger-69, ref:Anderson-87, ref:Lasserre-11}. Infinite LPs offer remarkable modeling power, subsuming general finite-dimensional optimization problems and the generalized moment problem as special cases. They are, however, often computationally intractable, motivating the study of approximations schemes.
	
	A particularly rich class of problems that can be modeled as infinite LPs involves Markov decision processes (MDP) and their optimal control. More often than not, it is impossible to obtain explicit solutions to MDP problems, making it necessary to resort to approximation techniques. Such approximations are the core of a methodology known as \emph{approximate dynamic programming} (ADP) \cite{ref:Bertsekas-96, ref:VanRoy2002, ref:Bertsekas-12}. Interestingly, a wide range of optimal control problems involving MDPs can be equivalently expressed as \emph{static} optimization problems over a closed convex set of measures, more specifically, as infinite LPs \cite{Manne, ref:Hernandez-96, ref:Hernandez-99, ref:chapter:Lerma}. This LP reformulation is particularly appealing for dealing with unconventional settings involving additional constraints \cite{ref:Hernandez-03, ref:Borkar-08}, secondary costs \cite{ref:Dufour-14}, information-theoretic considerations \cite{ref:Raginsky-15}, and reachability problems \cite{ref:Kariot-13,ref:Moh:motion-16}. In addition, the infinite LP reformulation allows one to leverage the developments in the optimization literature, in particular convex approximation techniques, to develop approximation schemes for MDP problems. This will also be the perspective adopted in the first part of this thesis.
	
	Historically, the LP formulation of MDPs is almost as old as the dynamic programming approach \cite{ref:Bellman-54} and was derived in the pioneering work \cite{Manne}. It was later extended to the so-called convex analytical approach to optimal control problems \cite{ref:Borkar-88, ref:chapter:Borkar}. In the setting of uncountable spaces, the underlying LPs are infinite-dimensional and computationally intractable in general and the study of approximation schemes for such infinite LPs, which is the content of Chapters~\ref{chap:approx:LP} and \ref{chap:approx:MDP}, is mainly unexplored in the literature so far.

\section{Infinite-dimensional LPs} \label{subsec:dual-pair}
	We introduce the general setting for the infinite-dimensional linear programming problems considered in this thesis. A more detailed exposition of this material can be found in \cite{ref:Anderson-87}.
	\begin{Def}[Dual pair]
	The triple $\big(\X,\C, \|\cdot\|\big)$ is called a \emph{dual pair} of normed vector spaces if
	\begin{itemize}
		\item $\X$ and $\C$ are vector spaces; 
		
		\item $\inner{\cdot}{\cdot}$ is a bilinear form on $\X\times \C$ that ``separates points", i.e., 
		\begin{itemize}
			\item for each nonzero $x \in \X$ there is some $c \in \C$ such that $\inner{x}{c} \neq 0$,
			\item for each nonzero $c \in \C$ there is some $x \in \X$ such that $\inner{x}{c} \neq 0$;
		\end{itemize}
		
		\item $\X$ is equipped with the norm $\|{\cdot}\|$, which together with the bilinear form induces a \emph{dual} norm on $\C$ defined through $\|{c}\|_* \Let \sup_{\|{x}\|\le 1}\inner{x}{c}$. 
	\end{itemize}
	\end{Def}
	The norm in the vector spaces is used as a means to quantify the performance of the approximation schemes. In particular, we emphasize that the vector spaces are not necessarily complete with respect to these norms.
	
	Let $\big(\B,\Y, \|\cdot\| \big)$ be another dual pair of normed vector spaces. As there is no danger of confusion, we use the same notation for the potentially different norm and bilinear form for each pair. Let $\op:\X \ra \B$ be a linear operator, and $\cone$ be a convex cone in $\B$. Given the fixed elements $c \in \C$ and $b \in \B$, we define a linear program, hereafter called the \emph{primal} program \ref{primal-inf}, as
	\begin{align} 
	\label{primal-inf} 
	\tag{$\Prim$}
	\Jp \Let \left\{ \begin{array}{ll}
	\Inf{x \in \X} & \inner{x}{c}   \\
	\subjectto & \op x \geqc{\cone} b, 
	\end{array} \right.
	\end{align}
	where the conic inequality $\op x \geqc{\cone} b$ is understood in the sense of $\op x - b \in \cone$. Throughout this study we assume that the program \ref{primal-inf} has an optimizer (i.e., the infimum is indeed a minimum), the cone $\cone$ is closed and the operator $\op$ is continuous where the corresponding topology is the weakest in which the topological duals of $\X$ and $\B$ are $\C$ and $\Y$, respectively. Let $\op^*:\Y \ra \C$ be the adjoint operator of $\op$ defined by
		\begin{align*}
		\inner{\op x}{y} = \inner{x}{\op^* y}, \qquad \forall x \in \X, \quad \forall y \in \Y. 
		\end{align*}
		Recall that if $\op$ is weakly continuous, then the adjoint operator $\op^*$ is well defined as its image is a subset of $\C$ \cite[Proposition 12.2.5]{ref:Hernandez-99}. The \emph{dual} program of \ref{primal-inf} is denoted by \ref{dual-inf} and is given by 
		\begin{align} 
		\label{dual-inf} 
		\tag{$\Dual$}
		\Jd \Let \left\{ \begin{array}{ll}
		\Sup{y\in \Y} & \inner{b}{y}   \\
		\subjectto & \op^* y = c \\
		& y \in \cone^*, 
		\end{array} \right.
		\end{align}
		where $\cone^*$ is the dual cone of $\cone$ defined as $\cone^* \Let \big\{ y \in \Y : \inner{b}{y} \ge 0, \  \forall \ b \in \cone \big\}.$ It is not hard to see that \emph{weak duality} holds, as
		\begin{align*}
		\Jp = \Inf{x \in \X} \Sup{y \in \cone^*} \inner{x}{c} - \inner{\op x - b}{y} \ge \Sup{y \in \cone^*} \Inf{x \in \X} \inner{x}{c} - \inner{\op x - b}{y} = \Jd.
		\end{align*}
		An interesting question is when the above assertion holds as an equality. This is known as \emph{zero duality gap}, also referred to as \emph{strong duality} particularly when both \ref{primal-inf} and \ref{dual-inf} admit an optimizer \cite[p.\ 52]{ref:Anderson-87}. Our study is not directly concerned with conditions under which strong duality between \ref{primal-inf} and \ref{dual-inf} holds; see \cite[Section 3.6]{ref:Anderson-87} for a comprehensive discussion of such conditions. The programs \ref{primal-inf} and \ref{dual-inf} are assumed to be \emph{infinite}, in the sense that the dimensions of the decision spaces ($\X$ in \ref{primal-inf}, and $\Y$ in \ref{dual-inf}) as well as the number of constraints are both infinite.

\section{Linear programming approach to MDPs} \label{section:MDP:LP}

We briefly recall some standard definitions and refer interested readers to \cite{ref:Hernandez-96, ref:Hernandez-03, ref:Arapostathis-93} for further details. Consider a \emph{Markov control model}
	$\big( S,A,\{A(s) : s\in S\},Q,\cost \big),$
	where $S$ (resp.\ $A$) is a metric space called the \emph{state space} (resp.\ \emph{action space}) and for each $s \in S$ the measurable set $A(s) \subseteq A$ denotes the set of \textit{feasible actions} when the system is in state $s\in S$. The \emph{transition law} is a stochastic kernel $Q$ on $S$ given the feasible state-action pairs in $K \Let\{(s,a):s\in S, a\in A(s)\}$. A stochastic kernel acts on real valued measurable functions $u$ from the left as
	\begin{align*}
	Qu(s,a):= \int_{S}u(s') Q(\drv s'|s,a), \quad \forall (s,a)\in K,
	\end{align*}
	and on probability measures $\mu$ on $K$ from the right as
	\begin{align*}
	\mu Q(B) := \int_{K}Q(B|s,a)\mu\big(\drv(s,a)\big), \quad \forall B\in \borel(S),
	\end{align*} 
	where $\borel(S)$ denotes the Borel $\sigma$-algebra on $S$.
	Finally $\cost:K \to\R_+$ denotes a measurable function called the \emph{one-stage cost function}. The \emph{admissible history spaces} are defined recursively as $H_{0}\Let S$ and $H_{t}\Let H_{t-1}\times K$ for $t\in\N$ and the canonical sample space is defined as $\Omega\Let(S\times A)^{\infty}$. All random variables will be defined on the measurable space $(\Omega,\Sigalg)$ where $\Sigalg$ denotes the corresponding product $\sigma$-algebra. A generic element $\omega\in\Omega$ is of the form $\omega=(s_{0},a_{0},s_{1},a_{1},\hdots)$, where $s_{i}\in S$ are the states and $a_{i}\in A$ the action variables. An \textit{admissible policy} is a sequence $\pi=(\pi_{t})_{t\in\N_{0}}$ of stochastic kernels $\pi_{t}$ on $A$ given $h_t\in H_{t}$, satisfying the constraints $\pi_{t}(A(s_{t})|h_{t})=1$. The set of admissible policies will be denoted by $\Pi$. 
	Given a probability measure $\nu\in \Prob(S)$ and policy $\pi\in\Pi$, by the Ionescu Tulcea theorem \cite[p.~140-141]{ref:Bertsekas-78} there exists a unique probability measure $\mathbb{P}^{\pi}_{\nu}$ on $\left(\Omega,\Sigalg \right)$ such that for all measurable sets $B \subset S$, $C \subset A$, $h_{t}\in H_{t}$, and $t\in\N_{0}$
	\begin{align*}
	\Probpi{s_{0}\in B}			&= \nu(B)  \\
	\Probpi{a_{t}\in C|h_{t}}		&= \pi_{t}(C|h_{t})  \\
	\Probpi{s_{t+1}\in B|h_{t},a_{t}}	&= Q(B|s_{t},a_{t}). 
	\end{align*}
	The expectation operator with respect to $\mathbb{P}^{\pi}_{\nu}$ is denoted by $\mathbb{E}^{\pi}_{\nu}$. The stochastic process $\big( \Omega,\Sigalg,\mathbb{P}^{\pi}_{\nu},(s_{t})_{t\in\N_{0}} \big)$ is called a \emph{discrete-time MDP}.
	For most of this thesis we consider optimal control problems where the aim is to minimise a long term \emph{average cost} (AC) over the set of admissible policies and initial state measures. We definite the optimal value of the optimal control problem by
	\begin{align} \label{AC}
	\Jac \Let  \inf_{(\pi,\nu) \in \Pi \times \Prob(S)}\limsup_{T\to\infty} \frac{1}{T} \Expecpi{\sum_{t=0}^{T-1}\cost(s_{t},a_{t})}.
	\end{align}	
	We emphasize, however, that the results also apply to other performance objective, including the long-run \emph{discounted cost} problem as shown in Section~\ref{sec:discounted:setting}.

The problem in \eqref{AC} admits an alternative LP characterization under some mild assumptions. 
	
	\begin{As}[Control model] \label{a:CM} We stipulate that
		\begin{enumerate}[label=(\roman*), itemsep = 1mm, topsep = -1mm]
			\item \label{a:CM:K} the set of feasible state-action pairs is the unit hypercube $K = [0,1]^{\dim(S\times A)}$;
			\item \label{a:CM:Q} the transition law $Q$ is Lipschitz continuous, i.e., there exists $L_Q>0$ such that for all $k, k'\in K$ and all continuous functions $u$  
			$$|Qu(k) - Qu(k')| \le L_Q \|u\|_\infty \|k - k'\|_{\ell_\infty};$$
			\item \label{a:CM:cost} the cost function $\cost$ is non-negative and Lipschitz continuous on $K$ with respect to the $\ell_\infty$-norm.
		\end{enumerate}
	\end{As}
	
	Assumption \ref{a:CM}\ref{a:CM:K} may seem restrictive, however, essentially it simply requires that the state-action set $K$ is compact. We refer the reader to Example~\ref{ex:fisheries} where a non-rectangular $K$ is transferred to a hypercube, and to \cite[Chapter~12.3]{ref:Hernandez-99} for further information about the LP characterization in more general settings.
	
	\begin{Thm}[LP characterization {\cite[Proposition~2.4]{ref:Dufour-15}}] 
		\label{thm:equivalent:LP}
		Under Assumption~\ref{a:CM}, 
		\begin{align}
		\label{AC-LP} 
		-\Jac = & \left\{ \begin{array}{ll}
		\inf\limits_{\rho, u} & -\rho   \\
		 \subjectto &\rho + u(s) - Qu(s,a) \leq \cost(s,a), \quad \forall (s,a)\in K \\
		& \rho\in\R, \quad u\in\Lip(S).
		\end{array} \right. 
		\end{align}
	\end{Thm}
	
	The LP \eqref{AC-LP} can be expressed as \ref{primal-inf}, i.e., in the standard conic form $\inf_{x \in \X} \big\{\inner{x}{c} : \op x - b \in \cone \big\}$ by introducing
	\begin{align}
	\label{AC-setting}
	\left\{
	\begin{array}{l}
	\X = \R\times\Lip(S) \\
	x = (\rho, u) \in \X\\
	c = (c_{1},c_{2})=(-1,0) \in \R\times\Meas(S) \\
	b(s,a) = -\cost(s,a) \\ 
	\inner{x}{c} = c_{1}\rho +\int_S u(s) c_{2}(\diff s) \\
	\op x(s,a) = -\rho - u(s) + Qu(s,a) \\
	\cone = \Lip_+(K),  
	\end{array}\right.
	\end{align}
	where $\Meas(S)$ is the set of finite signed measures supported on $S$, and $\Lip_+(K)$ is the cone of Lipschitz functions taking non-negative values. It should be noted that the choice of the positive cone $\K = \Lip_+(K)$ is justified since, thanks to Assumption~\ref{a:CM}\ref{a:CM:Q}, the linear operator $\op$ maps the elements of $\X$ into $\Lip(K)$. 
	
	\begin{Rem}[Constrained MDP]
		The LP characterization of MDP naturally allows us to incorporate constraints in the form of 
		\begin{align*}
		\limsup_{T\to\infty} \frac{1}{T}\, \Expecpi{\sum_{t=0}^{T-1}d_i(s_{t},a_{t})} \le \ell_i, \qquad \forall i \in \{1,\cdots,I\},
		\end{align*}
		where the functions $d_i: K \ra \R$ and constants $\ell_i$ reflect our desired specifications. To this end, it suffices to introduce auxiliary decision variables $\beta_i \in \R_+$, and in \eqref{AC-LP} replace $\rho$ in the objective with $\rho - \sum_{i=1}^{I}\beta_{i} \ell_{i}$ and in the constraint with $ \rho - \sum_{i=1}^{I}\beta_{i}d_{i}$, see \cite[Theorem 5.2]{ref:Hernandez-03}. 
	\end{Rem}
	
	Our aim is to derive an approximation scheme for a class of such infinite-dimensional LPs, including problems of the form \eqref{AC-LP}, that comes with an explicit bound on the approximation error.




\chapter{Approximation of infinite-dimensional LPs} \label{chap:approx:LP}

In this chapter, we present an approximation scheme for a class of infinite-dimensional linear programs \ref{primal-inf}, as formally introduced in Section~\ref{subsec:dual-pair}.
The scheme basically consists of two approximation steps: first the infinite LP \ref{primal-inf} is reduced to a semi-infinite (or robust) LP, that is the content of Section~\ref{subsec:semi}.
Then, the semi-infinite LP is approximated by means of a finite-dimensional convex program, for which we present to complementary approaches, one based on randomized sampling methods (Section~\ref{sec:semi-fin:rand}) and another based on accelerated first-order methods (Section~\ref{sec:semi-fin:smoothing}). The two-step process is finally combined in Section~\ref{sec:inf-fin} and establishes a link from the original infinite LP to finite-dimensional convex counterparts.

\section{Introduction} \label{subsec:approxLP:intro}
	Approximation schemes to tackle infinite LPs, as introduced in Section~\ref{subsec:dual-pair}, have historically been developed for particular classes of problems, e.g., the general capacity problem \cite{ref:Lai-92}, the generalized moment problem in a polynomial setting \cite{ref:Lasserre-11} or the LP approach to MDPs, in the setting described in Section~\ref{section:MDP:LP} \cite{ref:Duf-13, ref:Dufour-15}. An exception is the paper by Hernandez-Lerma and Lasserre \cite{ref:Hernandez-98}, where an approximation scheme for general linear programs is proposed. The main difficulty in practically using this scheme, however, is that the convergence proof is an existence proof but is not constructive. Furthermore, there are no explicit error bounds available.
	
To date, to the best of our knowledge, there are no results in the literature regarding approximation schemes for general linear programs on infinite dimensional spaces that are applicable in practice and that provides explicit error estimates. This chapter, that can be seen as the main contribution of the Part~\ref{part:ADP} of this thesis, fills this gap by investigating	
an approximation scheme  that involves a restriction of the decision variables from an infinite dimensional space to a finite dimensional subspace, followed by the approximation of the infinite number of constraints by a finite subset; we develop two complementary methods for performing the latter step. The structure of this section is illustrated in Figure~\ref{fig:overview}, where the contributions are summarized as follows:

	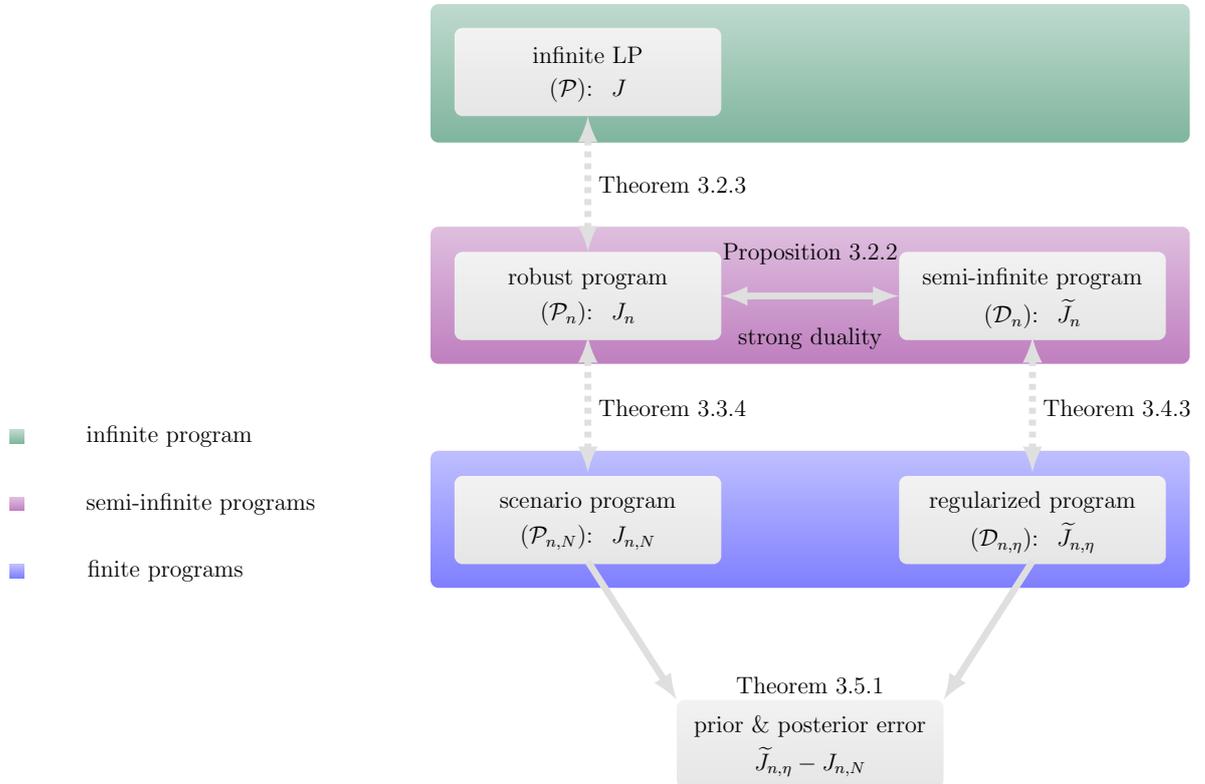
\begin{figure}[t!]
		\centering
		\scalebox{.75}{
\def \sca{1.2}

\def \xb{3.9*\sca}
\def \yb{1.3*\sca}

\def \r{6.5*\sca}

\def \y{3.3*\sca} 
\def \yy{2*\sca} 

\def \la{0.2*\sca}

\def \di{0.2*\sca}

\def \backx{-0.1*\sca}
\def \backy{-1.4*\sca}
\def \back{0.2*\sca}

\def \buf{0.35*\sca}

\begin{tikzpicture}[scale=1,auto, node distance=1cm,>=latex']

\shade[rounded corners,bottom color = darkgreen!50, top color = darkgreen!25] (\r-\buf,\yb+\buf)--(\r+\r+\xb+\buf,\yb+\buf)--(\r+\r+\xb+\buf,-0.5*\y+3.6*\buf)--(\r-\buf,-0.5*\y+3.6*\buf)--cycle;

\shade[rounded corners,bottom color = violet!50, top color = violet!25] (\r-\buf,-0.6*\y+\buf)--(\r+\r+\xb+\buf,-0.6*\y+\buf)--(\r+\r+\xb+\buf,-1.0*\y-\buf)--(\r-\buf,-1.0*\y-\buf)--cycle;

\shade[rounded corners,bottom color = blue!50, top color = blue!25] (\r-\buf,-1.6*\y+\buf)--(\r+\r+\xb+\buf,-1.6*\y+\buf)--(\r+\r+\xb+\buf,-2*\y-\buf)--(\r-\buf,-2*\y-\buf)--cycle;


\shade[rounded corners,bottom color = gray!20, top color = gray!10] (\r,0)--(\r+\xb,0)--(\r+\xb,\yb)--(\r,\yb)--cycle;
\node at (\r+0.5*\xb,0.7*\yb) {infinite LP};
\node at (\r+0.5*\xb,0.3*\yb) {$\eqref{primal-inf}$: \ $\Jp$};
\shade[rounded corners,bottom color = gray!20, top color = gray!10] (\r,-\y)--(\r+\xb,-\y)--(\r+\xb,\yb-\y)--(\r,\yb-\y)--cycle;
\node at (\r+0.5*\xb,0.7*\yb-\y) {robust program};
\node at (\r+0.5*\xb,0.3*\yb-\y) {$\eqref{primal-n}$: \ $\Jpn$};
\shade[rounded corners,bottom color = gray!20, top color = gray!10] (\r,-\y-\y)--(\r+\xb,-\y-\y)--(\r+\xb,\yb-\y-\y)--(\r,\yb-\y-\y)--cycle;
\node at (\r+0.5*\xb,0.7*\yb-\y-\y) {scenario program};
\node at (\r+0.5*\xb,0.3*\yb-\y-\y) {$\eqref{primal-n,N}$: \ $\JpnN$};

\shade[rounded corners,bottom color = gray!20, top color = gray!10] (\r+\r,-\y)--(\r+\r+\xb,-\y)--(\r+\r+\xb,\yb-\y)--(\r+\r,\yb-\y)--cycle;
\node at (\r+\r+0.5*\xb,0.7*\yb-\y) {semi-infinite program};
\node at (\r+\r+0.5*\xb,0.3*\yb-\y) {$\eqref{dual-n}$: \ $\Jdn$};

\shade[rounded corners,bottom color = gray!20, top color = gray!10] (\r+\r,-\y-\y)--(\r+\r+\xb,-\y-\y)--(\r+\r+\xb,\yb-\y-\y)--(\r+\r,\yb-\y-\y)--cycle;
\node at (\r+\r+0.5*\xb,0.7*\yb-\y-\y) {regularized program};
\node at (\r+\r+0.5*\xb,0.3*\yb-\y-\y) {$\eqref{dual-n-eta}$: \ $\Jdnr$};

\shade[rounded corners,bottom color = gray!20, top color = gray!10] (\r+0.5*\r,-\y-\y-\y)--(\r+0.5*\r+\xb,-\y-\y-\y)--(\r+0.5*\r+\xb,\yb-\y-\y-\y)--(\r+0.5*\r,\yb-\y-\y-\y)--cycle;
\node at (\r+0.5*\r+0.5*\xb,0.7*\yb-\y-\y-\y) {prior $\&$ posterior error};
\node at (\r+0.5*\r+0.5*\xb,0.3*\yb-\y-\y-\y) {$\Jdnr - \JpnN$};

    \draw[>=latex,<->,line width=1.2mm,gray!25] (\xb+\r,-\y+0.5*\yb) -- (\r+\r,-\y+0.5*\yb);
  \draw[>=latex,<->,dashed, line width=1.2mm,gray!25] (\r+0.5*\xb,0) -- (\r+0.5*\xb,-\y+\yb);
 \draw[>=latex,<->,dashed, line width=1.2mm,gray!25] (\r+0.5*\xb,-\y+0.25*\la) -- (\r+0.5*\xb,-\y+0.25*\la-\y+\yb);
 \draw[>=latex,<->,dashed, line width=1.2mm,gray!25] (\r+\r+0.5*\xb,-\y+0.25*\la) -- (\r+\r+0.5*\xb,-\y+0.25*\la-\y+\yb);
 
   \draw[>=latex,->,line width=1.2mm,gray!25] (\r+0.5*\xb,-\y-\y+0.2*\di) -- (\r+0.5*\r,\yb-\y-\y-\y);
   \draw[>=latex,->, line width=1.2mm,gray!25] (\r+\r+0.5*\xb,-\y-\y+0.2*\di) -- (\r+0.5*\r+\xb,\yb-\y-\y-\y);
 


  \shade[bottom color = darkgreen!50, top color = darkgreen!25] (0,-2.4*\y+1*\y)--(\di,-2.4*\y+1*\y)--(\di,-2.4*\y-\di+1*\y)--(0,-2.4*\y-\di+1*\y)--cycle;
  \node[] at (0.35*\r,-2.4*\y-0.5*\di+1*\y) {\  infinite program};

  \shade[bottom color = violet!50, top color = violet!25] (0,-2.7*\y+1*\y)--(\di,-2.7*\y+1*\y)--(\di,-2.7*\y-\di+1*\y)--(0,-2.7*\y-\di+1*\y)--cycle;
  \node[] at (0.35*\r,-2.7*\y-0.5*\di+1*\y) {\ \ \ \ \ \ \ \ \ semi-infinite programs};

  \shade[bottom color = blue!50, top color = blue!25] (0,-3.0*\y+1*\y)--(\di,-3.0*\y+1*\y)--(\di,-3.0*\y-\di+1*\y)--(0,-3.0*\y-\di+1*\y)--cycle;
  \node[] at (0.35*\r,-3.0*\y-0.5*\di+1*\y) {finite programs};

 \node[] at (0.5*\r+0.5*\xb+\r,-0.95*\y+0.5*\yb+1.3*\buf) {Proposition~$\ref{prop:SD}$};
 \node[] at (0.5*\r+0.5*\xb+\r,-\y+0.5*\yb -1.3*\buf) {}; 
 \node[] at (0.5*\r+0.5*\xb+\r,-1.05*\y+0.5*\yb -1.3*\buf) {strong duality};   
 \node[] at (0.19*\r+0.5*\xb+\r,-0.5*\y+0.5*\yb-0.0*\buf) {Theorem~$\ref{thm:inf-semi}$};
 \node[] at (0.19*\r+0.5*\xb+\r,-1.5*\y+0.5*\yb-0.0*\buf) {Theorem~$\ref{thm:semi-fin:rand}$};  
 \node[] at (1.19*\r+0.5*\xb+\r,-1.5*\y+0.5*\yb-0.0*\buf) {Theorem~$\ref{thm:semi-fin:smooth}$};     
  \node[] at (0.5*\r+0.5*\xb+\r,-2.6*\y+0.5*\yb -1.3*\buf) {Theorem~\ref{thm:inf-fin}};  
\end{tikzpicture}}
		\caption{Graphical representation of the chapter structure and its contributions}
		\label{fig:overview}
	\end{figure}
	
	\begin{enumerate}[label=$\bullet$, itemsep = 1mm, topsep = -1mm]
		\item \label{cont:1}
		We introduce a subclass of infinite LPs whose {\emph{regularized}} semi-infinite restriction enjoys analytical bounds for both primal and dual optimizers (Proposition~\ref{prop:SD}). 
		
		\item \label{cont:2}
		We derive an explicit error bound between the original infinite LP and the regularized semi-infinite counterpart, providing insights on the impact of the underlying norm structure as well as on how the choice of basis functions contributes to the approximation error (Theorem~\ref{thm:inf-semi}, Corollary~\ref{cor:Hilbert}). 
		
		\item \label{cont:3}
		We adopt recent developments from the randomized optimization literature to propose a finite convex program whose solution enjoys a priori probabilistic performance bounds (Theorem~\ref{thm:semi-fin:rand}). We extend the existing results to offer also an a posteriori bound under a generic underlying norm structure. 
		
		\item \label{cont:4}
		In parallel to the randomized approach, we also utilize recent developments in the structural convex optimization literature to propose an iterative algorithm for approximating the semi-infinite program. For this purpose, we extend the setting to incorporate unbounded prox-terms with a certain growth rate (Theorem~\ref{thm:semi-fin:smooth}). 
	\end{enumerate}

		\section{Semi-infinite approximation}\label{subsec:semi}
		
		Consider a family of linearly independent elements $\{x_n\}_{n \in \N} \subset \X$, and let $\X_n$ be the finite dimensional subspace generated by the first $n$ elements $\{x_i\}_{i\le n}$. Without loss of generality, we assume that $x_i$ are normalized, i.e., $\|x_i\| = 1$. Restricting the decision space $\X$ of \ref{primal-inf} to $\X_n$, along with an additional norm constraint, yields the program
		\begin{align} 
		\label{primal-n:old}
		\Jpn \Let \left\{ \begin{array}{ll}
		\Inf{\alpha \in \R^n} & \sum_{i = 1}^{n} \alpha_i \inner{x_i}{c}   \\
		\subjectto & \sum_{i = 1}^{n} \alpha_i \op x_i \geqc{\cone} b \vspace{1mm}\\
		& \|\alpha\|_{\Rnorm} \le \xnb \vspace{1mm}
		\end{array} \right.
		\end{align}
		where $\|\cdot\|_{\Rnorm}$ is a given norm on $\R^n$ and $\xnb$ determines the size of the feasible set. In the spirit of dual-paired normed vector spaces, one can approximate $(\X,\C,\|\cdot\|)$ by the finite dimensional counterpart $(\R^n,\R^n,\|\cdot\|_\Rnorm)$ where the bilinear form is the standard inner product. In this view, the linear operator $\op:\X \ra \B$ may also be approximated by the linear operator $\opn:\R^n \ra \B$ with the respective adjoint $\opn^*:\Y \ra \R^n$ defined as
		\begin{align}
		\label{Ln}
		\opn \alpha \Let \sum_{i=1}^{n} \alpha_i \op x_i, 
		\qquad 
		\opn^*y \Let \big[\inner{\op x_1}{y}, \cdots, \inner{\op x_n}{y} \big]\transp.
		\end{align}
		It is straightforward to verify the definitions \eqref{Ln} by noting that $\inner{\opn\alpha}{y} = \alpha \transp \opn^*y$ for all $\alpha\in\R^n$ and $y\in\Y$. 
		Defining the vector $\cnew \Let [\inner{x_1}{c}, \cdots, \inner{x_n}{c}]\transp$, we can rewrite the program \eqref{primal-n:old} as
		\begin{align} 
		\label{primal-n} 
		\tag{$\Prim_n$}	
		\Jpn \Let \left\{ \begin{array}{ll}
		\Inf{\alpha \in \R^n} & \alpha \transp \cnew   \\
		\subjectto & \opn \alpha \geqc{\cone} b \vspace{1mm}\\
		& \|\alpha\|_{\Rnorm} \le \xnb. 
		\end{array} \right.
		\end{align}
		We call \ref{primal-n} a \emph{semi-infinite} (or robust) program, as the decision variable is a finite dimensional vector $\alpha \in \R^n$, but the number of constraints is still in general infinite due to the conic inequality. The additional constraint on the norm of $\alpha$ in \ref{primal-n} acts as a \emph{regularizer} and is a key difference between the proposed approximation schemes and existing schemes in the literature. Methods for choosing the parameter $\xnb$ will be discussed later in Remark~\ref{rem:theta}.
		
		Dualizing the conic inequality constraint in \ref{primal-n} and using the dual norm definition leads to a dual counterpart
		\begin{align} 
		\label{dual-n} 
		\tag{$\Dual_n$}
		\Jdn \Let \left\{ \begin{array}{ll}
		\Sup{y \in \Y} & \inner{b}{y} - \xnb \|\opn^* y - \cnew\|_{\Rnorm^*}  \\
		\subjectto & y \in \cone^*, 
		\end{array} \right.
		\end{align}
		where$\|\cdot\|_{\Rnorm^*}$ denotes the dual norm of $\|\cdot\|_\Rnorm$. Note that setting $\xnb = \infty$ effectively implies that the second term of the objective in \ref{dual-n} introduces $n$ hard constraints $\opn^* y = \cnew$ (cf. \eqref{Ln}). We study further the connection between \ref{primal-n} and \ref{dual-n} under the following regularity assumption: 
		
		\begin{As}[Semi-infinite regularity]
			\label{a:reg}
			We stipulate that 
			\begin{enumerate}[label=(\roman*), itemsep = 1mm, topsep = -1mm]
				\item \label{a:reg:feas} the program \ref{primal-n} is feasible; 
				\item \label{a:reg:inf-sup} there exists a positive constant $\gamma$ such that $\|\opn^* y\|_{\Rnorm^*} \geq \gamma \|y\|_*$ for every $y \in \cone^*$, and $\xnb$ is large enough so that $\gamma \xnb > \|b\|$.
			\end{enumerate}
		\end{As}
		
		Assumption \ref{a:reg}\ref{a:reg:inf-sup} is closely related to the condition
		\begin{align*}
		\inf_{y \in \cone^*}\sup_{x \in \X_n} {\inner{\op x}{y} \over \|x\| \|y\|_*} \geq \gamma, 
		\end{align*}
		that in the literature of numerical algorithms in infinite dimensional spaces, in particular the Galerkin discretization methods for partial differential equations, is often referred to as the ``\emph{inf-sup}" condition, see \cite{ref:Ern-04} for a comprehensive survey. To see this, note that for every $x \in \X_n$ the definitions in \eqref{Ln} imply that
		$$\inner{\op x}{y} = \inner{\opn \alpha}{y} = \alpha \transp \opn^* y, \qquad x = \sum_{i=1}^{n} \alpha_i x_i.$$
		These conditions are in fact equivalent if the norm $\|\cdot\|_\Rnorm$ is induced by the original norm on $\X$, i.e., $\|\alpha\|_\Rnorm \Let \| \sum_{i=1}^{n} \alpha_i x_i\|$. We note that $\opn^*$ maps an infinite dimensional space to a finite dimensional one, and as such  Assumption~\ref{a:reg}\ref{a:reg:inf-sup} effectively necessitates that the null-space of $\opn^*$ intersects the positive cone $\cone^*$ only at $0$. In the following we show that this regularity condition leads to a zero duality gap between \ref{primal-n} and \ref{dual-n}, as well as an upper bound for the dual optimizers. The latter turns out to be a critical quantity for the performance bounds of this study. 
		
		\begin{Prop}[Duality gap \& bounded dual optimizers]
			\label{prop:SD}
			Under Assumption \ref{a:reg}\ref{a:reg:feas}, the duality gap between the programs \ref{primal-n} and \ref{dual-n} is zero, i.e., $\Jpn = \Jdn$. 
			If in addition Assumption \ref{a:reg}\ref{a:reg:inf-sup} holds, then for any optimizer $\yn$ of the program \ref{dual-n} and any lower bound $\Jlb \le \Jpn$ we have
			\begin{align}
			\label{yb}
			\|\yn\|_* \leq
			\ynb \Let {\xnb\|\cnew\|_{\Rnorm^*} - \Jlb \over \gamma \xnb - \|b\|} \le  {2\xnb\|\cnew\|_{\Rnorm^*} \over \gamma \xnb - \|b\|}.
			\end{align} 
		\end{Prop}
		
		\begin{proof}
			Since the elements $\{x_i\}_{i\le n}$ are linearly independent, the feasible set of the decision variable $\alpha$ in program \ref{primal-n} is a bounded closed subset of a finite dimensional space, and hence compact. Thus, thanks to the feasibility Assumption~\ref{a:reg}\ref{a:reg:feas} and compactness of the feasible set, the zero duality gap follows because
			\begin{align*}
			\Jpn = \inf_{\|\alpha\|_\Rnorm \le \xnb} \Big\{\alpha \transp \cnew + \sup_{y \in \cone^*} \inner{b - \opn \alpha}{y} \Big\} &=  \sup_{y \in \cone^*}\inf_{\|\alpha\|_\Rnorm \le \xnb}\Big\{ \inner{b}{y} - \alpha \transp (\opn^* y - \cnew)\Big\} = \Jdn,
			\end{align*}
			where the first equality holds by the definition of the dual cone $\cone^*$, and the second equality follows from Sion's minimax theorem \cite[Theorem 4.2]{ref:Sion-58}. Thanks to the zero duality gap above, we have 
			\begin{align*}
			\Jlb \le \Jpn = \Jdn & = \inner{b}{\yn} - \xnb\|\opn^* \yn - \cnew\|_{\Rnorm^*} \le \inner{b}{\yn}  - \xnb\|\opn^* \yn\|_{\Rnorm^*} + \xnb \|\cnew\|_{\Rnorm^*}.
			\end{align*}
			By Assumption \ref{a:reg}\ref{a:reg:inf-sup}, we then have
			\begin{align*}
			\Jpn & \le \|b\|\|{\yn}\|_* - \gamma \xnb\|\yn\|_* + \xnb\|\cnew\|_{\Rnorm^*} = \xnb\|\cnew\|_{\Rnorm^*} - \big(\gamma \xnb - \|b\|\big)\|\yn\|_*,
			\end{align*}
			which together with the simple lower bound $\Jlb \Let -\xnb\|\cnew\|_{\Rnorm^*} \le \Jpn$ concludes the proof. 
		\end{proof}
		
		Proposition~\ref{prop:SD} effectively implies that in the program \ref{dual-n} one can add a norm constraint $\|y\|_* \le \ynb$ without changing the optimal value. The parameter $\ynb$ depends on $\Jlb$, a lower bound for the optimal value of $\Jpn$. A simple choice for such a lower bound is $-\xnb\|\cnew\|_{\Rnorm^*}$, but in particular problem instances one may be able to obtain a less conservative bound. We validate the assertions of Proposition~\ref{prop:SD} for long-run average cost problems in Section~\ref{subsec:semi:MDP} and for long-run discounted cost problems in Section~\ref{sec:discounted:setting}.
		
		Program~\ref{primal-n} is a restricted version of the original program \ref{primal-inf} (also called an \emph{inner approximation} \cite[Definition 12.2.13]{ref:Hernandez-99}), and thus $\Jp \le \Jpn$. 
		However, under Assumption \ref{a:reg}, we show that the gap $\Jpn-\Jp$ can be quantified explicitly.
		To this end, we consider the projection mapping $\proj_{\set A}(x) \Let \arg\min_{x' \in \set A} \| x' - x\|$, the operator norm $\|{\op}\| \Let \sup_{\|x\| \le 1} \|{\op x}\|$, and define the set 
		\begin{align}
		\label{uball}
		\uball \Let \Big\{\sum_{i=1}^{n} \alpha_i x_i \in \X_n ~:~ \|\alpha\|_\Rnorm \le \xnb\Big\}.
		\end{align}
		
		\begin{Thm}[Semi-infinite approximation]
			\label{thm:inf-semi}
			Let $x\opt$ and $\yn$ be optimizers for the programs \ref{primal-inf} and \ref{dual-n}, respectively, and let $r_n \Let x\opt - \proj_{\uball}(x\opt)$ be the projection residual of the optimizer $x\opt$ onto  the set $\uball$ as defined in \eqref{uball}. Under Assumption~\ref{a:reg}\ref{a:reg:feas}, we have $0\leq \Jpn - \Jp \le \inner{r_n}{\op^*\yn - c}$ where $\Jpn$ and $J$ are the optimal value of the programs \ref{primal-n} and \ref{primal-inf}. In addition, if Assumption~\ref{a:reg}\ref{a:reg:inf-sup} holds, then 
			\begin{align}
			\label{inf-semi error}
			0 \le \Jpn - \Jp \le \big(\|c\|_*+\ynb\|\op\| \big)\|r_n\|,
			\end{align}
			where $\ynb$ is the dual optimizer bound introduced in \eqref{yb}.
		\end{Thm}
		
		\begin{proof}
			The lower bound $0\le \Jpn - \Jp$ is trivial, and we only need to prove the upper bound. Note that since the optimizer $x\opt \in \X$ is a feasible solution of \ref{primal-inf}, then $\op x\opt - b \in \cone$. By the definition of the dual cone $\cone^*$, this implies that $\inner{\op x\opt - b}{y} \ge 0$ for all $y \in \cone^*$. Since the dual optimizer $\yn$ belongs to the dual cone $\cone^*$, then 
			\begin{align*}
			\Jpn - \Jp &\le \Jpn - \Jp + \inner{\op x\opt - b}{\yn} = \Jpn - \inner{x\opt}{c} + \inner{\op x\opt}{\yn} - \inner{b}{\yn} \\
			& = \Jpn + \inner{x\opt}{\op^*\yn - c} - \inner{b}{\yn} \\
			& = \Jpn +  \inner{r_{n}}{\op^*\yn - c} + \inner{\proj_{\uball}(x\opt)}{\op^*\yn - c} - \inner{b}{\yn},\\
			& = \Jpn +  \inner{r_{n}}{\op^*\yn - c} + \wt{\alpha}\transp\big({\opn^*\yn - \cnew}\big) - \inner{b}{\yn},
			\end{align*}
			for some $\wt\alpha \in \R^n$ with norm $\|\wt\alpha\|_\Rnorm \le \xnb$; for the last line, see the definition of the operator $\opn$ in \eqref{Ln} as well as the vector $\cnew$ in the program \ref{primal-n}. Using the definition of the dual norm and the operators \eqref{Ln}, one can deduce from above that
			\begin{align*}
			\Jpn - \Jp & \le J_n + \inner{r_{n}}{\op^* \yn - c} + \xnb \|\opn^* \yn - \cnew\|_{\Rnorm^*} - \inner{b}{\yn} = J_n + \inner{r_{n}}{\op^* \yn - c} - \Jdn,
			\end{align*}
			which in conjunction with the zero duality gap ($\Jpn = \Jdn$) establishes the first assertion of the proposition. The second assertion is simply the consequence of the first part and the norm definitions, i.e.,
			\begin{align*}
			\inner{r_n}{\op^* \yn- c} = \inner{r_n}{-c} + \inner{\op r_n}{\yn} \le  \|r_n\|\|{c}\|_* + \|\op r_n\|\| \yn\|_{*} \le \|r_n\| \Big( \|c\|_{*} + \|\op\| \|\yn\|_* \Big).
			\end{align*}
			Invoking the bound on the dual optimizer $\yn$ from Proposition \ref{prop:SD} completes the proof. 
		\end{proof}		
		
		\begin{Rem}[Impact of norms on semi-infinite approximation]\label{rem:norm-semi} 
			We note the following concerning the impact of the choice of norms on the approximation error:
			\begin{enumerate} [label=(\roman*), itemsep = 1mm, topsep = -1mm]
				\item \label{rem:norm-semi:theta}
				The only norm that influences the semi-infinite program \ref{primal-n} is $\|\cdot\|_\Rnorm$ on $\R^n$. When it comes to the approximation error \eqref{inf-semi error}, the norm $\|\cdot\|_\Rnorm$ may have an impact on the residual $r_n$ only if the set $\uball$ in \eqref{uball} does not contain $\proj_{\X_n}(x\opt)$, the projection $x\opt$ on the subspace $\X_n$, where $x\opt$ is an optimizer of the infinite  program~\ref{primal-inf}. 	
				
				\item \label{rem:norm-semi:dual-pair}
				The norms of the dual pairs of vector spaces only appear in Theorem~\ref{thm:inf-semi} to quantify the approximation error. Note that in \eqref{inf-semi error} the stronger the norm on $\X$, the higher $\|r_n\|$, and the lower $\|c\|_*$ and $\|\op\|$. On the other hand, the stronger the norm on $\B$, the higher $\|b\|$ and $\|\op\|$ and the lower $\gamma$ (cf. Assumption~\ref{a:reg}\ref{a:reg:inf-sup}). 
			\end{enumerate}
		\end{Rem}				
		
		The error bound \eqref{inf-semi error} can be further improved when $\X$ is a Hilbert space. In this case, let $\comp{\X}_n$ denote the orthogonal complement of $\X_n$. We define the \emph{restricted} norms by
		\begin{align}
		\label{norm-res}
		\|c\|_{*n} \Let \sup_{x \in \comp{\X}_n} {\inner{x}{c} \over {\|{x}\|}}, \qquad \|\op\|_{n} \Let \sup_{x \in \comp{\X}_{n}} \frac{\|{\op x}\|}{\|{x}\|}. 
		\end{align}
		It is straightforward to see that by definition $\|c\|_{*n} \le \|c\|_{*}$ and $\|\op\|_{n} \le \|\op\|$.
		
		\begin{Cor}[Hilbert structure]\label{cor:Hilbert}
			Suppose that $\X$ is a Hilbert space and $\|\cdot\|$ is the norm induced by the corresponding inner product. Let $\{x_i\}_{i\in\N}$ be an orthonormal dense family and $\|\cdot\|_\Rnorm = \|\cdot\|_{\ell_2}$. Let $x\opt$ be an optimal solution for \ref{primal-inf} and chose $\xnb \ge \|x\opt\|$. Under the assymptions of Theorem~\ref{thm:inf-semi}, we have 
			\begin{align*}
			0 \le \Jpn - \Jp \le \big(\|c\|_{n}+\ynb\|\op\|_n \big)\big\|\proj_{\comp{\X}_{n}}(x\opt) \big\|.
			\end{align*}
		\end{Cor}
		
		\begin{proof}
			We first note that the $\ell_2$-norm on $\R^n$ is indeed the norm induced by $\|\cdot\|$, since due to the orthonormality of $\{x_i\}_{i\in\N}$ we have
			$$\|\alpha\|_\Rnorm \Let \Big\|\sum_{i=1}^{n} \alpha_i x_i \Big\| = \sqrt{\sum_{i=1}^{n} \alpha_i^2\|x_i\|^2} = \|\alpha\|_{\ell_2}.$$ 
			If $\xnb \ge \|x\opt\|$, then $\proj_{\uball}(x\opt) = \proj_{\X_n}(x\opt)$, i.e., the projection of the optimizer $x\opt$ on the ball $\uball$ is in fact the projection onto the subspace $\X_{n}$. Therefore, thanks to the orthonormality, the projection residual $r_n = x\opt - \proj_{\X_{n}}(x\opt)$ belongs to the orthogonal complement $\comp{\X}_{n}$. Thus, following the same reasoning as in the proof of Theorem~\ref{thm:inf-semi}, one arrives at a bound similar to \eqref{inf-semi error} but using the restricted norms \eqref{norm-res}; recall that the norm in a Hilbert space is self-dual. 
		\end{proof}

		\section{Semi-infinite to finite programs: randomized approach} 
		\label{sec:semi-fin:rand}
		We study conditions under which one can provide a finite approximation to the semi-infinite programs of the form \ref{primal-n}, that are in general known to be computationally intractable --- NP-hard \cite[p.~16]{ref:BenTal-09}. We approach this goal by deploying tools from two areas, leading to different theoretical guarantees for the proposed solutions. This section focuses on a randomized approach and the next section is dedicated to an iterative gradient-based descent method. The solution of each of these methods comes with a~priori as well as a posteriori performance certificates.

		We start with a lemma suggesting a simple bound on the norm of the operator $\opn$ in \eqref{Ln}. We will use the bound to quantify the approximation error of our proposed solutions. 
		
		\begin{Lem}[Semi-infinite operator norm]\label{lem:opn}
			Consider the operator $\opn:\R^n \ra \B$ as defined in \eqref{Ln}. Then,
			\begin{align}\label{opt}
			\|\opn\| \Let \sup_{\alpha \in \R^n} {\|\opn \alpha\| \over \|\alpha\|_\Rnorm} \le \|\op\|\ratio, \qquad  \ratio \Let \sup_{\|\alpha\|_\Rnorm \le 1} \|\alpha\|_{\ell_1},
			\end{align}
			where the constant $\ratio$ is the equivalence ratio between the norms $\|\cdot\|_\Rnorm$ and $\|\cdot\|_{\ell_1}$.\footnote{The constant $\ratio$ is indexed by $n$ as it potentially depends on the dimension of $\alpha \in \R^n$. } 
		\end{Lem}
		
		\begin{proof}
			The proof follows directly from the definition of the operator norm, that is, 
			$$\|\opn \alpha\| = \Big\|\sum_{i=1}^{n}\alpha_i \op x_i\Big\| \le \|\op\|\Big\|\sum_{i=1}^{n}\alpha_i x_i \Big\|,$$
			together with the inequality $\big\|\sum_{i=1}^{n}\alpha_i x_i \big \| \le \|\alpha\|_{\ell_1} \max_{i \le n} \|x_i\| = \|\alpha\|_{\ell_1}$, which concludes the proof. 
		\end{proof}
		
		Since $\cone$ is a closed convex cone, then $\cone^{**} = \cone$ \cite[p.\ 40]{ref:Anderson-87}, and as such the conic constraint in program \ref{primal-n} can be reformulated as 
		\begin{align}
		\label{conic-const}
		\opn \alpha \geqc{\cone} b \qquad \Llra \qquad \inner{\opn \alpha - b}{y} \ge 0, \quad \forall y \in \Kb \Let \ext\{y \in \cone^* : \|y\|_* = 1 \}, 
		\end{align}
		where $\ext\{B\}$ denotes the extreme points of the set $B$, i.e., the set of points that cannot be represented as a strict convex combination of some other elements of the set. Notice that the norm constraint as well as the restriction to the extreme points in the definition of $\Kb$ in \eqref{conic-const} does not sacrifice any generality, as conic constraints are homogeneous. These restrictions are introduced to improve the approximation errors. In what follows, however, one can safely replace the set $\Kb$ with any subset of the cone $\cone^*$ whose closure contains $\Kb$. This adjustment may be taken into consideration for computational advantages. Let $\PP$ be a Borel probability measure supported on $\Kb$, and $\{y_j\}_{j\le N}$ be independent, identically distributed (i.i.d.) samples generated from $\PP$. Consider the \emph{scenario} counterpart of the program \ref{primal-n} defined as
		\begin{align} 
		\label{primal-n,N} 
		\tag{$\Prim_{n,N}$}
		\JpnN \Let 
		\left\{ \begin{array}{ll}
		\Min{\alpha\in\R^n} & \alpha \transp \cnew   \\
		\subjectto & \alpha \transp \opn^* y_j \ge \inner{b}{y_j},  \quad  j \in \{1,\cdots,N\}\\
		& \|\alpha\|_{\Rnorm} \le \xnb,
		\end{array} \right.
		\end{align}
		where the adjoint operator $\opn^*:\B \ra \R^n$ is introduced in \eqref{Ln}. The optimization problem \ref{primal-n,N} is a standard finite convex program, and thus computationally tractable whenever the norm constraint $\|\alpha\|_\Rnorm \le \xnb$ is tractable. Program \ref{primal-n,N} is a relaxation of \ref{primal-n}, i.e., $ \Jpn \ge \JpnN$; note that $\JpnN$ is a random variable, therefore the relaxation error $\Jpn - \JpnN$ can only be interpreted in a probabilistic sense. 
		
		\begin{Def}[Tail bound]
			\label{def:tail}
			Given a probability measure $\PP$ supported on $\Kb$, we define the function $p:\R^n \times \R_+ \ra [0,1]$ as
			\begin{align*}
			p(\alpha,\zeta) \Let \PP\Big[y ~ : ~  \supp{\Kb}{-\opn \alpha + b}<\inner{-\opn \alpha + b}{y} + \zeta \Big],
			\end{align*}
			where $\sigma_{\Kb}(\cdot) \Let \sup_{y\in\Kb}\inner{\cdot}{y}$ is the support function of $\Kb$. We call $h:\R^n\times [0,1] \ra \R_+$ a \emph{tail bound} (TB) of the program \ref{primal-n,N}, if for all $\eps \in [0,1]$ and $\alpha$ we have	
			\begin{align*}
			h(\alpha,\eps) \ge \sup \big \{ \zeta ~:~ p(\alpha,\zeta) \le \eps \big \}.
			\end{align*}
		\end{Def}		
		
		The TB function in Definition~\ref{def:tail} can be interpreted as a {\em shifted} quantile function of the mapping $y \mapsto \inner{-\opn \alpha + b}{y}$ on $\Kb$--- the ``shift" is referred to the maximum value of the mapping which is $\supp{\Kb}{-\opn \alpha + b}$. 		
		TB functions depend on the probability measure $\PP$ generating the scenarios $\{y_j\}_{j\le N}$ in the program \ref{primal-n,N}, as well as the properties of the optimization problem. Definition \ref{def:tail} is rather abstract and not readily applicable. The following example suggests a more explicit, but not necessarily optimal, candidate for a TB.
		
		\begin{Ex}[TB candidate] \label{ex:TB}
			Let $g:\R_{+}\to [0,1]$ be a non-decreasing function such that for any $\kappa \in \Kb$ we have $g(\gamma) \le \PP\big[\ball{\kappa}{\gamma} \big]$, where $\ball{\kappa}{\gamma}$ is the open ball centered at $\kappa$ with radius $\gamma$; note that function $g$ depends on the choice of the norm on $\Y$. Then, a candidate for a TB function of the program \ref{primal-n,N} is 
			$$h(\alpha,\eps) \Let \| \opn \alpha - b \| g^{-1}(\eps) \le \big(\ratio \|\op\| \|\alpha\|_\Rnorm + \|b\|\big)g^{-1}(\eps),$$ 
			where the inverse function is understood as $g^{-1}(\eps) \Let \sup\{\gamma \in~\R_+~: g(\gamma) \le \eps\}$, and $\ratio$ is the constant ratio defined in \eqref{opt}. 
			To see this note that according to Definition~\ref{def:tail} we have 
			\begin{align*}
			p(\alpha,\zeta) 	&= \PP\Big[y ~ : ~  \sup_{\kappa\in\Kb}\inner{-\opn \alpha + b}{\kappa -y}<\zeta \Big] \\
			&= \inf_{\kappa\in\Kb}\PP\Big[y ~ : ~  \inner{-\opn \alpha + b}{\kappa -y}<\zeta \Big] \\
			&\geq \inf_{\kappa\in\Kb} \PP\Big[ y ~ : ~ \| \opn \alpha - b \| \| y-\kappa \|_{*} <\zeta \Big] \\
			&= \inf_{\kappa\in\Kb} \PP\Big[ \ball{\kappa}{\gamma(\zeta)} \Big] \geq g(\gamma(\zeta)), \quad \gamma(\zeta) \Let {\zeta \| \opn \alpha- b \|^{-1}}.
			\end{align*}
			Thus, if $p(\alpha,\zeta)\leq \eps$, then $g(\gamma(\zeta))\leq \eps$ and by construction of the inverse function $g^{-1}$ we have $\zeta \|\opn \alpha - b\|^{-1} \leq g^{-1}(\eps).$ In view of Definition \ref{def:tail}, this observation readily suggests that the function $h(\alpha,\varepsilon) \Let \| \opn \alpha - b \| g^{-1}(\eps)$ is indeed a TB candidate, and the suggested upper bound follows readily from Lemma~\ref{lem:opn}.
		\end{Ex}

		\begin{Thm}[Randomized approximation error]
			\label{thm:semi-fin:rand}
			Consider the programs \ref{primal-n} and \ref{primal-n,N} with the associated optimum values $\Jpn$ and $\JpnN$, respectively. Let Assumption \ref{a:reg} hold, $\alpha\opt_N$ be the optimizer of the program \ref{primal-n,N}, and the function $h$ be a TB as in Definition \ref{def:tail}. Given $\eps, \beta$ in $(0,1)$, we define 
			\begin{align} 
			\label{N}
			\NN(n, \eps, \beta) \Let \min \Big\{ N\in \N ~:~ \sum_{i=0}^{n-1}  {N \choose i} \eps^{i}(1-\eps)^{N-i}\leq\beta \Big\}.
			\end{align}
			For all positive parameters $\eps, \beta$ and $N \ge \NN(n,\eps,\beta)$ we have
			\begin{subequations}
				\label{eq:thm:semi-fin:rand}	
				\begin{align}
				\label{eq:thm:semi-fin:rand:1} 
				\PP^N \bigg[0\leq \Jpn - \JpnN \le \ynb h\big(\alpha\opt_N,\eps \big) \bigg] \ge 1-\beta,
				\end{align}
				where the constant $\ynb$ is defined as in \eqref{yb}. In particular, suppose the function $h$ is the TB candidate from Example \ref{ex:TB} with corresponding $g$ function, and 
				\begin{align}
				\label{NN}
				N \ge \NN\big(n,g(z_n\eps),\beta\big), \qquad z_n \Let \Big(\ynb\big(\xnb \ratio\|\op\| + \|b\|\big)\Big)^{-1}
				\end{align} 
				where $\ratio$ is the ratio constant defined in Lemma~\ref{lem:opn}. We then have 
				\begin{align}
				\label{eq:thm:semi-fin:rand:2} 	
				\PP^N \Big[0\leq \Jpn - \JpnN  \le  \eps \Big] \ge 1-\beta \,.
				\end{align}
			\end{subequations}
		\end{Thm}	
		
		Theorem~\ref{thm:semi-fin:rand} extends the result \cite[Theorem 3.6]{ref:MohSut-13} in two respects: 
		\begin{enumerate} [label=$\bullet$, itemsep = 1mm, topsep = -1mm]
			\item The bounds \eqref{eq:thm:semi-fin:rand} are described in terms of a generic norm and the corresponding dual optimizer bound.
			\item Through the optimizer of \ref{primal-n,N}, the bounds involve an a posteriori element (cf. \eqref{eq:thm:semi-fin:rand:1} to \eqref{eq:thm:semi-fin:rand:2}). 
		\end{enumerate}
		Before proceeding with the proof, we first remark on the complexity of the a~priori bound of Theorem~\ref{thm:semi-fin:rand}, its implications for an appropriate choice of $\xnb$, and its dependence on the dual pair norms. 
		
		\begin{Rem}[Curse of dimensionality]
			\label{rem:curse}
			The TB function $h$ of Example~\ref{ex:TB} may grow exponentially in the dimension of the support set $\Kb$ (i.e., $h(\alpha,\eps) \propto \eps^{-\dim(\Kb)}$). Since $\NN(n,\cdot,\beta)$ admits a linear growth rate, the a~priori bound \eqref{eq:thm:semi-fin:rand:2} effectively leads to an exponential number of samples in the precision level $\eps$, an observation related to the curse of dimensionality \cite[Remark~3.9]{ref:MohSut-13}. To mitigate this inherent computational complexity, one may resort to a more elegant sampling approach so that the required number of samples $\NN$ has a sublinear rate in the second argument, see for instance \cite{ref:NemShap-06}. 
		\end{Rem}
		
		\begin{Rem}[Optimal choice of $\xnb$]
			\label{rem:theta}
			In view of the a priori error in Theorem~\ref{thm:semi-fin:rand}, the parameter $\xnb$ may be chosen so as to minimize the required number of samples. To this end, it suffices to maximize $z_n$ defined in \eqref{NN} over all $\xnb> \|b\|\gamma^{-1}$, see Assumption~\ref{a:reg}\ref{a:reg:inf-sup}, where $\ynb$ is defined in \eqref{yb}. One can show that the optimal choice in this respect is analytically available as 
			\begin{align*}
			\xnb\opt \Let {\|b\| \over \gamma} + \sqrt{\Big({\|b\| \over \gamma} + {\|b\| \over \ratio\|\op\|} \Big)\Big({\|b\| \over \gamma}- {\Jlb \over \|\cnew\|_{\Rnorm*}} \Big)} \, ,
			\end{align*}
			where $\Jlb$ is a lower bound on the optimal value of \ref{primal-n} used in \eqref{yb}.
		\end{Rem}
		
		\begin{Rem}[Norm impact on finite approximation]
			\label{rem:norm-fin}
			Besides to what has already been highlighted in Remark~\ref{rem:norm-semi}, the choice of norms in the dual pairs of normed vector spaces also has an impact on the function $g^{-1}(\eps)$. More specifically, the stronger the norm in the space $\B$, the larger the balls in the dual space $\Y$, and thus the smaller the function $g^{-1}$. 
		\end{Rem}
		
		To prove Theorem~\ref{thm:semi-fin:rand} we need a few preparatory results.
		
		\begin{Lem}[Perturbation function]
			\label{lem:Lemma0}
			Given $\delta\in\B$, consider the $\delta$-\emph{perturbed} program of \ref{primal-n} defined as
			\begin{align} 
			\label{primal-n-pert} 
			\tag{$\Prim_n(\delta)$}
			\Jpnd \Let \left\{ \begin{array}{ll}
			\Inf{\alpha\in\R^n} & \alpha \transp \cnew   \\
			\subjectto & \opn \alpha \geqc{\cone} b - \delta \\
			& \|\alpha\|_\Rnorm \le \xnb.
			\end{array} \right.
			\end{align}
			Under Assumption \ref{a:reg}, we then have $\Jpn - \Jpnd \leq \inner{\delta}{\yn}$, where $\yn$ is an optimizer of \ref{dual-n}.
		\end{Lem}
		\begin{proof}
			For the proof we first introduce the dual program of \ref{primal-n-pert}:
			\begin{align} 
			\label{dual-n-pert} 
			\tag{$\Dual_n(\delta)$}
			\Jdnd \Let \left\{ \begin{array}{ll}
			\Sup{y} & \inner{b-\delta}{y} - \xnb\|\opn^* y-\cnew\|_{\Rnorm^*}  \\
			\subjectto & y\in\cone^{*}.
			\end{array} \right.
			\end{align}
	We then have 
			\begin{align*}
			\Jpn - \Jpnd & = \Jdn - \Jpnd= \inner{b}{\yn} - \xnb\|\opn^* \yn-\cnew\|_{\Rnorm^*}  - \Jpnd \\ 
			&= \inner{\delta}{\yn} + \inner{b-\delta}{\yn}- \xnb\|\opn^* \yn-\cnew\|_{\Rnorm^*}  - \Jpnd \\
			&\leq \inner{\delta}{\yn} + \Jdnd - \Jpnd \leq \inner{\delta}{\yn},
			\end{align*}
			where the first line follows from the strong duality (gap-free) between \ref{primal-n} and \ref{dual-n} by Proposition \ref{prop:SD}. The third line is due to the fact that $\yn$ is a feasible solution of \ref{dual-n-pert}, and the last line follows from weak duality between \ref{primal-n-pert} and \ref{dual-n-pert}.
		\end{proof}
						
		\begin{Lem}[Perturbation error] 
			\label{lem:Lemma1}
			Let $\alpha\opt_N$ be an optimal solution of \ref{primal-n,N} and assume that $\delta\in\B$ satisfies the conic inequality $ \opn \alpha\opt_N \geqc{\cone} b - \delta$. Then, under Assumption \ref{a:reg}, we have $0\leq \Jpn - \JpnN \leq \inner{\delta}{\yn}$.
			\end{Lem}
		\begin{proof}
			The lower bound on $\Jpn -\JpnN$ is trivial since \ref{primal-n,N} is a relaxation of \ref{primal-n}. For the upper bound the requirement on $\delta$ in the program \ref{primal-n-pert} implies that $\alpha\opt_N$ is a feasible solution of \ref{primal-n-pert}. We then have $\JpnN \geq \Jpnd$, and thus $0\leq \Jpn - \JpnN \leq \Jpn - \Jpnd$. Applying Lemma~\ref{lem:Lemma0} completes the proof.
		\end{proof}
		
		The following fact follows readily from Definition \ref{def:tail}. 
		
		\begin{Lem}[TB lower bound] \label{lem:ccp}
			If $\alpha \in \R^n$ satisfies $\PP\left[ y ~ : ~ \inner{\opn \alpha-b}{y} < 0 \right] \leq \eps$, then for any TB function in the sense of Definition~\ref{def:tail} we have $\supp{\Kb}{-\opn \alpha + b}\leq h(\alpha,\eps)$.
		\end{Lem}
		
		\begin{proof}
			By the definition of the support function we can equivalently write 
			\begin{align*}
			p(\alpha,\zeta)= \PP \left[ y ~ : ~ \inner{\opn \alpha - b }{y} < \zeta - \supp{\Kb}{-\opn \alpha + b} \right].
			\end{align*}
			Now setting $\zeta = \supp{\Kb}{-\opn \alpha + b}$ in the above relation together with the assumption of Lemma~\ref{lem:ccp} yields $p(\alpha,\zeta) \le \eps$, which in light of a TB in Definition \ref{def:tail} suggests that $\supp{\Kb}{-\opn \alpha + b}\leq h(\alpha,\eps)$. 
		\end{proof}

		We follow our discussion with a result from randomized optimization in a convex setting.

		\begin{Thm}[Finite-sample probabilistic feasibility {\cite[Theorem 1]{ref:CamGar-08}}] \label{thm:Campi}
			Assume that the program \ref{primal-n,N} admits a unique minimizer $\alpha\opt_N$.\footnote{The uniqueness assumption may be relaxed at the expense of solving an auxiliary convex program, see \cite[Section 3.3]{ref:MohSut-13}.} If $N \ge \NN(n,\eps,\beta)$ as defined in \eqref{N}, then with confidence at least $1 - \beta$ (across multi-scenarios $\{y_j\}_{j\le N}\subset \Kb$) we have $\PP \big[ y ~ : ~ \inner{\opn\alpha_N - b}{y} < 0  \big] \leq \eps$.
		\end{Thm}
		
		We are now in a position to prove Theorem~\ref{thm:semi-fin:rand}.
		
		\begin{proof}[Proof of Theorem~\ref{thm:semi-fin:rand}]
			By definition of the support function we know that $\supp{\Kb}{\delta} = \supp{\conv(\Kb)}{\delta}$ where $\conv(\Kb)$ is the convex hull of $\Kb$. Recall that by definition of the set $\Kb$ in \eqref{conic-const}, we also have $y/\|y\|_* \in \conv(\Kb)$ for any $y \in \cone^*$. Thus, for any $\delta \in \B$ and $y \in \cone^*$ we have $\inner{\delta}{y} \le \|y\|_* \supp{\Kb}{\delta}$. This leads to 
			\begin{align*}
			0\leq \Jpn-\JpnN \leq \inner{-\opn\alpha\opt_N + b}{\yn} \leq \|\yn\|_{*} \supp{\Kb}{-\opn\alpha\opt_N + b} 
			\end{align*}
			where the second inequality is due to Lemma~\ref{lem:Lemma1} as $\delta=-\opn\alpha\opt_N + b$ clearly satisfies the requirements. By  Lemma~\ref{lem:ccp} and Theorem~\ref{thm:Campi}, we know that with probability at least $1-\beta$ we have $\supp{\Kb}{-\opn\alpha_N + b} \le h(\alpha_N,\eps)$, which in conjunction with the dual optimizer bound in Proposition~\ref{prop:SD} results in \eqref{eq:thm:semi-fin:rand:1}. Now using the TB candidate in Example \ref{ex:TB} immediately leads to the first assertion of  \eqref{eq:thm:semi-fin:rand:2}. Recall that the solution \ref{primal-n,N} obeys the norm bound $\|\alpha\opt_N\|_\Rnorm \le \xnb$. Thus, by employing the triangle inequality together with Lemma~\ref{lem:opn} we arrive at the second assertion \eqref{eq:thm:semi-fin:rand:2}. 
		\end{proof}
		
		Theorem~\ref{thm:semi-fin:rand} quantifies the approximation error between programs \ref{primal-n} and \ref{primal-n,N} probabilistically in terms of the TB functions as introduced in Definition \ref{def:tail}. The natural question is under what conditions can the proposed bound be made arbitrarily small. This question is intimately related to the behavior of TB functions. For the TB candidate proposed in Example \ref{ex:TB}, the question translates to when does the measure of a ball $\ball{\kappa}{\gamma} \subset \Kb$ have a lower bound $g(\gamma)$ uniformly away from $0$ with respect to the location of its center: The answer to this question also depends on the properties of the norm on $(\B,\Y,\|\cdot\|)$. A positive answer to this question requires that the set $\Kb$ can be covered by finitely many balls, indicating that $\Kb$ is indeed compact with respect to the (dual) norm topology. In Section~\ref{subsec:rand:MDP} we study this requirement in more detail in the MDP setting.

		\section{Semi-infinite to finite programs: structural convex optimization} 
		\label{sec:semi-fin:smoothing}
		
		This section approaches the approximation of the semi-infinite program \ref{primal-n} from an alternative perspective relying on an iterative first order descent method. As opposed to the scenario approach presented in Section~\ref{sec:semi-fin:rand}, that is probabilistic and starts from the program \ref{primal-n}, the method of this section is deterministic and starts with the dual counterpart \ref{dual-n}, in particular a {\em regularized} version of whose solutions can be computed efficiently, see Figure~\ref{fig:overview} for a pictorial overview. It turns out that the regularized solution allows one to reconstruct a nearly optimal solution for both programs \ref{primal-n} and \ref{dual-n}, offering a meaningful performance bound for the approximation step from the semi-infinite program to a finite program.

		The basis of our approach is the fast gradient method that significantly improves the theoretical and, in many cases, also the practical convergence speed of the gradient method. The main idea is based on a well known technique of smoothing nonsmooth functions \cite{ref:Nest-05}. To simplify the notation, for a given $\xnb$ we define the sets
		\begin{align*}
		\Ab \Let \big\{ \alpha\in\R^{n} :  \| \alpha \|_\Rnorm \leq \xnb \big\}, \qquad
		\Yb \Let \bigg\{ y\in\cone^{*} : \|y\|_* \leq \ynb \bigg\},
		\end{align*}
		where $\ynb$ is the constant defined in \eqref{yb}. Recall that due to Proposition~\ref{prop:SD} we know that the decision variables of the dual program \ref{dual-n} may be restricted to the set $\Yb$ without loss of generality. We modify the program \ref{dual-n} with a regularization term scaled with the non-negative parameter $\eta$ and define the \emph{regularized} program
		\begin{align}  \label{dual-n-eta}
		\tag{$\Dual_{n,\eta}$}
		\Jneta \Let \Sup{y\in\Yb} 	\Big\{  \inner{b}{y} - \xnb\|\opn^* y-\cnew\|_{\Rnorm^*}   - \eta d(y) \Big\}, 		
		\end{align}
		where the regularization function $d: \Yb \ra \R_{+}$, also known as the {\em prox-function}, is strongly convex. The choice of the prox-function depends on the specific problem structure and may have significant impact on the approximation errors. Given the regularization term $\eta$ and the parameter $\alpha \in \R^n$, we introduce the auxiliary quantity
		\begin{align} \label{yeta}
		\yeta(\alpha)\Let \arg\max_{y\in\Yb} \Big\{  \inner{b-\opn \alpha }{y}  - \eta d(y) \Big\}.
		\end{align}
		It is computationally crucial for the solution method proposed in this part that the prox-function allows us to have access to the auxiliary variable $\yeta(\alpha)$ for each $\alpha \in \R^n$. This requirement is formalized as follows.
		
		\begin{As}[Lipschitz gradient]
			\label{a:yeta}
			Consider the adjoint operator $\opn^*$ in \eqref{Ln} and the optimizer $\yeta(\alpha)$ of the auxiliary quantity \eqref{yeta}. We assume that for each $\alpha \in \Ab$ the vector $\opn^*\yeta(\alpha) \in \R^n$ can be approximated to an arbitrary precision, and the mapping $\alpha \mapsto \opn^*\yeta(\alpha)$ is Lipschitz continuous with a constant $\tfrac{L}{\eta}$, i.e.,  
			\begin{align*}
			\|\opn^*\yeta(\alpha) - \opn^*\yeta(\alpha')\|_{\Rnorm^*} \le {L \over \eta} \|\alpha - \alpha'\|_\Rnorm, \qquad \forall  \alpha, \alpha' \in \Ab. 
			\end{align*}
		\end{As}
		
		Let $\vartheta > 0 $ be the strong convexity parameter of the mapping $\alpha \mapsto \tfrac{1}{2}\|\alpha\|_\Rnorm^2$ with respect to the $\Rnorm$-norm. We then define the operator $\T :\R^n\times\R^n \ra \R^n$ as
		\begin{align}
		\label{T}
		\T(q,\alpha) \Let \arg\min_{\beta \in \Ab} \Big\{q \transp \beta + {1 \over 2\vartheta} \|\beta - \alpha\|^2_\Rnorm \Big \}, 
		\end{align}
		More generally, a different norm can be used in the second term in \eqref{T} when $\vartheta$ is a different strong convexity parameter. However, we forgo this additional generality to keep the exposition simple. The operator $\T$ is defined implicitly through a finite convex optimization program whose computational complexity may depend on the $\Rnorm$-norm through the constraint set $\Ab$. For typical norms in $\R^n$ (e.g., $\|\cdot\|_{\ell_p}$) the pointwise evaluation of the operator $\T$ is computationally tractable. Furthermore, if $\|\cdot\|_\Rnorm = \|\cdot\|_{\ell_2}$, then the definition of \eqref{T} has an explicit analytical description for any pair $(q,\alpha)$ as follows.
		
		\begin{Lem}[Explicit description of {$\T$}]
			\label{lem:T}
			Suppose in the definition of the operator \eqref{T} the $\Rnorm$-norm is the classical $\ell_2$-norm. Then, the operator $\T$ admits the analytical description $\T(q,\alpha) = \xi\, (\alpha - q)$ where $\xi \Let \min\big\{1,\xnb\|q-\alpha\|_{\ell_2}^{-1}\big\}$.
		\end{Lem}
		
		\begin{proof}
			In case of the $\ell_2$-norm the strong convexity parameter is $\vartheta = 1$. Now using the classical duality theory, the objective function of \eqref{T} is equal to
			\begin{align}
			\label{pd}
			&\notag \min_{\|\beta\|_{\ell_2} \le \xnb } \Big\{q \transp \beta + {1 \over 2} \|\beta - \alpha\|^2_{\ell_2}\Big\} \\ 
			&\notag \quad = \max_{\lambda \ge 0} \Big\{-\lambda\xnb^2 + \min_\beta \big\{ q \transp \beta + {1 \over 2} \|\beta - \alpha\|^2_{\ell_2} + \lambda \|\beta\|^2_{\ell_2}\big\}\Big\} \\
			& \quad = \max_{\lambda \ge 0} \Big\{ - \lambda\xnb^2 + {q \transp \alpha -\|q\|^2_{\ell_2} \over 1 + 2\lambda} + {\|q+2\lambda\alpha\|^2_{\ell_2} \over 2(1+2\lambda)^2} + {\lambda \|\alpha - q\|^2_{\ell_2} \over (1+2\lambda)^2} \Big\},
			\end{align}
			where \eqref{pd} follows by substituting the explicit solution of the unconstrained inner problem described by
			\begin{align}
			\label{beta}
			\beta\opt(\lambda) \Let \arg\min_\beta \big\{ q \transp \beta + {1 \over 2} \|\beta - \alpha\|^2_{\ell_2} + \lambda \|\beta\|^2_{\ell_2}\big\} = {\alpha - q \over 1 + 2\lambda}.
			\end{align}
			To find the optimal $\lambda$ in the right-hand side of \eqref{pd}, it suffices to set the derivative to zero, which yields $\lambda\opt \Let \tfrac{1}{2\xnb}\max\{\|\alpha - q\| - \xnb, 0\}$. By substituting $\lambda\opt$ in \eqref{beta}, we have an optimal solution $\beta\opt(\lambda\opt) = \xi(\alpha - q)$ that is feasible since $\|\beta\opt(\lambda\opt)\|_{\ell_2} \le \xnb$. By virtue of the equality in \eqref{pd}, this concludes the desired assertion.
		\end{proof}
		
		Algorithm~\ref{alg} exploits the information revealed under Assumption~\ref{a:yeta} as well as the operator $\T$ to approximate the solution of the program \ref{dual-n}. The following theorem provides explicit error bounds for the solution provided by Algorithm \ref{alg} after $k$ iterations. The result is a slight extension of the classical smoothing technique in finite dimensional convex optimization \cite[Theorem 3]{ref:Nest-05} where the prox-function is not necessarily uniformly bounded, a potential difficulty in infinite dimensional spaces. We address this difficulty by considering a growth rate for the prox-function $d$ evaluated at the optimal solution $\yeta$. We later show, in Section~\ref{subsec:smooth:MDP}, how this extension will help in the MDP setting. 
		
		\begin{algorithm}[t!]
			\caption{\quad Optimal scheme for smooth convex optimization}
			\label{alg}
			\begin{algorithmic} 
				\State Choose some $w^{(0)} \in \Ab$
				\State {For $k \ge 0$ do}
				\begin{enumerate} [label=$(\roman*)$, itemsep = 1mm, topsep = 1mm, leftmargin = 2cm]
					
					\item [{\bf Step 1:}] Define $r^{(k)} \Let \frac{\eta}{L} \big(\cnew - \opn^*\yeta(w^{(k)})\big)$;
					
					\item [{\bf Step 2:}] Compute  $z^{(k)} \Let \T\big(\sum_{j=0}^{k} \frac{j+1}{2} r^{(j)}, 0\big), \quad \alpha^{(k)} \Let \T\big({1 \over \vartheta}r^{(k)}, w^{(k)}\big)$;
					
					\item [{\bf Step 3:}] Set $w^{(k+1)}=\frac{2}{k+3} z^{(k)} + \frac{k+1}{k+3} \alpha^{(k)}$.
				\end{enumerate}
			\end{algorithmic}
		\end{algorithm}

		\begin{Thm}[Smoothing approximation error] 
			\label{thm:semi-fin:smooth}
			Suppose Assumption \ref{a:yeta} holds with constant $L$ and $\vartheta$ is the strong convexity parameter in the definition of the operator $\T$ in \eqref{T}. Given the regularization term $\eta > 0$ and $k$ iterations of Algorithm \ref{alg}, we define 
			\begin{align*}
			\wh{\alpha}_\eta \Let \alpha^{(k)}, \qquad \wh{y}_\eta \Let \sum_{j=0}^{k}\frac{2(j+1)}{(k+1)(k+2)} \yeta(w^{(j)}).
			\end{align*}
			Under Assumption~\ref{a:reg}, the optimal value of the program \ref{primal-n} is bounded by $\Jnlb \le \Jpn \le \Jnub$ where 
			\begin{align}
			\label{Jn-algo}
			\Jnlb \Let \inner{b}{\wh{y}_\eta} - \xnb \|\opn^*\wh{y}_\eta -\cnew\|_{\Rnorm^*}, \qquad \Jnub \Let \wh{\alpha}_\eta \transp \cnew + \sup_{y \in \Yb} \inner{b-\opn \wh{\alpha}_\eta}{y}	
			\end{align} 
			Moreover, suppose there exist positive constants $c, C$ such that
			$$C \max\big\{\log\big(c\eta^{-1}\big),1\big\} \ge d\big(\yeta(\alpha)\big), \qquad \forall \eta > 0, \quad \forall \alpha \in \Ab,$$
			and, given an a~priori precision $\eps>0$, the regularization parameter $\eta$ and the number of iterations $k$ satisfy 
			\begin{align}
			\label{eta-k}
			\eta \le {\eps \over 2C\max\{2\log(2cC\eps^{-1}),1\}}, \qquad k \ge 2\xnb \ratio \frac{\sqrt{CL\max\{2\log(2cC\eps^{-1}),1\}}}{\sqrt{\vartheta}~\eps}, 
			\end{align} 
			where $\ratio$ is the constant defined in \eqref{opt}. Then, after $k$ iterations of Algorithm~\ref{alg} we have $\Jnub - \Jnlb \le \eps$. 
		\end{Thm}
		
		\begin{proof} 
			Observe that the bounds $\Jnlb$ and $\Jnub$ in \eqref{Jn-algo} are the values of the programs \ref{dual-n} and \ref{primal-n} evaluated at $\wh{y}_\eta$ and $\wh \alpha_\eta$, respectively. As such, the first assertion follows immediately. Towards the second part, thanks to the compactness of the set $\Ab$, the strong duality argument of Sion's minimax theorem \cite{ref:Sion-58} allows to describe the program \ref{dual-n-eta} through
			\begin{align} 
			\notag
			\Jneta & \Let \sup_{y \in \Yb} \inner{b}{y} - \Big[ \sup_{\alpha \in \Ab} \inner{\opn \alpha}{y} - \alpha \transp \cnew + \eta d(y)\Big] \\
			\notag 
			& = \inf_{\alpha \in \Ab} \alpha \transp \cnew + \sup_{y\in\Yb}  \Big[ \inner{b - \opn \alpha}{y}  - \eta d(y) \Big]\\
			\label{eq:dual:sion}	
			& = \inf_{\alpha \in \Ab} \alpha \transp \cnew + \inner{b - \opn \alpha}{\yeta(\alpha)}  - \eta d\big(\yeta(\alpha)\big), 
			\end{align}
			where the last equality follows from the definition in \eqref{yeta}. Note that the problem \eqref{eq:dual:sion} belongs to the class of smooth and strongly convex optimization problems, and can be solved using a fast gradient method developed by \cite{ref:Nest-05}. For this purpose, we define the function 
			\begin{align} 
			\label{eq:objective:smoothing}
			\phi_{\eta}(\alpha) &\Let \alpha \transp \cnew + \inner{b - \opn \alpha}{\yeta(\alpha)}  - \eta d\big(\yeta(\alpha)\big).
			\end{align}
			Invoking similar techniques to \cite[Theorem~1]{ref:Nest-05}, it can be shown that the mapping $\alpha \mapsto \phi_\eta(\alpha)$ is smooth with the gradient $\nabla \phi_{\eta}(\alpha) = \cnew - \opn^* \yeta(\alpha)$. The gradient $\nabla \phi_\eta(\alpha)$ is Lipschitz continuous by Assumption \ref{a:yeta} with constant $\tfrac{L}{\eta}$. Thus, following similar arguments as in the proof of \cite[Theorem~3]{ref:Nest-05} we have
			\begin{align}
			\label{eq:bound}
			0\le \Jnub - \Jnlb \le {L\|\alpha\opt\|^2_\Rnorm \over \vartheta (k+1)(k+2)\eta} + \eta d\big(\yeta(\alpha^*)\big) \le {L(\xnb \ratio)^2\over \vartheta k^2\eta} + C\eta\max\big\{\log(c\eta^{-1}),1\big\}.
			\end{align}
			Now, it is enough to bound each of the terms in the right-hand side of the above inequality by $\tfrac{1}{2}\eps$. It should be noted that this may not lead to an optimal choice of the parameter $\eta$, but it is good enough to achieve a reasonable precision order with respect to $\eps$. To ensure $\eta\log(\eta^{-1}) \le \eps$ for an $\eps \in (0,1)$ , it is not difficult to see that it suffices to set $\eta \le \tfrac{\eps}{2\log (\eps^{-1})}$. In this observation if we replace $\eta$ and $\eps$ with $\tfrac{1}{c}\eta$ and $\tfrac{1}{2cC}\eps$, respectively, we deduce that the second term on the right-hand side in \eqref{eq:bound} bounded by $\tfrac{1}{2}\eps$. Thus, the desired assertion follows by equating the first term on the right-hand side in \eqref{eq:bound} to $\tfrac{1}{2}\eps$ while the parameter $\eta$ is set as just suggested. 
		\end{proof}
		
		\begin{Rem}[Computational complexity]
			\label{rem:complex}
			Adding the prox-function to the problem~\ref{dual-n} ensures that the regularized counterpart \ref{dual-n-eta} admits an efficiency estimate (in terms of iteration numbers) of the order $\order\big(\sqrt{\tfrac{L}{\eta}\eps^{-1}}\big)$. To construct a smooth $\eps$- approximation for the original problem \ref{dual-n}, the Lipschitz constant $\tfrac{L}{\eta}$ can be chosen of the order $\order({\eps^{-1}}$ ${\log(\eps^{-1})})$. Thus, the presented gradient scheme has an efficiency estimate of the order $\order\big({\eps^{-1}}{\sqrt{\log(\eps^{-1})}}\big)$, see \cite{ref:Nest-05} for a more detailed discussion along similar objective.	
		\end{Rem}
		
		\begin{Rem}[Inexact gradient]
			\label{ref:inexact}
			The error bounds in Theorem~\ref{thm:semi-fin:smooth} are introduced based on the availability of the exact first-order information, i.e., it is assumed that at each iteration the vector $r^{(k)}$ that due to the bilinear form potentially involves a multi-dimensional integration can be computed exactly. In general, the evaluation of those vectors may only be available approximately. This gives rise to the question of how the fast gradient method performs in the case of \emph{inexact} first-order information. We refer the interested reader to \cite{ref:Devolver-13} for further details.
		\end{Rem}
		
		The a priori bound proposed by Theorem~\ref{thm:semi-fin:smooth} involves the positive constants $c, C$, which are used to introduce an upper bound for the proxy-term. These constants potentially depend on $\ynb$, the size of the dual feasible set, hence also on $\xnb$. Therefore, unlike the randomized approach in Section~\ref{sec:semi-fin:rand}, it is not immediately clear how $\xnb$ can be chosen to minimize the complexity of the proposed method, which in this case is the required number of iterations $k$ suggested in \eqref{eta-k} (cf. Remark~\ref{rem:theta}). In Section~\ref{subsec:smooth:MDP}, we shall discuss how to address this issue in the MDP setting for particular constants $c,C$.

		\section{Full infinite to finite programs} 
		\label{sec:inf-fin}
		
		The intention in this short section is to combine the two-step process from infinite to semi-infinite programs in Section~\ref{subsec:semi} and from semi-infinite to finite programs in Section~\ref{sec:semi-fin:rand} and \ref{sec:semi-fin:smoothing}, and hence establish a link from the original infinite program to finite counterparts. We only present the final result for the general infinite programs without discussing its implication in the MDP setting, as it is essentially a similar assertion. 
		
		\begin{Thm}[Infinite to finite approximation error]
			\label{thm:inf-fin}
			Consider the infinite program \ref{primal-inf} with a solution $\{x\opt,\Jp\}$, the finite (random) convex program \ref{primal-n,N} with the (random) solution $\big\{\alpha\opt_N,\JpnN\big\}$, and the output of Algorithm~\ref{alg} with values $\big\{\Jnlb, \Jnub\big\}$. Suppose Assumption~\ref{a:reg} holds and assume further that there exist constants $d, D$ so that the projection residual of the optimizer $x\opt$ onto the finite dimensional ball defined in Theorem~\ref{thm:inf-semi} is bounded by $\|r_n\| \le D n^{-1/d}$ for all $n\in\N$. Then, for any number of scenario samples $N$ and prox-term coefficient $\eta$, with probability $1-\beta$ we have 
			\begin{align*}
			\max\big\{\JpnN,\Jnlb\big\} - D\big(\|c\|_* + \ynb\|\op\|\big)n^{-1/d} \le \Jp \le \min\big\{\Jnub ~,~\JpnN+ \ynb h(\alpha\opt_N,\eps) \big\}.
			\end{align*}
			where $\ynb$ is as defined in \eqref{yb} and the function $h$ is a TB in the sense of Definition~\ref{def:tail}. Moreover, given an a priori precision level $\eps$, if 
			$$n \ge \Big(D\big(\|c\|_* + \ynb\|\op\|\big)\eps^{-1}\Big)^{d},$$
			and the number samples $N$ are chosen as in \eqref{NN} {\em or} the parameter $\eta$ together with the number of iterations of Algorithm~\ref{alg} is chosen as in \eqref{eta-k}, then with probability $1-\beta$ we have 
			\begin{align*}
			\min\Big\{|\Jp - \JpnN|, |\Jp - \Jnlb| \Big\} \le \eps \,.
			\end{align*}
		\end{Thm}
		
		The proof follows readily from the link between the infinite program~\ref{primal-inf} to the semi-infinite counterpart~\ref{primal-n} in Theorem~\ref{thm:inf-semi}, in conjunction with the link between \ref{primal-n} to the finite programs~\ref{primal-n,N} and \ref{dual-n-eta} in Theorems~\ref{thm:semi-fin:rand} and \ref{thm:semi-fin:smooth}, respectively.




\chapter{Approximation of MDPs} \label{chap:approx:MDP}

In this chapter we apply and simplify the approximation scheme developed in the previous chapter to the class of infinite-dimensional LPs that arise from discrete-time MDP as introduced in Section~\ref{section:MDP:LP}, which is the content of Sections~\ref{subsec:semi:MDP}, \ref{subsec:rand:MDP} and \ref{subsec:smooth:MDP}. The applicability of the theoretical results is demonstrated through a constrained linear quadratic optimal control problem and a fisheries management problem in Section~\ref{sec:sim}. The approximation method is generalized to the case where the transition kernel of the Markov process is unknown and has to be inferred from data, in Section~\ref{sec:RL}.

\section{Introduction} \label{subsec:approxMDP:intro}

The literature on control of MDP with uncountable state or action spaces mostly concentrates on approximation schemes with asymptotic performance guarantees \cite{ref:Hernandez-99, ref:Hernandez-98}, see also the comprehensive book \cite{ref:Kush-03} for controlled stochastic differential equations and \cite{ref:PrandiniHu-2007,ref:MohChatLyg-15} for reachability problems in a similar setting. From a practical viewpoint, a challenge using these schemes is that the convergence analysis is not constructive and does not lead to explicit error bounds. 
	A wealth of approximation schemes have been proposed in the literature under the names of approximate dynamic programming \cite{ref:Bert-75}, neuro-dynamic programming \cite{ref:Bertsekas-96}, reinforcement learning \cite{ref:Konda-03,ref:TsitVan-97}, and value and/or policy iteration \cite{ref:Bertsekas-12}. Most, however, deal with discrete (finite or at most countable) state and action spaces, while approximation over uncountable spaces remains largely unexplored. 

In addition, in many realistic applications the underlying dynamics are unknown and the decision maker needs to learn the optimal policy by trial-and-error, through interaction with the environment. In such a setting of unknown dynamics, the problem is particularly difficult and a prevalent approach in the existing literature consists of dynamic-programming-based reinforcement learning methods, also known as neuro-dynamic programming \cite{ref:Bertsekas-96, ref:Sutton-98,ref:Powell-07}. The two most common types of such reinforcement learning algorithms are $Q$-learning and actor-critic algorithms. $Q$-learning algorithms~\cite{ref:Watkins} are simulation-based schemes derived from value iteration, while actor-critic methods~\cite{ref:Konda-03} are simulation-based, two-time scale variants of policy iteration. $Q$-learning comes with asymptotic convergence guarantees but it is mostly considered in the case that state and action spaces are both discrete. On the other hand, while actor-critic algorithms can tackle continuous action and state spaces, since they are gradient-based, one can prove convergence to a local optimum.

The approach presented in this chapter is a particular instance of the approximation scheme developed in Chapter~\ref{chap:approx:LP}, i.e., by means of an infinite LP characterization that covers both long-run discounted and average cost performance criteria in the optimal control of MDP. The resulting approximation is based on finite convex programs that are different from the existing schemes. Closest in spirit to our proposed approximation is the linear programming approach based on constraint sampling in \cite{ref:FarVan-03,ref:FarVan-04,ref:sutter-CDC-14}. Unlike these works, however, we introduce an additional norm constraint that effectively acts as a \emph{regularizer}. We study in detail the conditions under which this regularizer can be exploited to bound the optimizers of the primal and dual programs, and hence provide an explicit approximation error for the proposed solution.

The structure of Sections~\ref{subsec:semi:MDP}, \ref{subsec:rand:MDP} and \ref{subsec:smooth:MDP} is illustrated in Figure~\ref{fig:overview:MDP} (that is the MDP simplification of Figure~\ref{fig:overview}), where the contributions are that the approximation results derived in the previous chapter are applied to the MDP setting with long-run average and discounted cost. Invoking the particular MDP structure of the underlying infinite LP several objects that were presented on a rather abstract level in Chapter~\ref{chap:approx:LP} can be simplified to explicit expressions, such as the tail bound function, presented in Defintion~\ref{def:tail} or the prox-function introduced in program~\ref{dual-n-eta}, that is chosen as the relative entropy.

Section~\ref{sec:RL} is an extension of the approximation scheme to the case where the transition kernel is is unknown but information is obtained by simulation. More specifically, in response to the current state and action, data about the next state is received. We derive a probabilistic explicit error bound between the data-driven finite convex program and the original infinite LP and discuss the sample complexity of the error bound, i.e., how many data are required for a certain approximation accuracy.

	\begin{figure}[t!]
		\centering
		\scalebox{.75}{
\def \sca{1.2}

\def \xb{3.9*\sca}
\def \yb{1.3*\sca}

\def \r{6.5*\sca}

\def \y{3.3*\sca} 
\def \yy{2*\sca} 

\def \la{0.2*\sca}

\def \di{0.2*\sca}

\def \backx{-0.1*\sca}
\def \backy{-1.4*\sca}
\def \back{0.2*\sca}

\def \buf{0.35*\sca}

\begin{tikzpicture}[scale=1,auto, node distance=1cm,>=latex']

\shade[rounded corners,bottom color = darkgreen!50, top color = darkgreen!25] (\r-\buf,\yb+\buf)--(\r+\r+\xb+\buf,\yb+\buf)--(\r+\r+\xb+\buf,-0.5*\y+3.6*\buf)--(\r-\buf,-0.5*\y+3.6*\buf)--cycle;

\shade[rounded corners,bottom color = violet!50, top color = violet!25] (\r-\buf,-0.6*\y+\buf)--(\r+\r+\xb+\buf,-0.6*\y+\buf)--(\r+\r+\xb+\buf,-1.0*\y-\buf)--(\r-\buf,-1.0*\y-\buf)--cycle;

\shade[rounded corners,bottom color = blue!50, top color = blue!25] (\r-\buf,-1.6*\y+\buf)--(\r+\r+\xb+\buf,-1.6*\y+\buf)--(\r+\r+\xb+\buf,-2*\y-\buf)--(\r-\buf,-2*\y-\buf)--cycle;


\shade[rounded corners,bottom color = gray!20, top color = gray!10] (0,0)--(\xb,0)--(\xb,\yb)--(0,\yb)--cycle;
\node at (0.5*\xb,0.7*\yb) {discrete time};
\node at (0.5*\xb,0.3*\yb) {MDP};
\shade[rounded corners,bottom color = gray!20, top color = gray!10] (\r,0)--(\r+\xb,0)--(\r+\xb,\yb)--(\r,\yb)--cycle;
\node at (\r+0.5*\xb,0.7*\yb) {infinite LP};
\node at (\r+0.5*\xb,0.3*\yb) {$\eqref{AC-LP}$: \ $\Jac$};
\shade[rounded corners,bottom color = gray!20, top color = gray!10] (\r,-\y)--(\r+\xb,-\y)--(\r+\xb,\yb-\y)--(\r,\yb-\y)--cycle;
\node at (\r+0.5*\xb,0.7*\yb-\y) {robust program};
\node at (\r+0.5*\xb,0.3*\yb-\y) {$\eqref{AC-LP-n}$: \ $\Jacn$};
\shade[rounded corners,bottom color = gray!20, top color = gray!10] (\r,-\y-\y)--(\r+\xb,-\y-\y)--(\r+\xb,\yb-\y-\y)--(\r,\yb-\y-\y)--cycle;
\node at (\r+0.5*\xb,0.7*\yb-\y-\y) {scenario program};
\node at (\r+0.5*\xb,0.3*\yb-\y-\y) {$\eqref{AC-LP-n,N}$: \ $\JacnN$};

\shade[rounded corners,bottom color = gray!20, top color = gray!10] (\r+\r,-\y)--(\r+\r+\xb,-\y)--(\r+\r+\xb,\yb-\y)--(\r+\r,\yb-\y)--cycle;
\node at (\r+\r+0.5*\xb,0.7*\yb-\y) {semi-infinite program};
\node at (\r+\r+0.5*\xb,0.3*\yb-\y) {$\eqref{dual-n}$: \ $\Jdn^{\text{AC}}$};

\shade[rounded corners,bottom color = gray!20, top color = gray!10] (\r+\r,-\y-\y)--(\r+\r+\xb,-\y-\y)--(\r+\r+\xb,\yb-\y-\y)--(\r+\r,\yb-\y-\y)--cycle;
\node at (\r+\r+0.5*\xb,0.7*\yb-\y-\y) {regularized program};
\node at (\r+\r+0.5*\xb,0.3*\yb-\y-\y) {$\eqref{dual-n-eta}$: \ $\Jdnr^{\text{AC}}$};

\shade[rounded corners,bottom color = gray!20, top color = gray!10] (\r+0.5*\r,-\y-\y-\y)--(\r+0.5*\r+\xb,-\y-\y-\y)--(\r+0.5*\r+\xb,\yb-\y-\y-\y)--(\r+0.5*\r,\yb-\y-\y-\y)--cycle;
\node at (\r+0.5*\r+0.5*\xb,0.7*\yb-\y-\y-\y) {prior $\&$ posterior error};
\node at (\r+0.5*\r+0.5*\xb,0.3*\yb-\y-\y-\y) {$\JacnN - \Jdnr^{\text{AC}}$};

  \draw[>=latex,<->,line width=1.2mm,gray!25] (\xb,0.5*\yb) -- (\r,0.5*\yb);
    \draw[>=latex,<->,line width=1.2mm,gray!25] (\xb+\r,-\y+0.5*\yb) -- (\r+\r,-\y+0.5*\yb);
  \draw[>=latex,<->,dashed, line width=1.2mm,gray!25] (\r+0.5*\xb,0) -- (\r+0.5*\xb,-\y+\yb);
 \draw[>=latex,<->,dashed, line width=1.2mm,gray!25] (\r+0.5*\xb,-\y+0.25*\la) -- (\r+0.5*\xb,-\y+0.25*\la-\y+\yb);
 \draw[>=latex,<->,dashed, line width=1.2mm,gray!25] (\r+\r+0.5*\xb,-\y+0.25*\la) -- (\r+\r+0.5*\xb,-\y+0.25*\la-\y+\yb);
 
   \draw[>=latex,->,line width=1.2mm,gray!25] (\r+0.5*\xb,-\y-\y+0.2*\di) -- (\r+0.5*\r,\yb-\y-\y-\y);
   \draw[>=latex,->, line width=1.2mm,gray!25] (\r+\r+0.5*\xb,-\y-\y+0.2*\di) -- (\r+0.5*\r+\xb,\yb-\y-\y-\y);
 


  \shade[bottom color = darkgreen!50, top color = darkgreen!25] (0,-2.4*\y+1*\y)--(\di,-2.4*\y+1*\y)--(\di,-2.4*\y-\di+1*\y)--(0,-2.4*\y-\di+1*\y)--cycle;
  \node[] at (0.35*\r,-2.4*\y-0.5*\di+1*\y) {\  infinite program};

  \shade[bottom color = violet!50, top color = violet!25] (0,-2.7*\y+1*\y)--(\di,-2.7*\y+1*\y)--(\di,-2.7*\y-\di+1*\y)--(0,-2.7*\y-\di+1*\y)--cycle;
  \node[] at (0.35*\r,-2.7*\y-0.5*\di+1*\y) {\ \ \ \ \ \ \ \ \ semi-infinite programs};

  \shade[bottom color = blue!50, top color = blue!25] (0,-3.0*\y+1*\y)--(\di,-3.0*\y+1*\y)--(\di,-3.0*\y-\di+1*\y)--(0,-3.0*\y-\di+1*\y)--cycle;
  \node[] at (0.35*\r,-3.0*\y-0.5*\di+1*\y) {finite programs};


 \node[] at (0.5*\r+0.5*\xb+\r,-0.95*\y+0.5*\yb+1.3*\buf) {Proposition~$\ref{prop:SD}$};
 \node[] at (0.5*\r+0.5*\xb+\r,-\y+0.5*\yb -1.3*\buf) {}; 
 \node[] at (0.5*\r+0.5*\xb+\r,-1.05*\y+0.5*\yb -1.3*\buf) {strong duality};   
 \node[] at (0.5*\xb+0.5*\r,0.5*\yb+1.3*\buf) {equivalent};
  \node[] at (0.5*\xb+0.5*\r,-0.3*\yb+1.3*\buf) {Theorem~$\ref{thm:equivalent:LP}$};
 \node[] at (0.19*\r+0.5*\xb+\r,-0.5*\y+0.5*\yb-0.0*\buf) {Corollary~$\ref{cor:MDP:inf-semi}$};
 \node[] at (0.19*\r+0.5*\xb+\r,-1.5*\y+0.5*\yb-0.0*\buf) {Corollary~$\ref{cor:adp:semi-finite}$};  
 \node[] at (1.19*\r+0.5*\xb+\r,-1.5*\y+0.5*\yb-0.0*\buf) {Corollary~$\ref{cor:adp:smooth}$};     
  \node[] at (0.5*\r+0.5*\xb+\r,-2.6*\y+0.5*\yb -1.3*\buf) {Theorem~\ref{thm:inf-fin}};  
\end{tikzpicture}}
		\caption{Graphical representation of the chapter structure and its contributions}
		\label{fig:overview:MDP}
	\end{figure}
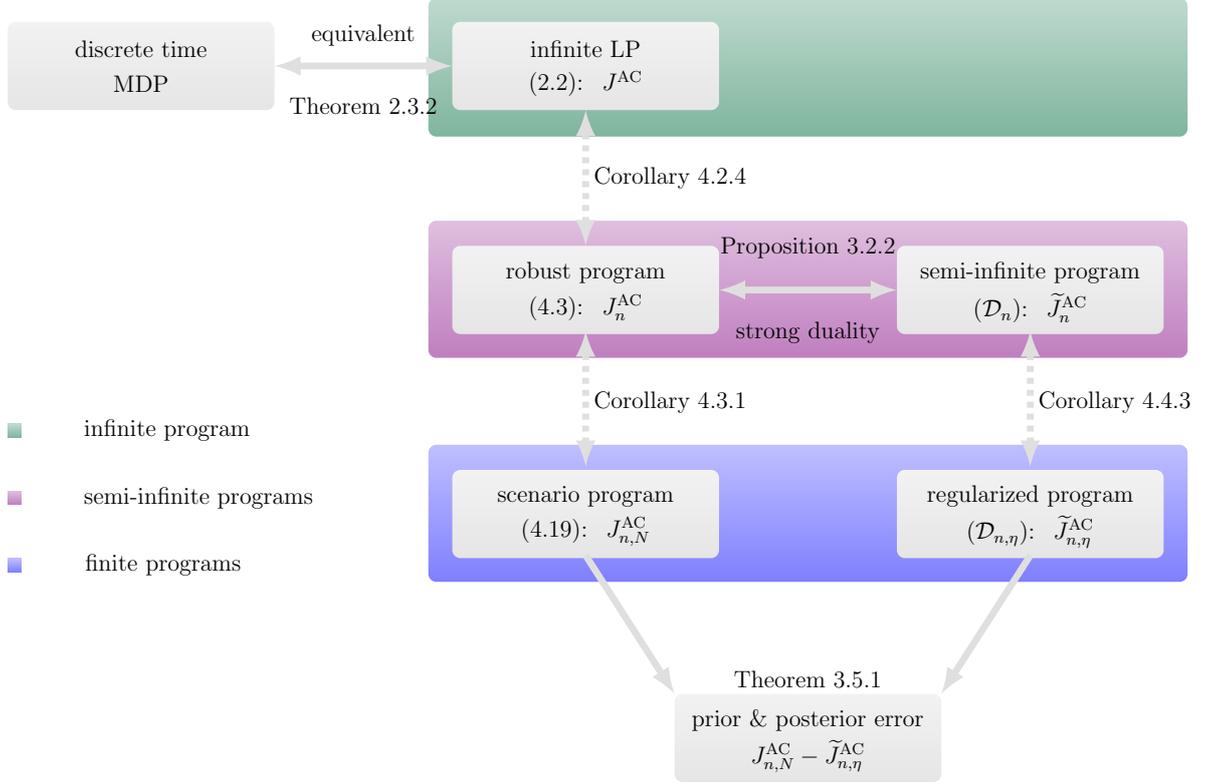

		\section{Semi-infinite results in the MDP setting}
		\label{subsec:semi:MDP}
		We now return to the MDP setting introduced in Section~\ref{section:MDP:LP}, and in particular the AC problem \eqref{AC-LP}, to investigate the application of the approximation scheme presented in Chapter~\ref{chap:approx:LP}. Recall that the AC problem \eqref{AC} can be recast in an LP framework in the form of \ref{primal-inf}, see \eqref{AC-setting}. To complete this transition to the dual pairs, we introduce the spaces 
		\begin{align} 
		\label{AC:pairs}
		\left\{\begin{array}{ll}
		\X = \R\times \Lip(S), & \C=\R\times\Meas(S), \\
		\B=\Lip(K), & \Y = \Meas(K), \\
		\cone=\Lip_+(K), & \cone^{*}=\Meas_+(K).
		\end{array}\right.
		\end{align}
		The bilinear form between each pair $(\X, \C)$ and $(\B,\Y)$ is defined in an obvious way (cf.\ \eqref{AC-setting}). The linear operator $\op:\X \ra \B$ is defined as $\op(\rho,u)(s,a) \Let -\rho - u(s) + Qu(s,a)$, and it can be shown to be weakly continuous \cite[p.~220]{ref:Hernandez-99}. On the pair $(\X, \C)$ we consider the norms
		\begin{subequations}
			\label{norm}
			\begin{align}
			\label{norm-1}
			\left\{
			\begin{array}{l}
			\|x\| = \|(\rho, u)\| := \max\big\{|\rho|, \|u\|_\lip \} = \max\big\{|\rho|, \|u\|_\infty, \sup_{s,s' \in S}{u(s) - u(s') \over \|s - s'\|_{\ell_\infty}}   \big\}, \vspace{2mm}\\
			\|c\|_* \Let \sup_{\|x\| \le 1} \inner{x}{c} =|c_1| + \sup_{\|u\|_\lip \le 1} \int_S u(s) c_2(\diff s) = |c_1| + \|c_2\|_{\wass}. 
			\end{array} \right. 
			\end{align}
			Recall that $\|\cdot\|_\lip$ is the Lipschitz norm on $\Lip(S)$ whose dual norm $\|\cdot\|_\wass$ in $\Meas(S)$ is known as the Wasserstein norm \cite[p.~105]{ref:Villani}. The adjoint operator $\op^{*}:\Y\to\C$ is given by $\op^{*}y(\cdot) \Let \big(-\inner{\ind}{y},-y(\cdot\times A)+y Q(\cdot)\big)$, where $\ind$ is the constant function in $\Lip(S)$ with value 1. In the second pair $(\B,\Y)$, we consider the norms
			\begin{align}
			\label{norm-2}
			\left\{
			\begin{array}{l}
			\|b\| = \|b\|_\lip \Let \max\big\{\|b\|_\infty, \sup_{k,k' \in K}{b(k) - b(k') \over \|k - k'\|_{\ell_\infty}}   \big\}, \vspace{2mm} \\
			\|y\|_* \Let \sup_{\|b\|_\lip \le 1} \inner{b}{y} = \|y\|_{\wass}. 
			\end{array} \right. 
			\end{align}
		\end{subequations}
		A commonly used norm on the set of measures is the total variation whose dual (variational) characterization is associated with $\|\cdot\|_\infty$ in the space of continuous functions \cite[p.~2]{ref:Hernandez-99}. We note that in the positive cone $\cone^* = \Meas_+(K)$ the total variation and Wasserstein norms indeed coincide. 
		
		Following the construction in \ref{primal-n}, we consider a collection of $n$-linearly independent, normalized functions $\{u_i\}_{i \le n}$, $\|u_i\|_\lip = 1$, and define the semi-infinite approximation of the AC problem \eqref{AC-LP} by
		\begin{align}
		\label{AC-LP-n} 
		-\Jacn = & \left\{ \begin{array}{ll}
		\inf\limits_{(\rho, \alpha)\in\R\times\R^n} & -\rho   \\
		\subjectto &\rho + \sum\limits_{i=1}^{n} \alpha_i\big(u_i(s) - Qu_i(s,a)\big) \leq \cost(s,a), \quad \forall (s,a)\in K \\
		&  \| \alpha \|_\Rnorm \le \xnb 
		\end{array} \right. 
		\end{align}
		Comparing with the program \ref{primal-n}, we note that the finite dimensional subspace $\X_n \subset \R \times \Lip(S)$ is the subspace spanned by the basis elements $x_0 = (1,0)$ and $x_i = (0,u_i)$ for all $i \in \{1,\cdots,n\}$, i.e., the subspace $\X_n$ is in fact $n+1$ dimensional.  Moreover, the norm constraint in \eqref{AC-LP-n} is only imposed on the second coordinate of the decision variables $(\rho,\alpha)$ (i.e., $\|\alpha\|_\Rnorm \le \xnb$). The following lemmas address the operator norm and the respective regularity requirements of Assumption~\ref{a:reg} for the program \eqref{AC-LP-n}.

\begin{Lem}[MDP operator norm] \label{lem:operator}
	In the AC problem \eqref{AC-LP} under Assumption~\ref{a:CM}\ref{a:CM:Q} with the specific norms defined in \eqref{norm}, the linear operator norm satisfies $\|I - Q\| \Let \sup_{\|u\|_\lip \le 1} \|u - Qu\|_\lip \le 1 + \max\{L_Q,1\}$.
\end{Lem}

\begin{proof}
	Using the triangle inequality it is straightforward to see that
	\begin{align*}
	\| I-Q \| 	&= 		\sup\limits_{u\in\Lip(S)} \frac{\| u - Qu \|_\lip}{\|u\|_\lip } \le 1 + \sup\limits_{u\in \Lip(S)} {\| Qu \|_\lip \over \| u \|_\lip} \le 1 + \sup\limits_{u\in \Lip(S)} {\| Qu \|_\lip \over \| u \|_\infty}\\ 
	& \le 1 + \max\Big\{L_Q, \sup\limits_{u\in \Lip(S)} {\| Qu \|_\infty \over \| u \|_\infty} \Big\} \le 1 + \max\{L_Q,1\},
	\end{align*}
	where the second line is an immediate consequence of Assumption~\ref{a:CM}\ref{a:CM:Q} and the fact that the operator $Q$ is a stochastic kernel. Hence, $|Qu(s,a)|=| \int_{S} u(y) Q(\drv y|s,a)| \leq \| u \|_\infty (\int_{S} Q(\drv y|s,a)) = \| u \|_\infty$.
\end{proof}

\begin{Lem}[MDP semi-infinite regularity]\label{lem:MDP:bd-dual}
	Consider the AC program \eqref{AC-LP} under Assumption \ref{a:CM}. Then, Assumption~\ref{a:reg} holds for the semi-infinite counterpart in \eqref{AC-LP-n} for any positive $\xnb$ and all sufficiently large $\gamma$. In particular, the dual optimizer bound in Proposition~\ref{prop:SD} simplifies to $\|\yn\|_\wass \le \ynb = 1$.
\end{Lem}
		
\begin{proof}
Since $K$ is compact, for any nonnegative $\xnb$, the program \eqref{AC-LP-n} is feasible and the optimal value is bounded; recall that $\|(Q-I)u_i\|_\lip \le 1+\max\{L_Q,1\}$ from Lemma~\ref{lem:operator} and $\|\cost\|_\infty < \infty$ thanks to Assumption~\ref{a:CM}\ref{a:CM:cost}. Hence, the optimal value of \eqref{AC-LP-n} is bounded and, without loss of generality, one can add a redundant constraint $|\rho| \le \omega^{-1}\xnb$, where $\omega$ is a sufficiently small positive constant. In this view, the last constraint $\|\alpha\|_\Rnorm \le \xnb$ may be replaced with 
\begin{align}
\label{norm-AC}
	\|(\rho,\alpha)\|_{\omega} \Let \max\{\omega|\rho|,\|\alpha\|_\Rnorm\} \le \xnb, 
\end{align}
where $\|\cdot\|_\omega$ can be cast as the norm on the pair $(\rho,\alpha) \in \R\times\R^{n+1}$. Using the $\omega$-norm as defined in \eqref{norm-AC}, we can now directly translate the program~\eqref{AC-LP-n} into the semi-infinite framework of \ref{primal-n}. As mentioned above, the feasibility requirement in  Assumption~\ref{a:reg}\ref{a:reg:feas} immediately holds. In addition, observe that for every $y\in\cone^*$ we have 
\begin{align*}
	\|\opn^* y\|_{\omega^*} &= \sup_{\|(\rho,\alpha)\|_{\omega} \le 1} (\rho,\alpha) \transp \big[-\inner{\ind}{y}, \inner{Qu_1-u_1}{y}, \cdots,\inner{Qu_n-u_n}{y}\big]\\	
	&= \sup_{\omega |\rho| \le 1}  - \rho\inner{\ind}{y} + \sup_{\|\alpha\|_\Rnorm\le 1} \alpha \transp \big[\inner{Qu_1-u_1}{y}, \cdots,\inner{Qu_n-u_n}{y}\big] \\
	& \ge \omega^{-1}\|y\|_\wass,
\end{align*}
where the third line above follows from the equality $\inner{\ind}{y} = \|y\|_\wass$ for every $y$ in the positive cone $\cone^*$, and the fact that the second term in the second line is nonnegative. Since $\omega$ can be arbitrarily close to 0, the inf-sup requirement Assumption~\ref{a:reg}\ref{a:reg:inf-sup} holds for all sufficiently large $\gamma = \omega^{-1}$. 
The second assertion of the lemma follows from the bound \eqref{yb} in Proposition \ref{prop:SD}. To show this, recall that in the MDP setting $c = (-1,0) \in \R \times \Meas(S)$ (cf. \eqref{AC-setting}) with the respective vector $\cnew = [-1,0,\cdots,0] \in \R\times \R^{n}$ (cf. \ref{primal-n}). Thus, $\|\cnew\|_{\omega^*} = \sup_{\|(\rho,\alpha)\|_{\omega}\le 1} {(\rho,\alpha)} \transp [-1,0,\cdots,0] = \omega^{-1}$, that helps simplifying the bound \eqref{yb} to
\begin{align*}
\|\yn\|_\wass \le \ynb \Let {\xnb\|\cnew\|_{\Rnorm^*} - \Jlb \over \gamma \xnb - \|b\|} = {\xnb \omega^{-1} + \|\cost\|_\infty  \over \omega^{-1} \xnb -\|\cost\|_\lip},  
\end{align*}
which delivers the desired assertion when $\omega$ tends to 0. 
\end{proof}
		
\begin{Rem}[AC dual optimizers bound]
	\label{rem:ynb:AC}
	As opposed to the general LP in Proposition~\ref{prop:SD}, Lemma~\ref{lem:MDP:bd-dual} implies that the dual optimizers for the AC problem is not influenced by the primal norm bound $\xnb$ and is uniformly bounded by $1$. In fact, this result can be strengthened to $\|\yn\|_\wass = 1$ due to the special minimax structure of the AC program~\eqref{AC-LP-n}. This refinement is not needed at this stage and we postpone the discussion to Section~\ref{subsec:smooth:MDP}. The feature discussed in this remark does, however, not hold for the class of long-run discounted cost problems, see Lemma~\ref{lem:MDP:DC} in Section~\ref{sec:discounted:setting}. 
\end{Rem}
	
Now we are in a position to translate Theorem~\ref{thm:inf-semi} to the MDP setting for the AC problem \eqref{AC-LP}.

\begin{Cor}[MDP semi-infinite approximation]
	\label{cor:MDP:inf-semi}
	Let $\Jac$ and $u\opt$ be the optimal value and an optimizer for the AC program~\eqref{AC-LP}, respectively. Consider the semi-infinite program \eqref{AC-LP-n} where $\xnb > \|\cost\|_\lip$, and let $\Uball \Let \{\sum_{i=1}^{n} \alpha_i u_i ~:~ \|\alpha\|_\Rnorm \le \xnb \}$. Then, the optimal value $\Jacn$ of \eqref{AC-LP-n} satisfies the inequality
		\begin{align*}
			0 \le \Jac - \Jacn \le \big(1 + \max\{L_Q, 1\}\big) \big\|u\opt - \proj_{\Uball}(u\opt)\big\|_\lip.
		\end{align*}
\end{Cor}
		
		\begin{proof}
			We first note that the existence of the optimizer $u\opt$ is guaranteed under Assumption \ref{a:CM} \cite[Theorem~12.4.2]{ref:Hernandez-99}.
			The proof is a direct application of Theorem~\ref{thm:inf-semi} under the preliminary results in Lemma \ref{lem:MDP:bd-dual} and \ref{lem:operator}. Observe that the projection error is $r_n \Let (\rho\opt,u\opt) - \proj_{\Uball}(\rho\opt,u\opt) = \big(0, u\opt - \proj_{\Uball}(u\opt)\big)$, resulting in $\inner{r_n}{c} = 0$. Thanks to this observation Lemma~\ref{lem:operator}, the assertion of Theorem~\ref{thm:inf-semi} translates to 
			\begin{align*}
			0 \le \Jac - \Jacn & = \Jpn - \Jp \le \inner{r_n}{\op^*\yn - c} = \inner{\op r_n}{\yn} \le \|I-Q\|\, \|r_n\|_\lip \, \|\yn\|_\wass\\
			& \le (1 + \max\{L_Q, 1\}) \|u\opt - \proj_{\Uball}(u\opt)\|_\lip.  \qedhere
			\end{align*}
		\end{proof}
		
		Observe that if from the beginning we consider the norm $\|\cdot\|_\infty$ on the spaces $\X$ and $\B$, it is not difficult to see that the operator norm in Lemma~\ref{lem:operator} simplifies to $2$ (recall that $Q$ is a stochastic kernel). Thus, the semi-infinite bound reduces to $\Jac - \Jacn \le 2 \|u\opt - \proj_{\Uball}(u\opt)\|_\infty$. One may arrive at this particular observation through a more straightforward approach: Using the shorthand notation $(Q-I)u \Let Qu - u$, we have 
		\begin{align*}
		\Jac - \Jacn & \le \min_{k\in K} \Big( \big(Q-I\big)u\opt(k) + \cost(k) \Big) - \min_{k\in K} \Big( \big(Q-I\big)\proj_{\Uball}(u\opt)(k) + \cost(k) \Big)\\
		& \le \max_{k\in K} \big(Q-I\big)\big(u\opt - \proj_{\Uball}(u\opt)\big)(k) \le \big\|\big(Q-I\big) \big(u\opt - \proj_{\Uball}(u\opt)\big)\big\|_\infty \\
		& \le 2\big\|u\opt - \proj_{\Uball}(u\opt)\big\|_\infty.
		\end{align*}
		Theorem~\ref{thm:inf-semi} is a generalization to the above observation in two respects: 
		\begin{itemize}
			\item It holds for a general LP that, unlike the AC problem \eqref{AC-LP}, may not necessarily enjoy a min-max structure.
			\item The result reflects how the bound on the decision space (i.e., $\xnb$ in \ref{primal-n}) influences the dual optimizers as well as the approximation performance in generic normed spaces. 
		\end{itemize}
	The latter feature is of particular interest as the boundedness of the decision space is often an a~priori requirement for optimization algorithms, see for instance \cite{ref:nesterov-book-04} and the results in Section~\ref{sec:semi-fin:smoothing}. The approximation error from the original infinite LP to the semi-infinite version is quantified in terms of the projection residual of the value function. Clearly, this is where the choice of the finite dimensional ball $\Uball$ plays a crucial role. We close this section with a remark on this point. 
		
		\begin{Rem}[Projection residual]\label{rem:projection}
			The residual error $\big\|u\opt - \proj_{\Uball}(u\opt)\big\|_\lip$ can be approximated by leveraging results from the literature on universal function approximation. Prior information about the value function $u\opt$ may offer explicit quantitative bounds. For instance, for MDP under Assumption~\ref{a:CM} we know that $u\opt$ is Lipschitz continuous. For appropriate choice of basis functions, we can therefore ensure a convergence rate of ${n}^{-1/\dim(S)}$ where $\dim(S)$ is the dimension of the state-action set $S$, see for instance \cite{ref:Farouki-12} for polynomials and \cite{ref:Olver-09} for the Fourier basis functions. 
			
		\end{Rem}

		\section{Randomized results in the MDP setting} 
		\label{subsec:rand:MDP}
		We return to the MDP setting and discuss the implication of Theorem~\ref{thm:semi-fin:rand} as the bridge from the semi-infinite program \ref{primal-n} to the finite counterpart \ref{primal-n,N}. Recall the dual pairs of vector spaces setting in \eqref{AC:pairs} with the assigned norms \eqref{norm}. To construct the finite program \ref{primal-n,N}, we need to sample from the set of extreme points of $\Prob(K)$, i.e., the set of point measures
		\begin{align*}
		\Kb \Let \ext\big(\Prob(K)\big) = \big\{ \dir{(s,a)} : (s,a)\in K \big\},
		\end{align*}
		where $\dir{(s,a)}$ denotes a point probability distribution at $(s,a) \in K$. In this view, in order to sample elements from $\Kb$ it suffices to sample from the state-action feasible pairs $(s,a) \in K$. 
		
		\begin{Cor}[MDP finite randomized approximation error]
			\label{cor:adp:semi-finite}
			Let $\{(s_j,a_j) \}_{j\le N}$ be $N$ i.i.d.~samples generated from the uniform distribution on $K$. Consider the program 
			\begin{align}
			\label{AC-LP-n,N} 
			-\JacnN = & \left\{ \begin{array}{ll}
			\inf\limits_{(\rho, \alpha)\in\R^{n+1}} & -\rho   \\
			\subjectto &\rho + \sum\limits_{i=1}^{n} \alpha_i\big(u_i(s_j) - Qu_i(s_j,a_j)\big) \leq \cost(s_j,a_j), \quad \forall j \in \{1,\cdots,N\}\\
			&  \| \alpha \|_{\Rnorm} \le \xnb.
			\end{array} \right. 
			\end{align}
			where the basis functions $\{u_i\}_{i \le n}$ introduced in \eqref{AC-LP-n} are normalized (i.e., $\|u_i\|_\lip = 1$). Let $L_Q$ be the Lipschitz constant from Assumption~\ref{a:CM}\ref{a:CM:Q}, and define the constant 
			\begin{align*}
			z_n \Let \big(\xnb \ratio (\max\{L_Q,1\}+ 1)+ \|\cost\|_\lip \big)^{-1},
			\end{align*}
			where $\ratio$ is the ratio constant introduced in \eqref{opt}. Then, for all $\eps, \beta$ in $(0,1)$ and $N \ge \NN\big(n+1,(z_n\eps)^{\dim(K)},\beta\big)$ defined in \eqref{N}, we have
			\begin{align*}
			\PP^N \Big[0\leq \JacnN - \Jacn \le   \eps \Big] \ge 1-\beta.
			\end{align*}
		\end{Cor}	
		
		\begin{proof}
			Let $(\rho\opt_N,\alpha\opt_N)$ be the optimal solution for \eqref{AC-LP-n,N}. Observe that in the MDP setting, Assumption~\ref{a:CM}\ref{a:CM:Q} implies
			\begin{align}
			\|\opn\alpha\opt_N - b\| & = \Big\|-\rho\opt_N+ \sum_{i=1}^{n} \alpha\opt_{N(i)}(Q-I)u_i + \cost\Big\|_\lip \\
			&\le (\max\{L_Q,1\}+1)\Big\|\sum_{i=1}^{n}\alpha\opt_{N(i)} u_i \Big\|_\lip + \|-\rho\opt_N+\cost\|_\lip \notag \\
			\label{Lg}  
			& \le (\max\{L_Q,1\}+1)\xnb \ratio \big(\max_{i\le n}\|u_i\|_\lip\big) + \|\cost\|_\lip,
			\end{align}
			where the equality $\|-\rho\opt_N+\cost\|_{\lip} = \|\cost\|_{\lip}$ leading to \eqref{Lg} follows from the fact that $\cost$ and $\rho\opt$ are non-negative (note that $\alpha = 0, \rho = 0$ is a trivial feasible solution for \eqref{AC-LP-n,N}). In the second step, we propose a TB candidate in the sense of Definition~\ref{def:tail}. Note that for any $k, k' \in K$, by the definition of the Wasserstein norm we have $\|\dir{\{k\}} - \dir{\{k'\}}\|_\wass = \min\{1,\|k - k'\|_\infty\}$. Thus, generating samples uniformly from $K$ leads to 
			\begin{align}
			\label{balls:AC}
			\PP\big[ \ball{\kappa}{\gamma} \big] \geq \PP\big[ \ball{k}{\gamma} \big] \ge {\gamma}^{\dim(K)}, \qquad \forall \kappa \in \Kb, \quad \forall k \in K,
			\end{align}
			where, with slight abuse of notation, the first ball $\ball{\kappa}{\gamma}$ is a subset of the infinite dimensional space $\Y$ with respect to the dual norm $\|\cdot\|_\wass$, while the second ball $\ball{k}{\gamma}$ is a subset of the finite dimensional space $K$ whose respective norm is $\|\cdot\|_\infty$. The relation \eqref{balls:AC} readily suggests a function $g:\R_+ \ra [0,1]$ for Example~\ref{ex:TB}, which together with \eqref{Lg} and the fact that the basis functions are normalized, it yields
			\begin{align*}
			h(\alpha,\eps) \Let \|\opn \alpha - b\| g^{-1}(\eps) &  \le \big(\xnb\ratio(\max\{L_Q,1\}+1) + \|\cost\|_\lip \big)\eps^{{1/\dim K}}.
			\end{align*}			
			Recall from Lemma~\ref{lem:MDP:bd-dual} that the dual multiplier bound is $\ynb = 1$, and feasible solutions $\alpha$ is bounded by $\xnb$. Finally, note that the decision variable of the program \eqref{AC-LP-n,N} is the $n+1$ dimensional pair $(\rho,\alpha)$. Given all the information above, the claim then readily follows from the second result of Theorem~\ref{thm:semi-fin:rand} in \eqref{eq:thm:semi-fin:rand:2}.
		\end{proof}
				
To select $\xnb$, one may minimize the complexity of the a priori bound in Corollary~\ref{cor:adp:semi-finite}, which is reflected through the required number of samples. At the same time, the impact of the bound $\xnb$ on the approximation step from infinite to semi-infinite in Corollary~\ref{cor:MDP:inf-semi} should also be taken into account. The first factor is monotonically decreasing with respect to $\xnb$, i.e., the smaller the parameter $\xnb$, the lower the number of the required samples. The second factor is presented through the projection residual (cf. Remark~\ref{rem:projection}). Therefore, an acceptable choice of $\xnb$ is an upper bound for the projection error of the optimal solution onto the ball $\Uball$ uniformly in $n \in \N$, i.e., 
\begin{subequations}
	\label{theta*:MDP}
	\begin{align}
	\label{theta*:MDP1}
	\xnb \ge \sup\bigg\{\|\alpha\opt\|_\Rnorm ~:~ \proj_{\Uball}(x\opt) = \sum_{i=1}^{n} \alpha\opt_iu_i, \quad n \in \N \bigg\}. 
	\end{align}
The above bound may be available in particular cases, e.g., when $\|\cdot\|_\Rnorm = \|\cdot\|_{\ell_2}$ it yields the bound
	\begin{align}
	\label{theta*:MDP2}
	\|\alpha\opt\|_{\ell_2} = \sqrt{\int_{S} {u\opt}^2(s) \diff s} \le \|u\opt\|_\lip \le \max\{L_Q,1\}\|\cost\|_\infty, 
	\end{align} 
\end{subequations}	
where $L_Q$ is the Lipschitz constant in Assumptions~\ref{a:CM}\ref{a:CM:Q}. We note that the first inequality in \eqref{theta*:MDP2} follows since $S$ is a unit hypercube, and the second inequality follows from \cite[Lemma~2.3]{ref:Dufour-15}, see also \cite[Section~5]{ref:Dufour-15} for further detailed analysis.

\section{Structural convex optimization results in the MDP setting}  
\label{subsec:smooth:MDP}
To link the approximation method presented in Section~\ref{sec:semi-fin:smoothing} to the AC program in \eqref{AC-LP-n}, let us recall the dual pairs \eqref{AC:pairs} equipped with the norms \eqref{norm}. To simplify the analysis, we refine the assertion in Lemma~\ref{lem:MDP:bd-dual} and argue that the dual optimizers are indeed probability measures, i.e.,
	\begin{align}
	\label{Y}
	\Yb \Let \bigg\{ y \in \Meas_+(K) \ : \|y\|_{\wass} = \ynb = 1 \bigg\}.
	\end{align}
To see this, one can consider the norm $\|(\rho,\alpha)\| \Let \|\alpha\|_{\Rnorm}$ and follow similar arguments in the proof of Proposition~\ref{prop:SD}. Strictly speaking, this is not a true norm on $\R^{n+1}$ but it does not affect the technical argument, in particular strong duality between \ref{primal-n} and \ref{dual-n}. The details are omitted here in the interest of space. We consider the prox-function as a relative entropy defined by
		\begin{align} 
		\label{d}
		d(y) \Let \left\{ 
		\begin{array}{cr} 
		\inner{\log\big(\frac{\diff y}{\diff\Leb}\big)}{y} & y \ll \Leb \\
		\infty & \text{o.w.},
		\end{array}\right.
		\end{align}
		where $\Leb$ is the uniform measure supported on the set $K$ and $\tfrac{\diff y}{\diff\Leb} \in \Func_+(K)$ is the Radon-Nikodym derivative between two measures $y$ and $\Leb$. One can inspect that the prox-function \eqref{d} is indeed a non-negative function. The optimizer of the regularized program \ref{dual-n-eta} for the AC program~\eqref{AC-LP-n} is
		\begin{align}
		\label{yeta:AC}
		\yeta(\rho,\alpha) \Let \arg\max_{y \in \Yb} \bigg\{ \inner{-\cost + \rho - \sum_{i=1}^{n}\alpha_i(Q-I)u_i}{y} - \eta \inner{\log\big(\tfrac{\diff y}{\diff\Leb}\big)}{y} \bigg\}.
		\end{align}
		To see \eqref{yeta:AC}, check \eqref{yeta} together with the definitions of the operator $\opn$ in \eqref{Ln} and the AC problem parameters in \eqref{AC-setting}. The main reason for such a choice of the regularization term is the fact that the optimizer of the regularized program \eqref{yeta:AC} admits an analytical expression.
		
		\begin{Lem}[Entropy maximization {\cite{ref:Csiszar-75}}] 
			\label{lem:entropy}
			Given a (measurable)  function $g:K\ra \R$ and the set $\Yb \subset \Meas_+(K)$ as defined in \eqref{Y} we have
			\begin{align*}
			y^\star(\diff k) \Let \arg\max_{y \in \Yb} \Big\{\inner{g}{y} - \eta d(y) \Big\} = \frac{\exp\big(\eta^{-1}g(k) \big) \Leb(\diff k)}{\inner{\exp\big(\eta^{-1}g(k)\big)}{\Leb}}.
			\end{align*}
		\end{Lem}
		
		Thanks to Lemma \ref{lem:entropy}, the analytical description of the dual optimizer in \eqref{yeta:AC} is readily available by setting
		\begin{align}
		\label{g}
		g(k) \Let [b-\opn \alpha](k) = -\cost(k) + \rho - \sum_{i=1}^n \alpha_i (Q-I)u_i(k).
		\end{align}		
		The last requirement to implement Algorithm~\ref{alg} is to verify Assumption~\ref{a:yeta}, i.e., we need to compute the Lipschitz constant of the mapping $(\rho,\alpha) \mapsto \opn^*\yeta(\rho,\alpha)$  in which the respective norm is $\|(\rho,\alpha)\| \Let \|\alpha\|_\Rnorm$. By definition of the adjoint operator $\opn^*$ in \eqref{Ln}, it is not difficult to observe that 
		\begin{align} 
		\label{Lny}
		\opn^* \yeta (\rho,\alpha) = 
		\left[ \begin{array}{c}
		\inner{-\ind}{\yeta (\rho,\alpha)} \\ \inner{(Q-I)u_1}{\yeta (\rho,\alpha)} \\ \vdots \\ \inner{(Q-I)u_n}{\yeta (\rho,\alpha)}
		\end{array} \right] 
		= 
		\left[ \begin{array}{c}
		-1\\ \inner{(Q-I)u_1}{\yeta (\rho,\alpha)} \\ \vdots \\ \inner{(Q-I)u_n}{\yeta (\rho,\alpha)}
		\end{array} \right].
		\end{align}
		The next lemma addresses the requirement of Assumption~\ref{a:yeta} for the mapping \eqref{Lny}. 
		
		\begin{Lem}[Lipschitz constant in MDP]
			\label{lem:Lip:AC}
			Consider the entropy maximizers in Lemma \ref{lem:entropy} with $g$ as defined in \eqref{g} and the adjoint operator in \eqref{Lny}. An upper bound for the Lipschitz constant in Assumption~\ref{a:yeta} is $L \le 4\ratio^2$ where the constant $\ratio$ is the equivalence ratio between the norms $\|\cdot\|_\Rnorm$ and $\|\cdot\|_{\ell_1}$ introduced in \eqref{opt}. 
		\end{Lem}
		
		\begin{proof}
			It is straightforward to see that \eqref{Lny} is differentiable with respect to the variable $(\rho, \alpha)$. Hence, it suffices to bound the norm of the matrix $\nabla \opn^* \yeta (\rho,\alpha) \in \R^{(n+1)\times(n+1)}$ uniformly on $(\rho,\alpha)$. Further, as the first element of the vector \eqref{Lny} is the constant 1, it only requires ro consider the gradient function with respect to the variable $\alpha \in \R^n$. A direct computation yields 
			\begin{align*}
			&|\big(\nabla_\alpha   \opn^* \yeta(\rho,\alpha) \big)_{ij} | \\ &\quad= \bigg| \frac{1}{\eta} \inner{(Q-I)u_i(Q-I)u_j}{\yeta(\rho,\alpha)} - \frac{1}{\eta}\inner{(Q-I)u_i}{\yeta(\rho,\alpha)} \inner{(Q-I)u_j}{\yeta(\rho,\alpha)} \bigg|\\
			& \quad= {4 \over \eta } \bigg | \inner{{(Q-I)u_i \over 2}{(Q-I)u_j \over 2}}{\yeta(\rho,\alpha)} - \inner{{(Q-I)u_i \over 2}}{\yeta(\rho,\alpha)} \inner{{(Q-I)u_j \over 2}}{\yeta(\rho,\alpha)} \bigg| \\
			& \quad\le  {4 \over \eta}, \qquad  \forall i,j \in \{1,\cdots,n\}.
			\end{align*}
			where the last inequality is a consequence of the Cauchy-Schwarz inequality and the fact that $\|(Q-I)u_j\|_\infty \le 2$ (recall that $Q$ is a stochastic kernel and all the basis functions are normalized). The Lipschitz constant of the desired mapping can then be upper bounded by
			\begin{align} 
			\label{grad-norm}
			{L \over \eta} & \leq \sup_{\tiny \begin{array}{cc} \|\alpha\|_\Rnorm \le \xnb \\ \|v\|_\Rnorm \le 1 \end{array}} \big\| \nabla_\alpha  \opn^* \yeta (\rho,\alpha) v \big\|_{\Rnorm*} 
			\le \sup_{\tiny \begin{array}{cc} \|\Phi_i\|_{\ell_\infty} \le 1 \\ \|v\|_\Rnorm \le 1 \end{array}}   {4 \over \eta} \Big \| (\Phi_1 \transp v, \cdots, \Phi_n \transp v) \Big \|_{\Rnorm^*}
			\end{align}	
			Recall that by the definition of the dual norm, we have $|\Phi_i \transp v| \le \|\Phi_i\|_{\Rnorm^*}$ for all $\|v\|_{\Rnorm}\le 1$. Thus, substituting the scalar variable $\mu_i \Let \Phi_i \transp v$ in the right-hand side of \eqref{grad-norm} and eliminating the factor $\eta$ lead to
			\begin{align*}
			L & \le \sup_{\|\Phi_i\|_{\ell_\infty} \le 1}  ~ \sup_{|\mu_i| \le \|\Phi_i\|_{\Rnorm^*}} {4} \big \| (\mu_1, \cdots, \mu_n) \big \|_{\Rnorm^*} = \sup_{|\mu_i| \le \ratio} {4} \big \| (\mu_1, \cdots, \mu_n)\big \|_{\Rnorm^*},
			\end{align*}
			where the last statement follows from the definition of the dual norm, and in particular the equality $$\sup_{\|\Phi_i\|_{\ell_\infty}\le 1} \|\Phi_i\|_{\Rnorm^*} = \sup_{\|\Phi_i\|_{\Rnorm}\le 1} \|\Phi_i\|_{\ell_1} =: \ratio.$$ Thus, using the same equality yields
			\[ L \le \sup_{\|\mu\|_{\ell_\infty} \le 1} {4 \ratio} \|\mu \|_{\Rnorm^*} = 4\ratio^2, \]
			which concludes the first desired assertion. 
		\end{proof}
		
		The performance of Algorithm~\ref{alg} can now be characterized through the following corollary. 
		
		\begin{Cor}[MDP smoothing approximation error]
			\label{cor:adp:smooth}
			Consider the operator \eqref{Ln} with the parameters described in \eqref{AC-setting} for the semi-infinite AC program \eqref{AC-LP-n}. Given this setting and the Lipschitz constant in Lemma~\ref{lem:Lip:AC}, we run Algorithm~\ref{alg} for $k$ iterations using the entropy function \eqref{d} with analytical solution \eqref{yeta:AC} as the prox-function. We define the constants 
			\begin{align*}
			C_1 \Let 2\e \big(\ratio \xnb (\max\{L_Q,1\}+1) + \|\cost\|_\lip\big), \qquad 
			C_2 \Let 4 \xnb \ratio^2 \sqrt{2\dim(K) \over \vartheta}.  
			\end{align*}
			For every $\eps \le C_1$ we set the smoothing factor $\eta$ and the number of iterations $k$ by
			\begin{align*}
			\eta \le {\eps \over 4 \dim(K)\log(C_1\eps^{-1})}, \qquad k \ge C_2 \frac{\sqrt{\log(C_1\eps^{-1})}}{\eps}.
			\end{align*} 
			Then, the outcome of Algorithm~\ref{alg} as defined in \eqref{Jn-algo} is an $\eps$ approximation of the optimal value $\Jacn$ in the sense of Theorem~\ref{thm:semi-fin:smooth}.
		\end{Cor}
		
		Corollary~\ref{cor:adp:smooth} requires one to compute the constants $c, C$ to quantify the a~priori bounds. The following two technical lemmas provide supplementary materials to address this issue.
		
		\begin{Lem}
			\label{lem:exp}
			Let $K \subseteq [0, 1]^m$ and $g: K \ra \R$ be a Lipschitz continuous function with constant $L_g > 0$ (with respect to the $\ell_\infty$-norm) and the maximum value $\gmax \Let \max_{k \in K} g(k)$. Then, for every $\eta > 0$ we have 
			\begin{align*}
			\int_{K} \exp\Big({\eta^{-1}\big(g(k) - \gmax\big)}\Big)\diff k \ge \min\Big\{\Big({m \eta \over L_g}\Big)^m, 1\Big\} \exp \big( - \min\big\{m, L_g  \eta^{-1}\big\} \big).
			\end{align*}
		\end{Lem}
		
		\begin{proof}
			Let us define the set $Z(\delta) \Let \{ k \in K : \gmax - g(k) < \delta\}$. Thanks to the Lipschitz continuity of the function $g$, we have $\gmax - g(k) \le L_g \|k\opt - k\|_{\ell_\infty}$ where $g(k\opt) = \gmax$. Thus, using this inequality one can bound the size of the set $Z(\delta)$ in the sense of 
			$$\int_{Z(\delta)} \diff k \ge \min\{(\delta L_g^{-1})^{m}, 1\}, \qquad \forall \delta \ge 0.$$
			By virtue of the above result, one can observe that for every $\delta > 0$
			\begin{align*}
			\int_{K} \exp\Big({\eta^{-1}\big(g(k) - \gmax\big)}\Big)\diff k & \ge \int_{Z(\delta)} \exp\Big({\eta^{-1}\big(g(k) - \gmax \big)}\Big)\diff k \\
			& \ge \exp(-\eta^{-1}\delta) \int_{Z(\delta)} \diff k \ge \exp(-\eta^{-1}\delta)\min\{(\delta L_g^{-1})^{m},1\}.
			\end{align*}
			Maximizing the right-hand side of the above inequality over $\delta$ suggests to set $\delta = \min\{m \eta, L_g\}$, which yields the desired assertion.
		\end{proof}
		
		In light of Lemma~\ref{lem:exp}, we can bound the entropy prox-function \eqref{d} evaluated at the optimizer \eqref{yeta:AC}. 
		
		\begin{Lem}[Entropy prox-bound]
			\label{lem:entropy:bound}
			Consider the prox-function \eqref{d} and let $\yeta(\rho,\alpha)$ be the optimizer of \eqref{yeta:AC}. Then, for every $\eta > 0$, $\rho$, and $\|\alpha\|_\Rnorm \le \xnb$, we have $d\big(\yeta(\rho, \alpha)\big) \le C\max\big\{\log(c\eta^{-1}),1\big\}$ where
			\begin{align*}
			\qquad C \Let \dim(K), \qquad c \Let {\e \over \dim(K)} \big(\xnb \ratio (\max\{L_Q,1\}+ 1)+ \|\cost\|_\lip \big),
			\end{align*}
			and $\ratio$ is the equivalence ratio between the norms $\|\cdot\|_{\ell_1}$ and $\|\cdot\|_\Rnorm$ as defined in \eqref{opt}. 
		\end{Lem}
		
		\begin{proof}
			The result is a direct application of Lemma~\ref{lem:exp}. Consider the function $g$ as defined in \eqref{g} with Lipschitz constant $L_g \ge 0$; note that the function $g$, as well as its Lipschitz constant $L_g$, depends also on the pair $(\rho,\alpha)$. Observe that 
			\begin{align}
			d\big(\yeta(\rho,\alpha)\big) &=  \inner{\log\big(\exp(\eta^{-1}g)\big)}{\yeta(\rho,\alpha)} - \log\big(\inner{\exp(\eta^{-1}g)}{\Leb}\big)\notag\\
			& = \inner{\eta^{-1}g}{\yeta(\rho,\alpha)} - \log\big(\inner{\exp(\eta^{-1}g)}{\Leb}\big)\notag\\
			& = \inner{\eta^{-1}g}{\yeta(\rho,\alpha)} - \eta^{-1}\gmax - \log\big(\inner{\exp(\eta^{-1}(g - \gmax)}{\Leb}\big) \notag\\
			&\le -\log\big(\inner{\exp(\eta^{-1}(g-\gmax))}{\Leb}\big) \notag\\
			& \le - \log\bigg( \min\Big\{\Big({\dim(K) \eta \over L_g}\Big)^{\dim(K)}, 1\Big\} \exp \big( - \min\big\{\dim(K), L_g  \eta^{-1}\big\} \big)\bigg) \label{eq:lemma:exp} \\
			& \le\dim(K)\max\Big\{\log\Big(\Big({\e L_g\over \dim(K)}\Big)\eta^{-1}\Big),1\Big\}\notag
			\end{align}
			where the inequality \eqref{eq:lemma:exp} follows from Lemma~\ref{lem:exp}. Note also that the Lipschitz constant $L_g$ for the function $g$ defined in \eqref{g} is upper bounded, uniformly in $(\rho,\alpha)$ where $\|\alpha\|_\Rnorm \le \xnb$, by
			\begin{align*}
			L_g \le \|g - \rho\|_{\lip} &\le \Big\| \sum_{i=1}^{n} \alpha_i(Q-I)u_i + \cost\Big\|_\lip \le (\max\{L_Q,1\}+1)\Big\|\sum_{i=1}^{n}\alpha_i u_i \Big\|_\lip + \|\cost\|_\lip \\
			&\le \xnb \ratio (\max\{L_Q,1\}+ 1)+ \|\cost\|_\lip.
			\end{align*} 
			We refer to the proof of Corollary~\ref{cor:adp:semi-finite}, and in particular the paragraph following \eqref{Lg}, for further discussions regarding $L_g$. The desired assertion follows from the last two inequalities and the definition of the constant $\ynb$ in \eqref{Y}.
		\end{proof} 	
		
		The proof of Corollary~\ref{cor:adp:smooth} follows by replacing the constants in Lemma~\ref{lem:entropy:bound} in Theorem~\ref{thm:semi-fin:smooth}. By contrast to the randomized approach in Corollary~\ref{cor:adp:semi-finite} where the computational complexity scales exponentially in dimensional of state-action space, the complexity of the smoothing technique grows effectively linearly (more precisely $\order\big(\eps^{-1}\sqrt{\log(\eps^{-1})}\big)$, cf. Remark~\ref{rem:complex}). The computational difficulty is, however, transferred to Step 1 of Algorithm~\ref{alg} for computation of $\opn^* \yeta$ as defined in \eqref{Lny}. The following remark elaborates this.
		
		\begin{Rem}[Efficient computation of \eqref{Lny}]
			When the transition kernel $Q$ and the basis functions $u_{i}$ are such that the relation \eqref{Lny} involves integration of exponentials of polynomials over simple sets (e.g., box or a simplex), one may utilize efficient methods that require solving a hierarchy of semidefinite programming problems to generate upper and lower bounds which asymptotically converge to the true value of integral, see \cite[Section~12.2]{ref:Lasserre-11} and \cite{ref:Bertsimas-08}. It is also worth noting that a straightforward computation of \eqref{Lny} for a small parameter $\eta$ may be numerically difficult due to the exponential functions. This issue can, however, be circumvented by a numerically stable technique presented in \cite[p.~148]{ref:Nest-05}. 
		\end{Rem}
		
		Regarding the choice of $\xnb$, in similar spirit to Section~\ref{sec:semi-fin:rand}, one can target minimizing the complexity of the a priori bound, in other words the number of iterations $k$ in \eqref{eta-k}. In the setting of  Corollary~\ref{cor:adp:smooth}, one can observe that the smaller the parameter $\xnb$, the lower the number of the required iterations, leading to the choice described as in \eqref{theta*:MDP}.

\section{Numerical examples} 
\label{sec:sim}
We present two numerical examples to illustrate the solution methods and corresponding performance bounds. Throughout this section we consider the norm $\|\cdot\|_\Rnorm = \|\cdot\|_{\ell_2}$, leading to $\ratio = \sqrt{n}$ in \eqref{opt}, and we choose the Fourier basis functions. 

\subsection{Example 1: truncated LQG} 
\label{ex:LQG}
Consider the linear system
	\begin{align*}
	s_{t+1}=\vartheta s_{t} + \rho a_{t} + \xi_{t}, \quad t \in \N,
	\end{align*}
with quadratic stage cost $\psi(s,a)=qs^{2}+ra^{2}$, where $q\geq 0$ and $r>0$ are given constants. We assume that $S=A=[-L,L]$ and the parameters $\vartheta, \rho \in\R$ are known. The disturbances $\{\xi_{t}\}_{t\in\N}$ are i.i.d.\ random variables generated by a truncated normal distribution with known parameters $\mu$ and $\sigma$, independent of the initial state $s_{0}$. Thus, the process $\xi_{t}$ has a distribution density
	\begin{align*}
	f(s,\mu,\sigma,L) = \left\{ 
	\begin{array}{cc} 
	\frac{\frac{1}{\sigma}\phi\left( \frac{s-\mu}{\sigma} \right)}{\Phi\left( \frac{L-\mu}{\sigma} \right)-\Phi\left( \frac{-L-\mu}{\sigma} \right)}, \quad & s\in[-L,L]\\
	0 & \text{o.w.},
	\end{array} \right.
	\end{align*}
where $\phi$ is the probability density function of the standard normal distribution, and $\Phi$ is its cumulative distribution function. The transition kernel $Q$ has a density function $q(y|s,a)$, i.e., $Q(B|s,a)=\int_{B}q(y|s,a)\drv y$ for all $B\in\Borel{S}$, that is given by
	\begin{align*}
	q(y|s,a)=f(y-\vartheta s - \rho a,\mu,\sigma,L).
	\end{align*}
In the special case that $L=+\infty$ the above problem represents the classical LQG problem, whose solution can be obtained via the algebraic Riccati equation \cite[p.~372]{ref:Bertsekas-12}. By a simple change of coordinates it can be seen that the presented system fulfills Assumption~\ref{a:CM}. The following lemma provides the technical parameters required for the proposed error bounds.
		
\begin{Lem}[Truncated LQG properties] \label{lem:LQG}
The error bounds provided by Corollaries~\ref{cor:adp:semi-finite} and \ref{cor:adp:smooth} hold with the norms $\|\psi\|_\infty = L^2(q+r)$, $\|\psi\|_\lip = 4L^2\sqrt{q^2+r^2}$, and the Lipschitz constant of the kernel is
	\begin{align*} 
	L_Q &=\frac{2L\max\{\vartheta,\rho\}}{\sigma^2\sqrt{2\pi}\left( \Phi\left( \frac{L-\mu}{\sigma} \right)-\Phi\left( \frac{-L-\mu}{\sigma} \right)\right)}\,.
	\end{align*}
\end{Lem}

\begin{proof}
	In regard to Assumption~\ref{a:CM}\ref{a:CM:K}, we consider the change of coordinates $\bar{s}_t \Let \frac{s_t}{2L}+\frac{1}{2}$ and $\bar{a}_t \Let \frac{a_t}{2L}+\frac{1}{2}$. In the new coordinates, the constants of Lemma~\ref{lem:LQG} follow from a standard computation.
\end{proof}

\paragraph{\emph{Simulation details:}}
For the simulation results we choose the numerical values $\vartheta = 0.8$, $\rho = 0.5$, $\sigma = 1$, $\mu = 0$, $q = 1$, $r = 0.5$, and $L=10$. In the first approximation step discussed in Section~\ref{subsec:semi:MDP}, we consider the Fourier basis $u_{2k-1}(s) = \frac{L}{k\pi}\cos\left(\frac{k \pi s}{L}\right)$ and $u_{2k}(s) = \frac{L}{k\pi}\sin\left(\frac{k \pi s}{L}\right)$. 
	
\begin{figure}[t]
	\subfigure[$n=2$ basis functions]{\scalebox{1}{
%
%
\definecolor{mycolor1}{rgb}{0.13, 0.55, 0.13}%
\definecolor{mycolor2}{rgb}{0.13, 0.55, 0.13}%

\definecolor{mycolor3}{rgb}{0.32,0.09,0.98}%
\definecolor{mycolor4}{rgb}{0.32,0.09,0.98}%
\begin{tikzpicture}

\begin{axis}[%
width=2.2in,
height=1.6in,
at={(1.011111in,0.813889in)},
scale only axis,
xmode=log,
xmin=1,
xmax=100000,
xminorticks=true,
xlabel={$N$},
xmajorgrids,
xminorgrids,
ymode=log,
ymin=0.5,
ymax=1000000,
ylabel={$J^{\text{AC}}_{n,N}$},
ymajorgrids,
ylabel style={yshift=-0.3cm},
legend style={legend cell align=left,align=left,draw=white!15!black,line width=1.0pt,font=\footnotesize}
]

\addplot[area legend,solid,fill=mycolor1,opacity=0.3,draw=none,forget plot]
table[row sep=crcr] {%
x	y\\
1	28.5805565420901\\
2	4.47632011102824\\
3	3.63625813510152\\
4	2.57054306764981\\
5	2.5705430451617\\
6	2.51737844086245\\
7	1.62086718607834\\
8	1.62086707873988\\
9	1.62086722353572\\
10	1.62086742950948\\
20	1.4816911948426\\
30	1.48169119341355\\
40	1.33395592104315\\
50	1.33395591974424\\
60	1.26186486900902\\
70	1.25936985320541\\
80	1.24252029228155\\
90	1.24252030145971\\
100	1.24177255933975\\
200	1.22146003433544\\
300	1.22138593743673\\
400	1.22138593696591\\
500	1.21990362266073\\
600	1.21984601184503\\
700	1.21984598782825\\
800	1.21984599252904\\
900	1.21851742859662\\
1000	1.21851742871054\\
2000	1.218517423794\\
3000	1.21851743224633\\
4000	1.21789756009852\\
5000	1.21784630283811\\
6000	1.21784630155865\\
7000	1.21780528641541\\
8000	1.21760200049366\\
9000	1.21744678652966\\
10000	1.21744678649881\\
90000      1.21744678649881\\
90000	1.2367915188808\\
10000	1.2367915188808\\
9000	1.23679151890685\\
8000	1.23874581894327\\
7000	1.24546711731161\\
6000	1.25055594704458\\
5000	1.250555946855\\
4000	1.25815881034301\\
3000	1.26872506989623\\
2000	1.30843172824957\\
1000	1.40877331859085\\
900	1.43408597381899\\
800	1.44803659772777\\
700	1.50904655068873\\
600	1.62281061306987\\
500	1.80402902141993\\
400	2.00288513809667\\
300	2.28684318616537\\
200	3.83320052838975\\
100	5.30781654327367\\
90	8.01504564206975\\
80	8.01504561930379\\
70	8.30009646023389\\
60	9.06756712328384\\
50	9.54571340420395\\
40	10.7947161495485\\
30	14.9844606772684\\
20	23.479991472872\\
10	70.1733036265233\\
9	78.1410743140534\\
8	91.9028975758403\\
7	94.797941357001\\
6	94.7979413611145\\
5	94.797941279872\\
4	98.2305489066226\\
3	99.1672795005246\\
2	99.991713310573\\
1	100.000001529317\\
}--cycle;

\addplot [color=mycolor2, solid,line width=1.5pt]
  table[row sep=crcr]{%
1	90.0988050872139\\
2	74.8559092713465\\
3	58.3740012753949\\
4	45.7955902519346\\
5	35.8635147694922\\
6	28.5860308398038\\
7	22.5827573210655\\
8	18.0333956261914\\
9	15.3877213318488\\
10	13.4436958792473\\
20	6.14068215774133\\
30	4.34954158372674\\
40	3.46122965614533\\
50	2.97739665229225\\
60	2.62103541037543\\
70	2.38307666634709\\
80	2.22704450558327\\
90	2.06498996159715\\
100	1.96438382986823\\
200	1.52532404666544\\
300	1.39413594434313\\
400	1.34570792651888\\
500	1.31627993472894\\
600	1.29786965282214\\
700	1.28371427700216\\
800	1.27386956207888\\
900	1.26695036359398\\
1000	1.26243951330834\\
2000	1.24033826626804\\
3000	1.23266956870538\\
4000	1.22871956187396\\
5000	1.22633726262058\\
6000	1.22479092312255\\
7000	1.22379566139958\\
8000	1.22287742352561\\
9000	1.22215100223844\\
10000	1.22162047102615\\
90000	1.22162047102615\\
};
\addlegendentry{$\xnb$ in \eqref{theta*:MDP2}};

\addplot[area legend,solid,fill=mycolor3,opacity=0.25,draw=none,forget plot]
table[row sep=crcr] {%
x	y\\
1	274221.033000622\\
2	469.613556032881\\
3	3.48721473507746\\
4	2.85193524555087\\
5	2.8519364270079\\
6	2.5965971629727\\
7	1.95753391047949\\
8	1.9575541777959\\
9	1.95755429557138\\
10	1.63669708450418\\
20	1.53433558516822\\
30	1.38210203317975\\
40	1.38210216885707\\
50	1.29335500621487\\
60	1.29335470494149\\
70	1.29335432955327\\
80	1.2933542216002\\
90	1.29335455602246\\
100	1.2466505804252\\
200	1.22404341249512\\
300	1.22404290796734\\
400	1.22053012115844\\
500	1.21990358559088\\
600	1.21984488819389\\
700	1.21984599034574\\
800	1.2198459844731\\
900	1.21984548946883\\
1000	1.21943079603014\\
2000	1.21943125742712\\
3000	1.21910571951159\\
4000	1.2177775156776\\
5000	1.21777707029804\\
6000	1.2177771371855\\
7000	1.21777713736774\\
8000	1.21760200096946\\
9000	1.21744678586919\\
10000	1.21744678595159\\
90000	1.21744678595159\\
90000	1.23327839942864\\
10000	1.23327839942864\\
9000	1.23329468709275\\
8000	1.24137528539087\\
7000	1.24137527394481\\
6000	1.24137528421462\\
5000	1.24494923824726\\
4000	1.25882152897169\\
3000	1.29000970091385\\
2000	1.29949424210475\\
1000	1.40877323061208\\
900	1.43408596392102\\
800	1.43542924566361\\
700	1.50904653496347\\
600	1.56708936228215\\
500	1.75134995924504\\
400	2.00288506198568\\
300	2.39057443930394\\
200	2.79984469509506\\
100	5.30961100123332\\
90	8.01503965956585\\
80	8.01503700224842\\
70	8.30009582075411\\
60	9.06756522040735\\
50	9.54571331006024\\
40	13.0100271086148\\
30	15.8465288076035\\
20	26.0095176807156\\
10	516138.231060319\\
9	563124.099372117\\
8	563124.102533274\\
7	817984.800975262\\
6	817984.800673086\\
5	891470.172514893\\
4	919705.358565746\\
3	921882.29880044\\
2	941777.972236723\\
1	945309.87173611\\
}--cycle;

\addplot [color=mycolor4,dashed,line width=1.5pt]
  table[row sep=crcr]{%
1	815296.000743184\\
2	570766.640331088\\
3	330784.106503749\\
4	187151.022981435\\
5	112220.758936073\\
6	65285.3958875767\\
7	34921.2444390044\\
8	15847.8435941127\\
9	10790.164414946\\
10	6999.1311452946\\
20	6.46017233573243\\
30	4.40786153197229\\
40	3.57754254365448\\
50	3.03359733715022\\
60	2.71604155729113\\
70	2.4146911555697\\
80	2.2565160912416\\
90	2.08240316384355\\
100	1.97999087837892\\
200	1.53403928117415\\
300	1.40535725279897\\
400	1.35232089143797\\
500	1.32083068402102\\
600	1.3009730333718\\
700	1.28655650264788\\
800	1.27561693749942\\
900	1.26919993033891\\
1000	1.26373398482813\\
2000	1.24143831516366\\
3000	1.23353693215556\\
4000	1.22911257498496\\
5000	1.22657498445719\\
6000	1.22497392900458\\
7000	1.22395425939084\\
8000	1.22310237614617\\
9000	1.22221048015023\\
10000	1.22176594472622\\
90000	1.22176594472622\\
};
\addlegendentry{$\xnb=\infty$}

\addplot [color=red,loosely dotted,line width=2pt]
  table[row sep=crcr]{%
1	1.3187\\
90000	1.3187\\
};
\addlegendentry{$\Jac$}

\end{axis} 
\end{tikzpicture}%
}\label{fig:LQG:ran:2}}
	\qquad 
	\subfigure[$n=10$ basis functions]{\scalebox{1}{
%
%
\definecolor{mycolor1}{rgb}{0.13, 0.55, 0.13}%
\definecolor{mycolor2}{rgb}{0.13, 0.55, 0.13}%

\definecolor{mycolor3}{rgb}{0.32,0.09,0.98}%
\definecolor{mycolor4}{rgb}{0.32,0.09,0.98}%
\begin{tikzpicture}

\begin{axis}[%
width=2.2in,
height=1.6in,
at={(1.011111in,0.813889in)},
scale only axis,
xmode=log,
xmin=1,
xmax=100000,
xminorticks=true,
xlabel={$N$},
xmajorgrids,
xminorgrids,
ymode=log,
ymin=0.5,
ymax=1000000,
ylabel={$J^{\text{AC}}_{n,N}$},
ymajorgrids,
ylabel style={yshift=-0.3cm},
legend style={legend cell align=left,align=left,draw=white!15!black,line width=1.0pt,font=\footnotesize}
]

\addplot[area legend,solid,fill=mycolor1,opacity=0.3,draw=none,forget plot]
table[row sep=crcr] {%
x	y\\
1	64.8626787674932\\
2	51.8028663321495\\
3	40.8698505042221\\
4	34.2057123029274\\
5	32.327740150497\\
6	26.5906149656799\\
7	21.8167512994986\\
8	18.2942570690963\\
9	17.509433217615\\
10	16.5109237494213\\
20	8.68290108849987\\
30	5.49612925906822\\
40	2.35413086420275\\
50	2.32862407932307\\
60	2.12306400135207\\
70	1.86120000680995\\
80	1.63535781691707\\
90	1.63535781690026\\
100	1.55959083049867\\
200	1.39333461801874\\
300	1.3643441040029\\
400	1.35394428022446\\
500	1.34195332050755\\
600	1.33856909460061\\
700	1.32859904928777\\
800	1.32458818880617\\
900	1.32445301927657\\
1000	1.32445302330009\\
2000	1.31914600579702\\
3000	1.31879268391564\\
4000	1.31862187925411\\
5000	1.31844109127923\\
6000	1.31830688543409\\
7000	1.31826094723498\\
8000	1.31826094843103\\
9000	1.3182005639989\\
100000	1.31813536293608\\
100000	1.32263803288161\\
9000	1.32500128260116\\
8000	1.32539113783232\\
7000	1.32655764138582\\
6000	1.32883437777301\\
5000	1.33550740507875\\
4000	1.34085918309822\\
3000	1.37977303081213\\
2000	1.41450337883357\\
1000	1.82423091203754\\
900	1.84987753113501\\
800	1.94883222433922\\
700	2.04079150965957\\
600	2.06445841923901\\
500	2.34248038369511\\
400	2.96176544947812\\
300	3.88985563333783\\
200	5.47246905081519\\
100	12.1911192973486\\
90	13.3995722506444\\
80	15.4831188210459\\
70	20.3527647402716\\
60	22.5255326834499\\
50	27.9046058790261\\
40	32.1987850816379\\
30	42.5032034823197\\
20	59.1930381508071\\
10	84.0304606027895\\
9	86.9142744662561\\
8	86.9142745304304\\
7	93.2294273884365\\
6	94.253548028958\\
5	97.06945360186\\
4	98.4481209823781\\
3	99.2617095114665\\
2	99.9999994398324\\
1	100.000002879261\\
}--cycle;

\addplot [color=mycolor2, solid,line width=1.5pt]
  table[row sep=crcr]{%
1	94.5768870629217\\
2	87.6422421294588\\
3	81.1060535945275\\
4	75.7166598326944\\
5	70.8550067501774\\
6	65.818793100411\\
7	61.4661914492361\\
8	57.5867852858906\\
9	54.506862634783\\
10	51.0690016796575\\
20	30.1255056462458\\
30	19.8540579079187\\
40	14.2465207906624\\
50	10.8717423880872\\
60	8.85011139713774\\
70	7.25012839633849\\
80	6.23207351231921\\
90	5.3655282834809\\
100	4.77894898497726\\
200	2.44822770289855\\
300	1.87735697901229\\
400	1.67372096917809\\
500	1.56577991405623\\
600	1.49838919411101\\
700	1.45440220544531\\
800	1.42286573705222\\
900	1.40333873367853\\
1000	1.38702187301104\\
2000	1.33566597760556\\
3000	1.32648669905007\\
4000	1.32290226192807\\
5000	1.32128509023382\\
6000	1.32039006435097\\
7000	1.31992148967192\\
8000	1.31948814984273\\
9000	1.31924603433781\\
100000	1.31906764370644\\
};
\addlegendentry{$\xnb$ in \eqref{theta*:MDP2}};


\addplot[area legend,solid,fill=mycolor3,opacity=0.25,draw=none,forget plot]
table[row sep=crcr] {%
x	y\\
1	639983.366524608\\
2	329098.537697569\\
3	270518.925077757\\
4	162716.677544189\\
5	99203.2256406953\\
6	86226.4301855034\\
7	47489.5274014969\\
8	47489.5282154978\\
9	30133.0276457532\\
10	21613.4199719176\\
20	7.0906145341405\\
30	4.65928077970097\\
40	2.39043513144145\\
50	2.3286038154872\\
60	2.12306379225585\\
70	2.09667479516671\\
80	1.635356391661\\
90	1.63535708245283\\
100	1.55958962460745\\
200	1.42441044263701\\
300	1.3729378008794\\
400	1.34201033025407\\
500	1.34200855870122\\
600	1.33393016386478\\
700	1.32548079869412\\
800	1.3245880787709\\
900	1.32445276232808\\
1000	1.32067603411791\\
2000	1.31914752559519\\
3000	1.31890352096265\\
4000	1.31862187891315\\
5000	1.3184410887336\\
6000	1.31830689055684\\
7000	1.31826094801364\\
8000	1.31820169759628\\
9000	1.31816451028211\\
100000	1.31813534897557\\
100000	1.32635030823924\\
9000	1.3263503050298\\
8000	1.32709601638159\\
7000	1.32710598610841\\
6000	1.32938918012012\\
5000	1.33621956240585\\
4000	1.3408591123418\\
3000	1.38643803937312\\
2000	1.41323205875657\\
1000	1.7807446249738\\
900	1.78586817791992\\
800	1.94883224270823\\
700	1.95467052048799\\
600	2.19564446831188\\
500	2.376049145298\\
400	2.84118283785502\\
300	4.42306322890742\\
200	5.948249385366\\
100	12.1234649487366\\
90	13.2643994103875\\
80	14.8895276860826\\
70	20.6106203279187\\
60	22.9578071142619\\
50	18731.5689369941\\
40	46154.8527946401\\
30	122416.784984592\\
20	245819.513244859\\
10	657237.808452924\\
9	733798.353290344\\
8	733798.349269297\\
7	879520.407606198\\
6	884146.221461426\\
5	916423.123492617\\
4	921568.783455853\\
3	938169.047545757\\
2	948954.054124091\\
1	949823.811310828\\
}--cycle;

\addplot [color=mycolor4,dashed,line width=1.5pt]
  table[row sep=crcr]{%
1	881068.308142879\\
2	753794.288347249\\
3	651989.021149723\\
4	561892.86158924\\
5	489715.217209112\\
6	427140.172458315\\
7	373606.121950658\\
8	324733.263256534\\
9	277359.534656746\\
10	235953.292052392\\
20	42857.0692688388\\
30	4299.73531162478\\
40	503.880225630512\\
50	57.5866681916299\\
60	8.75517131945854\\
70	7.16879104175466\\
80	6.03343592545909\\
90	5.30458286475803\\
100	4.74690136159816\\
200	2.47432622510224\\
300	1.90523920521026\\
400	1.67809177228141\\
500	1.56928466228527\\
600	1.50170304052627\\
700	1.44790616627579\\
800	1.42030204832343\\
900	1.40220936819888\\
1000	1.38672104646169\\
2000	1.3356980463223\\
3000	1.32613696787708\\
4000	1.32267627846345\\
5000	1.32124562408155\\
6000	1.32037996927841\\
7000	1.31984030493074\\
8000	1.31951848108317\\
9000	1.31925491020946\\
100000	1.31909475493793\\
};
\addlegendentry{$\xnb=\infty$}

\addplot [color=red,loosely dotted,line width=2pt]
  table[row sep=crcr]{%
1	1.3187\\
100000	1.3187\\
};
\addlegendentry{$\Jac$}

\end{axis}
\end{tikzpicture}%
}\label{fig:LQG:ran:10}}
	\\	
	\subfigure[$n=100$ basis functions]{\scalebox{1}{
%
%
\definecolor{mycolor1}{rgb}{0.13, 0.55, 0.13}%
\definecolor{mycolor2}{rgb}{0.13, 0.55, 0.13}%

\definecolor{mycolor3}{rgb}{0.32,0.09,0.98}%
\definecolor{mycolor4}{rgb}{0.32,0.09,0.98}%
\begin{tikzpicture}

\begin{axis}[%
width=2.2in,
height=1.6in,
at={(1.011111in,0.813889in)},
scale only axis,
xmode=log,
xmin=1,
xmax=100000,
xminorticks=true,
xlabel={$N$},
xmajorgrids,
xminorgrids,
ymode=log,
ymin=0.5,
ymax=1000000,
ylabel={$J^{\text{AC}}_{n,N}$},
ymajorgrids,
ylabel style={yshift=-0.3cm},
legend style={legend cell align=left,align=left,draw=white!15!black,line width=1.0pt, font=\footnotesize}
]

\addplot[area legend,solid,fill=mycolor1,opacity=0.3,draw=none,forget plot]
table[row sep=crcr] {%
x	y\\
1	73.0725008122354\\
2	61.4824626630137\\
3	51.1246287160035\\
4	44.1816795345505\\
5	41.624170651969\\
6	41.6241708288271\\
7	41.3325160220246\\
8	41.3026068831755\\
9	40.6233055772378\\
10	37.8298523875866\\
20	27.1803843820131\\
30	24.5439377663766\\
40	21.1587220918693\\
50	18.5528520648749\\
60	16.6437499563526\\
70	16.322223991574\\
80	15.8713158532552\\
90	14.8174019929416\\
100	13.3405730708024\\
200	10.0409771992707\\
300	8.04018316672937\\
400	7.60747751315085\\
500	6.78171835223416\\
600	5.89411527131351\\
700	5.34601605040283\\
800	5.02765091315349\\
900	4.77759951871282\\
1000	4.33511024741078\\
2000	2.85423098452847\\
3000	2.49758098138473\\
4000	2.13430309683285\\
5000	1.90766192491257\\
6000	1.87269917912076\\
7000	1.7562695919021\\
8000	1.7135652779055\\
9000	1.69424841286886\\
10000	1.69378703711643\\
20000	1.47208388988814\\
30000	1.42459098794049\\
40000	1.39714278405354\\
50000	1.38376127859128\\
60000	1.3739062869345\\
70000	1.37355295449137\\
80000	1.36373362951018\\
90000	1.35874593505327\\
90000	1.44662276075467\\
80000	1.45652639971006\\
70000	1.47712914703438\\
60000	1.51349283270477\\
50000	1.53008822831339\\
40000	1.61330074144611\\
30000	1.68707340458998\\
20000	1.86627528768281\\
10000	2.3131049866413\\
9000	2.42296812312227\\
8000	2.61777890277017\\
7000	2.68786005791279\\
6000	2.89127196564123\\
5000	3.510806434344\\
4000	3.88588637605768\\
3000	4.44770418607772\\
2000	5.55554394108026\\
1000	8.3041309640639\\
900	8.61069152874115\\
800	9.22388343498101\\
700	9.73895406512103\\
600	10.0887545803061\\
500	11.7423421318748\\
400	13.3781368228904\\
300	15.0578515420467\\
200	19.1904050392464\\
100	26.5167653669121\\
90	26.8915076038712\\
80	29.9242066761444\\
70	31.7378829721566\\
60	36.5436400216705\\
50	41.2533356354866\\
40	43.1544597275795\\
30	49.6880996420488\\
20	65.6018066727078\\
10	86.0282494739563\\
9	88.9376572262689\\
8	92.1518901496837\\
7	93.9846699180999\\
6	95.0637467475903\\
5	96.2394683425346\\
4	99.196295378522\\
3	99.1962963229621\\
2	99.93608188773\\
1	99.9999999628775\\
}--cycle;

  \addplot [color=mycolor2, solid,line width=1.5pt]
  table[row sep=crcr]{%
1	95.0700417717928\\
2	90.0671167740208\\
3	84.9944490270267\\
4	79.9920905070121\\
5	76.0027072571778\\
6	72.0642442192834\\
7	68.5216005663884\\
8	65.3043977366562\\
9	62.3114318130003\\
10	59.4892715267976\\
20	43.4644954999383\\
30	35.5965024418207\\
40	30.8718050780024\\
50	27.5461133266723\\
60	25.1746912498526\\
70	23.3269440347622\\
80	21.782611418054\\
90	20.5563711968901\\
100	19.5603805628542\\
200	13.9065438625585\\
300	11.4420274150589\\
400	9.97217331313981\\
500	8.8661674800461\\
600	8.08608694410586\\
700	7.44042640087966\\
800	6.93601395110263\\
900	6.49772069280504\\
1000	6.1011506875886\\
2000	4.12697750456289\\
3000	3.30182507653445\\
4000	2.84903992360772\\
5000	2.56617423483291\\
6000	2.37412705435564\\
7000	2.22851604396447\\
8000	2.11901250524476\\
9000	2.03455393456047\\
10000	1.96698930950154\\
20000	1.64938776316658\\
30000	1.54311305739475\\
40000	1.48459019402139\\
50000	1.45122052152376\\
60000	1.42868871164553\\
70000	1.41198856115809\\
80000	1.39930609641481\\
90000	1.39159597512432\\
};
\addlegendentry{$\xnb$ in \eqref{theta*:MDP2}};

\addplot[area legend,solid,fill=mycolor3,opacity=0.25,draw=none,forget plot]
table[row sep=crcr] {%
x	y\\
1	726032.932683365\\
2	575497.422786663\\
3	487235.424507427\\
4	433795.962742556\\
5	373175.234062319\\
6	311798.973653394\\
7	268995.513623822\\
8	258287.077277272\\
9	227632.440252786\\
10	220081.695430595\\
20	112157.835636989\\
30	87800.1396742205\\
40	65844.4903492122\\
50	50680.4553501273\\
60	43162.6798842125\\
70	35355.7953123166\\
80	32661.7650283694\\
90	26002.954959635\\
100	22773.0465341619\\
200	2411.72466093664\\
300	21.0573383250428\\
400	16.046915139724\\
500	12.8109912873503\\
600	11.2102195393165\\
700	10.7167490565531\\
800	10.1411754565593\\
900	9.30056376015061\\
1000	8.2707566654869\\
2000	5.24154994749335\\
3000	4.03795564502112\\
4000	3.26947862218561\\
5000	3.03295845707396\\
6000	2.77332330384995\\
7000	2.49505756211084\\
8000	2.36738517154882\\
9000	2.25014532078066\\
10000	2.21109518688523\\
20000	1.69815240236852\\
30000	1.62161878310671\\
40000	1.51748252946769\\
50000	1.47201423243165\\
60000	1.45660169781521\\
70000	1.44077296343534\\
80000	1.42369917319803\\
90000	1.40053639945081\\
90000	1.59392070952854\\
80000	1.62476181187348\\
70000	1.68447114549425\\
60000	1.7304245328959\\
50000	1.82134680815319\\
40000	1.97163710481217\\
30000	2.20338202233787\\
20000	2.47092552819889\\
10000	3.66438170316786\\
9000	3.88629035695664\\
8000	4.37476900292604\\
7000	4.48013062748919\\
6000	4.82971826032916\\
5000	5.46485996971368\\
4000	6.26061109589694\\
3000	7.79724788905368\\
2000	10.3707742851764\\
1000	14.5345960731987\\
900	15.2854850836583\\
800	16.928960911425\\
700	19.0934194379443\\
600	38.9268305307339\\
500	607.995945437792\\
400	2447.8486587807\\
300	10365.3366176473\\
200	25996.4007632041\\
100	52013.7262786537\\
90	65575.2480556314\\
80	76659.0671714717\\
70	87594.2707644928\\
60	114733.56462992\\
50	166787.067405686\\
40	170002.388063472\\
30	306333.320936457\\
20	395872.245625063\\
10	677071.315309723\\
9	699076.981019667\\
8	701474.473762693\\
7	822322.595975613\\
6	911948.680060861\\
5	923400.838429024\\
4	936081.988079509\\
3	945357.985362924\\
2	947776.774626934\\
1	951818.750313226\\
}--cycle;

\addplot [color=mycolor4,dashed,line width=1.5pt]
  table[row sep=crcr]{%
1	898168.581130224\\
2	802140.564420225\\
3	711397.054174735\\
4	631227.507176506\\
5	564596.867924508\\
6	509081.327498531\\
7	461185.52507959\\
8	419045.290316326\\
9	388003.069338418\\
10	354681.204123067\\
20	195414.667488606\\
30	133418.7035383\\
40	99905.9936620545\\
50	79327.3864603108\\
60	65543.7009034799\\
70	55436.5410869081\\
80	47602.4289447553\\
90	41155.4071614626\\
100	36113.5661551971\\
200	10498.8978786933\\
300	1302.02082425537\\
400	97.9984650481927\\
500	24.2219285693239\\
600	16.2337383070317\\
700	14.4287485468089\\
800	13.1608781833076\\
900	12.1746393095505\\
1000	11.361749131535\\
2000	7.27993084063056\\
3000	5.65572319796055\\
4000	4.74967312763189\\
5000	4.1522738928749\\
6000	3.73741043135761\\
7000	3.41604996101106\\
8000	3.18587668073005\\
9000	2.99136381697669\\
10000	2.8428530214657\\
20000	2.10067242226305\\
30000	1.83987392899804\\
40000	1.70657578675137\\
50000	1.62699148118933\\
60000	1.57452717762884\\
70000	1.53602078018279\\
80000	1.50792906440744\\
90000	1.4855385240109\\
};
\addlegendentry{$\xnb=\infty$}

\addplot [color=red,loosely dotted,line width=2pt]
  table[row sep=crcr]{%
1	1.3187\\
90000	1.3187\\
};
\addlegendentry{$\Jac$}

\end{axis}
\end{tikzpicture}%
}\label{fig:LQG:ran:100}}
	~
	\subfigure[Zoomed version of the average cost for different $n$]{\scalebox{1}{
%
%

\definecolor{mycolor1}{rgb}{0.0, 0.5, 1.0}
\definecolor{mycolor2}{rgb}{0.0, 0.8, 0.6}
\definecolor{mycolor3}{rgb}{1.0, 0.5, 0.31}
\definecolor{mycolor4}{rgb}{0.13, 0.55, 0.13}%

\begin{tikzpicture}

\begin{axis}[%
width=2.2in,
height=1.6in,
at={(1.011111in,0.813889in)},
scale only axis,
xmode=log,
xmin=100,
xmax=1000000,
xminorticks=true,
xlabel={$N$},
xmajorgrids,
xminorgrids,
ymode=log,
ymin=1.2,
ymax=2,
ylabel={$J^{\text{AC}}_{n,N}$},
ymajorgrids,
ylabel style={yshift=-0.1cm},
legend style={legend cell align=left,align=left,draw=white!15!black,line width=1.0pt, font=\footnotesize}
]

\addplot [color=mycolor1, solid, mark=otimes*, mark size=2 ,line width=1.5pt]
  table[row sep=crcr]{%
1	90.0988050872139\\
2	74.8559092713465\\
3	58.3740012753949\\
4	45.7955902519346\\
5	35.8635147694922\\
6	28.5860308398038\\
7	22.5827573210655\\
8	18.0333956261914\\
9	15.3877213318488\\
10	13.4436958792473\\
20	6.14068215774133\\
30	4.34954158372674\\
40	3.46122965614533\\
50	2.97739665229225\\
60	2.62103541037543\\
70	2.38307666634709\\
80	2.22704450558327\\
90	2.06498996159715\\
100	1.96438382986823\\
200	1.52532404666544\\
300	1.39413594434313\\
400	1.34570792651888\\
500	1.31627993472894\\
600	1.29786965282214\\
700	1.28371427700216\\
800	1.27386956207888\\
900	1.26695036359398\\
1000	1.26243951330834\\
2000	1.24033826626804\\
3000	1.23266956870538\\
4000	1.22871956187396\\
5000	1.22633726262058\\
6000	1.22479092312255\\
7000	1.22379566139958\\
8000	1.22287742352561\\
9000	1.22215100223844\\
10000	1.22162047102615\\
20000 1.22162047102615\\
30000 1.22162047102615\\
40000 1.22162047102615\\
50000 1.22162047102615\\
60000 1.22162047102615\\
70000 1.22162047102615\\
80000 1.22162047102615\\
90000	1.22162047102615\\
100000 1.22162047102615\\
200000 1.22162047102615\\
300000 1.22162047102615\\
400000 1.22162047102615\\
500000 1.22162047102615\\
600000 1.22162047102615\\
700000 1.22162047102615\\
800000 1.22162047102615\\
900000 1.22162047102615\\
1000000 1.22162047102615\\
};
\addlegendentry{$n=2$};

\addplot [color=mycolor2, solid, mark=star, mark size=3, line width=1.5pt]
  table[row sep=crcr]{%
1	94.5768870629217\\
2	87.6422421294588\\
3	81.1060535945275\\
4	75.7166598326944\\
5	70.8550067501774\\
6	65.818793100411\\
7	61.4661914492361\\
8	57.5867852858906\\
9	54.506862634783\\
10	51.0690016796575\\
20	30.1255056462458\\
30	19.8540579079187\\
40	14.2465207906624\\
50	10.8717423880872\\
60	8.85011139713774\\
70	7.25012839633849\\
80	6.23207351231921\\
90	5.3655282834809\\
100	4.77894898497726\\
200	2.44822770289855\\
300	1.87735697901229\\
400	1.67372096917809\\
500	1.56577991405623\\
600	1.49838919411101\\
700	1.45440220544531\\
800	1.42286573705222\\
900	1.40333873367853\\
1000	1.38702187301104\\
2000	1.33566597760556\\
3000	1.32648669905007\\
4000	1.32290226192807\\
5000	1.32128509023382\\
6000	1.32039006435097\\
7000	1.31992148967192\\
8000	1.31948814984273\\
9000	1.31924603433781\\
10000  1.31906764370644\\
20000  1.31906764370644\\
30000  1.31906764370644\\
40000  1.31906764370644\\
50000  1.31906764370644\\
60000  1.31906764370644\\
70000  1.31906764370644\\
80000  1.31906764370644\\
90000  1.31906764370644\\
100000	1.31906764370644\\
200000 1.31906764370644\\
300000 1.31906764370644\\
400000 1.31906764370644\\
500000 1.31906764370644\\
600000 1.31906764370644\\
700000 1.31906764370644\\
800000 1.31906764370644\\
900000 1.31906764370644\\
1000000   1.31906764370644\\
};
\addlegendentry{$n=10$};

 \addplot [color=mycolor3, solid, mark=diamond, mark size=3, line width=1.5pt]
  table[row sep=crcr]{%
1	95.0727659242603\\
2	89.3293464937204\\
3	83.9600007375867\\
4	78.0428561341892\\
5	73.7229205184179\\
6	70.4072926106444\\
7	66.6424865458847\\
8	63.334550659593\\
9	60.216269546597\\
10	58.0616678649646\\
20	42.5296820114253\\
30	34.3797500254576\\
40	29.669338715003\\
50	26.831507897844\\
60	24.6676317740388\\
70	22.7390049288413\\
80	21.4440070632338\\
90	20.3367569025039\\
100	19.2554694316222\\
200	13.76144920044\\
300	11.3359882537192\\
400	9.84420511083659\\
500	8.82623441506369\\
600	7.9885333739873\\
700	7.38758245792849\\
800	6.81009576743343\\
900	6.42393695130645\\
1000	6.02474149436536\\
2000	4.09354171022974\\
3000	3.23896547092362\\
4000	2.78533674555221\\
5000	2.53689020223726\\
6000	2.35122639488588\\
7000	2.20059192075798\\
8000	2.09602080177611\\
9000	2.00843218001926\\
10000	1.94483860948849\\
20000	1.63302526595443\\
30000	1.53134706861736\\
40000	1.4734347662596\\
50000	1.44301446129694\\
60000	1.42129288813096\\
70000	1.4058035303561\\
80000	1.39424179901048\\
90000	1.38593192917479\\
100000	1.379296445348\\
200000	1.34534237322102\\
300000	1.33368847028753\\
400000	1.32412199975997\\
500000	1.32166204586395\\
600000	1.32034235423453\\
700000  1.31986648762340\\
800000  1.31955634230423\\
800000  1.31920979847234\\
1000000 1.31906764370644\\
};
\addlegendentry{$n=100$};

\addplot [color=red,loosely dotted,line width=2pt]
  table[row sep=crcr]{%
100	1.3187\\
1000000	1.3187\\
};

\addlegendentry{$\Jac$}

\end{axis}
\end{tikzpicture}%
}\label{fig:LQG:ran:tail}}
	\caption{The objective performance $\JacnN$ is computed using \eqref{AC-LP-n,N} for Example~\ref{ex:LQG}. The red dotted line denoted by $\Jac$ is the optimal solution approximated by $n = 10^3$ and $N = 10^6$.}  
	\label{fig:LQG:ran}
\end{figure}
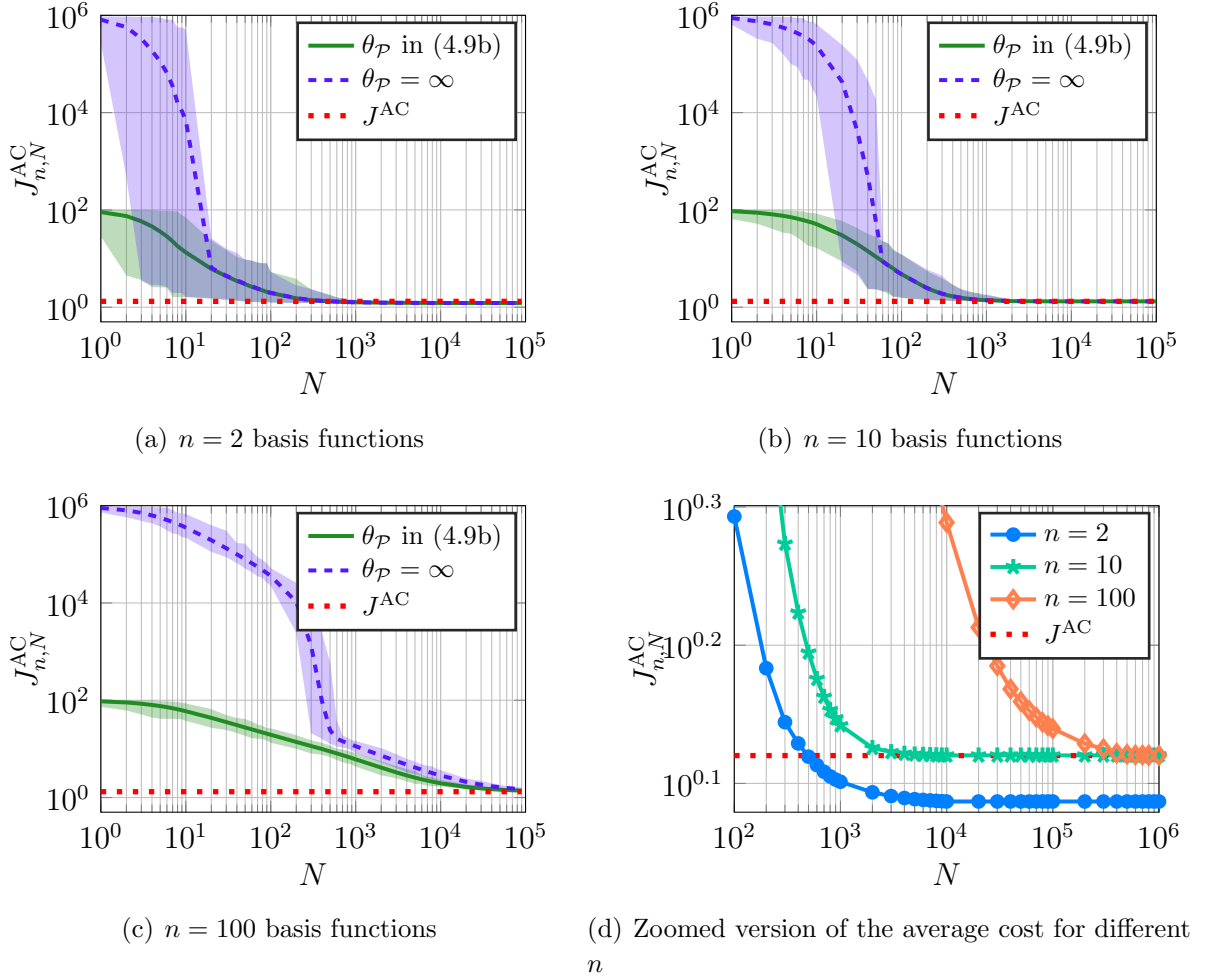

\subsubsection*{Randomized approach:} 			
We implement the methodology presented in Section~\ref{subsec:rand:MDP}, resulting in a finite random convex program as in \eqref{AC-LP-n,N}, where the uniform distribution on $K = S\times A = [-L,L]^2$ is used to draw the random samples. Figures~\ref{fig:LQG:ran:2}, \ref{fig:LQG:ran:10}, and \ref{fig:LQG:ran:100} visualize three cases with different number of basis functions $n \in \{2,10,100\}$, respectively. To show the impact of the additional norm constraint, in each case two approximation settings are examined: the constrained (regularized) one proposed in this article (i.e., $\xnb < \infty$), and the unconstrained one (i.e., $\xnb = \infty$). In the former we choose the bound suggested by \eqref{theta*:MDP2}. In the latter, the resulting optimization programs of \eqref{AC-LP-n,N} may happen to be unbounded, particularly when the number of samples $N$ is low; numerically, we capture the behavior of the unbounded $\xnb$ through a large bound such as $\theta = 10^6$. In each sub-figure, the colored tubes represent the results of $400$ independent experiments (shaded areas) as well as the mean value across different experiments (solid and dashed lines) of the objective performance $\JacnN$ as a function of the sample size $N$. 
		
Figure~\ref{fig:LQG:ran:tail} depicts a zoomed perspective of the means for the three cases of $n$. All the results in Figure~\ref{fig:LQG:ran} are obtained based on 400 independent simulation experiments. It is perhaps not surprising that the optimal value depicted in red dotted line is very close to the classical LQG example whose exact solution is analytically available. It can be seen that the randomized approximations asymptotically converge, as suggested by Theorem~\ref{thm:inf-fin}.  		

The simulation results suggest three interesting features concerning  $n$, the number of basis functions: The higher the number of basis functions,
\begin{enumerate}[label=(\roman*), itemsep = 1mm, topsep = -1mm]
	\item \label{sim:err} 
	the smaller the approximation error (i.e., asymptotic distance for $N\to\infty$ to the red dotted line),
	\item \label{sim:var}
	the lower the variance of approximation with respect to the sampling distribution for each $N$, and 
	\item \label{sim:conv}
	the slower the convergence behavior with respect to the sample size $N$.
\end{enumerate}
The features~\ref{sim:err} and \ref{sim:var} are positive impacts of increasing the number of basis functions. While \ref{sim:err} is predicted by Corollary~\ref{cor:MDP:inf-semi}, since the error due to the projection term becomes smaller, it is not entirely clear how to formally explain \ref{sim:var}. On the contrary, the feature~\ref{sim:conv} is indeed a negative impact, as a high number of basis functions requires a large number of samples $N$ to produce reasonable approximation errors. This phenomena can be justified through the lens of Corollary~\ref{cor:adp:semi-finite} where the approximation errors grows proportionally to $n$. 

	
\subsubsection*{Structural convex optimization:} 
Algorithm \ref{alg} was implemented with the parameters described in Corollary~\ref{cor:adp:smooth} leading to deterministic upper and lower bounds ($\Jnub$ and $\Jnlb$, respectively) for the cost function $\Jacn$, see also Theorem~\ref{thm:semi-fin:smooth}. These bounds are computationally appealing as they provide a posteriori bounds on the approximation error that often is significantly smaller than the a priori bounds given by Theorem~\ref{thm:semi-fin:smooth}. This behavior can be seen in the simulation results summarized in Figure~\ref{fig:LQG:smoothing} where the number of basis functions is $n=10$. Similar to Figure~\ref{fig:LQG:ran}, the red dotted line is the optimal value of the original infinite program \ref{primal-inf}, which we approximated by using $10^3$ basis functions and $10^6$ iterations of Algorithm~\ref{alg}; it coincides with the one from the randomized method.
\begin{figure}[t]
	\subfigure[A priori error $\varepsilon$ and a posteriori error $\Jnub-\Jnlb$]{\scalebox{1}{
%
%
\begin{tikzpicture}

\begin{axis}[%
width=2.2in,
height=1.8in,
at={(1.011111in,0.641667in)},
scale only axis,
xmin=10,
xmax=100000,
xmode=log,
xlabel={number of iterations $k$},
xmajorgrids,
xminorgrids,
ymin=0,
ymax=600,
ymode=log,
ylabel={prior $\&$ posterior error},
ymajorgrids,
ylabel style={yshift=-0.15cm},
legend style={legend cell align=left,align=left,draw=white!15!black,line width=1.0pt,font=\footnotesize}
]
\addplot [color=black,solid,line width=1.5pt]
  table[row sep=crcr]{%
10	536.129757314671\\
20	297.757439244727\\
30	209.453263910916\\
40	162.72697105966\\
50	133.599201535883\\
60	113.618546830154\\
70	99.0202124606833\\
80	87.8655207223592\\
90	79.0515988085919\\
100	71.9035232973631\\
200	38.3663375914169\\
300	26.4817232476722\\
400	20.3309563638865\\
500	16.5512539578754\\
600	13.9852128898149\\
700	12.1253407196401\\
800	10.7133433712918\\
900	9.6036263953016\\
1000	8.70775656335686\\
2000	4.56109548359483\\
3000	3.11906336071083\\
4000	2.38020824257256\\
5000	1.92921540328592\\
6000	1.62455856340165\\
7000	1.40460792895481\\
8000	1.23815997581314\\
9000	1.10769960273813\\
10000	1.00262523720079\\
20000	0.51969893006755\\
30000	0.353458689836964\\
40000	0.268758966720061\\
50000	0.217256676546119\\
60000	0.182565901433752\\
70000	0.157578014270967\\
80000	0.138704324779933\\
90000	0.123935179488576\\
100000	0.112056573667818\\
200000	0.0576958143278744\\
300000	0.0391007132359004\\
400000	0.0296601216738937\\
500000	0.0239338161100979\\
};
\addlegendentry{$\varepsilon$};

\addplot [color=black,dashed,line width=1.5pt]
  table[row sep=crcr]{%
10	24.3565005293573\\
20	17.5364265215274\\
30	13.2451063157308\\
40	9.31312109408406\\
50	9.03581976197761\\
60	7.29124989487846\\
70	6.85165949467133\\
80	6.73215972985745\\
90	7.91906192077684\\
100	6.43867285663217\\
200	3.45711494005305\\
300	1.81584251291752\\
400	1.55869807407919\\
500	1.46754111427825\\
600	1.66757043728\\
700	1.65550462437804\\
800	1.37002395686236\\
900	1.26175419668461\\
1000	1.33156112752499\\
2000	0.510705599116073\\
3000	0.320497504636412\\
4000	0.29995740865895\\
5000	0.285731965916484\\
6000	0.190662075258335\\
7000	0.169055891413946\\
8000	0.19042832425306\\
9000	0.184585055707421\\
10000	0.14371286558865\\
20000	0.126900649247346\\
30000	0.0655305859838577\\
40000	0.0253164317947498\\
50000	0.0334016204747398\\
60000	0.0405570683669749\\
70000	0.0271062817729779\\
80000	0.0168730222321867\\
90000	0.0179113676236098\\
100000	0.02217519505106\\
200000	0.00879183854352084\\
300000	0.00619908849403994\\
400000	0.00513727563791022\\
500000	0.00364347933428188\\
};
\addlegendentry{$\Jnub-\Jnlb$};

\end{axis}
\end{tikzpicture}%
}}
	\qquad 
	\subfigure[Upper bound $\Jnub$ and lower bound $\Jnlb$]{\scalebox{1}{
%
%
\begin{tikzpicture}

\begin{axis}[%
width=2.2in,
height=1.8in,
at={(1.011111in,0.641667in)},
scale only axis,
xmin=10,
xmax=100000,
xmode=log,
xlabel={number of iterations $k$},
xmajorgrids,
xminorgrids,
ymin=0.5,
ymax=3,
ylabel={posterior approximation},
ymajorgrids,
ylabel style={yshift=-0.15cm},
legend style={legend cell align=left,align=left,legend pos=north east,draw=white!15!black,line width=1.0pt,font=\footnotesize}
]
\addplot [color=black,solid,line width=1.5pt]
  table[row sep=crcr]{%
10	25.1527051087009\\
20	18.5453358562866\\
30	14.2530400670406\\
40	10.3615635341208\\
50	10.1650692726669\\
60	8.37690974774427\\
70	7.94622728775838\\
80	7.88020667877186\\
90	9.082242328564\\
100	7.63091987563802\\
200	4.69978256390843\\
300	3.07692493690422\\
400	2.8212316189458\\
500	2.73122526180391\\
600	2.94530887515802\\
700	2.93846732574994\\
800	2.65795277930856\\
900	2.55324120683406\\
1000	2.62579210015638\\
2000	1.8153459536679\\
3000	1.62867228923354\\
4000	1.6090745751905\\
5000	1.59625396327131\\
6000	1.50194678447926\\
7000	1.48000020459195\\
8000	1.50290052371669\\
9000	1.49728138060725\\
10000	1.45712194885304\\
20000	1.44419732361832\\
30000	1.38230623729533\\
40000	1.3432552583054\\
50000	1.35079180338104\\
60000	1.35837949527842\\
70000	1.34481075737979\\
80000	1.33488846388037\\
90000	1.33579752492145\\
100000	1.34013935255392\\
200000	1.32686085101654\\
300000	1.32434929714319\\
400000	1.32325490194924\\
500000	1.32179986534431\\
};
\addlegendentry{$\Jnub$};

\addplot [color=black,dashed,line width=1.5pt]
  table[row sep=crcr]{%
10	0.796204579343647\\
20	1.00890933475926\\
30	1.00793375130979\\
40	1.04844244003675\\
50	1.12924951068927\\
60	1.08565985286581\\
70	1.09456779308705\\
80	1.14804694891441\\
90	1.16318040778716\\
100	1.19224701900585\\
200	1.24266762385537\\
300	1.2610824239867\\
400	1.26253354486661\\
500	1.26368414752567\\
600	1.27773843787802\\
700	1.2829627013719\\
800	1.28792882244619\\
900	1.29148701014945\\
1000	1.29423097263139\\
2000	1.30464035455183\\
3000	1.30817478459713\\
4000	1.30911716653155\\
5000	1.31052199735482\\
6000	1.31128470922093\\
7000	1.310944313178\\
8000	1.31247219946363\\
9000	1.31269632489983\\
10000	1.31340908326439\\
20000	1.31729667437098\\
30000	1.31677565131148\\
40000	1.31793882651065\\
50000	1.3173901829063\\
60000	1.31782242691145\\
70000	1.31770447560681\\
80000	1.31801544164819\\
90000	1.31788615729784\\
100000	1.31796415750286\\
200000	1.31806901247302\\
300000	1.31815020864915\\
400000	1.31811762631133\\
500000	1.31815638601002\\
};
\addlegendentry{$\Jnlb$};

\addplot [color=red,loosely dotted,line width=2pt]
  table[row sep=crcr]{%
10	1.3187\\
100000	1.3187\\
};
\addlegendentry{$\Jac$}

\end{axis}
\end{tikzpicture}%
}}
	\caption{The results and error bounds are obtained by Algorithm~\ref{alg} with $n=10$ for Example~\ref{ex:LQG}. The red dotted line is the optimal solution computed as indicated in Figure~\ref{fig:LQG:ran}.}  
	\label{fig:LQG:smoothing}
\end{figure}
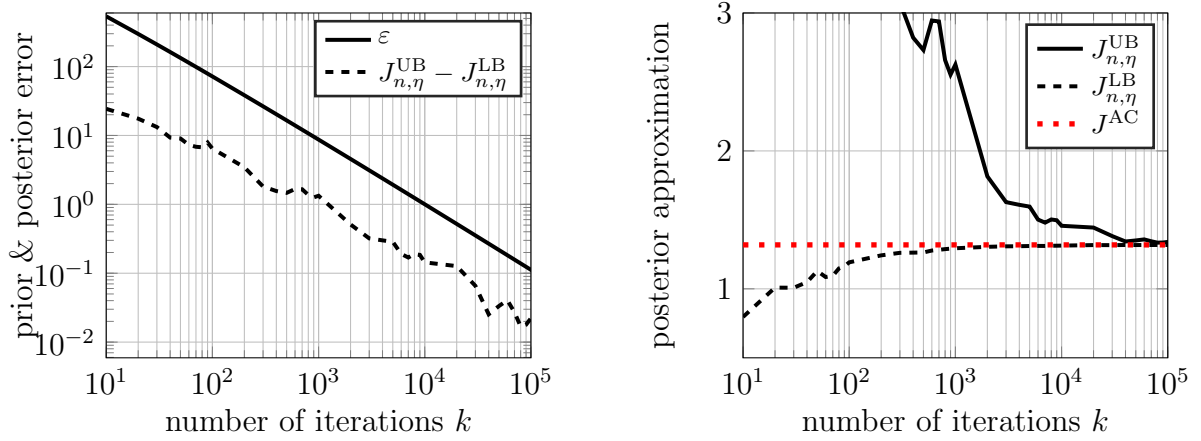

\subsection{Example 2: A fisheries management problem} 
\label{ex:fisheries}
A natural approximation approach toward dynamic programming problems goes through a discretization scheme (e.g., discretization the state and/or action spaces). The main objective of this example is to compare the proposed LP-based approximation of this article with more standard discretization schemes. To this end, we borrow an example from \cite[Section~1.3]{ref:Hernandez-96} and compare our results with the recent discretization method proposed by \cite{ref:Saldi-15}. Consider the population growth model, known as Ricker model, 
	\begin{align*} 
	s_{t+1} = \vartheta_1 a_t  \exp(-\vartheta_2 a_t + \xi_t), \quad t \in \N,
	\end{align*}
where $\vartheta_1, \vartheta_2 \in \Rp$, $s_t$ is the population size in season $t$, and $a_t$ is the population to be left for spawning for the next season, i.e., the difference $s_t -a_t$ is the amount of fish captured in season $t$. The running reward function, to be maximized is $\cost(a,s) = \varphi(s - a)$, where $\varphi$ is the so-called shifted isoelastic utility function $\varphi(z) := 3(z + 0.5)^{1/3} - (0.5)^{1/3}$ \cite[13, Section 4.1]{ref:Dufour-12}. The state space is $S = [\underline{\kappa}, \overline{\kappa}]$, for some $\underline{\kappa}, \overline{\kappa}\in\Rp$.
Since the population left for spawning cannot be greater than the total population, for each $s\in S$, the set of admissible actions is $A(s) = [\underline{\kappa}, s]$. To fulfill Assumption~\ref{a:CM}\ref{a:CM:K}, following the transformation suggested by \cite{ref:Saldi-15}, we equivalently reformulate the above problem using the dynamics 
	\begin{align*}
	s_{t+1} = \vartheta_1 \min(a_t, s_t) \exp( -\vartheta_2 \min( a_t, s_t) + \xi_t), \quad  t\in\N,
	\end{align*}
where the admissible actions set is now the state-independent set $A=[\underline{\kappa}, \overline{\kappa}]$, and the running reward function is $\cost(a,s) = \varphi(s - a)\indic{\{s\geq a\}}$. The noise process $(\xi_t)_{t\in\N}$ is a sequence of i.i.d.\ random variables which have a uniform density function $g$ supported on the interval $[0, \lambda]$. Thus, the corresponding kernel is
	\begin{align*}
	Q(B|s,a) = \int_B g\Big( \log \xi - \log\big( \vartheta_1 \min(a,s) \big) + \vartheta_2 \min(a,s) \Big)\frac{1}{\xi} \drv \xi, \quad \forall B \in \Borel{\R}.
	\end{align*}
Note that to make the model consistent, we must have $ \vartheta_1 a  \exp(-\vartheta_2 a + \xi)\in [\underline{\kappa}, \overline{\kappa}]$ for all $(a,\xi) \in  [\underline{\kappa}, \overline{\kappa}]\times [0,\lambda]$. By defining an appropriate change of coordinate similar to Lemma~\ref{lem:LQG}, Assumption~\ref{a:CM} are fulfilled; we refer the reader to \cite[Section~7.2]{ref:Saldi-15} for further information and detailed analysis. 
		
\paragraph{\emph{Simulation details:}} The chosen numerical values are $\lambda = 0.5$, $\vartheta_1 = 1.1$, $\vartheta_2 = 0.1$, $\overline{\kappa} = 7$, and $\underline{\kappa} = 0.005$. In the first approximation step discussed in Section~\ref{subsec:semi:MDP}, we consider the Fourier basis $u_{2k-1}(s) = \frac{L}{k\pi}\cos\left(\frac{k \pi s}{L}\right)$ and $u_{2k}(s) = \frac{L}{k\pi}\sin\left(\frac{k \pi s}{L}\right)$, where $L=\frac{\overline{\kappa}-\underline{\kappa}}{2}$.
		
\subsubsection*{Randomized approach:}
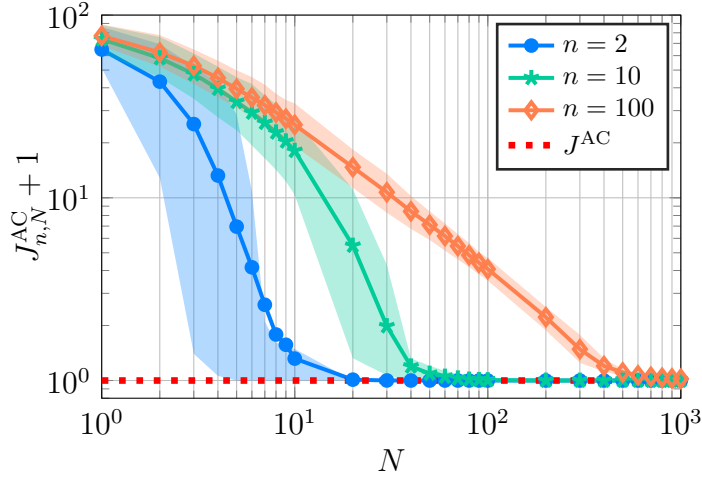
\begin{figure}[t]	
	\centering
	\scalebox{1}{
%

\definecolor{mycolor1}{rgb}{0.0, 0.5, 1.0}
\definecolor{mycolor2}{rgb}{0.0, 0.8, 0.6}
\definecolor{mycolor3}{rgb}{1.0, 0.5, 0.31}

\begin{tikzpicture}

\begin{axis}[%
width=3.0in,
height=2.0in,
at={(1.011111in,0.813889in)},
scale only axis,
xmode=log,
xmin=1,
xmax=1000,
xminorticks=true,
xlabel={$N$},
xmajorgrids,
xminorgrids,
ymode=log,
ymin=0.8,
ymax=100,
ylabel={$J^{\text{AC}}_{n,N}+1$},
ymajorgrids,
ylabel style={yshift=-0.3cm},
legend style={legend cell align=left,align=left,draw=white!15!black,line width=1.0pt,font=\footnotesize}
]

\addplot[area legend,solid,fill=mycolor1,opacity=0.3,draw=none,forget plot]
table[row sep=crcr] {%
x	y\\
1	51.2714090446456\\
2	12.8554017093566\\
3	1.40007751878754\\
4	1.0599599497917\\
5	0.999999999854423\\
6	0.999999999390651\\
7	0.999999996711568\\
8	0.999999997545897\\
9	0.999999992735313\\
10	0.9999999758598\\
20	0.999999911396561\\
30	0.999999912762618\\
40	0.999999932608935\\
50	0.999999893219666\\
60	0.999999948093206\\
70	0.999999954220393\\
80	0.999999948779772\\
90	0.999999919783682\\
100	0.999999901173917\\
200	0.99999984415569\\
300	0.99999979630737\\
400	0.999999943935447\\
500	0.999999993464536\\
600	0.999999996961352\\
700	0.999999990202266\\
800	0.999998293665241\\
900	0.999998615876402\\
1000	0.999996699106386\\
2000	0.999999985687337\\
3000	0.999989210045804\\
4000	0.999999924682763\\
5000	0.999999833349447\\
6000	0.999999423026571\\
7000	0.999999090032058\\
8000	0.999998471996984\\
9000	0.999997358591236\\
9000	0.999999195491182\\
8000	0.999999617460104\\
7000	0.999999774712436\\
6000	0.999999877010055\\
5000	0.9999999959124\\
4000	0.999999999136364\\
3000	0.999999999972595\\
2000	0.999999999988912\\
1000	0.999999999998817\\
900	0.99999999994029\\
800	0.999999999999791\\
700	0.999999999999923\\
600	0.999999999999886\\
500	0.999999999999884\\
400	0.999999999999899\\
300	0.999999999999948\\
200	0.999999999999869\\
100	0.999999999999846\\
90	0.999999999999752\\
80	0.999999999999545\\
70	0.999999999999477\\
60	0.999999999999598\\
50	0.999999999999218\\
40	0.999999999998829\\
30	0.999999999996703\\
20	1.00024554448601\\
10	1.48899308726068\\
9	1.5196942809493\\
8	1.61651417636124\\
7	2.09786857409122\\
6	10.8757806859436\\
5	28.5359276660769\\
4	42.7389974282208\\
3	52.3683901801154\\
2	69.3949509625149\\
1	86.1440806603169\\
}--cycle;

\addplot [color=mycolor1, solid, mark=otimes*, mark size=2, line width=1.5pt]
  table[row sep=crcr]{%
1	64.8116476992506\\
2	43.2200402474537\\
3	25.3208400082223\\
4	13.2378459411505\\
5	6.95400880847415\\
6	4.1645557208202\\
7	2.59765905699936\\
8	1.78597820279529\\
9	1.56911686964621\\
10	1.31828429549788\\
20	1.01094395309579\\
30	1.0005830505906\\
40	1.00003302116884\\
50	1.00000205987336\\
60	1.00000206653471\\
70	1.00000206850419\\
80	0.99999997657582\\
90	0.999999968903861\\
100	0.999999964625778\\
200	0.999999935920859\\
300	0.99999992700282\\
400	0.999999928315666\\
500	0.999999929350169\\
600	0.999999936205877\\
700	0.999999886023148\\
800	0.999999662794959\\
900	0.999999534722168\\
1000	0.999998869174429\\
2000	0.999999953159715\\
3000	0.99999594678029\\
4000	0.999999973894095\\
5000	0.999999928422408\\
6000	0.999999681436228\\
7000	0.999999445076314\\
8000	0.999999128285646\\
9000	0.999998338701263\\
};

\addlegendentry{$n=2$};

\addplot[area legend,solid,fill=mycolor2,opacity=0.3,draw=none,forget plot]
table[row sep=crcr] {%
x	y\\
1	62.0635428418405\\
2	45.4907583447821\\
3	34.8661716108083\\
4	27.5780996871585\\
5	23.4371815502727\\
6	19.3410194752331\\
7	16.4192316649149\\
8	14.2143588995335\\
9	12.2853108616049\\
10	10.3548051842559\\
20	1.32372410515241\\
30	1.06677351950785\\
40	1.00258500702997\\
50	0.999999999597828\\
60	0.99999999198904\\
70	0.999999967241146\\
80	0.999999940412202\\
90	0.999999841083342\\
100	0.999999771221102\\
200	0.999999331824728\\
300	0.999999364204503\\
400	0.999999206187848\\
500	0.999998848910086\\
600	0.999998604033411\\
700	0.999998052096279\\
800	0.999997284019949\\
900	0.99999738050084\\
1000	0.999997315087737\\
2000	0.999995178237145\\
3000	0.999996328271209\\
4000	0.999996808641027\\
5000	0.999997634015652\\
6000	0.999995434770158\\
7000	0.999996348074189\\
8000	0.999972110635102\\
9000	0.999962594479052\\
9000	0.999999994928664\\
8000	0.999999986912433\\
7000	0.999999797623133\\
6000	0.999999717247929\\
5000	0.999999925327239\\
4000	0.999999918127995\\
3000	0.999999909692193\\
2000	0.999999992677334\\
1000	0.999999999988935\\
900	0.999999999996555\\
800	0.99999999999805\\
700	0.999999999999275\\
600	0.999999999999309\\
500	0.99999999999914\\
400	0.999999999999298\\
300	0.999999999999739\\
200	0.999999999999811\\
100	1.02079622260696\\
90	1.0332283363589\\
80	1.04709122950839\\
70	1.0896262225957\\
60	1.12796089441312\\
50	1.20471800371942\\
40	1.32158263907062\\
30	4.30081376243941\\
20	11.1588934476533\\
10	26.9545659296004\\
9	28.8018740545273\\
8	32.5077842536438\\
7	36.4033900498261\\
6	41.9735615277747\\
5	45.7345034400618\\
4	52.2250648726374\\
3	61.6520034208808\\
2	75.9041241759474\\
1	88.2182321130333\\
}--cycle;

  \addplot [color=mycolor2, solid, mark=star, mark size=3, line width=1.5pt]
  table[row sep=crcr]{%
1	73.8302920074995\\
2	58.1542776172087\\
3	47.5591469514312\\
4	39.5706429806797\\
5	33.5678486857135\\
6	29.308406871455\\
7	25.7669609046978\\
8	22.8140914008158\\
9	20.3760188749272\\
10	18.0893824553476\\
20	5.47700257642174\\
30	1.98562187091641\\
40	1.19563504907069\\
50	1.09062600261757\\
60	1.05133197742151\\
70	1.02833220788591\\
80	1.01304501304375\\
90	1.00824096966948\\
100	1.00525946032438\\
200	1.00028246262316\\
300	0.999999832159385\\
400	0.999999773293682\\
500	0.999999698027969\\
600	0.999999624157324\\
700	0.99999953512617\\
800	0.999999379231007\\
900	0.999999309921831\\
1000	0.999999241536277\\
2000	0.999998468308918\\
3000	0.999998348911666\\
4000	0.999998610478249\\
5000	0.999999042889131\\
6000	0.999997904755132\\
7000	0.999998337482491\\
8000	0.999989616639842\\
9000	0.999988832750042\\
};
\addlegendentry{$n=10$};

\addplot[area legend,solid,fill=mycolor3,opacity=0.3,draw=none,forget plot]
table[row sep=crcr] {%
x	y\\
1	67.8690534540419\\
2	52.2612795934801\\
3	41.9930519029856\\
4	35.5583041889373\\
5	31.163891548069\\
6	27.4857986940591\\
7	24.6440284341841\\
8	22.6655163940452\\
9	20.7684873627083\\
10	19.1639296310229\\
20	11.3378230160899\\
30	8.28462156648752\\
40	6.7635474192051\\
50	5.87476899522253\\
60	5.14115205101883\\
70	4.59200741242686\\
80	4.17482720258423\\
90	3.82252040618674\\
100	3.52208519225872\\
200	1.88222360624751\\
300	1.25111275018784\\
400	1.1162646347967\\
500	1.06370357977148\\
600	1.03964332372105\\
700	1.02782614685429\\
800	1.01993544772025\\
900	1.01419236916502\\
1000	1.01082178185698\\
2000	1.00069499630434\\
3000	0.9999999775057\\
4000	0.999996125101665\\
5000	0.999988556459958\\
6000	0.999980896205392\\
7000	0.999974567120992\\
8000	0.9999724240661\\
9000	0.999962964927076\\
9000	1.00014125240953\\
8000	1.00025063084954\\
7000	1.00048073263709\\
6000	1.00085152006655\\
5000	1.0013740348419\\
4000	1.00226272134291\\
3000	1.00479319424886\\
2000	1.01019634051764\\
1000	1.03428135454121\\
900	1.04308176090712\\
800	1.05267979198464\\
700	1.07093047218955\\
600	1.09939990869127\\
500	1.14899772891232\\
400	1.32252456087992\\
300	1.77827850206735\\
200	2.52572275567535\\
100	4.68562210069998\\
90	5.03852191858689\\
80	5.66239537840355\\
70	6.33252829149337\\
60	7.23688099336444\\
50	8.43776777706861\\
40	10.2330499972663\\
30	13.536244061436\\
20	18.3823442529203\\
10	32.8611148900239\\
9	34.3392185911427\\
8	37.1128813801639\\
7	41.5634111891501\\
6	45.7656103411951\\
5	49.8625472008749\\
4	56.2751983205023\\
3	64.8591821763853\\
2	77.757401716288\\
1	88.8666275970295\\
}--cycle;

  \addplot [color=mycolor3, solid, mark=diamond ,mark size=3, line width=1.5pt]
  table[row sep=crcr]{%
1	76.9230952118403\\
2	62.4364888956889\\
3	52.7623730955701\\
4	45.2716002676724\\
5	39.4561345240028\\
6	35.4160542125093\\
7	32.1146392918633\\
8	29.3189568150483\\
9	27.1117850737004\\
10	25.1075537436364\\
20	14.6993274998819\\
30	10.7249814939413\\
40	8.43176859523845\\
50	7.11155586098732\\
60	6.17381624981037\\
70	5.43866702773198\\
80	4.85095004329511\\
90	4.41026336803648\\
100	4.07416339295127\\
200	2.22087509487071\\
300	1.4842572895281\\
400	1.20129227879928\\
500	1.10813237633227\\
600	1.06794358532117\\
700	1.04851764544339\\
800	1.03581933290154\\
900	1.02815211737581\\
1000	1.02237183396532\\
2000	1.00519818105117\\
3000	1.00188382094289\\
4000	1.00085255845987\\
5000	1.00043024094767\\
6000	1.00024471921017\\
7000	1.00012846278167\\
8000	1.00007483581836\\
9000	1.00004553716265\\
};
\addlegendentry{$n=100$};

\addplot [color=red,loosely dotted,line width=2.5pt]
  table[row sep=crcr]{%
1	1\\
900000	 1\\
};
\addlegendentry{$\Jac$}

\end{axis}
\end{tikzpicture}%
}
	\caption{The objective performance $\JacnN$ is computed using \eqref{AC-LP-n,N} for Example~\ref{ex:fisheries}.  The red dotted line is the optimal value approximated by $n = 10^3$ and $N = 10^6$, which amounts to 0 as also reported in \cite{ref:Saldi-15}.}
	\label{fig:fishery:rand}
\end{figure}		
We implement the methodology presented in Section~\ref{subsec:rand:MDP}, resulting in a finite random convex program \eqref{AC-LP-n,N}, where the uniform distribution on $K = S\times A = [\underline{\kappa},\overline{\kappa}]^2$ is used to draw the random samples. Figure~\ref{fig:fishery:rand} illustrates three cases with the number of basis functions $n \in \{2,10,100\}$ and the bound \eqref{theta*:MDP2}. The colored tubes represent the results between $[10\%,90\%]$ quantiles (shaded areas) as well as the means (solid lines) across $400$ independent experiments of the objective performance $\JacnN$ as a function of the sample size $N$. It is interesting to note that in this example the optimal solution is captured even with $2$ basis functions and only $N = 20$ random samples. This becomes even more attractive when we compare the results with a direct discretization scheme depicted in \cite[Figure~2]{ref:Saldi-15}. 

\subsubsection*{Structural convex optimization:}  
Similar to the LQG example in Section~\ref{ex:LQG}, we also implement the smoothing methodology for the case of $n = 10$. The simulation results are reported in Figure~\ref{fig:fishery:smoothing}. 

\begin{figure}[t]
	\subfigure[A priori error $\varepsilon$ and a posteriori error $\Jnub-\Jnlb$]{\scalebox{1}{
%
%
\begin{tikzpicture}

\begin{axis}[%
width=2.2in,
height=1.8in,
at={(1.011111in,0.641667in)},
scale only axis,
xmin=10,
xmax=100000,
xmode=log,
xlabel={number of iterations $k$},
xmajorgrids,
xminorgrids,
ymin=0,
ymax=600,
ymode=log,
ylabel={prior $\&$ posterior error},
ymajorgrids,
ylabel style={yshift=-0.15cm},
legend style={legend cell align=left,align=left,draw=white!15!black,line width=1.0pt,font=\footnotesize}
]
\addplot [color=black,solid,line width=1.5pt]
  table[row sep=crcr]{%
10	90.7596988873165\\
20	50.4063697325145\\
30	35.4576514031928\\
40	27.5475115641207\\
50	22.6165676646474\\
60	19.2341086084525\\
70	16.7628047887621\\
80	14.8744638585725\\
90	13.3823841724815\\
100	12.1723100224072\\
200	6.49491059977694\\
300	4.483003482218\\
400	3.44176046714649\\
500	2.80190712795399\\
600	2.36751050897174\\
700	2.05265889084282\\
800	1.81362652235531\\
900	1.62576619993161\\
1000	1.47410735437071\\
2000	0.772132793037242\\
3000	0.528015936747699\\
4000	0.402937593601818\\
5000	0.326590505080941\\
6000	0.275016154676782\\
7000	0.237781438079292\\
8000	0.209604013726553\\
9000	0.187518807967235\\
10000	0.169731115595792\\
20000	0.087978115751977\\
30000	0.0598358544320587\\
40000	0.0454973179960288\\
50000	0.0367786653603076\\
60000	0.0309059785953666\\
70000	0.0266758616911061\\
80000	0.0234807971207534\\
90000	0.0209805772841802\\
100000	0.0189696873296157\\
200000	0.00976713388785029\\
300000	0.00661923062763622\\
400000	0.00502106405626882\\
500000	0.00405167669644238\\
};
\addlegendentry{$\varepsilon$};

\addplot [color=black,dashed,line width=1.5pt]
  table[row sep=crcr]{%
10	2.75114402364085\\
20	1.64989484200388\\
30	1.08662273474772\\
40	0.838413529167986\\
50	0.716303058391313\\
60	0.751629176764794\\
70	0.527843195838203\\
80	0.453976846348352\\
90	0.370513448724604\\
100	0.475541564126903\\
200	0.252205934333526\\
300	0.129153469504749\\
400	0.056317239774328\\
500	0.0599573624198011\\
600	0.0546622918758253\\
700	0.0493018029438565\\
800	0.0338735832992433\\
900	0.0423541661739733\\
1000	0.0510683108240164\\
2000	0.020656372203495\\
3000	0.00933625317231349\\
4000	0.0124975949205784\\
5000	0.00724609981313697\\
6000	0.0108189115952417\\
7000	0.00695926150953026\\
8000	0.00451336505017449\\
9000	0.00501445742333247\\
10000	0.00285021265837832\\
20000	0.00103120305498905\\
30000	0.000962708277496393\\
40000	0.00099187467494551\\
50000	0.000843458013181959\\
60000	0.000731958241294043\\
70000	0.000820653367694514\\
80000	0.000552985426491319\\
90000	0.000470169711212189\\
100000	0.000549664006081856\\
200000	0.000161542871196692\\
300000	0.000157478416186615\\
400000	0.000114286160560574\\
500000	9.27186052237982e-05\\
};
\addlegendentry{$\Jnub-\Jnlb$};

\end{axis}
\end{tikzpicture}%
}}
	\qquad
	\subfigure[Upper bound $\Jnub$ and lower bound $\Jnlb$]{\scalebox{1}{
%
%
\begin{tikzpicture}

\begin{axis}[%
width=2.2in,
height=1.8in,
at={(1.011111in,0.641667in)},
scale only axis,
xmin=10,
xmax=100000,
xmode=log,
xlabel={number of iterations $k$},
xmajorgrids,
xminorgrids,
ymin=-0.1,
ymax=3,
ylabel={posterior approximation},
ymajorgrids,
ylabel style={yshift=-0.15cm},
legend style={legend cell align=left,align=left,legend pos=north east,draw=white!15!black,line width=1.0pt,font=\footnotesize}
]
\addplot [color=black,solid,line width=1.5pt]
  table[row sep=crcr]{%
10	2.69782240051088\\
20	1.62579317933833\\
30	1.07115935272095\\
40	0.828171926336044\\
50	0.70607517282797\\
60	0.740878175245887\\
70	0.522115480764568\\
80	0.444594005149921\\
90	0.362379183009112\\
100	0.470639892582908\\
200	0.248729379903276\\
300	0.127021159055569\\
400	0.0549353014708736\\
500	0.0590281281029852\\
600	0.0536519263944911\\
700	0.0483419461767818\\
800	0.0325118716723125\\
900	0.0416840173242761\\
1000	0.05061899782999\\
2000	0.0204534961184887\\
3000	0.00915265017565384\\
4000	0.0123235985696712\\
5000	0.00716959309992658\\
6000	0.010693155730566\\
7000	0.00685894282273861\\
8000	0.00442100487590579\\
9000	0.004930653756291\\
10000	0.00279382182970044\\
20000	0.00101572126869705\\
30000	0.000955529434186856\\
40000	0.000982917919282597\\
50000	0.000836282492199456\\
60000	0.000721253912624484\\
70000	0.000815445567292819\\
80000	0.000543737656212934\\
90000	0.000465564036385646\\
100000	0.000543300400007567\\
200000	0.000157944041236142\\
300000	0.000154977794771278\\
400000	0.000112678016109268\\
500000	9.1779790071263e-05\\
};

\addlegendentry{$\Jnub$};

\addplot [color=black,dashed,line width=1.5pt]
  table[row sep=crcr]{%
 10	-0.0533216231299718\\
20	-0.0241016626655506\\
30	-0.0154633820267671\\
40	-0.0102416028319417\\
50	-0.0102278855633426\\
60	-0.0107510015189079\\
70	-0.0057277150736354\\
80	-0.00938284119843086\\
90	-0.0081342657154919\\
100	-0.00490167154399528\\
200	-0.00347655443025035\\
300	-0.00213231044917997\\
400	-0.00138193830345437\\
500	-0.000929234316815946\\
600	-0.00101036548133423\\
700	-0.000959856767074679\\
800	-0.00136171162693085\\
900	-0.000670148849697179\\
1000	-0.000449312994026358\\
2000	-0.000202876085006332\\
3000	-0.000183602996659646\\
4000	-0.00017399635090721\\
5000	-7.65067132103906e-05\\
6000	-0.000125755864675721\\
7000	-0.000100318686791647\\
8000	-9.23601742686991e-05\\
9000	-8.38036670414648e-05\\
10000	-5.63908286778793e-05\\
20000	-1.54817862920044e-05\\
30000	-7.17884330953662e-06\\
40000	-8.95675566291295e-06\\
50000	-7.1755209825029e-06\\
60000	-1.07043286695595e-05\\
70000	-5.20780040169506e-06\\
80000	-9.24777027838479e-06\\
90000	-4.60567482654289e-06\\
100000	-6.36360607428934e-06\\
200000	-3.59882996055008e-06\\
300000	-2.50062141533679e-06\\
400000	-1.60814445130632e-06\\
500000	-9.388151525352e-07\\
};
\addlegendentry{$\Jnlb$};

\addplot [color=red,loosely dotted,line width=2pt]
  table[row sep=crcr]{%
10	0\\
100000	0\\
};
\addlegendentry{$\Jac$}

\end{axis}
\end{tikzpicture}%
}}
	\caption{The results and error bounds are obtained by Algorithm~\ref{alg} with $n=10$ for Example~\ref{ex:fisheries}. The red dotted line is the optimal solution computed as indicated in Figure~\ref{fig:fishery:rand}.}  
	\label{fig:fishery:smoothing}
\end{figure}
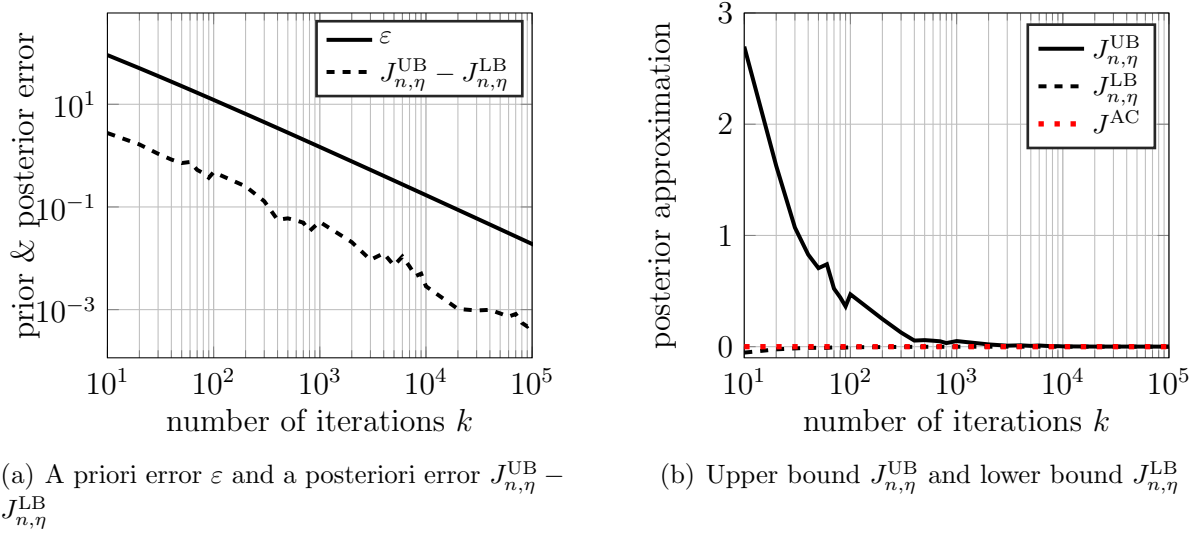

\section{Unknown dynamics} \label{sec:RL}	
This section is an extension of the approximation scheme presented in Sections~\ref{subsec:semi:MDP} and \ref{subsec:rand:MDP} to the case where the transition kernel is is unknown but information is obtained by simulation, i.e. we assume that data drawn from the probability distribution $Q(\cdot|x,a)$ is available. We consider the same setting as in Sections~\ref{subsec:semi:MDP} and \ref{subsec:rand:MDP}, and introduce a linear operator $\RQ: \Lip(X) \to \Lip(X\times A)$, defined by
\begin{equation} \label{eq:operator}
\begin{aligned}
\RQ u(x,a):=& u(x) - \int_X u(y)Q(\drv y|x,a) \\
		 =& u(x) -  \ExpecQ{u(y)},
\end{aligned}
\end{equation}		
and recall (Theorem~\ref{thm:equivalent:LP}) that using this notation the LP~\eqref{thm:equivalent:LP} can be written as
\begin{align}\label{J}
\Jac = & \left\{ \begin{array}{cl}
   \sup\limits_{\rho, u} & \rho   \\
   \subjectto &\rho + \RQ u(x,a) \leq c(x,a), \ \forall (x,a)\in \K \\
   & \rho\in\R, \ u\in \Lip(X).
\end{array} \right.
\end{align}		
Let $\{(x_j,a_j) \}_{j\le N}$ be i.i.d.~samples generated with respect to some probability measure $\PP_1$ supported on $\mathbb{K}$.
We propose the following finite-dimensional convex program as an approximation to the infinite LP~\eqref{J} and hence to the optimal control problem~\eqref{AC}
\begin{align}\label{AC-LP-n,N}
J_{n,N}^m = & \left\{ \begin{array}{cl}
  \!\!\! \sup\limits_{(\rho, \alpha)\in\R^{n+1}}\!\!\!\!\!\!\!\! &  \rho   \\
   \subjectto &\rho + \sum\limits_{i=1}^{n} \alpha_i\RQm u_i(x_j,a_j) \leq c(x_j,a_j), \quad  \forall j \in \{1,\cdots,N\}\\
   &  \| \alpha \|_{2} \le \theta,
\end{array} \right.
\end{align}
where $\{u_i\}_{i=1}^n \subset \Lip(X)$ is a family of linearly independent elements, called the basis functions and $\theta>0$ is the regularization parameter as introduced in \eqref{theta*:MDP}. Moreover,
we use the following notation 
\begin{equation*}
\RQm u(x,a):= u(x)  - \frac{1}{m} \sum_{i=1}^m u(y_i), \quad y_i \overset{\text{i.i.d.}}{\sim}Q(\cdot|x,a).
\end{equation*}
Note that the program \eqref{AC-LP-n,N} does not require knowledge of the transition kernel $Q$, but instead, it uses simulations to learn $Q$ via the samples $y_i$ in the operator $\RQm$.
In the following, we quantify the approximation error of $\eqref{AC-LP-n,N}$ with respect to $\eqref{J}$.		
		
\begin{Thm} \label{thm:main}
Given Assumptions~\ref{a:CM}, let $\eps,\beta\in(0,1)$, $\mathbb{U}_n := \{\sum_{i=1}^{n} \alpha_i u_i : \|\alpha\|_2 \le \theta \}$ and consider the finite convex program \eqref{AC-LP-n,N} where the number of sampled constraints satisfies $N\geq \NN(n+1, (\frac{\eps z_n}{2})^{\dim(\mathbb{K})}, \frac{\beta}{2})$, where $\NN$ is as in \eqref{N}, $z_n := \big(\theta \sqrt{n} (\max\{L_Q,1\}+ 1)+ \| c \|_L \big)^{-1}$ and $m\geq\frac{8Cn\theta^2\log(4 n N /\beta)}{\eps^2}$. Then, with probability $1-\beta$
\begin{align*}
&\left| \Jac - J_{n,N}^m \right | \leq \big(1 + \max\{L_Q, 1\}\big) \big\|u^\star - \Pi_{\mathbb{U}_n}(u^\star)\big\|_L+\eps.
\end{align*}
\end{Thm}		
Note that that $J_{n,N}^m$ is a real valued random variable on the space $(\mathbb{K}\times X^m)^{N}$. Strictly speaking, the error bound of Theorem~\ref{thm:main} has to be interpreted with respect to $\PP_2^N$, where $\PP_2$ is a probability measure on $\mathbb{K}\times X^m$ defined by $\PP_2[\drv(x,a,y_1,\hdots,y_m)]:=Q^m(\drv y|x,a)\PP_1[\drv(x,a)]$ and $\PP_2^N$ stands for the N-fold product probability measure. For simplicity we slightly abuse the notation and use $\PP$ instead of $\PP_2^N$, and will be doing so hereinafter.		

\subsection{Proof or Theorem~\ref{thm:main}} \label{proof}
Some preliminaries are needed in order to prove Theorem~\ref{thm:main}. Consider the finite convex program
\begin{align}\label{J_n_N}
J_{n,N} := & \left\{ \begin{array}{cl}
  \!\!\! \sup\limits_{(\rho, \alpha)\in\R^{n+1}} \!\!\!\!\!\!\!\!& \rho   \\
   \subjectto &\rho + \sum\limits_{i=1}^{n} \alpha_i\RQ u_i(x_j,a_j) \leq c(x_j,a_j), \quad \forall j \in \{1,\cdots,N\}\\
   &  \| \alpha \|_{2} \le \theta.
\end{array} \right.
\end{align}

\begin{Lem} \label{lemma:MC}
For any $\eps>0$ and $n,m,N>0$
\begin{equation*}
\PP \Big[ \left| J_{n,N}^m - J_{n,N} \right | \leq \varepsilon \Big] \geq 1-2 n N \exp\left( \frac{-\varepsilon^2 m }{2n\theta^2} \right).
\end{equation*}
\end{Lem}

\begin{proof}
As the first step, we invoke the Hoeffding inequality \cite{ref:Bouch-13} together with the subadditivity of probability measures\footnote{i.e., $\PP(A \cap B) \geq \PP( A ) + \PP( B ) - 1.$} which states that for any $\varepsilon>0$
\begin{align*}
&\PP\Big[ \forall i=1,\hdots,n, j=1,\hdots,N, \left| \RQ u_i(x_j,a_j) - \RQm u_i(x_j,a_j) \right|  \leq \varepsilon  \Big]\geq 1-2 n N \exp\left( \frac{-\varepsilon^2 m}{2} \right).
\end{align*}
Hence, for all $\varepsilon>0$
\begin{align*}
&\PP\left[ \forall j=1,\hdots,N \sup_{\|\alpha\|_2\leq \theta} \left| \sum_{i=1}^n \alpha_i \RQ u_i(x_j,a_j) \right. -\left| \sum_{i=1}^n \alpha_i \RQm u_i(x_j,a_j)  \right| \leq \varepsilon \right] \\
&\quad \geq \PP\Big[ \forall i=1,\hdots,n, j=1,\hdots,N \ \|\alpha \|_1 \left| \RQ u_i(x_j,a_j) \right. \left. - \RQm u_i(x_j,a_j) \right| \leq \varepsilon \Big] \\
& \quad \geq 1-2 n N \exp\left( \frac{-\varepsilon^2 m}{2 n \theta^2} \right),
\end{align*}
where we have used the normalization of the basis functions (i.e., $\|u_i\|_L\leq 1$) leading to $\| \sum_{i=1}^n \alpha_i u_i \|_L \leq \|\alpha\|_1 \leq \sqrt{n} \theta$.
Therefore, with confidence $1 - 2 N \exp\left( \frac{-\varepsilon^2 m}{2n\theta^2} \right)$ we have
\begin{align*}
J_{n,N}^m &=  \left\{ \begin{array}{cl}
   \sup\limits_{(\rho, \alpha)\in\R^{n+1}} & \rho   \\
    \subjectto &\rho + \sum\limits_{i=1}^{n} \alpha_i\big(\RQ u_i(x_j,a_j) \\
   & \quad + \RQm u_i(x_j,a_j)- \RQ u_i(x_j,a_j)\big) \\
   &\quad \leq c(x_j,a_j), \quad \forall j \in \{1,\cdots,N\} \\
   &  \| \alpha \|_{2} \le \theta.
\end{array} \right. \\
&\geq \left\{ \begin{array}{cl}
   \sup\limits_{(\rho, \alpha)\in\R^{n+1}} & \rho-\varepsilon   \\
    \subjectto &\rho + \sum\limits_{i=1}^{n} \alpha_i\RQ u_i (x_j,a_j) \\
   & \quad \leq c(x_j,a_j), \quad \forall j \in \{1,\cdots,N\}\\
   &  \| \alpha \|_{2} \le \theta.
\end{array} \right. \\
&=J_{n,N}-\varepsilon
\end{align*}
and similarly one can show $J_{n,N}^m \leq J_{n,N}+\varepsilon$, which completes the proof.
\end{proof}

\emph{Proof of Theorem~\ref{thm:main}:}
The proof consists of combining three results. First, recall that \cite[Corollary~3.9]{ref:Peyman-17} for the given setting of Theorem~\ref{thm:main}
\begin{equation} \label{proof:step1}
0\leq J-J_n \leq \big(1 + \max\{L_Q, 1\}\big) \big\|u^\star - \Pi_{\mathbb{U}_n}(u^\star)\big\|_L,
\end{equation}
where
\begin{align*}\label{J_n}
J_{n} := & \left\{ \begin{array}{cl}
   \sup\limits_{(\rho, \alpha)\in\R^{n+1}} & \rho   \\
   \subjectto &\rho + \sum\limits_{i=1}^{n} \alpha_i\RQ u_i(x,a) \leq c(x,a), \quad \forall (x,a)\in X\times A\\
   &  \| \alpha \|_{2} \le \theta.
\end{array} \right.
\end{align*}
Next, \cite[Corollary~3.9]{ref:Peyman-17} states that for $N\geq \NN(n+1, (\eps z_n)^{\dim(K)}, \beta)$, where $z_n := \big(\theta \sqrt{n} (\max\{L_Q,1\}+ 1)+ \| c \|_L \big)^{-1}$
\begin{equation} \label{proof:step2}
\PP^N \Big[ |J_n - J_{n,N}|\leq \eps \Big] \geq \beta,
\end{equation}
where $J_{n,N}$ is defined in \eqref{J_n_N}. Finally, a simple union bound of \eqref{proof:step1}, \eqref{proof:step2} and Lemma~\ref{lemma:MC} concludes the proof. $\blacksquare$

\subsection{Numerical example} \label{sec:example}

Consider the linear system
	\begin{align*}
	x_{t+1}=\vartheta x_{t} + \rho a_{t} + \xi_{t}, \quad t \in \N,
	\end{align*}
with quadratic stage cost $\cost(x,a)=qx^{2}+ra^{2}$, where $q\geq 0$ and $r>0$ are given constants. We assume that $X=A=[-L,L]$ and the parameters $\vartheta, \rho \in\R$ are known. The disturbances $\{\xi_{t}\}_{t\in\N}$ are i.i.d.\ random variables generated by a truncated normal distribution with known parameters $\mu$ and $\sigma$, independent of the initial state $x_{0}$. Thus, the process $\xi_{t}$ has a distribution density
	\begin{align*}
	f(x,\mu,\sigma,L) = \left\{
	\begin{array}{cc}
	\frac{\frac{1}{\sigma}\phi\left( \frac{x-\mu}{\sigma} \right)}{\Phi\left( \frac{L-\mu}{\sigma} \right)-\Phi\left( \frac{-L-\mu}{\sigma} \right)}, \quad & x\in[-L,L]\\
	0 & \text{o.w.},
	\end{array} \right.
	\end{align*}
where $\phi$ is the probability density function of the standard normal distribution, and $\Phi$ is its cumulative distribution function. The transition kernel $Q$ has a density function $q(y|x,a)$, i.e., $Q(B|x,a)=\int_{B}q(y|x,a)\drv y$ for all $B\in\Borel{X}$, that is given by
	\begin{align*}
	q(y|x,a)=f(y-\vartheta x - \rho a,\mu,\sigma,L).
	\end{align*}
In the special case that $L=+\infty$ the above problem represents the classical LQG problem, whose solution can be obtained via the algebraic Riccati equation \cite[p.~372]{ref:Bertsekas-12}. By a simple change of coordinates it can be seen that the presented system fulfills Assumption~\ref{a:CM}. Moreover, the following lemma provides the technical parameters required for the proposed error bounds.
		
\begin{Lem}[Truncated LQG properties] \label{lem:LQG}
The error bounds provided by Theorem~\ref{thm:main} hold with the norms $\|\cost\|_\infty = L^2(q+r)$, $\|\cost\|_L = 4L^2\sqrt{q^2+r^2}$, and the Lipschitz constant of the kernel is
	\begin{align*}
	L_Q &=\frac{2L\max\{\vartheta,\rho\}}{\sigma^2\sqrt{2\pi}\left( \Phi\left( \frac{L-\mu}{\sigma} \right)-\Phi\left( \frac{-L-\mu}{\sigma} \right)\right)}\,.
	\end{align*}
\end{Lem}

\begin{proof}
	In regard to Assumption~\ref{a:CM}\ref{a:CM:K}, we consider the change of coordinates $\bar{x}_t \Let \frac{x_t}{2L}+\frac{1}{2}$ and $\bar{a}_t \Let \frac{a_t}{2L}+\frac{1}{2}$. In the new coordinates, the constants of Lemma~\ref{lem:LQG} follow from a standard computation.
\end{proof}

\paragraph{Simulation details}
For the simulation results we choose the numerical values $\vartheta = 0.8$, $\rho = 0.5$, $\sigma = 1$, $\mu = 0$, $q = 1$, $r = 0.5$, and $L=10$. Throughout this section we used the Fourier basis $u_{2k-1}(s) = \frac{L}{k\pi}\cos\left(\frac{k \pi s}{L}\right)$ and $u_{2k}(s) = \frac{L}{k\pi}\sin\left(\frac{k \pi s}{L}\right)$ and the uniform distribution on $K = X\times A = [-L,L]^2$ to draw the random samples $\{x_j,a_j\}_{j=1}^N$ in program \eqref{AC-LP-n,N}.

\paragraph{Simulation results} \

\begin{figure}[!h]
	\subfigure[varying constraint samples $N$, where $m=10^6$]{\scalebox{1}{
%
%
\definecolor{mycolor1}{rgb}{0.13, 0.55, 0.13}%
\definecolor{mycolor2}{rgb}{0.13, 0.55, 0.13}%

\definecolor{mycolor3}{rgb}{0.32,0.09,0.98}%
\definecolor{mycolor4}{rgb}{0.32,0.09,0.98}%
\begin{tikzpicture}

\begin{axis}[%
width=2.2in,
height=1.6in,
at={(1.011111in,0.813889in)},
scale only axis,
xmode=log,
xmin=1,
xmax=100000,
xminorticks=true,
xlabel={$N$},
xmajorgrids,
xminorgrids,
ymode=log,
ymin=0.5,
ymax=1000000,
ylabel={$J^{\text{AC}}_{n,N}$},
ymajorgrids,
ylabel style={yshift=-0.3cm},
legend style={legend cell align=left,align=left,draw=white!15!black,line width=1.0pt,font=\footnotesize}
]

\addplot[area legend,solid,fill=mycolor1,opacity=0.3,draw=none,forget plot]
table[row sep=crcr] {%
x	y\\
1	64.8626787674932\\
2	51.8028663321495\\
3	40.8698505042221\\
4	34.2057123029274\\
5	32.327740150497\\
6	26.5906149656799\\
7	21.8167512994986\\
8	18.2942570690963\\
9	17.509433217615\\
10	16.5109237494213\\
20	8.68290108849987\\
30	5.49612925906822\\
40	2.35413086420275\\
50	2.32862407932307\\
60	2.12306400135207\\
70	1.86120000680995\\
80	1.63535781691707\\
90	1.63535781690026\\
100	1.55959083049867\\
200	1.39333461801874\\
300	1.3643441040029\\
400	1.35394428022446\\
500	1.34195332050755\\
600	1.33856909460061\\
700	1.32859904928777\\
800	1.32458818880617\\
900	1.32445301927657\\
1000	1.32445302330009\\
2000	1.31914600579702\\
3000	1.31879268391564\\
4000	1.31862187925411\\
5000	1.31844109127923\\
6000	1.31830688543409\\
7000	1.31826094723498\\
8000	1.31826094843103\\
9000	1.3182005639989\\
100000	1.31813536293608\\
100000	1.32263803288161\\
9000	1.32500128260116\\
8000	1.32539113783232\\
7000	1.32655764138582\\
6000	1.32883437777301\\
5000	1.33550740507875\\
4000	1.34085918309822\\
3000	1.37977303081213\\
2000	1.41450337883357\\
1000	1.82423091203754\\
900	1.84987753113501\\
800	1.94883222433922\\
700	2.04079150965957\\
600	2.06445841923901\\
500	2.34248038369511\\
400	2.96176544947812\\
300	3.88985563333783\\
200	5.47246905081519\\
100	12.1911192973486\\
90	13.3995722506444\\
80	15.4831188210459\\
70	20.3527647402716\\
60	22.5255326834499\\
50	27.9046058790261\\
40	32.1987850816379\\
30	42.5032034823197\\
20	59.1930381508071\\
10	84.0304606027895\\
9	86.9142744662561\\
8	86.9142745304304\\
7	93.2294273884365\\
6	94.253548028958\\
5	97.06945360186\\
4	98.4481209823781\\
3	99.2617095114665\\
2	99.9999994398324\\
1	100.000002879261\\
}--cycle;

\addplot [color=mycolor2, solid,line width=1.5pt]
  table[row sep=crcr]{%
1	94.5768870629217\\
2	87.6422421294588\\
3	81.1060535945275\\
4	75.7166598326944\\
5	70.8550067501774\\
6	65.818793100411\\
7	61.4661914492361\\
8	57.5867852858906\\
9	54.506862634783\\
10	51.0690016796575\\
20	30.1255056462458\\
30	19.8540579079187\\
40	14.2465207906624\\
50	10.8717423880872\\
60	8.85011139713774\\
70	7.25012839633849\\
80	6.23207351231921\\
90	5.3655282834809\\
100	4.77894898497726\\
200	2.44822770289855\\
300	1.87735697901229\\
400	1.67372096917809\\
500	1.56577991405623\\
600	1.49838919411101\\
700	1.45440220544531\\
800	1.42286573705222\\
900	1.40333873367853\\
1000	1.38702187301104\\
2000	1.33566597760556\\
3000	1.32648669905007\\
4000	1.32290226192807\\
5000	1.32128509023382\\
6000	1.32039006435097\\
7000	1.31992148967192\\
8000	1.31948814984273\\
9000	1.31924603433781\\
100000	1.31906764370644\\
};
\addlegendentry{$\xnb$ in \eqref{theta*:MDP2}};


\addplot[area legend,solid,fill=mycolor3,opacity=0.25,draw=none,forget plot]
table[row sep=crcr] {%
x	y\\
1	639983.366524608\\
2	329098.537697569\\
3	270518.925077757\\
4	162716.677544189\\
5	99203.2256406953\\
6	86226.4301855034\\
7	47489.5274014969\\
8	47489.5282154978\\
9	30133.0276457532\\
10	21613.4199719176\\
20	7.0906145341405\\
30	4.65928077970097\\
40	2.39043513144145\\
50	2.3286038154872\\
60	2.12306379225585\\
70	2.09667479516671\\
80	1.635356391661\\
90	1.63535708245283\\
100	1.55958962460745\\
200	1.42441044263701\\
300	1.3729378008794\\
400	1.34201033025407\\
500	1.34200855870122\\
600	1.33393016386478\\
700	1.32548079869412\\
800	1.3245880787709\\
900	1.32445276232808\\
1000	1.32067603411791\\
2000	1.31914752559519\\
3000	1.31890352096265\\
4000	1.31862187891315\\
5000	1.3184410887336\\
6000	1.31830689055684\\
7000	1.31826094801364\\
8000	1.31820169759628\\
9000	1.31816451028211\\
100000	1.31813534897557\\
100000	1.32635030823924\\
9000	1.3263503050298\\
8000	1.32709601638159\\
7000	1.32710598610841\\
6000	1.32938918012012\\
5000	1.33621956240585\\
4000	1.3408591123418\\
3000	1.38643803937312\\
2000	1.41323205875657\\
1000	1.7807446249738\\
900	1.78586817791992\\
800	1.94883224270823\\
700	1.95467052048799\\
600	2.19564446831188\\
500	2.376049145298\\
400	2.84118283785502\\
300	4.42306322890742\\
200	5.948249385366\\
100	12.1234649487366\\
90	13.2643994103875\\
80	14.8895276860826\\
70	20.6106203279187\\
60	22.9578071142619\\
50	18731.5689369941\\
40	46154.8527946401\\
30	122416.784984592\\
20	245819.513244859\\
10	657237.808452924\\
9	733798.353290344\\
8	733798.349269297\\
7	879520.407606198\\
6	884146.221461426\\
5	916423.123492617\\
4	921568.783455853\\
3	938169.047545757\\
2	948954.054124091\\
1	949823.811310828\\
}--cycle;

\addplot [color=mycolor4,dashed,line width=1.5pt]
  table[row sep=crcr]{%
1	881068.308142879\\
2	753794.288347249\\
3	651989.021149723\\
4	561892.86158924\\
5	489715.217209112\\
6	427140.172458315\\
7	373606.121950658\\
8	324733.263256534\\
9	277359.534656746\\
10	235953.292052392\\
20	42857.0692688388\\
30	4299.73531162478\\
40	503.880225630512\\
50	57.5866681916299\\
60	8.75517131945854\\
70	7.16879104175466\\
80	6.03343592545909\\
90	5.30458286475803\\
100	4.74690136159816\\
200	2.47432622510224\\
300	1.90523920521026\\
400	1.67809177228141\\
500	1.56928466228527\\
600	1.50170304052627\\
700	1.44790616627579\\
800	1.42030204832343\\
900	1.40220936819888\\
1000	1.38672104646169\\
2000	1.3356980463223\\
3000	1.32613696787708\\
4000	1.32267627846345\\
5000	1.32124562408155\\
6000	1.32037996927841\\
7000	1.31984030493074\\
8000	1.31951848108317\\
9000	1.31925491020946\\
100000	1.31909475493793\\
};
\addlegendentry{$\xnb=\infty$}

\addplot [color=red,loosely dotted,line width=2pt]
  table[row sep=crcr]{%
1	1.3187\\
100000	1.3187\\
};
\addlegendentry{$\Jac$}

\end{axis}
\end{tikzpicture}%
}\label{fig:LQG:varying:N}}
	\qquad
	\subfigure[varying number of basis functions $n$, where $m=10^6$]{\scalebox{1}{
%
%
\definecolor{mycolor1}{rgb}{0.13, 0.55, 0.13}%
\definecolor{mycolor2}{rgb}{0.13, 0.55, 0.13}%

\definecolor{mycolor3}{rgb}{0.32,0.09,0.98}%
\definecolor{mycolor4}{rgb}{0.32,0.09,0.98}%

\definecolor{mycolor5}{rgb}{0.59, 0.29, 0.0}
\begin{tikzpicture}

\begin{axis}[%
width=2.4in,
height=1.8in,
at={(1.011111in,0.813889in)},
scale only axis,
xmode=log,
xmin=2,
xmax=180,
xminorticks=true,
xlabel={$n$},
xmajorgrids,
xminorgrids,
ymode=log,
ymin=1.22,
ymax=10,
ylabel={$J^{m}_{n,N}$},
ymajorgrids,
ylabel style={yshift=-0.1cm},
legend style={at={(0.05,0.95)},anchor=north west, legend cell align=left,align=left,draw=white!15!black,line width=1.0pt,font=\footnotesize}
]

\addplot[area legend,solid,fill=mycolor1,opacity=2.500000e-01,draw=none,forget plot]
table[row sep=crcr] {%
x	y\\
2	1.21851798976243\\
4	1.30443684498261\\
6	1.31819413982273\\
8	1.31956680978645\\
10	1.32063874267718\\
12	1.32622125757727\\
14	1.33511028133313\\
16	1.33743318967419\\
18	1.3397104227538\\
20	1.34655065886455\\
40	1.46720085341757\\
60	1.76360370903982\\
80	2.26438668404005\\
100	4.55630500632129\\
120	4.68562958900335\\
140	4.90413361260079\\
160	5.0593249508123\\
180	5.21278776532087\\
180	9.11800395566789\\
160	9.02971157705474\\
140	8.85093554389685\\
120	8.69465516614353\\
100	8.47554283180835\\
80	5.52583037035595\\
60	3.9949639596318\\
40	3.13653687283261\\
20	1.98504074609239\\
18	1.93519498920906\\
16	1.85518514993278\\
14	1.77572363257799\\
12	1.76308012712116\\
10	1.754747441654\\
8	1.71271950522839\\
6	1.61984949346223\\
4	1.47844480442941\\
2	1.37467167830166\\
}--cycle;

\addplot [color=mycolor1, solid,line width=1.5pt]
  table[row sep=crcr]{%
2	1.26251547326883\\
4	1.33728090660148\\
6	1.35610228421885\\
8	1.37090969606023\\
10	1.38707930240325\\
12	1.40599599816037\\
14	1.43076613797033\\
16	1.45623650085315\\
18	1.4836108542614\\
20	1.51349432509347\\
40	1.9267441013783\\
60	2.54238485952415\\
80	3.35596507129358\\
100	5.96668784867735\\
120	6.27381638798601\\
140	6.45756983451891\\
160	6.59459279795208\\
180	6.71340696544437\\
};
\addlegendentry{$N=10^3$};


\addplot[area legend,solid,fill=mycolor3,opacity=2.500000e-01,draw=none,forget plot]
table[row sep=crcr] {%
x	y\\
2	1.21735378875179\\
4	1.30203122260995\\
6	1.31436443314809\\
8	1.31717590676467\\
10	1.31810772967677\\
12	1.31819094464691\\
14	1.31826734558735\\
16	1.31825574502779\\
18	1.31827251212186\\
20	1.31844467071736\\
40	1.31966487484554\\
60	1.3222302280342\\
80	1.32726567732878\\
100	1.68017892109808\\
120	1.70897122532832\\
140	1.73403577602776\\
160	1.74600880928492\\
180	1.7561456552525\\
180	2.67096848508025\\
160	2.63113833620444\\
140	2.5520652959155\\
120	2.51065957060969\\
100	2.39852268554186\\
80	1.38019526688269\\
60	1.35497404688767\\
40	1.34131524473376\\
20	1.32680792981274\\
18	1.32575873606598\\
16	1.32476385245458\\
14	1.32372124962789\\
12	1.32284276756395\\
10	1.32217280109155\\
8	1.32081248065311\\
6	1.31914554578041\\
4	1.30952720517522\\
2	1.23326155041944\\
}--cycle;

\addplot [color=mycolor3, solid,line width=1.5pt]
  table[row sep=crcr]{%
2	1.22197298436012\\
4	1.30431192046822\\
6	1.31569801317773\\
8	1.31815228975377\\
10	1.31890867229767\\
12	1.31918521190929\\
14	1.31945129317029\\
16	1.31972592992585\\
18	1.31998273607678\\
20	1.32035210294587\\
40	1.3247706780631\\
60	1.33175461536792\\
80	1.34149006800039\\
100	1.93200493351356\\
120	2.00934924669849\\
140	2.04397706749593\\
160	2.07230216709334\\
180	2.10070013985845\\
};
\addlegendentry{$N=10^4$};

\addplot[area legend,solid,fill=mycolor5,opacity=0.3,draw=none,forget plot]
table[row sep=crcr] {%
x	y\\
2	1.21711431227416\\
4	1.30185158400481\\
6	1.31420552973648\\
8	1.31686326380858\\
10	1.31761241468131\\
12	1.31000701073858\\
14	1.31279512090143\\
16	1.31395556034419\\
18	1.31315007002564\\
20	1.31365417514198\\
40	1.3138889232756\\
60	1.31451022992511\\
80	1.31448627651163\\
100	1.35171770649448\\
120	1.35438824208047\\
140	1.35339515585583\\
160	1.35614830580725\\
180	1.36328039354759\\
180	1.44512431207682\\
160	1.43976618328982\\
140	1.43838854039343\\
120	1.43526343213205\\
100	1.42416020573977\\
80	1.31842088139675\\
60	1.31815332020937\\
40	1.31744801473338\\
20	1.31729020611809\\
18	1.3167972631667\\
16	1.31632661989196\\
14	1.316908935996\\
12	1.31716697928315\\
10	1.31810085573192\\
8	1.31734333272814\\
6	1.31475269412566\\
4	1.30275153583959\\
2	1.21869697963803\\
}--cycle;

\addplot [color=mycolor5, solid,line width=1.5pt]
  table[row sep=crcr]{%
2	1.21754818498613\\
4	1.30207458668551\\
6	1.31439587701071\\
8	1.3171335016792\\
10	1.31792385826077\\
12	1.31470844418087\\
14	1.3141814856865\\
16	1.31522983661363\\
18	1.31539488406366\\
20	1.31536240109319\\
40	1.31568604356944\\
60	1.31621446223175\\
80	1.31628958360746\\
100	1.38116074999332\\
120	1.38739085412177\\
140	1.39121184207134\\
160	1.39299609428222\\
180	1.39299609428222\\
};
\addlegendentry{$N=10^5$};


\addplot [color=red,loosely dotted,line width=2pt]
  table[row sep=crcr]{%
1	1.3187\\
100000	1.3187\\
};
\addlegendentry{$J^{\text{AC}}$}

\end{axis}
\end{tikzpicture}%
}\label{fig:LQG:varying:n}} \\
	\center{
	\subfigure[varying kernel-learning samples $m$, where $N=10^3$]{\scalebox{1}{
%
%
\definecolor{mycolor1}{rgb}{0.13, 0.55, 0.13}%
\definecolor{mycolor2}{rgb}{0.13, 0.55, 0.13}%

\definecolor{mycolor3}{rgb}{0.32,0.09,0.98}%
\definecolor{mycolor4}{rgb}{0.32,0.09,0.98}%

\definecolor{mycolor5}{rgb}{0.59, 0.29, 0.0}
\begin{tikzpicture}

\begin{axis}[%
width=2.4in,
height=1.8in,
at={(1.011111in,0.813889in)},
scale only axis,
xmode=log,
xmin=10,
xmax=100000,
xminorticks=true,
xlabel={$m$},
xmajorgrids,
xminorgrids,
ymode=linear,
ymin=0.5,
ymax=1.5,
ylabel={$J^{m}_{n,N}$},
ymajorgrids,
ylabel style={yshift=-0.1cm},
legend style={at={(0.63,0.42)},anchor=north west, legend cell align=left,align=left,draw=white!15!black,line width=1.0pt,font=\footnotesize}
]

\addplot[area legend,solid,fill=mycolor1,opacity=0.3,draw=none,forget plot]
table[row sep=crcr] {%
x	y\\
10	0.232881213307227\\
20	0.356148544877288\\
30	0.457087634122338\\
40	0.581657489011523\\
50	0.593824252699872\\
60	0.65037782949258\\
70	0.693298832988233\\
80	0.739132685981094\\
90	0.800725934647246\\
100	0.798407294631585\\
200	0.920436520647634\\
300	0.947109176091395\\
400	1.0030964689061\\
500	1.03104870502016\\
600	1.04909972599431\\
700	1.05793913052624\\
800	1.07642749242451\\
900	1.06274365938646\\
1000	1.08101353703667\\
2000	1.12955368980325\\
3000	1.13351058323775\\
4000	1.15374875511558\\
5000	1.16369119296186\\
6000	1.15997843087475\\
7000	1.1724531493355\\
8000	1.16904828206887\\
9000	1.1754227968988\\
10000	1.17438492876108\\
20000	1.18384476181721\\
30000	1.19589568656287\\
40000	1.19719809353353\\
50000	1.20234037205148\\
60000	1.20385507469031\\
70000	1.20691301836696\\
80000	1.20719276565617\\
90000	1.20376774641464\\
100000	1.2070915553172\\
100000	1.21430856590813\\
90000	1.2141756540796\\
80000	1.21421862248823\\
70000	1.2120650506349\\
60000	1.21198151532706\\
50000	1.21102866463212\\
40000	1.21240951401625\\
30000	1.2111715295248\\
20000	1.20557271651245\\
10000	1.19734815842297\\
9000	1.19943105230394\\
8000	1.19363321169517\\
7000	1.19580544282455\\
6000	1.19645814023396\\
5000	1.18613087571596\\
4000	1.18283030600192\\
3000	1.17739752401551\\
2000	1.1661731477206\\
1000	1.13402212856781\\
900	1.13036230432177\\
800	1.13275106765916\\
700	1.11156456422729\\
600	1.1143706798054\\
500	1.11434230268183\\
400	1.09775768824323\\
300	1.04323430710002\\
200	1.00055620336308\\
100	0.949438088581632\\
90	0.922048006723959\\
80	0.953005767591718\\
70	0.910417682509249\\
60	0.8658799558025\\
50	0.813510306518754\\
40	0.805122886213041\\
30	0.702748485629701\\
20	0.653209902330234\\
10	0.429409740513393\\
}--cycle;

\addplot [color=mycolor1, solid,line width=1.5pt]
  table[row sep=crcr]{%
10	0.326340555123576\\
20	0.518395930678215\\
30	0.629952091540719\\
40	0.689869073825465\\
50	0.706810817420131\\
60	0.759889859640624\\
70	0.81162387362111\\
80	0.827266235386798\\
90	0.856658003214213\\
100	0.866626527602622\\
200	0.970428787753394\\
300	1.01047660932627\\
400	1.04494220862698\\
500	1.06304746600157\\
600	1.08284708792655\\
700	1.0876096575484\\
800	1.10470350362355\\
900	1.10354426816513\\
1000	1.10718836060328\\
2000	1.14748505645753\\
3000	1.15930053976622\\
4000	1.1699703751703\\
5000	1.17663260078688\\
6000	1.1784894285926\\
7000	1.18471565274468\\
8000	1.18281704376986\\
9000	1.18673200354038\\
10000	1.18662915342755\\
20000	1.19863993037484\\
30000	1.2039298164394\\
40000	1.20572813538461\\
50000	1.20684696456388\\
60000	1.20831076900203\\
70000	1.20949423730441\\
80000	1.20999836479319\\
90000	1.20978950856812\\
100000	1.21014295051457\\
};
\addlegendentry{$n=2$};



\addplot[area legend,solid,fill=mycolor3,opacity=0.3,draw=none,forget plot]
table[row sep=crcr] {%
 x	y\\
10	0.582669363461729\\
20	0.713503361815557\\
30	0.801000078637008\\
40	0.863041628115513\\
50	0.824184995225978\\
60	0.889919933692065\\
70	0.873848303880541\\
80	0.909886179338578\\
90	0.958289915223431\\
100	0.98760889816251\\
200	1.05406034777605\\
300	1.08401923080819\\
400	1.09595956607934\\
500	1.10437748757874\\
600	1.14232532881519\\
700	1.14903823778039\\
800	1.15121573811896\\
900	1.16366896883282\\
1000	1.16508331472741\\
2000	1.20617547112529\\
3000	1.21823466980118\\
4000	1.23410583557331\\
5000	1.23973982148308\\
6000	1.25285258699463\\
7000	1.24968272214581\\
8000	1.25326127787138\\
9000	1.25360055410181\\
10000	1.26594076445989\\
20000	1.27727458413504\\
30000	1.28195072169741\\
40000	1.28672894704855\\
50000	1.29018872102953\\
60000	1.28952125125057\\
70000	1.29265275424633\\
80000	1.29397062927464\\
90000	1.29317860570628\\
100000	1.29627912281282\\
100000	1.30100822266649\\
90000	1.30132761431559\\
80000	1.29879764304691\\
70000	1.29854059554445\\
60000	1.29949320278511\\
50000	1.29588216698003\\
40000	1.29442722101335\\
30000	1.292972990012\\
20000	1.28898076932809\\
10000	1.27373684200409\\
9000	1.27349106537047\\
8000	1.27203183831843\\
7000	1.26640947302218\\
6000	1.27174683747886\\
5000	1.26404700731583\\
4000	1.25374389635512\\
3000	1.24339563295959\\
2000	1.23650361336035\\
1000	1.20200129745288\\
900	1.19547091188644\\
800	1.1911456777277\\
700	1.18754074549618\\
600	1.19795034930343\\
500	1.1716239992355\\
400	1.17580457016585\\
300	1.13336996393232\\
200	1.10714329844758\\
100	1.03946240982216\\
90	1.04846082595235\\
80	1.02340813419033\\
70	1.00975258942825\\
60	0.994380491241151\\
50	0.979326836233709\\
40	0.953316608897577\\
30	0.940936712742852\\
20	0.814555175264589\\
10	0.747557351656977\\
}--cycle;

\addplot [color=mycolor3, solid,line width=1.5pt]
  table[row sep=crcr]{%
10	0.660487099929831\\
20	0.775815467538281\\
30	0.859712032254272\\
40	0.900261087787939\\
50	0.922951897688918\\
60	0.941358498910951\\
70	0.952793337310516\\
80	0.971282075837294\\
90	0.998399567105143\\
100	1.01144139222619\\
200	1.07956267945674\\
300	1.11303038185869\\
400	1.13490006962336\\
500	1.14478324792376\\
600	1.16126266565427\\
700	1.17056370161563\\
800	1.1743114147005\\
900	1.18345748045973\\
1000	1.18749278072144\\
2000	1.22062626889082\\
3000	1.23355747276165\\
4000	1.24255784330174\\
5000	1.25131363354072\\
6000	1.25964742095556\\
7000	1.25981461460671\\
8000	1.26279845406276\\
9000	1.26562337884783\\
10000	1.26892313889254\\
20000	1.28167105982375\\
30000	1.28804883405422\\
40000	1.29088486362748\\
50000	1.2929072327024\\
60000	1.29420283783397\\
70000	1.29609995011779\\
80000	1.29678980749725\\
90000	1.29775263063731\\
100000	1.29859115571563\\
};

\addlegendentry{$n=10$};


\addplot[area legend,solid,fill=mycolor5,opacity=0.3,draw=none,forget plot]
table[row sep=crcr] {%
x	y\\
10	1.06944728090579\\
20	1.10528977391431\\
30	1.14010972549472\\
40	1.14527742769216\\
50	1.16074895642217\\
60	1.15975929463208\\
70	1.17012602946377\\
80	1.1766421045435\\
90	1.19157709285196\\
100	1.18954164408351\\
200	1.21284802537536\\
300	1.22764491839548\\
400	1.24000596723195\\
500	1.24520556034427\\
600	1.24857605312802\\
700	1.25578359030127\\
800	1.25764483382203\\
900	1.26101572642139\\
1000	1.26331016873695\\
2000	1.27612246866084\\
3000	1.2816098979132\\
4000	1.28367210450358\\
5000	1.28781297518126\\
6000	1.2919032574996\\
7000	1.29263967154221\\
8000	1.29293194596838\\
9000	1.29502496901465\\
10000	1.29601675054081\\
20000	1.30139317640108\\
30000	1.30339056831679\\
40000	1.30515843934738\\
50000	1.30684274914745\\
50000	1.31121929831254\\
40000	1.31032016723131\\
30000	1.30901815865109\\
20000	1.30787009292914\\
10000	1.30409185505738\\
9000	1.30540017379745\\
8000	1.30335596976406\\
7000	1.30342914282741\\
6000	1.30298036165659\\
5000	1.30193288184888\\
4000	1.29675319648424\\
3000	1.30486099433243\\
2000	1.29565418873804\\
1000	1.29024220062187\\
900	1.28714969177944\\
800	1.28644840566333\\
700	1.29017719558066\\
600	1.28173999384842\\
500	1.2793013740183\\
400	1.28109523734467\\
300	1.27557181919701\\
200	1.27133384925198\\
100	1.26840833147492\\
90	1.2702497335596\\
80	1.27940384316883\\
70	1.25896627245158\\
60	1.25944065566103\\
50	1.29938732183462\\
40	1.25951888814438\\
30	1.3059204014388\\
20	1.28892065489674\\
10	1.29694088349055\\
}--cycle;

\addplot [color=mycolor5, solid,line width=1.5pt]
  table[row sep=crcr]{%
10	1.18043930558203\\
20	1.18498240198077\\
30	1.21503897145779\\
40	1.20321892869736\\
50	1.21641671491051\\
60	1.21560820966806\\
70	1.21482441882843\\
80	1.22390667864723\\
90	1.22609963514618\\
100	1.23927661089482\\
200	1.24354087820241\\
300	1.25174369626898\\
400	1.26099611866841\\
500	1.26500474861283\\
600	1.27368694684411\\
700	1.28073461049789\\
800	1.28458754430366\\
900	1.28069950334232\\
1000	1.2899396141424\\
2000	1.28995185521436\\
3000	1.31136998811437\\
4000	1.29203050411657\\
5000	1.30131721700882\\
6000	1.31841274664937\\
7000	1.31044950370411\\
8000	1.29878705645434\\
9000	1.3089660489702\\
10000	1.30689386084153\\
20000	1.31185446361884\\
30000	1.32169608799924\\
40000	1.33653176763586\\
50000	1.32035813572648\\
};

\addlegendentry{$n=100$};


\addplot [color=red,loosely dotted,line width=2pt]
  table[row sep=crcr]{%
1	1.3187\\
100000	1.3187\\
};
\addlegendentry{$J^{\text{AC}}$}

\end{axis}
\end{tikzpicture}%
}\label{fig:LQG:varying:m}}
	\caption{The objective performance $J_{n,N}^m$ is computed according to \eqref{AC-LP-n,N}. The colored tubes represent the results between $[10\%,90\%]$ quantiles (shaded areas) as well as the means (solid lines) across 200 independent experiments of the objective performance $J_{n,N}^m$. The red dotted line denoted by $J^\text{AC}$ is the optimal solution approximated by $n = 10^3$, $m=10^6$ and $N = 10^6$.}}
	\label{fig:LQG:ran}
\end{figure}
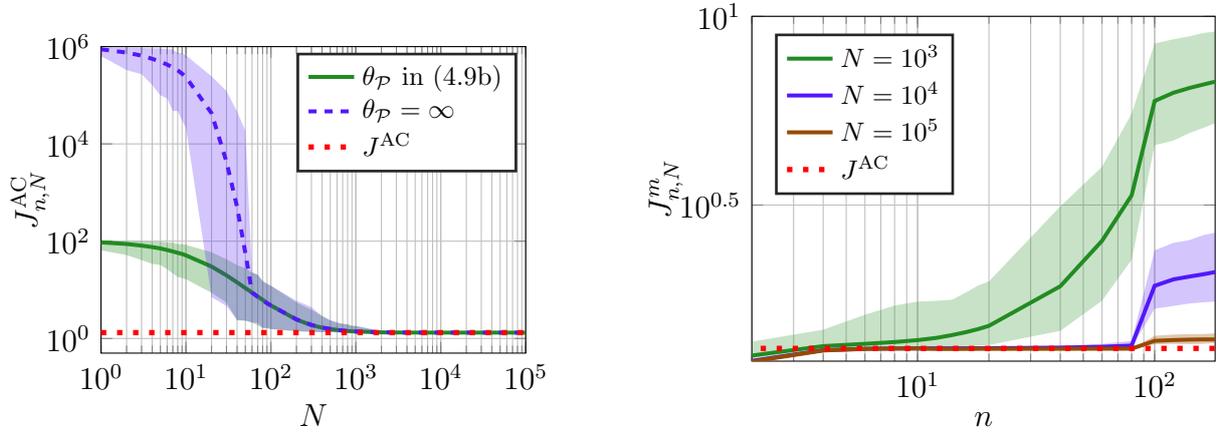
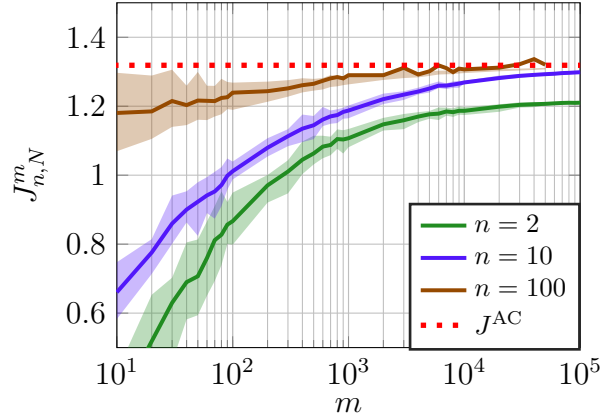
The simulation results are shown in Figure~\ref{fig:LQG:ran}.
Figure~\ref{fig:LQG:varying:N} suggests three interesting features concerning  $n$, the number of basis functions: The higher the number of basis functions,
\begin{enumerate}[label=(\roman*), itemsep = 1mm, topsep = -1mm]
	\item \label{sim:err}
	the smaller the approximation error (i.e., asymptotic distance for $N\to\infty$ to the red dotted line),
	\item \label{sim:var}
	the lower the variance of approximation with respect to the sampling distribution for each $N$, and
	\item \label{sim:conv}
	the slower the convergence behavior with respect to the sample size $N$.
\end{enumerate}
The feature~\ref{sim:conv}, namely that a high number of basis functions requires a large number of sampled constraints $N$ to produce reasonable approximation errors can also be seen in Figure~\ref{fig:LQG:varying:n}. Moreover, the higher the number of sampled constraints $N$ the lower the variance of the approximation. Figure~\ref{fig:LQG:varying:n} suggests that there a sweet spot, namely given a certain number $n$ of basis functions, there is a minimum number of sampled constraints $N$ required for an acceptable approximation accuracy.
Finally, Figure~\ref{fig:LQG:varying:m} indicates that the more basis functions $n$, the less samples from the kernel $m$ are required for $J_{n,N}^m$ to be close to the optimal value.

\section{Infinite-horizon discounted-cost problems} \label{sec:discounted:setting}
In the Markov decision process setting, introduced in Section~\ref{section:MDP:LP}, let us consider \emph{long-run $\tau$-discounted cost} (DC) problems with the discount factor $\tau\in(0,1)$ and initial distribution $\nu\in\mathcal{P}(X)$ described as
	\begin{equation} \label{eq:discounted:problem}
	\Jdc(\nu)\Let \inf_{\pi\in\Pi} \lim_{n\to\infty} \Expecpi{\sum_{t=0}^{n-1}\tau^t \cost(x_{t},a_{t})}.
	\end{equation}
		
As in the average cost setting, in Section~\ref{section:MDP:LP}, we assume that the control model satisfies Assumption~\ref{a:CM}. We refer to \cite[Chapter~4]{ref:Hernandez-96} and \cite[Chapter~8]{ref:Hernandez-99} for a detailed exposition and required technical assumptions in more general settings. As for the AC problems, it is well known that the DC problem~\eqref{eq:discounted:problem} can be alternatively characterized by means of infinite LPs \eqref{primal-inf} and \eqref{dual-inf} introduced in Section~\ref{subsec:dual-pair}, where
	\begin{align}
	\label{AC-setting:discounted}
	\left\{
	\begin{array}{l}
		(\X,\C) \Let (\Cont(S),\Meas(S)) \\
		(\B,\Y) \Let (\Cont(K),  \Meas(K)) \\
		\cone \Let \Cont_{+}(K)\\
		\cone^{*} \Let \Meas_{+}(K) \\
		c(B) = -\nu(B), \quad B\in\borel(S)\\
		b(s,a) = -\cost(s,a) \\
		\op : \X \ra \B, \quad \op x(s,a) \Let - x(s) + \tau Qx(s,a) \\
		\op^{*}: \Y \to \C, \quad \op^{*}y (B) \Let y(B\times A) - \tau y Q(B), \quad B\in \borel(S).
		\end{array}\right.
	\end{align}
	
\begin{Thm}[LP characterization {\cite[Theorem~6.3.8]{ref:Hernandez-96}}] \label{thm:equivalent:LP:discounted}
	Under Assumption~\ref{a:CM}, the optimal value $\Jdc$ of the DC problem in \eqref{eq:discounted:problem} can be characterized by the LP problem \eqref{primal-inf} in the setting \eqref{AC-setting:discounted}, in the sense that $\Jp = -\Jdc$.
\end{Thm}

It is known that under similar conditions as in Assumption~\ref{a:CM} on the control model, the value function $u^{\star}$ in the $\tau$-discounted cost optimality equation is Lipschitz continuous \cite[Section~2.6]{ref:HernandezLerma-89} or \cite[Theorem~3.1]{ref:Duf-13}. We use the norms similar to the AC-setting \eqref{AC:pairs}. The next step toward studying the approximation error \eqref{inf-semi error} for the DC-setting readily follows by Theorem~\ref{thm:inf-semi} combined with the following lemma.
		
		\begin{Lem}[DC semi-infinite regularity]\label{lem:MDP:DC}
			For the DC-problem \eqref{eq:discounted:problem}, characterized by the dual-pair vector spaces in \eqref{AC-setting:discounted}, under Assumption~\ref{a:CM} we have the operator norm $\|\op\| \leq 1+ \max\{L_Q,1\}\tau$, the inf-sup constant of Assumption~\ref{a:reg}\ref{a:reg:inf-sup} $\gamma = 1-\tau$, and the dual optimizer norm
			\begin{align}
			\label{yb:DC}
			\|y\opt\|_\wass \le \ynb = \frac{\xnb+ (1-\tau)^{-1}\|\cost\|_{\infty}}{(1-\tau) \xnb - \|\cost\|_{\lip}} \,.
			\end{align}
		\end{Lem}
		\begin{proof}
			With the norms considered and following a proof similar to Lemma~\ref{lem:operator}, the operator norm $\|\op\|$ can be upper bounded as
			$\| \op \| \leq 1+\tau.$ The \emph{inf-sup} condition, Assumption~\ref{a:reg}\ref{a:reg:inf-sup}, holds with $\gamma = 1-\tau$, since
			\begin{align*}
			\inf_{y \in \cone^*}\sup_{x \in \X_n} {\inner{\op x}{y} \over \|x\| \|y\|_{\wass}} \geq \inf_{y \in \cone^*} \frac{(1-\tau)\inner{\ind}{y}}{\| y \|_{\wass}} = 1-\tau.
			\end{align*}
			Moreover $\|\nu\|_{\wass} = 1$ since it is a probability measure. Thus, given the lower bound for the optimal value $\Jdcn \ge {-(1-\tau)^{-1}\|\cost\|_{\infty}}$, the assertion of Proposition~\ref{prop:SD} (i.e, the dual optimizers bound in \eqref{yb}) leads to the desired assertion \eqref{yb:DC}. 
		\end{proof}
		
		Note that when the norm constraint is neglected, the dual program enforces that any solution $\yn$ in the program \ref{dual-n} satisfies $\inner{x}{\op^{*}\yn -c}=0$ for all $x\in \X_{n}$ (cf. the program~\ref{dual-inf}). Assume that a constant function belongs to the set $\X_{n}$. Then, the constraint evaluated at the constant function reduces to $(1-\tau)\inner{\ind}{\yn} = (1-\tau)\|\yn\|_\wass = 1$. It is worth noting that this observation can consistently be captured by Lemma~\ref{lem:MDP:DC} when $\xnb$ tends to $\infty$, in which the bound \eqref{yb:DC} reduces to $\|\yn\|_{\wass} \le ({1-\tau})^{-1}$.

\part{Information theoretic problems}\label{part:IT}



\chapter{Channel capacity approximation} \label{chap:channel:cap}

The second part of the thesis studies optimization problems arising from an information theoretic perspective, namely the channel capacity problem and the maximum entropy estimation problem. In this chapter we present a novel iterative method for approximately computing the capacity of discrete memoryless channels, possibly under additional constraints on the input distribution. Based on duality of convex programming, we derive explicit upper and lower bounds for the capacity. 
The presented method requires $O(M^2 N \sqrt{\log N}/\varepsilon)$ to provide an estimate of the capacity to within $\varepsilon$, where $N$ and $M$ denote the input and output alphabet size; a single iteration has a complexity $O(M N)$. We also show how to approximately compute the capacity of memoryless channels having a bounded continuous input alphabet and a countable output alphabet under some mild assumptions on the decay rate of the channel's tail.
It is shown that discrete-time Poisson channels fall into this problem class. As an example, we compute sharp upper and lower bounds for the capacity of a discrete-time Poisson channel with a peak-power input constraint.

\section{Introduction}
A discrete memoryless channel (DMC) comprises a finite input alphabet $\mathcal{X} = \{1,2,\hdots,N\}$, a finite output alphabet $\mathcal{Y} = \{1,2,\hdots,M\}$, and a conditional probability mass function expressing the probability of observing the output symbol $y$ given the input symbol $x$, denoted by $W(y|x)$. In his seminal 1948 paper \cite{shannon48}, Shannon proved that the channel capacity for a DMC is
\begin{equation}
C(W)=\max \limits_{p \in \Delta_N} \I{p}{W},  \label{eq:shannon48}
\end{equation} 
where $\Delta_{N}\!\!:=\!\{  x\in\R^{N} \!: \ \! x\geq 0,  \sum_{i=1}^{N} x_{i}=1\}$\! denotes the $N$-simplex and $\I{p}{W}\!\!:=\!\sum_{x \in \mathcal{X}}  p(x)$ $ \D{W(\cdot|x)}{(pW)(\cdot)}$ the mutual information. $W(y|x)=\ProbIT{Y=y|X=x}$ describes the channel law and $(pW)(\cdot)$ is the probability distribution of the channel output induced by $p$ and $W$, i.e., $(pW)(y):=\sum_{x \in \mathcal{X}} p(x) W(y|x)$. $\D{\cdot}{\cdot}$ denotes the relative entropy that is defined as $\D{W(\cdot|x)}{(pW)(\cdot)}:=\sum_{y \in \mathcal{Y}} W(y|x) \log\left(\tfrac{W(y|x)}{(pW)(y)}\right)$.
Shannon also showed that in case of an additional average cost constraint on the input distribution of the form $\E{s(X)}\leq S$, where $s:\mathcal{X}\to\Rp$ denotes a cost function and $S\geq 0$, the capacity is given by
\begin{equation} \label{eq:DMC_capacity_const}
C_S(W)= \left\{
\begin{array}{lll}
			&\max\limits_{p} 		& \I{p}{W} \\
			&\subjectto			& \E{s(X)}\leq S\\
			& 					& p\in \Delta_{N}.
	\end{array} \right.
\end{equation}
For a few DMCs it is known that the capacity can be computed analytically, however in general there is no closed-form solution. It is therefore of interest to have an algorithm that solves \eqref{eq:DMC_capacity_const} in a reasonable amount of time. Since for a fixed channel the mutual information is a concave function in $p$, the optimization problem \eqref{eq:DMC_capacity_const} is a finite dimensional convex optimization problem. Solving \eqref{eq:DMC_capacity_const} with convex programming solvers, however, turned out to be computationally inefficient even for small alphabet sizes \cite{blahut72}. 

Shannon's formula for the capacity of a DMC generalizes to the case of memoryless channels with continuous input and output alphabets, i.e. $\mathcal{X}=\mathcal{Y}=\R$.  However, when considering such channels, it is essential to introduce additional constraints on the channel input to obtain physically meaningful results, more details can be found in \cite[Chapter~7]{gallager68}.
In addition to average cost type constraints, peak-power constraints are also often considered. A peak-power constraint demands that $X\in \A$ for some compact set $\A\subset\mathcal{X}$ with probability one. For such a setup, i.e., having average and peak-power constraints, the capacity is given by
\begin{equation} \label{eq:cont_DMC_capacity_const}
C_{\A,S}(W)= \left\{
\begin{array}{lll}
			&\sup\limits_{p} 		& \I{p}{W} \\
			&\subjectto			& \E{s(X)}\leq S\\
			& 					& p\in \mathcal{P}(\A),
	\end{array} \right.
\end{equation}
where $\mathcal{P}(\A)$ denotes the set of all probability distributions on the Borel $\sigma$-algebra $\mathcal{B}(\A)$ and the mutual information is defined as $\I{p}{W}:=\int_{\A} \D{W(\cdot|x)}{(pW)(\cdot)} p(\drv x)$. The channel is described by a transition density defined by $\ProbIT{Y\in\drv y|X=x}=W(y|x)\drv y$ and $(pW)(\cdot)$ is the probability distribution of the channel output induced by $p$ and $W$ which is given by $(pW)(y):=\int_{\A}W(y|x)p(\drv x)$ and the relative entropy that is defined as $\D{W(\cdot|x)}{(pW)(\cdot)}:=\int_{\mathcal{Y}} W(y|x) \log\left(\tfrac{W(y|x)}{(pW)(y)}\right)\drv y$.
The optimization problem \eqref{eq:cont_DMC_capacity_const} is an infinite dimensional convex optimization problem and as such in general computationally intractable (NP-hard).

\vspace{3mm}
\subsection*{Previous Work and Contributions}
Historically one of the first attempts to numerically solve \eqref{eq:DMC_capacity_const} is the so-called \emph{Blahut-Arimoto algorithm} \cite{blahut72,arimoto72}, that exploits the special structure of the mutual information and approximates iteratively the capacity of any DMC. Each iteration step has a computational complexity $O(MN)$. It was shown that this algorithm, in case of no additional input constraints has an \textit{a priori} error bound of the form $|C(W)-C_{\textnormal{approx}}^{(n)}(W)|\leq O(\tfrac{\log(N)}{n})$, where $n$ denotes the number of iterations \cite[Corollary~1]{arimoto72}. Hence, the overall computational complexity of finding an additive $\varepsilon$-solution is given by $O(\tfrac{MN\log(N)}{\varepsilon})$. As such the computational cost required for an acceptable accuracy for channels with large input alphabets can be considerable. 
This undesirable property together with the complexity per iteration prevents the algorithm from being useful for a large class of channels, e.g., a Rayleigh channel with a discrete input alphabet \cite{shamai01}.
There have been several improvements of the Blahut-Arimoto algorithm \cite{sayir00,matz04,yaming10}, which achieve a better convergence for certain channels. However, since they all rely on the original Blahut-Arimoto algorithm they inherit its overall computational complexity as well as its complexity per iteration step. Therefore, even with improved Blahut-Arimoto algorithms, approximating the capacity for channels having large input alphabets remains computationally expensive. 
Based on sequential Monte-Carlo integration methods (a.k.a.\ particle filters), the Blahut-Arimoto algorithm has been extended to memoryless channels with continuous input and output alphabets \cite{dauwels05,ref:Chen-13,ref:Chen-14-1, ref:Chen-14-2}. As shown in several examples, this approach seems to be powerful in practice, however a rate of convergence has not been proven.

Another recent approach towards approximating \eqref{eq:DMC_capacity_const} is presented in \cite{chiang04} by Mung and Boyd, where 
they introduce an efficient method to derive upper bounds on the channel capacity problem, based on geometric programming.
Huang and Meyn \cite{meyn05} developed a different approach based on cutting plane methods, where the mutual information is iteratively approximated by linear functionals and in each iteration step, a finite dimensional linear program is solved. It has been shown that this method converges to the optimal value, however no rate of convergence is provided. An interesting insight provided by the work \cite{meyn05} and \cite{ref:Pandit-04} is that the in the setting of a continuous alphabet, the optimal input distributions typically have finite support.

In this article, we present a new approach to solve \eqref{eq:DMC_capacity_const} that is based on its dual formulation. It turns out that the dual problem of \eqref{eq:DMC_capacity_const} has a particular structure that allows us to apply Nesterov's smoothing method \cite{ref:Nest-05}. In the absence of input cost constraints, this leads to an a priori error bound of the order $|C(W)-C_{\textnormal{approx}}^{(n)}(W)|\leq O(\tfrac{M \sqrt{\log(N)}}{n})$, where $n$ denotes the number of iterations and each iteration step has a  computational complexity of $O\!\left(NM \right)$.
Thus, the overall computational complexity of finding an $\varepsilon$-solution is given by $O(\tfrac{M^{2} N \sqrt{\log(N)}}{\varepsilon})$.
In particular for large input alphabets our method has a computational advantage over the Blahut-Arimoto algorithm. In addition the novel method provides primal and dual optimizers leading to an \emph{a posteriori} error which is often much smaller than the a priori error. 

Due to the favorable structure of the capacity problem and its dual formulation, the presented method can be extended to approximate the capacity of memoryless channels having a bounded continuous input alphabet and a countable output alphabet, under some assumptions on the tail of $W(\cdot|x)$, i.e., problem \eqref{eq:cont_DMC_capacity_const} is addressed for a countable output alphabet. As a concrete example, this is demonstrated on the discrete-time Poisson channel with a peak-power constraint. To the best of our knowledge, for this scenario up to now only lower bounds exist \cite{lapidoth09}.

\subsection*{Notation}
The logarithm with basis 2 is denoted by $\log(\cdot)$ and the natural logarithm by $\ln(\cdot)$. In Section~\ref{sec:classicalCapacity} we consider DMCs with a finite input alphabet $\mathcal{X}=\{ 1,2,\hdots,N \}$ and a finite output alphabet $\mathcal{Y}=\{ 1,2,\hdots,M \}$. The channel law is summarized in a matrix $\W\in\R^{N\times M}$, where $\W_{ij}:=\ProbIT{Y=j|X=i}=W(j|i)$. We define the standard $n-$simplex as $\Delta_{d}:=\left\{  x\in\R^{d} : x\geq 0, \sum_{i=1}^{d} x_{i}=1\right\}$. The input and output probability mass functions are denoted by the vectors $p\in \Delta_{N}$ and $q\in\Delta_{M}$. The input cost constraint can be written as $\E{s(X)} = p\transp s\leq S$, where $s\in \Rp^{N}$ denotes the cost vector and $S\in \Rp$ is the given total cost. The binary entropy function is denoted by $\Hb(\alpha):=-\alpha \log(\alpha)-(1-\alpha)\log(1-\alpha)$, for $\alpha\in [0,1]$. For a probability mass function $p \in \Delta_{N}$ we denote the entropy by $H(p):=\sum_{i=1}^N -p_i \log(p_i)$. It is convenient to introduce an additional variable for the conditional entropy of $Y$ given $\{X=i\}$ as $r\in\R^{N}$, where $r_{i}=-\sum_{j=1}^{M}\W_{ij}\log(\W_{ij})$. For a probability density $p$ supported at a measurable set $B\subset \R$ we denote the differential entropy by $h(p)=-\int_{B} p(x) \log(p(x)) \drv x$.  For two vectors $x,y \in \R^n$, we denote the canonical inner product by $\left \langle x,y \right \rangle := x \transp y$. We denote the maximum (resp.~minimum) between $a$ and $b$ by $a \vee b$ (resp.~$a\wedge b$). For $\A\subset\R$ and $1\leq p \leq \infty$, let $\Lp{p}(\A)$ denote the space of $\Lp{p}$-functions on the measure space $(\A, \mathcal{B}(\A)\!, \drv x)$, where $\mathcal{B}(\A)$ denotes the Borel $\sigma$-algebra and $\drv x$ the Lebesgue measure.
The capacity of a channel $W$ is denoted by $C(W)$. 
For the channel law matrix  $\W\in\R^{N\times M}$ we consider the norm
$\|\W\|:= \max\limits_{\lambda\in\R^{M}, \ p\in\R^{N}}\left\{ \inprod{\W\lambda}{p} \ : \ \|\lambda\|_{2}=1, \ \|p\|_{1}=1 \right\},$
and note that an upper bound is given by 
\begin{equation} \label{eq:operator:norm}
\|\W\| 			=		\max\limits_{\|p\|_{1}=1} \max\limits_{\|\lambda\|_{2}=1} \lambda\transp \W\transp p \leq 	\max\limits_{\|p\|_{1}=1} \| \W\transp p \|_{2}\leq 	\max\limits_{\|p\|_{1}=1} \| \W\transp p \|_{1} =	\max\limits_{\|p\|_{1}=1} \| p\|_{1}=	1.
\end{equation}

\section{Discrete memoryless channel} \label{sec:classicalCapacity}
To keep notation simple we consider a single average-input cost constraint as the extension to multiple average-input cost constraints is straightforward. In a first step, we introduce the output distribution $q\in\Delta_{M}$ as an additional decision variable, as done in \cite{benTal88,chiang04,chiang05} and note that the mutual information $I(X;Y)$ is equal to $H(Y)-H(Y|X)$.
\begin{Lem} \label{lem:equivalent:primal:problem}
Let $\mathcal{F}:=\arg\max\limits_{p\in\Delta_{N}} \I{p}{W}$ and $S_{\max}:=\min\limits_{p\in\mathcal{F}}s\transp p$. If $S\geq S_{\max}$ the optimization problem \eqref{eq:DMC_capacity_const} has the same optimal value as
\begin{equation} \label{opt:primal:equivalent:no:power:constraints}
 	\mathsf{P}: \quad \left\{ \begin{array}{lll}
			&\max\limits_{p,q} 		&- r\transp p + H(q) \\
			&\subjectto					& \W\transp p = q\\
			& 					& p\in \Delta_{N}, \ q\in\Delta_{M}.
	\end{array} \right.
\end{equation}
If $S<S_{\max}$ the optimization problem \eqref{eq:DMC_capacity_const} has the same optimal value as
\begin{equation} \label{opt:primal:equivalent}
 	\mathsf{P}: \quad \left\{ \begin{array}{lll}
			&\max\limits_{p,q} 		&- r\transp p + H(q) \\
			&\subjectto					& \W\transp p = q\\
			&					& s\transp p = S \\
			& 					& p\in \Delta_{N}, \ q\in\Delta_{M}.
	\end{array} \right.
\end{equation}
\end{Lem}
\begin{proof}
The mutual information $\I{p}{W}$ can be expressed as
\begin{align*}
\I{p}{W} 	&= 	\sum_{i=1}^{N} \sum_{j=1}^{M} \W_{ij}p_{i}\log\left( \frac{\W_{ij}}{\sum_{k=1}^{N}\W_{kj}p_{k}} \right) \\
		&=	\sum_{i=1}^{N} \sum_{j=1}^{M} \left[ p_{i} \W_{ij}\log(\W_{ij}) - p_{i}\W_{ij}\log\left(\sum_{k=1}^{N}\W_{kj}p_{k} \right) \right].
\end{align*}
By adding the constraint $\sum_{i=1}^{N}p_{i}\W_{ij}=q_{j}$ for all $j=1,\hdots, M$, 
\begin{align*}
\I{p}{W} 	&= 	\sum_{i=1}^{N} \sum_{j=1}^{M} \left[ p_{i} \W_{ij}\log(\W_{ij}) - p_{i}\W_{ij}\log (q_{j}) \right]  \\
		&=	\sum_{i=1}^{N} \sum_{j=1}^{M}  p_{i} \W_{ij}\log(\W_{ij})  -\sum_{j=1}^{M} q_{j }\log(q_{j}) \\
		&= 	- r\transp p + H(q),
\end{align*}
where $p\in\Delta_{N}$. Since $q=\W\transp p$ and $\W\transp$ is a stochastic matrix, this implies $q\in\Delta_{M}$. 
By definition of $S_{\max}$ it is obvious that the input cost constraint $s\transp p \leq S$ is inactive for $S\geq S_{\max}$, leading to the first optimization problem in Lemma~\ref{lem:equivalent:primal:problem}. It remains to show that for
$S<S_{\max}$, the input constraint can be written with equality, leading to the second optimization problem in Lemma~\ref{lem:equivalent:primal:problem}. 
In oder to keep the notation simple we define $\mathsf{C}(S):=C_{S}(W)$ for a fixed channel $W$.
We show that $\mathsf{C}(S)$ is concave in $S$ for $S\in[0,S_{\max}]$. Let $S^{(1)},S^{(2)} \in [0,S_{\max}]$, $0\leq \lambda \leq 1$ and $p^{(i)}$ probability mass functions that achieve $\mathsf{C}(S^{(i)})$ for $i \in \{1,2 \}$. Consider the probability mass function $p^{(\lambda)}=\lambda p^{(1)}+ (1-\lambda) p^{(2)}$. We can write
\begin{align}
s\transp p^{(\lambda)} &= \lambda s\transp p^{(1)} + (1-\lambda) s\transp p^{(2)} \nonumber\\
 &\leq \lambda S^{(1)}+(1-\lambda) S^{(2)} \nonumber\\
 &=: S^{(\lambda)} \label{eq:slam} \in [0,S_{\max}].
\end{align}
Using the concavity of the mutual information in the input distribution, we obtain
\begin{align*}
\lambda \mathsf{C}(S^{(1)}) + (1-\lambda) \mathsf{C}(S^{(2)}) &= \lambda \I{p^{(1)}}{W}+(1-\lambda) \I{p^{(2)}}{W}  \\
&\leq \I{p^{(\lambda)}}{W} \\
&\leq \mathsf{C}(S^{(\lambda)}),
\end{align*}
where the final inequality follows by Shannon's formula for the capacity given in \eqref{eq:shannon48}. $\mathsf{C}(S)$ clearly is non-decreasing in $S$ since enlarging $S$ relaxes the input cost constraint. Furthermore, we show that
\begin{equation} \label{eq:proof:ineq:equality:epsilon:step}
\mathsf{C}(S_{\max}-\varepsilon)<\mathsf{C}(S_{\max}), \quad \text{for all }\varepsilon>0.
\end{equation}
Suppose $\mathsf{C}(S_{\max}-\varepsilon)=\mathsf{C}(S_{\max})$ and denote $\mathsf{C}^{\star}:=\max\limits_{p\in\Delta_{N}}\I{p}{W}$. This then implies that there exists $\bar{p}\in\Delta_{N}$ such that $\I{\bar{p}}{W}=\mathsf{C}^{\star}$ and $s\transp \bar{p}\leq S_{\max}-\varepsilon$, which contradicts the definition of $S_{\max}$. Hence, the concavity of $\mathsf{C}(S)$ together with the non-decreasing property and \eqref{eq:proof:ineq:equality:epsilon:step} imply that $\mathsf{C}(S)$ is strictly increasing in $S$. 
\end{proof}
Note that we later add an assumption on our channel (Assumption~\ref{ass:channel}) that guarantees uniqueness of the optimizer maximizing the mutual information, i.e., $\mathcal{F}$ is a singleton. In this case the optimizer to \eqref{opt:primal:equivalent} (resp. \eqref{opt:primal:equivalent:no:power:constraints}) is also feasible for the original problem \eqref{eq:DMC_capacity_const}. Computing $S_{\max}$ is straightforward once $\mathcal{F}$ is known. The singleton $\mathcal{F}$ can be seen as the maximizer of a channel capacity problem with no additional input cost constraint and can as such be computed with the scheme we present in this article. 

For the rest of the section we restrict attention to \eqref{opt:primal:equivalent}, since the less constrained problem \eqref{opt:primal:equivalent:no:power:constraints} can be solved in a similar, more direct way.
We tackle this optimization problem through its Lagrangian dual problem. The dual function turns out to be a non-smooth function. As such, it is known that the efficiency estimate of a black-box first-order method is of the order $O\left( \tfrac{1}{\varepsilon^{2}}\right)$ if no specific problem structure is used, where $\varepsilon$ is the desired abolute accuracy of the approximate solution in function value \cite{ref:nesterov-book-04}. 
We show, however, that $\mathsf{P}$ has a certain structure that allows us to use Nesterov's approach for approximating non-smooth problems with smooth ones \cite{ref:Nest-05} leading to an efficiency estimate of the order $O\left( \tfrac{1}{\varepsilon}\right)$.  This, together with the low complexity of each iteration step in the approximation scheme leads to a numerical method for the channel capacity problem that has a very attractive computational complexity.

\subsection{Preliminaries}
Some preliminaries are needed in order to present our capacity approximation method. We begin by recalling Nesterov's seminal work \cite{ref:Nest-05} in the context of structural convex optimization, which is our main tool in the proposed capacity approximation scheme. 
 \subsection*{Nesterov's smoothing approach \cite{ref:Nest-05}} 
Consider finite-dimensional real vector spaces $E_i$ endowed with a norm $\|\cdot\|_i$ and denote its dual space by $E^\star_i$ for $i=1,2$. Each dual pair of vector spaces comes with a bilinear form $\inprod{\cdot}{\cdot}_i:E^\star_i\times E_i\to \R$. For a linear operator $A:E_1\to E_2^\star$ the operator norm is defined as $\|A\|_{1,2}=\max_{x,u}\{ \inprod{Ax}{u}_2 \ : \ \|x\|_1=1, \|u\|_2=1 \}$. We are interested in the following optimization problem
\begin{equation}\label{eq:Nesterov:1}
\min_x\{ f(x) \ : \ x\in Q_1\},
\end{equation}
where $Q_1\subset E_1$ is a compact convex set and $f$ is a continuous convex function on $Q_1$. We assume that the objective function has the following structure
\begin{equation} \label{eq:Nesterov:2:structure}
f(x) = \hat{f}(x) + \max_u \{ \inprod{Ax}{u}_2 -\hat{\phi}(u) \ : \ u\in Q_2\},
\end{equation}
where $Q_2\subset E_2$ is a compact convex set, $\hat{f}$ is a continuously differentiable convex function whose gradient is Lipschitz continuous with constant $L$ on $Q_1$ and $\hat{\phi}$ is a continuous convex function on $Q_2$. It is assumed that $\hat{\phi}$ and $Q_2$ are simple enough such that the maximization in \eqref{eq:Nesterov:2:structure} is available in closed form.
The dual program to \eqref{eq:Nesterov:1} can be given as
\begin{equation}\label{eq:Nesterov:dual}
\max_u \left\{ -\hat{\phi}(u) + \min_x \{ \inprod{Ax}{u}_2 + \hat{f}(x) \ : \ x\in Q_1 \}  \ : \ u\in Q_2 \right\}.
\end{equation}
The main difficulty in solving \eqref{eq:Nesterov:1} efficiently is its non-smooth objective function. Without using any specific problem structure the complexity for subgradient-type methods is $O\left( \tfrac{1}{\varepsilon^{2}}\right)$, where $\varepsilon$ is the desired abolute accuracy of the approximate solution in function value. Nesterov's work suggests that when approximating problems with the particular structure \eqref{eq:Nesterov:2:structure} by smooth ones, a solution to the 
non-smooth problem can be constructed with complexity in order of $O\left( \tfrac{1}{\varepsilon}\right)$. In addition, Nesterov shows that when solving the smooth problem, a solution to the dual problem \eqref{eq:Nesterov:dual} can be obtained, and as such an a posteriori statement about the duality gap is available that often is significantly tighter than the $O\left( \tfrac{1}{\varepsilon}\right)$ complexity bound.
Consider the the smooth approximation to problem \eqref{eq:Nesterov:1} given by
\begin{equation} \label{eq:Nesterov:smooth:problem}
\min_x\{ f_\nu (x) \ : \ x\in Q_1\},
\end{equation}
where $\nu>0$ and the objective function is given by
\begin{equation}\label{eq:Nesterov:2:smooth:objective}
f_\nu(x) = \hat{f}(x) + \max_u \{ \inprod{Ax}{u}_2 -\hat{\phi}(u)-\nu d(u) \ : \ u\in Q_2\},
\end{equation}
where $d:Q_2\to\R$ is continuous and strongly convex with convexity parameter $\sigma$. It can be shown that $f_\nu$ has a Lipschitz continuous gradient with Lipschitz constant $L+\tfrac{\| A \|_{1,2}^2}{\nu \sigma}$ \cite[Theorem~1]{ref:Nest-05}. In this light, the optimization problem \eqref{eq:Nesterov:smooth:problem} belongs to a class of problems that can be solved in  
 $O\left( \tfrac{1}{\sqrt{\varepsilon}}\right)$ using a fast gradient method. The result \cite[Theorem~3]{ref:Nest-05} explicitly details how, having solved the smooth problem \eqref{eq:Nesterov:smooth:problem}, primal and dual solutions to the non-smooth problems \eqref{eq:Nesterov:1} and \eqref{eq:Nesterov:dual} can be obtained and how good they are.

 \subsection*{Entropy maximization} 
As a second preliminary result for some $c\in\R^{N}$ we consider the following optimization problem, that, if feasible, has an analytical solution
\begin{equation} \label{opt:cover}
 	\left\{ \begin{array}{lll}
			&\underset{p}{\max} 		&H(p) - c\transp p \\
			&\text{s.t. } 			& s\transp p = S\\
			& 					& p\in \Delta_{N}.
	\end{array} \right.
\end{equation}
\begin{Lem} \label{lem:cover}
Let $p^\star=[p_1^\star, \ldots, p_N^\star]$ with $p_i^\star =2^{\mu_1 - c_i + \mu_2 s_i}$, where $\mu_1$ and $\mu_2$ are chosen such that $p^\star$ satisfies the constraints in \eqref{opt:cover}. Then $p^\star$ uniquely solves \eqref{opt:cover}.
\end{Lem}
\begin{proof}
This proof is similar to the proof given in \cite[Theorem~12.1.1]{cover}. Let $q$ satisfy the constraints in \eqref{opt:cover}. Then
\begin{subequations}
\begin{align}
J(q)&= H(q)-c \transp q = -\sum_{i=1}^N q_i \log(q_i) -c\transp q   \nonumber \\
 &= -\sum_{i=1}^N q_i \log\left( \frac{q_i}{p^\star_i}p^\star_i \right) -c\transp q=-\D{q}{p^\star}- \sum_{i=1}^N q_i \log(p^\star_i) - c \transp q  \nonumber \\
 & \leq - \sum_{i=1}^N q_i \log(p^\star_i) - c \transp q \label{eq:ineq}\\
 & = - \sum_{i=1}^N q_i \left(\mu_1 + \mu_2 s_i \right) \label{eq:step1}\\
 &= -\sum_{i=1}^N p^{\star}_i \left(\mu_1 + \mu_2 s_i \right) -c\transp p^\star +c \transp p^\star \label{eq:step2}\\ 
 &= -\sum_{i=1}^N p^\star_i \log(p^\star_i) - c\transp p^\star= J(p^\star). \nonumber
\end{align}
\end{subequations}
The inequality follows form the non-negativity of the relative entropy. Equality \eqref{eq:step1} follows by the definition of $p^\star$ and \eqref{eq:step2} uses the fact that both $p^\star$ and $q$ satisfy the constraints in \eqref{opt:cover}. Note that equality holds in \eqref{eq:ineq} if and only if $q = p^\star$. This proves the uniqueness.
\end{proof}

\subsection{Capacity approximation scheme}
In the following we focus on the input constrained channel capacity problem \eqref{opt:primal:equivalent} and the scenario of no input constraints \eqref{opt:primal:equivalent:no:power:constraints} is discussed as a special case within this section.
Consider the convex optimizaton problem \eqref{opt:primal:equivalent}, whose optimal value, according to Lemma~\ref{lem:equivalent:primal:problem} is the capacity $C_{S}(W)$.
The Lagrange dual program to \eqref{opt:primal:equivalent} is
\begin{align}\label{Lagrange:Dual:Program}
\mathsf{D}:\quad \left\{ \begin{array}{ll}
			\underset{\lambda}{\min} 		&G(\lambda) + F(\lambda) \\
			\text{s.t. } 				& \lambda\in\R^{M},
	\end{array}\right. 
\end{align}
where $F, G: \R^{M}\to\R$ are given by
\begin{align} \label{equation:F:and:G} 
G(\lambda)= \left\{ \begin{array}{ll}
			\underset{p}{\max} 		&-r\transp p + \lambda\transp \W\transp p \\
			\text{s.t. } 				&s\transp p=  S \\
								& p\in\Delta_{N}
	\end{array} \right.
	\quad \textnormal{and} \qquad
	F(\lambda)= \left\{ \begin{array}{ll}
			\underset{q}{\max} 		&H(q)-\lambda\transp q \\
			\text{s.t. } 				& q\in\Delta_{M}.
	\end{array}\right. 
\end{align}
Note that since the coupling constraint $ \W\transp p = q$ in the primal program \eqref{opt:primal:equivalent} is affine, the set of optimal solutions to the dual program \eqref{Lagrange:Dual:Program} is nonempty \cite[Proposition~5.3.1]{ref:Bertsekas-09} and as such the optimum is attained. 
It can be seen that the dual program \eqref{Lagrange:Dual:Program} structurally resembles the problem \eqref{eq:Nesterov:1} with \eqref{eq:Nesterov:2:structure}, without a bounded feasible set, however. 
To ensure that the set of dual optimizers is compact, we need to impose the following assumption on the channel matrix $\W$, that we will maintain for the remainder of Section~\ref{sec:classicalCapacity}.

\begin{As} \label{ass:channel}
$\gamma:=\min\limits_{i,j}\W_{ij}>0$
\end{As}
Assumption~\ref{ass:channel} excludes situations where the channel matrix has zero entries. Even though this may seem restrictive at first glance, it holds for a large class of channels. Moreover, in a finite dimensional setting, for a fixed input distribution, the mutual information is well known to be continuous in the channel matrix entries. Therefore, singular cases where the channel matrix contains zero entries can be avoided by slight perturbations of those entries. (This is discussed in more detail in Remark~\ref{rmk:perturbation}.)
Under Assumption~\ref{ass:channel} for a fixed channel, the mutual information can be seen to be a strictly concave function in the input distribution. Therefore, the capacity achieving input distribution is unique.
With Assumption~\ref{ass:channel} one can derive an explicit bound on the norm of the dual optimizers, which is crucial in the subsequent derivation of the main result in this section, namely Theorem~\ref{thm:error:bound:capacity}.   
\begin{Lem} \label{lem:compact:set}
Under Assumption~\ref{ass:channel}, the dual program \eqref{Lagrange:Dual:Program} is equivalent to 
\begin{equation}\label{eq:dual:finite:compact}
\begin{aligned}
\left\{ \begin{array}{ll}
			\underset{\lambda}{\min} 		&G(\lambda) + F(\lambda) \\
			\textnormal{s.t. } 				& \lambda\in Q,
	\end{array}\right. 
\end{aligned}
\end{equation}
where $Q:= \left\{ \lambda\in\R^{M} \ : \ \Norm{\lambda}_{2}\leq M \left( \log(\gamma^{-1}) \vee \tfrac{1}{\ln 2} \right) \right\}$. 
\end{Lem}
\begin{proof}
Consider the following two convex optimization problems
\begin{align*} 
\mathsf{P}_{\beta}:\quad  \left\{ \begin{array}{ll}
			\max\limits_{p,q,\varepsilon} 		&-r\transp p + H(q) - \beta\varepsilon\\
			\text{s.t. }						&\|\W\transp p - q\|_{\infty}\leq \varepsilon  \\
										& s\transp p = S \\
			 							& p\in\Delta_{N}, \ q\in\Delta_{M}, \ \varepsilon \in\Rp
	\end{array} \right.  \quad \textnormal{and} \quad
	\quad 
\mathsf{D}_{\beta}: \quad  \left\{ \begin{array}{ll}
			\min\limits_{\lambda} 		&F(\lambda) + G(\lambda) \\
			\text{s.t. } 					& \Norm{\lambda}_{1} \leq \beta \\
									& \lambda\in \R^{M}.
	\end{array}\right. 
\end{align*}
\begin{Claim}
Strong duality holds between $\mathsf{P}_{\beta}$ and $\mathsf{D}_{\beta}$.
\end{Claim}
\begin{proof}
According to the identity $\Norm{\W\transp p - q}_{\infty}=\max_{\Norm{\lambda}_{1}\leq 1} \lambda\transp \left( \W\transp p - q \right)$ \cite[p.~7]{holevo_book} the optimization problem $\mathsf{P}_{\beta}$ can be rewritten as
\begin{align*} 
\mathsf{P}_{\beta}:\quad  \left\{ \begin{array}{ll}
			\max\limits_{p,q} 			&-r\transp p + H(q) + \min\limits_{\Norm{\lambda}_{1}\leq \beta}\lambda\transp \left( \W\transp p - q \right)\\
			\text{s.t. }					& s\transp p = S \\
			 							& p\in\Delta_{N}, \ q\in\Delta_{M},
	\end{array} \right.
\end{align*}
whose dual program, where strong duality holds according to \cite[Proposition~5.3.1, p.~169]{ref:Bertsekas-09} is given by
\begin{align*}
 \quad  \left\{ \begin{array}{lll}
			\min\limits_{\Norm{\lambda}_{1}\leq \beta}  &\max\limits_{p,q} 		&-r\transp p + H(q) + \lambda\transp \left( \W\transp p - q \right) \\
			&\text{s.t. } 				  &s\transp p = S \\
								&	&  p\in\Delta_{N}, \ q\in\Delta_{M},
	\end{array}\right. .
\end{align*}
which clearly is equivalent to $\mathsf{D}_{\beta}$ with $F(\cdot)$ and $G(\cdot)$ as given in \eqref{equation:F:and:G}.
\end{proof}
Denote by $\varepsilon^{\star}(\beta)$ the optimizer of $\mathsf{P}_{\beta}$ with the respective optimal value $J^{\star}_{\beta}$. We show that for a sufficiently large $\beta$ the optimizer $\varepsilon^{\star}(\beta)$ of $\mathsf{P}_{\beta}$ is equal to zero. Hence, in light of the duality relation, the constraint $\Norm{\lambda}_{1} \leq \tfrac{\beta}{2}$ in $\mathsf{D}_{\beta}$ is inactive and as such $\mathsf{D}_{\beta}$ is equivalent to $\mathsf{D}$ in equation~\eqref{Lagrange:Dual:Program}. Note that for
\begin{align}  \label{eq:J(eps)}
J(\varepsilon):=  \left\{ \begin{array}{ll}
			\max\limits_{p,q} 		&-r\transp p + H(q) \\
			\text{s.t. }						&\|\W\transp p - q\|_{\infty}\leq \varepsilon \\
									 & s\transp p = S \\
			 							& p\in\Delta_{N}, \ q\in\Delta_{M}
	\end{array} \right. ,
	\end{align}
the mapping $\varepsilon \mapsto J(\varepsilon)$, the so-called perturbation function, is concave \cite[p.~268]{ref:BoyVan-04}. In the next step we write the optimization problem \eqref{eq:J(eps)} in another equivalent form
\begin{align}  \label{eq:J(eps):equiv}
J(\varepsilon)=  \left\{ \begin{array}{ll}
			\max\limits_{p,v} 		&-r\transp p + H(\W\transp p + \varepsilon v) \\
			\text{s.t. }						&\Norm{v}_{\infty} \leq 1 \\
								      & s\transp p = S \\
			 							& p\in\Delta_{N}, \ v\in \mathsf{Im}(\W\transp)\subset\R^{M}
	\end{array} \right. .
	\end{align}
By using Taylor's theorem, there exists $y_{\varepsilon}\in[0,\varepsilon]$ such that the entropy term in the objective function of \eqref{eq:J(eps):equiv} can be bounded as
\begin{align}
H(\W\transp p + \varepsilon v) 	&=		H(\W\transp p) - \left( \log(\W\transp p) + \tfrac{1}{\ln 2}\boldsymbol{1} \right)\transp v \varepsilon - \sum_{j=1}^{M}\frac{v_{j}^{2}}{\sum_{i=1}^{N}\W_{ij}p_{i}+y_{\varepsilon}v_{j}}\varepsilon^{2} \tfrac{1}{\ln 2} \nonumber \\
						&\leq		H(\W\transp p) - \left( \log(\W\transp p) + \tfrac{1}{\ln 2}\boldsymbol{1}  \right)\transp v \varepsilon + \frac{M}{\gamma \ln 2}\varepsilon^{2} . \label{eq:Taylor:bound}
\end{align}
Thus, the optimal value of problem $\mathsf{P}_{\beta}$ can be expressed as
\begin{subequations}
\begin{align}
J_{\beta}^{\star} 	&\leq 		\max\limits_{\varepsilon} \left\{ J(\varepsilon)-\beta \varepsilon \right\} \nonumber \\
			&\leq		\max\limits_{\varepsilon} \left\{ \max\limits_{p,v}\left[ -r\transp p + H(\W\transp p) - \left( \log(\W\transp p) + \tfrac{1}{\ln 2} \boldsymbol{1} \right)\transp v \varepsilon :  s\transp p = S  \right]  + \frac{M}{\gamma \ln 2}\varepsilon^{2}  -\beta \varepsilon \right\} \label{eq:proof:compact:first:step}\\
			&\leq		\max\limits_{\varepsilon} \left\{ \max\limits_{p,v}\left[ -r\transp p + H(\W\transp p) \ : \ s\transp p = S  \right] + (\rho - \beta)\varepsilon + \frac{M}{\gamma \ln 2}\varepsilon^{2} \right\} \label{eq:proof:compact:second:step}\\
			&=		J(0) + \max\limits_{\varepsilon} \left\{  (\rho - \beta)\varepsilon +\frac{M}{\gamma \ln 2}\varepsilon^{2} \right\}, \label{eq:proof:compact:third:step}
\end{align}
\end{subequations}
where $\rho = M \left( \log(\gamma^{-1}) \vee \tfrac{1}{\ln 2}\right)$. Note that \eqref{eq:proof:compact:first:step} follows from $\eqref{eq:J(eps):equiv}$ and \eqref{eq:Taylor:bound}. The equation \eqref{eq:proof:compact:second:step} uses $\|v\|_\infty \leq 1$, $- \left( \log(\W\transp p) + \tfrac{1}{\ln 2}\boldsymbol{1} \right)\transp v \leq M \left( \log(\gamma^{-1}) \vee \tfrac{1}{\ln 2}\right)$. Thus, for $\beta>\rho$ and $\varepsilon_{1}=\frac{\gamma\ln 2}{M}(\beta - \rho)$, we have $\max\limits_{\varepsilon\leq \varepsilon_{1}} \left\{  (\rho - \beta)\varepsilon + \tfrac{M}{\gamma \ln 2}\varepsilon^{2}\right\} = 0.$ Therefore, \eqref{eq:proof:compact:third:step} together with the concavity of the mapping $\varepsilon\mapsto J(\varepsilon)$ imply that $J(0)$ is the global optimum of $J(\varepsilon)$ and as such $\varepsilon^{\star}(\beta)=0$ for $\beta>\rho$, indicating that $\mathsf{P}_{\beta}$ is equivalent to $\mathsf{P}$ in the sense that $J^{\star}_{\beta}=J^{\star}_{0}$. By strong duality this implies that the constraint $\Norm{\lambda}_{1} \leq \beta$ in $\mathsf{D}_{\beta}$ is inactive. Finally, $\Norm{\lambda}_{2}\leq\Norm{\lambda}_{1}$ concludes the proof. 
\end{proof}
\begin{Lem} \label{lem:strong:duality:finite}
Strong duality holds between \eqref{opt:primal:equivalent} and \eqref{Lagrange:Dual:Program}.
\end{Lem}
\begin{proof}
The proof follows by a standard strong duality result of convex optimization, see \cite[Proposition~5.3.1, p.~169]{ref:Bertsekas-09}.
\end{proof}
Note that the optimization problem defining $F(\lambda)$ is of the form given in \eqref{opt:cover}. Hence, according to Lemma~\ref{lem:cover}, $F(\lambda)$ has a unique optimizer $q^{\star}$ with components
$q_{j}^{\star} = 2^{\mu-\lambda_{j}}$,
where $\mu\in\R$ needs to be chosen such that $q^{\star}\in\Delta_{M}$, i.e.,
\begin{equation*}
\mu = - \log\left(\sum_{j=1}^M 2^{-\lambda_j} \right).
\end{equation*}
Therefore, 
\begin{equation} \label{analytical:solution:F}
\begin{aligned}
F(\lambda) 	&=\sum_{j=1}^{M} \left(- q_{j}^{\star}\log(q_{j}^{\star}) - \lambda_{j}q_{j}^{\star} \right) =	-\sum_{j=1}^{M} \mu \, 2^{\mu-\lambda_{j}}\\ 
&= -\mu \, 2^{\mu}\sum_{j=1}^{M}2^{-\lambda_{j}} =    \log\left(\sum_{j=1}^M 2^{-\lambda_j} \right).
\end{aligned}
\end{equation}
$F(\lambda)$ is a smooth function with gradient
\begin{align} \label{eq:gradient:F}
(\nabla F(\lambda))_{i} = \frac{-2^{-\lambda_{i}}}{\sum_{j=1}^M 2^{-\lambda_j} }.
\end{align}
According to \cite[Theorem~1]{ref:Nest-05} and the fact that the negative entropy is strongly convex with convexity parameter 1 \cite[Lemma~3]{ref:Nest-05}, $\nabla F(\lambda)$ is Lipschitz continuous with Lipschitz constant $1$.
The main difficulty in solving \eqref{eq:dual:finite:compact} efficiently is that $G(\cdot)$ is non-smooth. Following Nesterov's smoothing technique \cite{ref:Nest-05}, we alleviate this difficulty by approximating $G(\cdot)$ by a function with a Lipschitz continuous gradient. This smoothing step is efficient in our case because of the particular structure of $\eqref{eq:dual:finite:compact}$.
Following \cite{ref:Nest-05} and \eqref{eq:Nesterov:2:smooth:objective}, consider
\begin{align} \label{eq:discrete:G:nu}
G_{\nu}(\lambda) 	=    \left\{ \begin{array}{ll}
	\max\limits_{p}				&  \lambda\transp \W\transp p - r\transp p + \nu H( p) -\nu\log(N)\\
			\text{s.t. } 			&  s\transp p=  S \\
							&  p\in\Delta_{N},
	\end{array}\right. 
\end{align}
with smoothing parameter $\nu\in\R_{>0}$ and denote by $p_{\nu}(\lambda)$ the optimizer to \eqref{eq:discrete:G:nu}, which is unique because the objective function is strictly concave. Clearly for any $\lambda \in Q$, $G_{\nu}(\lambda)$ is a uniform approximation of the non-smooth function $G(\lambda)$, since $G_{\nu}(\lambda)\leq G(\lambda)\leq G_{\nu}(\lambda) + \nu \log(N)$.
Using Lemma~\ref{lem:cover}, the optimizer $p_{\nu}(\lambda)$ to \eqref{eq:discrete:G:nu} is analytically given by
\begin{align} \label{eq:finite:optimizer:pmu}
p_{\nu}(\lambda,\mu)_{i} = 2^{\mu_{1} + \tfrac{1}{\nu} (\W \lambda \, -\,  r)_{i} + \mu_{2}s_{i}},
\end{align}
where $\mu_{1},\mu_{2}\in\R$ have to be chosen so that $s\transp p_{\nu}(\lambda,\mu)=S$ and $p_{\nu}(\lambda,\mu)\in\Delta_{N}$; for this choice of $\mu_1$, $\mu_2$ we denote the solution by $p_{\nu}(\lambda)$. 
\begin{Rem}\label{rmk:stabilization:optimizer}
In case of no input constraints, the unique optimizer to \eqref{eq:discrete:G:nu} is given by
\begin{equation*}
p_{\nu}(\lambda)_{i} = \frac{2^{ \tfrac{1}{\nu} (\W\lambda -  r)_{i}} }{\sum_{i=1}^{N}2^{ \tfrac{1}{\nu} (\W\lambda -  r)_{i}}}\quad \text{for }i=1,\hdots,N,
\end{equation*}
whose straightforward evaluation is numerically difficult for small $\nu$. One can circumvent this problem, however, by following the numerically stable technique that we present in Remark~\ref{rmk:finite:no:input:const}.
By Dubin's theorem it can be shown that the capacity of a memoryless channel with a discrete output alphabet of size $M$ and input alphabet size $N\geq M$, is achieved by a discrete input distribution with $M$ mass points \cite{gallager68,witsenhausen}. Computing the exact positions and weights of this optimal input distribution may be difficult, though it is worth noting that our analytical solution in \eqref{eq:finite:optimizer:pmu} converges to this optimal input distribution as $\nu$ tends to $0$.
\end{Rem}
\begin{Rem}[Additional input constraints]\label{rmk:finite:constraint:optimizer}
In case of additional input constraints, we need an efficient method to find the coefficients $\mu_{1}$ and $\mu_{2}$ in \eqref{eq:finite:optimizer:pmu}. In particular if there are multiple input constraints (leading to multiple $\mu_{i}$) the efficiency of the method computing them becomes important. Instead of solving a system of nonlinear equations, one can show (\cite[Theorem~4.8]{ref:Borwein-91}, \cite[p.~257 ff.]{ref:Lasserre-11}) that the coefficients $\mu_{i}$ are the unique maximizers to the following convex optimization problem
\begin{equation} \label{eq:opt:problem:find:mu:finite}
\max\limits_{\mu\in\R^{2}}\left\{ y\transp \mu - \sum_{i=1}^{N}p_{\nu}(\lambda,\mu)_{i} \right\}, 
\end{equation}
where $y:=(1,S)$. Notice that $\eqref{eq:opt:problem:find:mu:finite}$ is an unconstrained maximization of a strictly concave function, whose gradient and Hessian can be directly computed as
\begin{equation*}
\left( \begin{array}{c} y_1 - \ln 2 \sum_{i=1}^N p_{\nu}(\lambda,\mu)_i \\ y_2 - \ln 2 \sum_{i=1}^N s_i p_{\nu}(\lambda,\mu)_i
\end{array} \right),
\left( \begin{array}{cc} - (\ln 2)^2 \sum_{i=1}^N p_{\nu}(\lambda,\mu)_i & -(\ln 2)^2 \sum_{i=1}^N s_i p_{\nu}(\lambda,\mu)_i \\ -(\ln 2)^2 \sum_{i=1}^N s_i p_{\nu}(\lambda,\mu)_i & -(\ln 2)^2 \sum_{i=1}^N s^2_i p_{\nu}(\lambda,\mu)_i
\end{array} \right),
\end{equation*}
which allows the use of efficient second-order methods such as Newton's method. This method directly extends to multiple input constraints. 
Let us point out that Theorem~\ref{thm:error:bound:capacity}, quantifying the approximation error of the presented algorithm, is based on the assumption that the maximum entropy solution \eqref{eq:finite:optimizer:pmu} is available, meaning that one can solve \eqref{eq:opt:problem:find:mu:finite} for optimality. In the case of a finite input alphabet this assumption is not restrictive as we have argued that \eqref{eq:opt:problem:find:mu:finite} is easy to solve. For a continuous input alphabet, that we shall discuss in the subsequent section, however, finding the maximum entropy solution is numerically difficult as it involves integration problems. Therefore, in Remark~\ref{rmk:finite:constraint:optimizer:cts}, we comment on how the presented channel capacity algorithm behaves, when having access only to an approximate solution to the mentioned maximum entropy problem. 
\end{Rem}

Finally, we can show that the uniform approximation $G_{\nu}(\lambda)$ is smooth and has a Lipschitz continuous gradient, with known Lipschitz constant.
\begin{Prop} \label{prop:Lipschitz:continuity:DMC}
$G_{\nu}(\lambda)$ is well defined and continuously differentiable at any $\lambda\in Q$. Moreover, it is convex and its gradient $\nabla G_{\nu}(\lambda)=\W\transp p_{\nu}(\lambda)$ is Lipschitz continuous with Lipschitz constant $\tfrac{1}{\nu}$.
\end{Prop}
\begin{proof}
The proof follows directly from the proof of Theorem 1 and Lemma 3 in \cite{ref:Nest-05} together with \eqref{eq:operator:norm}.
\end{proof}
We consider the smooth, convex optimization problem
\begin{align}\label{Lagrange:Dual:Program:smooth}
 \mathsf{D}_{\nu}:\quad \left\{ \begin{array}{ll}
	\min\limits_{\lambda} 		& F(\lambda) + G_{\nu}(\lambda) \\
			\text{s.t. } 					& \lambda\in Q,
	\end{array}\right.
\end{align}
whose objective function has a Lipschitz continuous gradient with Lipschitz constant $1+\tfrac{1}{\nu}$. As such $\mathsf{D}_{\nu}$ can be  be approximated with Nesterov's optimal scheme for smooth optimization \cite{ref:Nest-05}, which is summarized in Algorithm~\hyperlink{algo:1}{1}, where $\pi_{Q}(x)$ denotes the projection operator of the set $Q$, defined in Lemma~\ref{lem:compact:set}, with $R:=M \left( \log(\gamma^{-1}) \vee \tfrac{1}{\ln 2}\right)$
\begin{equation*}
\pi_{Q}(x):=\left\{ \begin{array}{ll} R \tfrac{x}{\Norm{x}_{2}}, & \Norm{x}_{2}>R \\ x, & \text{otherwise.} \end{array} \right. 
\end{equation*}

 \begin{table}[!htb]
\centering 
\begin{tabular}{c}
  \Xhline{3\arrayrulewidth}  \hspace{1mm} \vspace{-3mm}\\ 
\hspace{22mm}{\bf{\hypertarget{algo:1}{Algorithm 1: } }} Optimal scheme for smooth optimization \hspace{22mm} \\ \vspace{-3mm} \\ \hline \vspace{-0.5mm}
\end{tabular} \\
\vspace{-5mm}
 \begin{flushleft}
  {\hspace{1.7mm}Choose some $\lambda_0 \in Q$}
 \end{flushleft}
 \vspace{-6mm}
 \begin{flushleft}
  {\hspace{1.7mm}\bf{For $k\geq 0$ do$^{*}$}}
 \end{flushleft}
 \vspace{-5mm}
 
  \begin{tabular}{l l}
{\bf Step 1: } & Compute $\nabla F(x_{k})+\nabla G_{\nu}(x_{k})$ \\
{\bf Step 2: } & $y_k = \pi_{Q}\left(-\frac{1}{L_{\nu}}\left( \nabla F(x_k)+\nabla G_{\nu}(x_k) \right) + x_k\right)$\\
{\bf Step 3: } &   $z_k=\pi_{Q}\left(-\frac{1}{L_{\nu}} \sum_{i=0}^{k} \frac{i+1}{2} \left(  \nabla F(x_i)+\nabla G_{\nu}(x_i) \right)\right)$\\
{\bf Step 4: } & $x_{k+1}=\frac{2}{k+3}z_{k} + \frac{k+1}{k+3}y_{k}$\\ 
  \end{tabular}
   \begin{flushleft}
  {\hspace{1.7mm}[*The stopping criterion is explained in Remark~\ref{remark:stopping}]}
  \vspace{-10mm}
 \end{flushleft}  
\begin{tabular}{c}
\hspace{22.7mm} \phantom{ {\bf{Algorithm:}} Optimal Scheme for Smooth Optimization}\hspace{22.7mm} \\ \vspace{-1.0mm} \\\Xhline{3\arrayrulewidth}
\end{tabular}
\end{table}

The following theorem provides explicit error bounds for the solution provided by Algorithm~\hyperlink{algo:1}{1} after $n$ iterations. Define the constants $D_{1}:=\tfrac{1}{2}(M \log(\gamma^{-1})\vee\tfrac{1}{\ln 2})^2 $ and $D_{2}:=\log(N)$.
\begin{Thm}[\cite{ref:Nest-05}] \label{thm:error:bound:capacity}
Under Assumption~\ref{ass:channel}, for $n\in\mathbb{N}$ consider a smoothing parameter
\begin{align*}
\nu = \nu(n) = \frac{2}{n+1}\sqrt{\frac{D_{1}}{D_{2}}}.
\end{align*}
Then after $n$ iterations of Algorithm~\hyperlink{algo:1}{1} we can generate the approximate solutions to the problems \eqref{Lagrange:Dual:Program} and \eqref{eq:DMC_capacity_const}, namely,
\begin{align}
\hat{\lambda} = y_{n} \in Q\qquad \textnormal{and} \qquad \hat{p}=\sum_{k=0}^{n}\frac{2(k+1)}{(n+1)(n+2)} p_{\nu}(x_{k})\in \Delta_{N}, \label{eq:optInPut}
\end{align}
which satisfy 
\begin{align}
0\leq F(\hat{\lambda}) + G(\hat{\lambda}) - \I{\hat{p}}{W} \leq \frac{4}{n+1} \sqrt{D_{1} D_{2}} + \frac{4 D_{1}}{(n+1)^{2}}. \label{eq:EBB}
\end{align}
Thus, the complexity of finding an $\varepsilon$-solution to the problems \eqref{Lagrange:Dual:Program} and \eqref{eq:DMC_capacity_const} does not exceed
\begin{align}\label{eq:finite:complexity:thm}
4  \sqrt{D_{1} D_{2}} \ \frac{1}{\varepsilon} + 2 \sqrt{\frac{D_{1}}{\varepsilon}}.
\end{align}
\end{Thm}
\begin{proof}
The proof follows along the lines of \cite[Theorem~3]{ref:Nest-05} and in particular requires Lemma~\ref{lem:compact:set}, Lemma~\ref{lem:strong:duality:finite} and Proposition~\ref{prop:Lipschitz:continuity:DMC}.
\end{proof}
Note that Theorem~\ref{thm:error:bound:capacity} provides an explicit error bound \eqref{eq:EBB}, also called \emph{a priori error}. In addition this theorem gives an approximation to the optimal input distribution \eqref{eq:optInPut}, i.e., the optimizer of the primal problem. Thus, by comparing the values of the primal and the dual optimization problem, one can also compute an \emph{a posteriori error} which is the difference of the dual and the primal problem, namely $F(\hat{\lambda}) + G(\hat{\lambda}) - \I{\hat{p}}{W}$.

\begin{Rem}[Stopping criterion of Algorithm~\hyperlink{algo:1}{1}] \label{remark:stopping}
There are two immediate approaches to define a stopping criterion for Algorithm~\hyperlink{algo:1}{1}.
\begin{enumerate}[label=(\roman*), itemsep = 1mm, topsep = -1mm]
\item \emph{A priori stopping criterion}: Choose an a priori error $\varepsilon>0$. Setting the right hand side of \eqref{eq:EBB} equal to $\varepsilon$ defines a number of iterations $n_{\varepsilon}$ required to ensure an $\varepsilon$-close solution.
\item \emph{A posteriori stopping criterion}: Choose an a posteriori error $\varepsilon>0$. Choose the smoothing parameter $\nu(n_{\varepsilon})$ for $n_{\varepsilon}$ as defined above in the a priori stopping criterion. Fix a (small) number of iterations $\ell$ that are run using Algorithm~\hyperlink{algo:1}{1}. Compute the a posteriori error $\mathrm{e}_{\ell}:= F(\hat{\lambda}) + G(\hat{\lambda}) - \I{\hat{p}}{\rho}$ according to Theorem~\ref{thm:error:bound:capacity}. If $\mathrm{e}_{\ell}\leq \varepsilon$ terminate the algorithm otherwise continue with another $\ell$ iterations. Continue until the a posteriori error is below $\varepsilon$.
\end{enumerate}
\end{Rem}

\begin{Rem}[Computational stability]\label{rmk:finite:no:input:const}
In the special case of no input cost constraints, one can derive an analytical expression for $G_{\nu}(\lambda)$ and its gradient as
\begin{align}
G_{\nu}(\lambda) &= \nu \log\left( \sum_{i=1}^{N} 2^{\frac{1}{\nu}( \W \lambda\, - \, r)_{i}} \right) -\nu \log(N) \nonumber \\
\nabla G_{\nu}(\lambda) &=  \frac{1}{S(\lambda)}\sum_{i=1}^{N} 2^{\frac{1}{\nu}( \W \lambda\, - \, r)_{i}}\W_{i,\cdot}, \label{eq:gradient:G}
\end{align}
where $S(\lambda):=\sum_{i=1}^{N} 2^{\frac{1}{\nu}( \W \lambda\, - \, r)_{i}}$. In order to achieve an $\varepsilon$-precise solution the smoothing factor $\nu$ has to be chosen in the order of $\varepsilon$, according to Theorem~\ref{thm:error:bound:capacity}. A straightforward computation of $\nabla G_{\nu}(\lambda)$ via \eqref{eq:gradient:G} for a small enough $\nu$ is numerically difficult. In the light of  \cite[p.~148]{ref:Nest-05}, we present a numerically stable technique for computing $\nabla G_{\nu}(\lambda)$. By considering the functions $\R^{M}\ni \lambda \mapsto f(\lambda)=\W\lambda -r \in \R^{N}$ and $\R^{N}\ni x \mapsto R_{\nu}(x)=\nu \log \left( \sum_{i=1}^{N}2^{\tfrac{x_{i}}{\nu}} \right)\in\R$ it is clear that $\nabla_{\lambda} R_{\nu}(f(\lambda))=\nabla G_{\nu}(\lambda)$. The basic idea is to define $\bar{f}(\lambda):=\max_{1\leq i \leq N} f_{i}(\lambda)$ and then consider a function $g:\R^{M}\to \R^{N}$ given by $g_{i}(\lambda)=f_{i}(\lambda)-\bar{f}(\lambda)$, such that all components of $g(\lambda)$ are non-positive. One can show that
\begin{equation*} \label{eq:numerical:stability}
\nabla_{\lambda} R_{\nu}(f(\lambda))=\nabla_{\lambda} R_{\nu}(g(\lambda))+ \nabla \bar{f}(\lambda),
\end{equation*}
where the term on the right-hand side can be computed with a small numerical error.
\end{Rem}

\begin{Rem}[Computational complexity]\label{rmk:finite:complexity}
In case of no input cost constraint, one can see by \eqref{eq:gradient:G} that the computational complexity of a single iteration step of Algorithm~\hyperlink{algo:1}{1} is $O(MN)$. Furthermore, according to \eqref{eq:finite:complexity:thm}, the complexity in terms of number of iterations to achieve an $\varepsilon$-precise solution is $O\left( \tfrac{M\sqrt{\log N}}{\varepsilon} \right)$. This finally gives a computational complexity for finding an additive $\varepsilon$-solution of $O(\tfrac{M^2 N \sqrt{\log N}}{\varepsilon})$. 
Let us point out that that the constants in the computational complexity, explicitly given in \eqref{eq:finite:complexity:thm} and in particular the dependency on the parameter $\gamma$, can have a significant impact on the runtime of the proposed approximation method in practice. In the following remark, however, we presents a way to circumvent ill-conditioned channels with very small (or even vanishing) $\gamma$ parameter.
\end{Rem}

\begin{Rem}[Removing Assumption~\ref{ass:channel}]\label{rmk:perturbation}
The continuity of the channel capacity can be used to remove Assumption~\ref{ass:channel}. Let $\W_1\in\R^{N\times M}$ be an channel transition matrix that does not satisfy Assumption~\ref{ass:channel}, i.e., that contains zero entries. Define a new channel matrix $\W_2\in\R^{N\times M}$ by adding a perturbation $\varepsilon>0$ to all zero entries of $\W_1$ and then normalizing the rows. According to \cite{leung09}
\begin{equation} \label{eq:continuity:smith}
|C(\W_1) - C(\W_2)| \leq 3 \Norm{\W_1-\W_2}_{\triangleright} \log(M\vee N) + 2 \eta( \Norm{\W_1-\W_2}_{\triangleright} ),
\end{equation}
where $\eta(t)=-t\log t$ and the norm $\Norm{\cdot}_{\triangleright}$ on $\R^{N\times M}$ is defined as $\Norm{A}_{\triangleright}:=\max_{b\in\Delta_N} \Norm{b b\transp A}_{\text{tr}}$.
 Since $\W_2$ by construction satisfies Assumption~\ref{ass:channel}, we can run Algorithm~\hyperlink{algo:1}{1} for channel $\W_2$ and as such get the following upper and lower bounds for the capacity of the singular channel $\W_1$
\begin{align*}
C_{\text{LB}}(\W_1) &:= C_{\text{LB}}(\W_2) - 3 \Norm{\W_1-\W_2}_{\triangleright} \log(M\vee N) - 2 \eta( \Norm{\W_1-\W_2}_{\triangleright} ) \\
C_{\text{UB}}(\W_1) &:= C_{\text{UB}}(\W_2) + 3 \Norm{\W_1-\W_2}_{\triangleright} \log(M\vee N) + 2 \eta( \Norm{\W_1-\W_2}_{\triangleright} ).
\end{align*}
See in Example~\ref{ex:BEC} how this perturbation method behaves numerically.
\end{Rem}

\subsection{Numerical example}
This section presents two examples to illustrate the theoretical results developed in the preceding sections and their performance. All the simulations in this section are performed on a 2.3 GHz Intel Core i7 processor with 8 GB RAM.
\begin{Ex} \label{ex:one}
Consider a DMC $W$ having a channel matrix $\W \in \R^{N \times M}$ with $N=10000$ and $M=100$, such that $\W_{ij}=\tfrac{V_{ij}}{\sum_{j=1}^M V_{ij}}$, where $V_{ij}$ is chosen i.i.d. uniformly distributed in $[0,1]$ for all $1\leq i \leq N$ and $1 \leq j \leq M$. The parameter $\gamma$ happens to be $1.0742\cdot 10^{-8}$.
Figure~\ref{fig:ex1} and Table~\ref{tab:ex1} compare the performance of the Blahut-Arimoto algorithm with that of Algorithm~\hyperlink{algo:1}{1}, which has the a priori error bound predicted by Theorem~\ref{thm:error:bound:capacity}, namely
\begin{equation*}
C_{\textnormal{UB}}(W)-C_{\textnormal{LB}}(W) \leq  \frac{2M\sqrt{2\log(N)}}{n+1}\left(\log(\gamma^{-1})\vee\tfrac{1}{\ln 2})\right)+\frac{2 M^{2}}{(n+1)^{2}}\left(\log(\gamma^{-1})\vee\tfrac{1}{\ln 2})\right)^{2},
\end{equation*}
where $n$ denotes the number of iterations and $\gamma$ is equal to the smallest entry in the channel matrix $W$. Recall that the Blahut-Arimoto algorithm has an a priori error bound of the form $C(W)-C_{\textnormal{LB}}(W) \leq  \tfrac{\log(N)}{n}$ \cite[Corollary~1]{arimoto72}. Moreover, the new method provides us with an a posteriori error, which the Blahut-Ariomoto algorithm does not. 
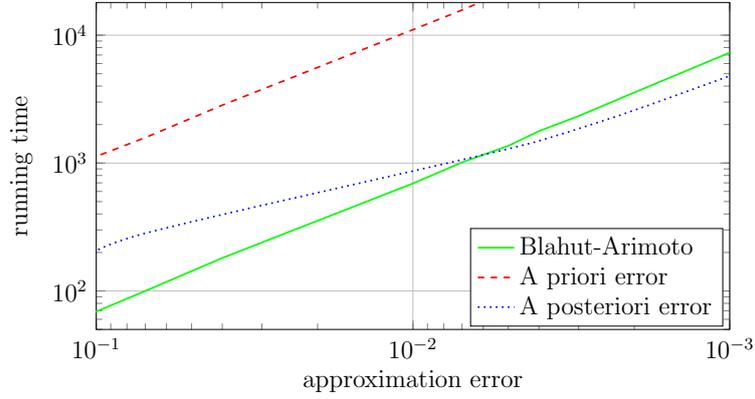
\begin{figure}[t]	
	\centering
	\scalebox{0.8}{  \begin{tikzpicture}
	\begin{axis}[
		height=7cm,
		width=12cm,
		grid=major,
		xlabel=approximation error,
		ylabel=running time,
		xmode=log,
		ymode=log,
		xmin=0.001,
		xmax=0.1,
		ymax=18000,
		ymin=50,
		x dir=reverse,
		legend style={at={(0.790,0.310)},anchor=north,legend cell align=left} 
	]

	\addplot[green,thick] coordinates {
	     (1,7.4)	
		(0.7,10)
		(0.4,18)	
		(0.1,69)	
		(0.07,100)	
		(0.04,181)
		(0.01,693)			
		(0.007,1013)
		(0.005,1364)
		(0.004,1779)	
		(0.003, 2331)	
		(0.002, 3564)
		(0.001,7306)																				
	};
	\addlegendentry{Blahut-Arimoto}
	
		\addplot[red,thick,dashed] coordinates {
	     (1,114)	
		(0.7,162)
		(0.4,282)	
		(0.1,1127)	
		(0.07,1581)	
		(0.04,2837)
		(0.01,11036)			
		(0.007,15718)	
		(0.004,28456)	
	};
		\addlegendentry{A priori error}	

		\addplot[blue,thick,smooth,dotted] coordinates {
	     (0.1325,114)	
		(0.1154,162)
		(0.0702,282)	
		(0.0063,1127)	
		(0.0037,1581)	
		(0.0018,2837)
		(0.0003976,11036)		
	};
	\addlegendentry{A posteriori error}

	\end{axis}  
\end{tikzpicture}}
	\caption{For Example~\ref{ex:one}, this plot depicts the runtime of Algorithm~1 with respect to the a priori and a posteriori stopping criterion, as explained in Remark~\ref{remark:stopping}. As a reference, the runtime of the Blahut-Arimoto algorithm is shown.  }  
	\label{fig:ex1}
\end{figure}

 \begin{table}[!htb]
 \small{
\centering 
\caption{Some specific simulation points of Example~\ref{ex:one}. }
\label{tab:ex1}
\hspace{30mm} {Blahut-Arimoto Algorithm} \hspace{18mm} Algorithm~\hyperlink{algo:1}{1}
\vspace{3mm} \phantom{..}
  \begin{tabular}{c@{\hskip 3mm} | c@{\hskip 2mm} c@{\hskip 2mm} c@{\hskip 2mm} c  | c@{\hskip 2mm} c@{\hskip 3mm} c@{\hskip 3mm} c  }
 A priori error \hspace{1mm}  & \hspace{1mm}    1  &$0.1$ & $0.01$ & $0.001$ \hspace{1mm}   &\hspace{1mm}  1  &$0.1$ & $0.01$ & $0.001$   \\ 
 $C_{\textnormal{UB}}(W)$ & \hspace{1mm} --- & --- & ---  & --- \hspace{1mm}   &\hspace{1mm}  0.4419 & 0.4131 & 0.4092 & 0.4088   \\
 $C_{\textnormal{LB}}(W)$ & \hspace{1mm} 0.2930 & 0.4008 & 0.4088  & 0.4088 \hspace{1mm}   & \hspace{1mm}   0.3094 & 0.4069 & 0.4088 & 0.4088 \\
  A posteriori error & \hspace{1mm} --- & --- & ---  & --- \hspace{1mm}  & \hspace{1mm}   0.1325 & 0.0063 & 4.0$\cdot 10^{-4}$ & 3.7$\cdot 10^{-5}$ \\
 Time [s] & \hspace{1mm} 7.4  &69 &693  & 7306 \hspace{1mm} & \hspace{1mm}  114 & 1127 & 11\,036 & 110\,987  \\
  Iterations & \hspace{1mm} 14  &133 &1329  & 13\,288 \hspace{1mm} & \hspace{1mm}  27\,797 & 273\,447 & 2\,729\,860 & 27\,294\,000
  \end{tabular} }
\end{table} 
\end{Ex}

\begin{Ex} \label{ex:BEC}
Consider a binary erasure channel with erasure probability $\alpha$ whose channel transition matrix is given by 
$\W =  \left(\begin{smallmatrix}
1-\alpha & \alpha & 0 \\ 0 & \alpha & 1-\alpha
\end{smallmatrix} \right)$ and as such does not satisfy Assumption~\ref{ass:channel}. We use the perturbation method introduced in Remark~\ref{rmk:perturbation} to approximate its capacity that is analytically known to be $1-\alpha$ \cite[p.~189]{cover}. Table~\ref{tab:BEC} shows the performance of this perturbation method and Algorithm~\hyperlink{algo:1}{1}.

 \begin{table}[!htb]
 \small{
\centering 
\caption{Some specific simulation points of Example~\ref{ex:BEC} for $\alpha = 0.4$.}
\label{tab:BEC}
\vspace{3mm} \phantom{..}
  \begin{tabular}{c@{\hskip 4mm} | c@{\hskip 4mm} c@{\hskip 4mm} c@{\hskip 4mm} c}
  \hspace{1mm} Perturbation $\varepsilon$  & \hspace{1mm}    $10^{-4}$  &$10^{-5}$ & $10^{-6}$ & $10^{-7}$  \\ 
A priori error & \hspace{1mm} 0.01 & 0.01 & 0.01  & 0.01 \hspace{1mm}      \\
 $C_{\textnormal{UB}}(W)$ & \hspace{1mm} 0.6024 & 0.6003 & 0.6000  & 0.6000 \hspace{1mm}      \\
 $C_{\textnormal{LB}}(W)$ & \hspace{1mm} 0.5949 & 0.5994 & 0.5999  & 0.6000 \hspace{1mm}    \\
  A posteriori error & \hspace{1mm} 0.0075 & $9.2\cdot 10^{-4}$ & $1.1\cdot 10^{-4}$  & $1.2\cdot 10^{-5}$ \hspace{1mm}   \\
 Time [s] & \hspace{1mm} 0.70  & 0.54 & 0.66 & 0.78 \hspace{1mm}  \\
  Iterations & \hspace{1mm} 9056  & 7402 & 8523  & 9896 \hspace{1mm} 
  \end{tabular} }
\end{table}
\end{Ex}

\section{Channels with continuous input and countable output alphabets}   \label{sec:cont:Channels}
In this section we generalize the approximation scheme introduced in Section \ref{sec:classicalCapacity} to memoryless channels with continuous input and countable output alphabets. The class of discrete-time Poisson channels is an example of such channels with particular interest in applications, for example to model direct detection optical communication systems \cite{moser_phd,shamai90,ref:Chen-13}. Consider $\X \subseteq \R$ as the input alphabet set and $\Y = \N$ as the output alphabet set. The channel is described by the conditional probability $W(i|x) := \ProbIT{Y = i ~|~ X=x}$. 
	Given a channel $W$ and an integer $M$,  we introduce an $M$-\emph{truncated} version of the channel by 
		\begin{align} \label{W_M} 
			W_M(i|x)&:=\left \lbrace \begin{array}{ll}W(i|x) + \frac{1}{M}\sum\limits_{j \ge M}W(j|x) , & i\in \{0,1,\ldots,M-1\} \\ 0, & i\geq M.
				\end{array} \right.
		\end{align}
	$W_M$ can be seen as a channel with input alphabet $\X$ and output alphabet $\{0,1,\ldots,M-1\}$. Figure~\ref{fig:inf:channel:plot} shows a pictorial representation of a channel and its $M$-truncated counterpart. 
The finiteness of the output alphabet of $W_{M}$ allows us to deploy an approximation scheme similar to the one developed in Section \ref{sec:classicalCapacity} to numerically approximate $C(W_{M})$.  

	 	\begin{figure}[htb!] 
 			  \centering 
 			  \includegraphics[scale = 0.82]{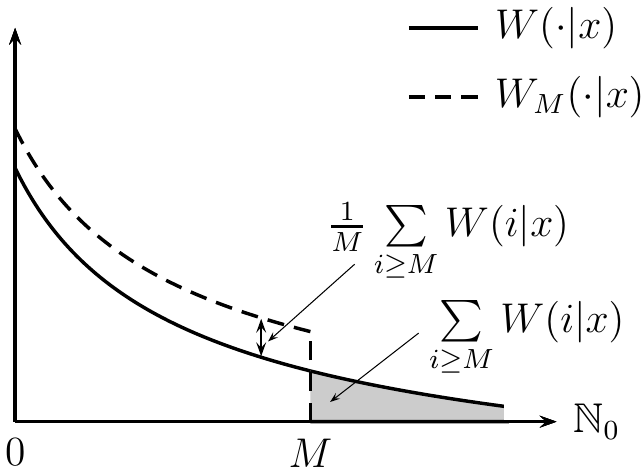} 
 			  \caption{Pictorial representation of the $M$-truncated channel counterpart.} 
 			  \label{fig:inf:channel:plot} 
 	 	\end{figure} 	
The following definition is a key feature of the channel required for the theoretical results developed in this section which, roughly speaking, imposes a certain decay rate for the output distribution uniformly in the input alphabet.
	\begin{Def} [Polynomial tail]
	\label{def:tail}
		The channel $W$ features a \emph{$k$-ordered polynomial tail} if for $M\in\mathbb{N}_{0}$ and $k\in\Rp$
		\begin{align}\label{R}
			R_k(M) \Let \sum\limits_{i\ge M} \big(\sup_{x \in \X}W(i|x)\big)^k  < \infty.
		\end{align}
	\end{Def}
	The following assumptions hold throughout this section. 
	\begin{As} \label{a:channel}  \ 
	\begin{enumerate}[label=(\roman*), itemsep = 1mm, topsep = -1mm]
		\item The channel $W$ has a $k$-ordered polynomial tail for some $k \in (0,1)$ in the sense of Definition \ref{def:tail}. \label{ass:channel:i}
		\item The mapping $x\mapsto W(i|x)$ is Lipschitz continuous for any $i\in\mathbb{N}_{0}$ with Lipschitz constant $L$.	\label{ass:channel:ii}
	\end{enumerate}
	\end{As}

Assumption~\ref{a:channel} allows us to relate the capacity of the original channel to that of its truncated counterpart. 
	
	\begin{Thm} \label{thm:C}
	Suppose channel $W$ satisfies Assumption~\ref{a:channel}\eqref{ass:channel:i} with the order $k \in (0,1)$. Then, for any $M \in \N$ and for any probability distribution $p\in\mathcal{P}(\mathcal{X})$ we have
		\begin{align*} 
			\big|  I(p &, W)  - I(p,W_{M}) \big|  \leq \frac{2\log(\e)}{\e(1-k)} \Big[ M^{1-k}\big(R_1(M)\big)^k + R_k(M) \Big],
		\end{align*}
	where $R_k(M)$ is as defined in \eqref{R}. 
	\end{Thm}

To prove Theorem \ref{thm:C} we need a preliminary lemma.
	
	\begin{Lem}
	\label{lem:log}
		Given $k \in (0,1)$ and $p\in[0,1]$, we have for all $x\in[0,1-p]$ 
		\begin{align*}
			\big| (p+x)\log(p+x) - p\log(p) \big| \le \frac{\log(\e)}{\e(1-k)} x^k.
		\end{align*}
	\end{Lem}

	\begin{proof}
		Note that for a fixed $x \in [0,1]$, the mapping $p \mapsto  (p+x)\log(p+x) - p\log(p)$ is non-decreasing; observe that the derivative of the mapping is non-negative for all $x \in [0,1]$. Therefore, it suffices to verify the claim for $p \in \{0,1\}$. For $p = 1$ and accordingly $x = 0$, Lemma~\ref{lem:log} holds trivially. Let $p = 0$ and $h(x) \Let \frac{\log(\e)}{\e(1-k)} x^{k-1} + \log(x)$. Note that $h(1) = \tfrac{\log(\e)}{\e(1-k)} >0$ and $h(x) \rightarrow \infty$ as $x \rightarrow 0$. Hence, by setting $\frac{\diff }{\diff x}h(x^\star) = 0$, it can be easily seen that 
			$$\min_{x \in (0,1]}h(x) = h(x^\star)= 0, \qquad x^\star \Let \e^{\frac{1}{k-1}}.$$
		Thus $h(x)\ge 0$, and consequently $xh(x)\ge 0$ for all $x \in (0,1]$, which concludes the proof. 
	\end{proof}

	\begin{proof}[Proof of Theorem \ref{thm:C}]
	We bound the mutual information difference uniformly in the input probability distribution $p \in \meas(\X)$. Observe that
	\begin{align*}
		\big|  I(p &, W)  - I(p,W_{M}) \big| \\
		& = \bigg| \int_{\X} \Big[-h\big(W(\cdot,x)\big) + h\big(W_M(\cdot,x)\big)\Big] p(\diff x) \\ 
		&\qquad + h\Big(\int_{\X}W(\cdot,x) p(\diff x) \Big) -  h\Big(\int_{\X}W_M(\cdot,x) p(\diff x) \Big) \bigg| \\
		& = \bigg|  \int_{\X} \Big[ \sum\limits_{i \in \N} W(i|x)\log(W(i|x)) - W_M(i|x)\log(W_M(i|x))\Big] p(\diff x) \\
		& \qquad \qquad + \sum\limits_{i \in \N}  -\Big(\int_{\X} W(i|x)p(\diff x)  \Big)\log \Big( \int_{\X} W(i|x)p(\diff x) \Big) \\
		& \qquad \qquad + \Big(\int_{\X} W_M(i|x) p(\diff x) \Big)\log \Big( \int_{\X} W_M(i|x) p(\diff x) \Big) \bigg|.
	\end{align*}
	By the definition of the truncated channel in \eqref{W_M} and applying Lemma \ref{lem:log} to the above relation, we have
	\begin{align*}
		\big|  I(p , W)   - I(p,W_{M}) \big| & \le \frac{\log(\e)}{\e(1-k)} \bigg(  \int_{\X} \bigg[ \sum\limits_{i < M } \Big(\frac{1}{M}\sum\limits_{j \ge M}W(j|x)\Big)^k + \sum\limits_{i \ge M} \big(W(i|x)\big)^k \bigg] p(\diff x) \\ 
		& \quad   + \sum\limits_{i < M } \Big(\frac{1}{M}\sum\limits_{j \ge M}  \int_{\X} W(j|x)p(\diff x) \Big)^k + \sum\limits_{i \ge M}   \Big( \int_{\X} W(i|x)p(\diff x) \Big)^k  \bigg) \\
		& \le \frac{2\log(\e)}{\e(1-k)} \bigg( M \Big(\frac{R_1(M)}{M}\Big)^k + R_k(M)\bigg),
	\end{align*}
	which concludes the proof.
	\end{proof}

Note that Theorem~\ref{thm:C} directly implies an upper bound to the capacity since
	\begin{align*}
		|C(W)-C(W_M)| =  \big| \sup\limits_{p\in \meas(\X)} I(p,W)- \sup \limits_{p \in \meas(\X)} I(p,W_{M}) \big| \leq \sup \limits_{p \in \meas(\X)} \big| I(p,W) - I(p,W_{M}) \big|.
	\end{align*} 
We consider two types of input cost constraints: a peak-power constraint $\ProbIT{X\in \A}=1$ for a compact set $\A\subseteq\mathcal{X}$ and an average-power constraint $\E{s(X)}\leq S$ for $S\in\Rp$ and a continuous function $s$ on $\mathcal{X}$.  
The primal capacity problem for the channel $W_{M}$ is given by
\begin{equation} \label{eq:inf:channel:primal}
C_{\A,S}(W_{M})= \left\{
\begin{array}{lll}
			&\sup\limits_{p} 		& \I{p}{W_{M}} \\
			&\subjectto					& \E{s(X)}\leq S\\
			& 						& p\in \mathcal{P}(\A),
	\end{array} \right.
\end{equation}
where $\mathcal{P}(\A)$ denotes the space of all probability distributions supported on $\A$, cf.~\eqref{eq:cont_DMC_capacity_const}. Our method always requires a peak-power constraint, whereas the average-power constraint is optimal.
The following proposition allows us to restrict the optimization variables from probability distributions to probability densites.
\begin{Prop} \label{lem:density:dense}
The optimization problem \eqref{eq:inf:channel:primal} is equivalent to
\begin{equation*} 
C_{\A,S}(W_{M})= \left\{
\begin{array}{lll}
			&\sup\limits_{p} 		& \I{p}{W_{M}} \\
			&\subjectto					& \E{s(X)}\leq S\\
			& 						& p\in \mathcal{D}(\A),
	\end{array} \right.
\end{equation*}
where $\mathcal{D}(\A)$ is the set of probability densities functions,~i.e., $\mathcal{D}(\A):=\{ f\in\Lp{1}(\A) \, : \, f\geq 0, \, \int_{\A}f(x)\drv x = 1 \}$. 
\end{Prop}
\begin{proof} 
We show that the optimization problem \eqref{eq:inf:channel:primal} is equivalent to
\begin{equation*}
C_{\A,S}(W_{M}) = \sup\limits_{p\in\mathfrak{D}(\A)} \left\{ I(p,W_{M}) \ : \ \E{s(X)}\leq S \right\},
\end{equation*}
where $\mathfrak{D}(\A)$ is the space of probability measures that are absolutely continuous with respect to the Lebesgue measure. This completes the proof since optimizing over $\mathfrak{D}(\A)$ is equivalent to optimizing over the space of probability densities $\mathcal{D}(\A)$ according to the Radon-Nikod\'ym Theorem \cite[Theorem~3.8, p.~90]{ref:Folland-99}.

It is known that the mapping $p \mapsto I(p,W_{M})$ is weakly lower semicontinuous \cite{ref:verdu-12}. It then suffices to show that $\mathfrak{D}(\A)$ is weakly dense in $\mathcal{P}(\A)$. Let $\B$ be a countable dense subset of $\A$, and $\Delta(\B)$ be the family of probability measures whose supports are finite subsets of $\B$. It is well known that $\Delta(\B)$ is weakly dense in $\mathcal{P}(\A)$, i.e., $\mathcal{P}(\A) = \overline{{\Delta}(\B)}$ \cite[Theorem~4, p.~237]{ref:billingsley-68}, where $\overline{\Delta}$ is the weak closure of $\Delta$. Moreover, thanks to the Lebesgue differentiation theorem \cite[Theorem~3.21, p.~98]{ref:Folland-99}, we know that for any $b \in \B$ the point measure $\delta_{\{b\}} \in \Delta(\B)$ can be arbitrarily weakly approximated by measures in $\mathfrak{D}(\A)$, i.e., $\delta_{\{b\}} \in \overline{\mathfrak{D}(\A)}$. Hence, we have $\Delta(\B) = \overline{\mathfrak{D}(\A)}$, which in light of the preceding assertion implies $\mathcal{P}(\A) = \overline{\mathfrak{D}(\A)}$.
\end{proof}
We consider the pair of vector spaces
$(\Lp{1}(\A),\Lp{\infty}(\A))$ together with the bilinear form
\begin{align*}
\inprod{f}{g}:=\int_{\mathcal{X}}f(x)g(x)\drv x.
\end{align*}
In the light of \cite[Theorem~243G]{ref:fremlin-03} this is a dual pair of vector spaces; we refer to \cite[Section~3]{ref:Anderson-87} for the details of the definition of dual pairs of vector spaces. 
Considering the standard inner product as a bilinear form on the dual pair $(\R^{M},\R^{M})$, we define the linear operator $\WW: \R^{M}\to \Lp{\infty}(\A)$ and its adjoint operator $\WW^{\star}:\Lp{1}(\A)\to\R^{M}$, given by
\begin{align*}
\WW \lambda (x) := \sum_{i=1}^{M}W_{M}(i-1|x)\lambda_{i}, \qquad  (\WW^{\star}p)_{i} := \int_{X}W_{M}(i-1|x)p(x)\drv x.
\end{align*}
Let $S_{\max}:=\inf_{p\in\mathcal{D}(\A)}\{ \inprod{p}{s} \ : \ I(p,W_{M}) = \sup_{q\in\mathcal{D}(\A)} I(q,W_{M})\}$. Following similar lines as in Lemma~\ref{lem:equivalent:primal:problem}, one can deduce that in problem \eqref{eq:inf:channel:primal} the inequality input constraint can be replaced by equality (resp. removed) is $S<S_{\max}$ (resp. $S\geq S_{\max}$). That is, in view of  Proposition~\ref{lem:density:dense}, Lemma~\ref{lem:equivalent:primal:problem} and the discussion there, problem \eqref{eq:inf:channel:primal} (under Assumption~\ref{ass:channel:Poisson}, that we require later) is equivalent to
\begin{equation} \label{opt:primal:equivalent:Poisson}
 	\mathsf{P}: \quad \left\{ \begin{array}{lll}
			&\sup\limits_{p,q} 		&- \inprod{p}{r} + H(q) \\
			&\subjectto					& \WW^{\star}p = q\\
			&					& \inprod{p}{s} = S\\
			& 					& p\in\mathcal{D}(\A), \ q\in\Delta_{M},
	\end{array} \right.
\end{equation}
where $r(\cdot): = -\sum_{j=0}^{M-1}W_{M}(j|\cdot)\log(W_{M}(j|\cdot))$ is an element in $\Lp{\infty}(\A)$ by Assumption~\ref{a:channel}\eqref{ass:channel:ii}. For the rest of the section we restrict attention to \eqref{opt:primal:equivalent:Poisson}, since unconstrained problem can be solved in a similar way.
We call \eqref{opt:primal:equivalent:Poisson} the primal program.
Thanks to the dual vector space framework, the Lagrange dual program of $\mathsf{P}$ is given by
\begin{align} \label{eq:Dual:Program:Poisson}\mathsf{D}: \quad\left\{ \begin{array}{ll}
	\underset{\lambda}{\inf} 		&G(\lambda) + F(\lambda) \\
			\text{s.t. } 					& \lambda\in \R^{M},
	\end{array}\right.
\end{align}
where 
\begin{align*} 
G(\lambda)= \left\{ \begin{array}{ll}
			\underset{p}{\sup} 		& \inprod{p}{\WW \lambda}-\inprod{p}{r} \\
			\text{s.t. } 				& \inprod{p}{s} =  S \\
								& p\in \mathcal{D}(\A)
	\end{array} \right. \text{ and}
	\qquad 
	F(\lambda)= \left\{ \begin{array}{ll}
			\underset{q}{\max} 		&H(q)-\lambda\transp q \\
			\text{s.t. } 				& q\in\Delta_{M}.
	\end{array}\right. 
\end{align*}
\begin{Lem} \label{lem:strong:duality:poisson}
Strong duality holds between \eqref{opt:primal:equivalent:Poisson} and \eqref{eq:Dual:Program:Poisson}.
\end{Lem}
\begin{proof}
Note that the dualized constraint is a linear equality constraint. Therefore the conditions of (1) in \cite[Theorem~5]{mitter08} holds and as such strong duality follows by \cite[Theorem~4]{mitter08}.
\end{proof}
In the remainder of this article we impose the following assumption on the channel.
\begin{As} \label{ass:channel:Poisson}
$\gamma_{M}:=\min\limits_{y\in\{0,1,\hdots,M-1\}}\min\limits_{x\in\A}W_{M}(y|x)>0$
\end{As}
In case $\sum_{j\geq M} W(j|x) > 0$ for all $x$, Assumption~\ref{ass:channel:Poisson} holds according to \eqref{W_M} and a lower bound can be given by $\gamma_{M} \geq \tfrac{1}{M} \min_{x} \sum_{j\geq M} W(j|x)$.
Under Assumption~\ref{ass:channel:Poisson} we can show that we can again assume without loss of generality that $\lambda$ takes values in a compact set.
\begin{Lem} \label{lem:compact:set:Poisson}
Under Assumption~\ref{ass:channel:Poisson}, the dual program \eqref{eq:Dual:Program:Poisson} is equivalent to 
\begin{align*}
\left\{ \begin{array}{ll}
			\underset{\lambda}{\min} 		&G(\lambda) + F(\lambda) \\
			\textnormal{s.t. } 				& \lambda\in Q,
	\end{array}\right. 
\end{align*}
where $Q:= \left\{ \lambda\in\R^{M} \ : \ \Norm{\lambda}_{2}\leq M \left( \log(\gamma_{M}^{-1}) \vee \tfrac{1}{\ln 2} \right) \right\}$.
\end{Lem}
\begin{proof}
The proof follows the same lines as in the proof of Lemma~\ref{lem:compact:set}.
\end{proof}

Note that $F(\lambda)$ is the same as in Section~\ref{sec:classicalCapacity} and therefore given by \eqref{analytical:solution:F} and its gradient by \eqref{eq:gradient:F}.
As in Section~\ref{sec:classicalCapacity}, we consider the smooth approximation
\begin{align} \label{eq:Gnu:cts}
G_{\nu}(\lambda)= \left\{ \begin{array}{ll}
			\underset{p}{\sup} 		& \inprod{p}{\WW \lambda}-\inprod{p}{r} + \nu \Hdiff{p} - \nu \log(\rho) \\
			\text{s.t. } 				& \inprod{p}{s} =  S \\
				 				& p\in\mathcal{D}(\A),
	\end{array} \right. 
\end{align}
with smoothing parameter $\nu\in\R_{>0}$ and $\rho$ denoting the Lebesgue measure of $\A$. 
To analyze the properties of $G_{\nu}(\lambda)$ we need one more auxiliary lemma. 
\begin{Lem}\label{lem:strong:convexity:cts}
The function $\mathcal{D}(\A)\ni p \mapsto  -\Hdiff{p}+\log(\rho)\in\Rp$ is strongly convex with convexity parameter $\sigma = 1$.
\end{Lem}
\begin{proof}
The proof follows the ideas of \cite{ref:Nest-05}. It can easily be shown that for $d(p):=-\Hdiff{p}+\log(\rho)$
\begin{equation*}
\inprod{d''(p)\cdot g}{g} = \int_{\A}\frac{g(x)^{2}}{p(x)}\drv x.
\end{equation*}
Cauchy-Schwarz then implies
\begin{align*}
\inprod{d''(p)\cdot g}{g} \geq \frac{\left( \int_{\A} g(x) \drv x \right)^{2}}{\int_{\A} p(x) \drv x} = \Norm{g}^{2}.
\end{align*} 
\end{proof}
Furthermore, we can show that the uniform approximation $G_{\nu}(\lambda)$ is smooth and has a Lipschitz continuous gradient, with known constant.
The following result is a generalization of Proposition~\ref{prop:Lipschitz:continuity:DMC}.
\begin{Prop} \label{prop:Lipschitz:cts:channel}
$G_{\nu}(\lambda)$ is well defined and continuously differentiable at any $\lambda\in \R^{M}$. Moreover, this function is convex and its gradient $\nabla G_{\nu}(\lambda)=\WW^{\star} p_{\nu}^{\lambda}$ is Lipschitz continuous with constant $L_{\nu} =\tfrac{1}{\nu}$.
\end{Prop}
\begin{proof}
It is known, according to Theorem~5.1 in \cite{ref:devolder-12}, that $G_{\nu}(\lambda)$ is well defined and continuously differentiable at any $\lambda\in \R^{M}$ and that  this function is convex and its gradient $\nabla G_{\nu}(\lambda)=\WW^{\star} p_{\nu}^{\lambda}$ is Lipschitz continuous with constant $L_{\nu} =\tfrac{1}{\nu}\Norm{\WW}^{2}$, where we have also used Lemma~\ref{lem:strong:convexity:cts}.
The operator norm can be simplified to
\begin{align}
\Norm{\WW} 	&=		\sup\limits_{\lambda\in\R^{M}\!, \, p\in\Lp{1}(\A)} \left\{ \inprod{p}{\WW \lambda} \ : \ \Norm{\lambda}_{2}=1, \ \Norm{p}_{1}=1 \right\} \nonumber \\
			&\leq 	\sup\limits_{\lambda\in\R^{M}\!, \, p\in\Lp{1}(\A)} \left\{ \Norm{\WW^{\star}p}_{2} \Norm{\lambda}_{2} \ : \ \Norm{\lambda}_{2}=1, \ \Norm{p}_{1}=1\right\}\label{eq:norm:W:proof:step:CS} \\
			&\leq 	\sup\limits_{ p\in\Lp{1}(\A)} \left\{ \Norm{\WW^{\star}p}_{1} \ : \ \Norm{p}_{1}=1\right\}\nonumber \\
			&= 		\sup\limits_{ p\in\Lp{1}(\A)} \left\{ \sum_{i=0}^{M-1}  \int_{\mathcal{X}} W_{M}(i|x) p(x) \drv x   \ : \ \Norm{p}_{1}=1\right\}\nonumber \\
			&= 		\sup\limits_{ p\in\Lp{1}(\A)} \left\{   \int_{\mathcal{X}} \Norm{ W_{M}(\cdot|x)}_{1} p(x) \drv x   \ : \ \Norm{p}_{1}=1\right\} \nonumber\\
			&\leq 	\sup\limits_{x\in \A} \Norm{ W_{M}(\cdot|x)}_{1} \nonumber \\
			&\leq 	1, \nonumber
\end{align}
where \eqref{eq:norm:W:proof:step:CS} is due to Cauchy-Schwarz.
\end{proof}
We denote by $p_{\nu}^{\lambda}$ the optimizer to \eqref{eq:Gnu:cts}, that is unique since the objective function is strictly concave. 
To analyze the solution to \eqref{eq:Gnu:cts} we consider the following optimization problem, that, if feasible, has a closed form solution
\begin{equation} \label{opt:cover:cont}
 	\left\{ \begin{array}{lll}
			&\underset{p}{\sup} 		&\Hdiff{p} + \inprod{p}{c} \\
			&\text{s.t. } 			& \inprod{p}{s} = S\\
			& 					& p\in\mathcal{D}(\A),
	\end{array} \right.
\end{equation}
with $c,s\in \Lp{\infty}(\A)$.
\begin{Lem} \label{lem:cover:cont}
Let $p^{\star}(x)=2^{\mu_1 + c(x) + \mu_2 s(x)}$, where $\mu_1, \mu_2\in\R$ are chosen such that $p^\star$ satisfies the constraints in \eqref{opt:cover:cont}. Then $p^\star$ uniquely solves \eqref{opt:cover:cont}.
\end{Lem}
The proof directly follows from \cite[p.~409]{cover} and the proof of Lemma~\ref{lem:cover}. Hence, $G_{\nu}(\lambda)$ has a (unique) analytical optimizer
\begin{align} \label{eq:finite:optimizer:pmu:cts}
p_{\nu}^{\lambda}(x,\mu) = 2^{\mu_{1} + \tfrac{1}{\nu}\left( \WW \lambda(x) - r(x)\right) + \mu_{2}s(x)}, \quad x\in \mathcal{X},
\end{align}
where $\mu_{1},\mu_{2}\in\R$ have to be chosen such that $\inprod{p_{\nu}^{\lambda}(\cdot,\mu)}{s}=S$ and $p_{\nu}^{\lambda}(\cdot,\mu)\in\mathcal{D}(\A)$; for this choice of $\mu_1$, $\mu_2$ we denote the solution by $p_{\nu}^{\lambda}(\cdot)$. 
\begin{Rem}[No input constraints]\label{rmk:stabilization:optimizer:cts}
In case of no input constraints, the unique optimizer to \eqref{eq:Gnu:cts} is given by
\begin{equation*}
p_{\nu}^{\lambda}(x) = \frac{2^{ \tfrac{1}{\nu} (\WW\lambda(x) -  r(x))} }{\int_{\A} 2^{ \tfrac{1}{\nu} (\WW\lambda(x) -  r(x))} \drv x},
\end{equation*}
whose numerical evaluation can be done in a stable way by following Remark~\ref{rmk:finite:no:input:const}.
\end{Rem}
\begin{Rem}[Additional input constraints]\label{rmk:finite:constraint:optimizer:cts}
As in Remark~\ref{rmk:finite:constraint:optimizer}, in case of additional input constraints we need an efficient method to find the coefficients $\mu_{i}$ in \eqref{eq:finite:optimizer:pmu:cts}. This problem can again be reduced to a finite dimensional convex optimization problem (\cite[Theorem~4.8]{ref:Borwein-91},\cite[p.~257 ff.]{ref:Lasserre-11}), in the sense that the coefficients $\mu_i$ are the unique maximizers to
\begin{equation} \label{eq:opt:problem:find:mu:cts}
\max\limits_{\mu\in\R^{2}}\left\{ y\transp \mu -\int_{\A}p_{\nu}^{\lambda}(x,\mu)\drv x \right\},
\end{equation}
where $y:=(1,S)$. Note that $\eqref{eq:opt:problem:find:mu:cts}$ is an unconstrained maximization of a striclty concave function. The evalutation of the gradient and the Hessian of this objective function involves computing moments of the measure $p_{\nu}^{\lambda}(x,\mu)\drv x$, which unlike to the finite input alphabet case (Remark~\ref{rmk:finite:constraint:optimizer}) is numerically difficult. In \cite[p.~259 ff.]{ref:Lasserre-11}, an efficient approximation of the mentioned gradient and Hessian in terms of two single semidefinite programs involving two linear matrix inequalities (LMI) is presented, where the desired accuracy is controlled by the size of the LMI constraints. As mentioned in Remark~\ref{rmk:finite:constraint:optimizer}, this will provide a suboptimal solution to the maximum entropy problem \eqref{eq:Gnu:cts} and as such the error bounds of Theorem~\ref{thm:error:bound:capacity:continuous:channel} do not hold. By following \cite{ref:Devolver-13}, however, one can quantify the approximation error of Algorithm~\hyperlink{algo:1}{1} in case of an inexact gradient. We also refer the interested reader to \cite{ref:D:Sutter-16}, for a related work on channel capacity approximation under inexact first-order information.
\end{Rem}

Note that the differential entropy $\Hdiff{p}\leq \log(\rho) $ for all $p\in\mathcal{D}(\A)$ and that there exists a function $\iota:\Rsp\to\Rp$ such that
 \begin{equation} \label{eq:uniform:bound:cts}
 G_{\nu}(\lambda)\leq G(\lambda)\leq G_{\nu}(\lambda) + \iota(\nu) \ \text{for all }\lambda\in Q,
 \end{equation}
 i.e., $G_{\nu}(\lambda)$ is a uniform approximation of the non-smooth function $G(\lambda)$.  
The following lemma, Lemma~\ref{lem:iota:new}, provides an explicit expression for the function $\iota$ in \eqref{eq:uniform:bound:cts} under some Lipschitz continuity assumptions, implying in particular that $\iota(\nu)\to 0$ as $\nu\to 0$. 
\begin{Lem}\label{lem:Lf}
Under Assumption~\ref{a:channel}\eqref{ass:channel:ii} and Assumption~\ref{ass:channel:Poisson} the function $f_{\lambda}(\cdot):= \WW\lambda(\cdot) -  r(\cdot)$ is Lipschitz continuous uniformly in $\lambda\in Q$ with constant $L_{f}=LM^2 (\log \tfrac{1}{\gamma_M} \vee \tfrac{1}{\ln 2})+M L |\log \tfrac{1}{\gamma_M} - \frac{1}{\ln 2} |$.
\end{Lem}
\begin{proof}
Let $x_1,x_2 \in \mathcal{X}$, then by definition of $f_{\lambda}(\cdot)$ we obtain
\begin{subequations}
\begin{align}
&|f_{\lambda}(x_1)-f_{\lambda}(x_2)| \nonumber \\
&\hspace{5mm}= \left| \sum_{i=1}^M W_{M}(i-1|x_1) \lambda_i + \sum_{j=1}^{M} W_{M}(j-1|x_1) \log W_{M}(j-1|x_1) \right. \nonumber \\
&\hspace{11mm}\left. - \sum_{i=1}^M W_{M}(i-1|x_2) \lambda_i - \sum_{j=1}^{M} W_{M}(j-1|x_2) \log W_{M}(j-1|x_2) \right| \nonumber \\
&\hspace{5mm}\leq \left| \sum_{i=1}^M \left(W_{M}(i-1|x_1)-W_{M}(i-1|x_2) \right)\lambda_i \right| + \left|\Hh{W_{M}(\cdot|x_1)}-\Hh{W_{M}(\cdot|x_2)} \right| \label{eq:triang}\\
&\hspace{5mm}\leq  \sum_{i=1}^M \left|\left(W_{M}(i-1|x_1)-W_{M}(i-1|x_2) \right)\lambda_i \right| + \left|\Hh{W_{M}(\cdot|x_1)}-\Hh{W_{M}(\cdot|x_2)} \right| \label{eq:ass1}\\
&\hspace{5mm}\leq L M \|\lambda \|_1 |x_1-x_2 | + \left|\Hh{W_{M}(\cdot|x_1)}-\Hh{W_{M}(\cdot|x_2)} \right|\label{eq:LambaSet}\\
&\hspace{5mm}\leq L M^2  \left(\log \frac{1}{\gamma_M}\vee \frac{1}{\ln 2} \right)  |x_1-x_2| + \left|\Hh{W_{M}(\cdot|x_1)}-\Hh{W_{M}(\cdot|x_2)} \right| \label{eq:LambdaSet}\\
&\hspace{5mm}\leq L M^2  \left(\log \frac{1}{\gamma_M}\vee \frac{1}{\ln 2} \right)  |x_1-x_2| + ML \left|\log \frac{1}{\gamma_M}- \frac{1}{\ln 2} \right|  |x_1-x_2|. \label{eq:fini}
\end{align}
\end{subequations}
Inequalities~\eqref{eq:triang} and \eqref{eq:ass1} use the triangle inequality. Inequality~\eqref{eq:LambaSet} follows by Assumption~\ref{a:channel}\eqref{ass:channel:ii} and \eqref{eq:LambdaSet} can be derived by following the proof of Lemma~\ref{lem:compact:set:Poisson}, which is similar to the one of Lemma~\ref{lem:compact:set}. Finally, \eqref{eq:fini} follows from the fact that the function $\Delta_n \ni x^n \mapsto \Hh{x^n} \in \Rp $ with $\min_{1\leq i \leq n} x_i < c$ is Lipschitz continuous with constant $n \left|\log \tfrac{1}{c}-\tfrac{1}{\ln 2} \right|$ and from Assumption~\ref{a:channel}\eqref{ass:channel:ii}. 
\end{proof}
\begin{As}[Lipschitz continuity of the average-power constraint function] \label{a:constraint_fct:s}	
The average-power constraint function $s(\cdot)$ is Lipschitz continuous with constant $L_{s}$. 
\end{As}
\begin{Lem} \label{lem:iota:new}
Under Assumptions~\ref{a:channel}\eqref{ass:channel:ii}, \ref{ass:channel:Poisson} and \ref{a:constraint_fct:s} a possible choice of the function $\iota$ in \eqref{eq:uniform:bound:cts} is given by
\begin{equation*}
\iota(\nu) = \left\{ 
  \begin{array}{l l}
    \nu \left( \log\left( \frac{T_{1}}{\nu}+T_{2} \right) +1 \right), & \quad \nu<\tfrac{T_{1}}{1-T_{2}} \text{ or } T_{2}>1\\
    \nu, & \quad \text{otherwise},
  \end{array} \right.
\end{equation*}
where $T_{1}\!:=L_{f}\rho + 2L_{f}L_{s}\rho^{2}\!\left( \tfrac{1}{-\underline{s}} \vee \tfrac{1}{\overline{s}} \right)$, $T_{2}\!:= L_{s}\rho (\underline{\mu} \, \vee \, \overline{\mu})$, $\underline{\mu}\!:=\tfrac{2}{-\underline{s}}\log\!\left( \tfrac{2L_{s}\rho}{-\underline{s}}\vee 1 \right)$, $\overline{\mu}\!:=\tfrac{2}{\overline{s}}\log\!\left( \tfrac{2L_{s}\rho}{\overline{s}}\vee 1 \right)$, $\rho:=\int_{\A}\drv x$, $\underline{s} := -S + \min_{x\in\A} s(x)$ and $\overline{s} := -S + \max_{x\in\A} s(x)$.
\end{Lem}
\begin{proof} 
We start by the following definitions that simplify the proof below
\begin{alignat*}{3}
f_{\lambda, \nu}(x)&:=\WW\lambda(x)-r(x) + \nu \mu_{\nu}s(x), \qquad  &\bar{f}_{\lambda,\nu}&:= \max\limits_{x\in\A}f_{\lambda,\nu}(x) \\
B_{\lambda,\nu}(\varepsilon)&:=\left\{ x\in\A \ | \ \bar{f}_{\lambda,\nu} - f_{\lambda,\nu}(x) <\varepsilon \right\}, \qquad  &\eta_{\lambda,\nu}(\varepsilon) &:= \int_{B_{\lambda,\nu}(\varepsilon)}\drv x.
\end{alignat*}
By the Lipschitz continuity of $f_{\lambda}(\cdot)$ and $s(\cdot)$ we get the uniform lower bound 
\begin{equation} \label{eq:eta:lower:bound}
\eta_{\lambda,\nu}(\varepsilon) \geq  \frac{\varepsilon}{L_{f}+|\nu\mu_{\nu}|L_{s}}\wedge \rho.
\end{equation}
By using the solution to $G_{\nu}(\lambda)$, according to \eqref{eq:finite:optimizer:pmu:cts} we can write
\begin{subequations}
\begin{align}
G_{\nu}(\lambda) 	&= 		-\nu \log(\rho) + \nu \log\left( \int_{\A}2^{\tfrac{1}{\nu}f_{\lambda,\nu}(x)}\drv x \right) \label{eq:proof:Gnu}\\
					&\leq	\inf\limits_{\ell\in\R}\max\limits_{x\in\A} \left\{ f_{\lambda}(x)+\ell s(x) \right\} \label{eq:step:Gnu:smaller:G0} \\
					&= G(\lambda), \label{eq:step:strong:duality}
\end{align}
\end{subequations}	
where the equality \eqref{eq:step:strong:duality} follows as \eqref{eq:step:Gnu:smaller:G0} is the dual program to $G(\lambda)$ and strong duality holds. The inequality \eqref{eq:step:Gnu:smaller:G0} then is due to $G_{\nu}(\lambda)\leq G(\lambda)$ for any $\lambda$, see \eqref{eq:uniform:bound:cts}. Therefore,

\begin{subequations}
\begin{align}
&G(\lambda)-G_{\nu}(\lambda)		\leq		\bar{f}_{\lambda,\nu}-G_{\nu}(\lambda) \label{eq:iota:G} \\
		&\qquad =		\nu \left( -\log \left(  \int_{B_{\lambda,\nu}(\varepsilon)}2^{\frac{1}{\nu}\left( f_{\lambda,\nu}(x) - \bar{f}_{\lambda,\nu}\right)} \drv x +  \int_{B_{\lambda,\nu}^{\setC}(\varepsilon)}2^{\frac{1}{\nu}\left( f_{\lambda,\nu}(x) - \bar{f}_{\lambda,\nu}\right)} \drv x\right) + \log(\rho)  \right) \label{eq:iota:Gnu}\\
		&\qquad \leq 	\nu \left( -\log \left(  \int_{B_{\lambda,\nu}(\varepsilon)}2^{\frac{1}{\nu}\left( f_{\lambda,\nu}(x) - \bar{f}_{\lambda,\nu}\right)} \drv x \right) + \log(\rho)  \right) \nonumber \\
		&\qquad \leq 	\nu \left( -\log \left( \eta_{\lambda,\nu}(\varepsilon) 2^{-\frac{\varepsilon}{\nu}} \right) + \log(\rho)  \right) \label{eq:iota:B:eta} \\
		&\qquad \leq 	\nu \left( -\log \left( \tfrac{\varepsilon}{L_{f}+|\nu\mu_{\nu}|L_{s}}\vee \rho \right) +\frac{\varepsilon}{\nu} + \log(\rho)  \right) \label{eq:iota:lipschitz} \\
		&\qquad =		\nu  \log \left( \tfrac{(L_{f}+|\nu\mu_{\nu}|L_{s})\rho}{\varepsilon} \vee 1 \right) + \varepsilon, \nonumber
\end{align}
\end{subequations}
where \eqref{eq:iota:G} follows from \eqref{eq:step:strong:duality} and \eqref{eq:iota:Gnu} is due to \eqref{eq:proof:Gnu}. The inequality \eqref{eq:iota:B:eta} results from the definitions of $B_{\lambda,\nu}(\varepsilon)$ and $\eta_{\lambda,\nu}(\varepsilon)$ above and \eqref{eq:iota:lipschitz} is implied by \eqref{eq:eta:lower:bound}. Finally, it can be seen that for $\nu<(L_{f}+|\nu\mu_{\nu}|L_{s})\rho$, the optimal choice for $\varepsilon$ is $\nu$, which leads to
\begin{equation} \label{eq:proof:itoa:almost:done}
G(\lambda)-G_{\nu}(\lambda) \leq \nu \left( 1 +  \log \left( \tfrac{(L_{f}+|\nu\mu_{\nu}|L_{s})\rho}{\nu} \vee 1 \right) \right). 
\end{equation}
It remains to upper bound the term $|\nu\mu_{\nu}|$. Define $\underline{f}:=\min_{x,\lambda}f_{\lambda}(x)$, $\overline{f}:=\max_{x,\lambda}f_{\lambda}(x)$, $\Delta_{f}:=\overline{f}-\underline{f}$ and note that $\Delta_{f}\leq L_{f}\rho$. By \eqref{eq:proof:Gnu}, \eqref{eq:uniform:bound:cts} and the fact that adding an additional constraint to a maximization problem cannot increase its objective value
\begin{equation*}
G_{\nu}(\lambda) =  \nu \log\left( \int_{\A}2^{\tfrac{1}{\nu}\left( f_{\lambda}(x)+\nu\mu_{\nu}s(x)\right)}\drv x \right) -\nu \log(\rho) \leq \overline{f}=\nu \log \left( 2^{\tfrac{1}{\nu}\overline{f}} \right),
\end{equation*}
which is equivalent to
 $\int_{\A}2^{\tfrac{1}{\nu}\left( f_{\lambda}(x) - \overline{f} +\nu\mu_{\nu}s(x)\right)}\drv x \leq \rho $ and implies
 \begin{equation} \label{eq:proof:iota:inbetween}
 \int_{\A} 2^{\mu_{\nu}s(x)} \drv x \leq \rho \, 2^{\tfrac{\Delta_{f}}{\nu}}.
 \end{equation}
From \eqref{eq:proof:iota:inbetween} two bounds can be derived. First, \eqref{eq:proof:iota:inbetween} implies that $\left(\rho \wedge \tfrac{\varepsilon}{L_{s}}\right) 2^{\mu_{\nu}(\overline{s}-\varepsilon)}\leq \rho 2^{\tfrac{\Delta_{f}}{\nu}}$, which by choosing $\varepsilon=\tfrac{\overline{s}}{2}$ leads to $2^{\mu_{\nu}\tfrac{\overline{s}}{2}}\leq \left( \tfrac{2L_{s}\rho}{\overline{s}}\vee 1 \right)2^{\tfrac{\Delta_{f}}{\nu}}$ and finally, 
\begin{equation} \label{eq:itoa:upper:bound}
\nu\mu_{\nu}\leq \tfrac{2}{\overline{s}}\log\left( \tfrac{2L_{s}\rho}{\overline{s}}\vee 1 \right)\nu + \tfrac{2\Delta_{f}}{\overline{s}}.
\end{equation}
Similarly one can derive a lower bound
\begin{equation} \label{eq:itoa:lower:bound}
\nu\mu_{\nu}\geq \tfrac{2}{\underline{s}}\log\left( \tfrac{2L_{s}\rho}{-\underline{s}}\vee 1 \right)\nu +  \tfrac{2\Delta_{f}}{\underline{s}}.
\end{equation}
Equation \eqref{eq:proof:itoa:almost:done} together with \eqref{eq:itoa:upper:bound} and \eqref{eq:itoa:lower:bound} complete the proof.
\end{proof}
We consider the smooth, finite dimensional, convex optimization problem
\begin{align}\label{Lagrange:Dual:Program:smooth:poisson}
 \mathsf{D}_{\nu}: \left\{ \begin{array}{ll}
	\underset{\lambda}{\inf} 		& F(\lambda) + G_{\nu}(\lambda) \\
			\text{s.t. } 					& \lambda\in Q,
	\end{array}\right.
\end{align}
whose solution can be approximated with Algorithm~\hyperlink{algo:1}{1} presented in Section~\ref{sec:classicalCapacity}, as follows. Define the constant $D_{1}:=\tfrac{1}{2}(M \log(\gamma^{-1})\vee\tfrac{1}{\ln 2})^2$.
\begin{Thm}\label{thm:error:bound:capacity:continuous:channel}
Under Assumptions~\ref{a:channel}\ref{ass:channel:ii}, \ref{ass:channel:Poisson} and \ref{a:constraint_fct:s}, let $\alpha := 2(T_1 + T_2 + 1)$ where $T_1$ and $T_2$ are as defined in Lemma~\ref{lem:iota:new}. Given $\varepsilon \in (0, \tfrac{\alpha}{4})$, we set the smoothing parameter $\nu = \tfrac{\varepsilon / \alpha}{\log \left( \alpha / \varepsilon \right)}$ and number of iterations $n\geq \tfrac{1}{\varepsilon} \sqrt{8 D_{1}\alpha}\sqrt{\log(\varepsilon^{-1}) + \log(\alpha) + \tfrac{1}{4}}$.  Consider 
\begin{align} 
\hat{\lambda} = y_{n} \in Q\qquad \text{and} \qquad \hat{p}=\sum_{k=0}^{n}\frac{2(i+1)}{(n+1)(n+2)} p_{\nu}^{x_{k}}\in \mathcal{D}(\A), \label{eq:optInPut:cts}
\end{align}
where $y_k$ computed at the $\text{k}^{\text{th}}$ iteration of Algorithm~\hyperlink{algo:1}{1} and $p_{\nu}^{x_{k}}$ is the analytical solution in \eqref{eq:finite:optimizer:pmu:cts}. Then, $\hat{\lambda}$ and $\hat{p}$ are the approximate solutions to the problems \eqref{eq:Dual:Program:Poisson} and \eqref{eq:inf:channel:primal}, i.e., 
\begin{align}
0\leq F(\hat{\lambda}) + G(\hat{\lambda}) - \I{\hat{p}}{W_{M}} \leq \varepsilon. \label{eq:EBB:cts}
\end{align}
Therefore, Algorithm~\hyperlink{algo:1}{1} requires $O\left( \tfrac{1}{\varepsilon}\sqrt{\log\left( \varepsilon^{-1} \right)} \right)$ iterations to find an $\varepsilon$-solution to the problems \eqref{eq:Dual:Program:Poisson} and \eqref{eq:inf:channel:primal}.
\end{Thm}
\begin{proof}
Following \cite{ref:Nest-05} and using Lemma~\ref{lem:compact:set:Poisson}, Lemma~\ref{lem:strong:duality:poisson}, Propostion~\ref{prop:Lipschitz:cts:channel} and Lemma~\ref{lem:iota:new}, after $n$ iterations of Algorithm~\hyperlink{algo:1}{1} the following approximation error is obtained
\begin{equation}\label{eq:proof:rate:error:term}
0\leq F(\hat{\lambda}) + G(\hat{\lambda}) - \I{\hat{p}}{W} \leq \iota(\nu) + \frac{4D_{1}}{\nu (n+1)^{2}}+\frac{4D_{1}}{(n+1)^{2}}=:\mathsf{err}(\nu,n),
\end{equation}
where for $\nu<\tfrac{T_{1}}{1-T_{2}}$ or $T_{2}>1$ we have $\iota(\nu)=\nu \left( \log\left( \tfrac{T_{1}}{\nu} + T_{2} \right) +1 \right)$, which is strictly increasing in $\nu$. Let us redefine the smoothing term by $\nu:=\tfrac{\delta}{\log\left( \delta^{-1} \right)}$ for $\delta \in (0,1)$ and define the function $g(\delta):=\left( \tfrac{\log\left( T_{1} \log\left( \delta^{-1} \right) + T_{2}\delta \right)+1}{\log\left( \delta^{-1} \right)} +1 \right)$. One can see that $\iota(\nu)=\delta g(\delta)$ and that $\lim_{\delta\to 0}g(\delta)=1$. Furthermore $\delta\leq 2^{-1} \wedge 2^{-\tfrac{1}{T_{1}+T_{2}}}$ implies 
\begin{equation} \label{eq:bound:g:delta}
g(\delta) - 1 \leq \frac{\log \left( 2(T_{1}+T_{2})\log\left( \delta^{-1} \right) \right) }{\log\left( \delta^{-1} \right)} \leq T_{1}+T_{2},
\end{equation}
where the first inequality is due to $\delta\leq 2^{-1}$ and the second follows from $\delta\leq 2^{-\tfrac{1}{T_{1}+T_{2}}}$. We seek for a lower bound of $n$ and upper bound $\delta$ such that the error term \eqref{eq:proof:rate:error:term} is smaller than the preassigned $\varepsilon>0$, i.e.,
\begin{equation} \label{eq:proof:leq:vareps}
\mathsf{err}(\tfrac{\delta}{\log\left( \delta^{-1} \right)},n) = g(\delta) \delta + \frac{4D_{1}}{(n+1)^{2}} \left( \frac{\log\left( \delta^{-1} \right)}{\delta} +1 \right) \leq \varepsilon
\end{equation}
To this end, we introduce an auxiliary variable $\zeta \in (0,1)$ such that such that $g(\delta) \delta = (1-\zeta)\varepsilon$ and $\tfrac{4D_{1}}{(n+1)^{2}} \left( \tfrac{\log\left( \delta^{-1} \right)}{\delta} +1 \right)\leq \zeta \varepsilon$, which implies \eqref{eq:proof:leq:vareps}. Observe that $g(\delta) \delta = (1-\zeta)\varepsilon$ is equivalent to $\delta = \tfrac{(1-\zeta)}{g(\delta)}\varepsilon=:\beta \varepsilon$. Hence $\zeta=1-\beta g(\delta)$ for $\beta \in [0,\tfrac{1}{g(\delta)}]$. Moreover,
\begin{align*}
\frac{4D_{1}}{(n+1)^{2}} \left( \frac{\log\left( \delta^{-1} \right)}{\delta} +1 \right) =  \frac{4D_{1}}{(n+1)^{2}} \left( \frac{\log\left( (\beta \varepsilon)^{-1} \right)}{\beta \varepsilon} +1 \right)\leq \zeta \varepsilon \nonumber
\end{align*}
is equivalent to
\begin{align}
 4 D_{1} \left(  \frac{\log\left( (\beta \varepsilon)^{-1} \right) + \beta \varepsilon}{\beta (1-g(\delta)\beta)\varepsilon^{2}}\right)
= 4 D_{1} \left( \frac{\log(\varepsilon^{-1}) + \log 2 g(\delta) + \tfrac{\varepsilon}{2 g(\delta)}}{\tfrac{\varepsilon^{2}}{4g(\delta)}} \right) \leq (n+1)^{2} \label{eq:proof:no:beta},
\end{align}
where we have chosen $\beta=\tfrac{1}{2g(\delta)}$ and as such is equivalent to
\begin{align}
\frac{4}{\varepsilon} \sqrt{D_{1}\left( g(\delta) \log\left( \varepsilon^{-1} \right) + g(\delta)\log\left( 2 g(\delta) \right) + \tfrac{\varepsilon}{2} \right)} \leq n+1 \nonumber
\end{align}
 Finally, using \eqref{eq:bound:g:delta} implies for $\nu = \tfrac{\varepsilon / \alpha}{\log \left( \alpha / \varepsilon \right)}$, where $\alpha := 2(T_1 + T_2 + 1)$
\begin{equation*}
\mathsf{err}(\nu,n) \leq \varepsilon \quad \text{ for } \quad n\geq \tfrac{1}{\varepsilon} \sqrt{8 D_{1}\alpha}\sqrt{\log(\varepsilon^{-1}) + \log(\alpha) + \tfrac{1}{4}}.
\end{equation*} 
\end{proof}
Hence, under Assumption~\ref{a:constraint_fct:s} we can quantify the approximation error of the presented method to find the capacity of any channel $W$, satisfying Assumptions~\ref{a:channel} and \ref{ass:channel:Poisson}, by
\begin{align*}
\left| C(W) -  C_{\text{approx}}^{(n)}(W_{M}) \right| 	&\leq   \underbrace{\left| C(W) -  C(W_{M}) \right|}_{(\star)} + \underbrace{\left| C(W_{M}) -  C_{\text{approx}}^{(n)}(W_{M}) \right|}_{(\star \star)},
\end{align*}
where $(\star)$ and $(\star \star)$ are addressed by Theorem~\ref{thm:C} and Theorem~\ref{thm:error:bound:capacity:continuous:channel}, respectively. Let us highlight that for the term $(\star \star)$ we have two different quantitative bounds: First, the \textit{a priori} bound $\varepsilon$ for which Theorem~\ref{thm:error:bound:capacity:continuous:channel} prescribes a lower bound for the required number of iterations; second, the \textit{a posteriori} bound $F(\hat{\lambda}) + G(\hat{\lambda}) - I (\hat p, W_{M})$
which can be computed after a number of iterations have been executed. In practice, the a posteriori bound often approaches $\varepsilon$ much faster than the a priori bound. Note also that 
by \eqref{eq:uniform:bound:cts} and Theorem~\ref{thm:error:bound:capacity:continuous:channel}
\begin{align*}		
0\leq F(\hat{\lambda}) + G_{\nu}(\hat{\lambda}) +\iota(\nu) - \I{\hat{p}}{W_{M}} \leq \iota(\nu) + \varepsilon,
\end{align*}
which shows that $ F(\hat{\lambda}) + G_{\nu}(\hat{\lambda}) +\iota(\nu)$ is an upper bound for the channel capacity with a priori error $\iota(\nu) + \varepsilon$. This bound can be particularly helpful in cases where an evaluation of $G(\lambda)$ for a given $\lambda$ is hard.
\begin{Rem}[Optimal tail truncation] \label{rmk:optimal:truncation}
Given a fixed number of iterations, the term $(\star\star)$ above is effected by the truncation level $M$ for two reasons: the higher $M$ the larger the size of the output as well as the lower the parameter $\gamma_{M}$. Therefore, term $(\star\star)$ increases as M increases, which can be quantified by \eqref{eq:proof:rate:error:term}. On the other hand, term $(\star)$ obviously has the opposite behavior. Namely, the higher M leads to the better approximation of the channel W by the truncated version $W_M$ as quantified in Theorem~\ref{thm:C}. Hence, given a channel $W$ with the polynomial tail order $k$, there is an optimal value for the truncation parameter $M$, which thanks to the monotonicity explained above can be effectively computed in practice by techniques such as bisection.   \\
Note, that this truncation procedure could also be applied to a finite output alphabet, given that the channel satisfies Assumption~\ref{a:channel}\eqref{ass:channel:i}, and for example improve the performance of the method presented in Section~\ref{sec:classicalCapacity}.
\end{Rem}

\begin{Rem}[Without average-power constraint] \
In case of considering only a peak-power constraint and no average-power constraint, our proposed methodology allows us to access a closed form expression for $G_{\nu}(\lambda)$ and its gradient, 
\begin{align} 
G_{\nu}(\lambda) 	&= 	\nu \log\left( \int_{\A} 2^{\frac{1}{\nu}\left( \WW\lambda(x) -  r(x) \right)} \drv x \right) - \nu \log(\rho) \label{eq:poiss:closed:form:Gnu} \\
\nabla G_{\nu}(\lambda) 	&= 	\frac{ \int_{\A} 2^{\frac{1}{\nu}\left( \WW\lambda(x) -  r(x) \right)}W_{M}(\cdot|x) \drv x}{\int_{\A} 2^{\frac{1}{\nu}\left( \WW\lambda(x) -  r(x) \right)} \drv x}. \nonumber
\end{align}
\end{Rem}

 \subsection*{Discrete-time Poisson channel} 
 The discrete-time Poisson channel is a mapping from $\Rp$ to $\mathbb{N}_0$, such that conditioned on the input $x\geq 0$ the output is Poisson distributed with mean $x + \eta$, i.e.,
\begin{equation}
W(y|x)=e^{-(x+\eta)}\frac{(x+\eta)^y}{y!}, \quad y\in \mathbb{N}_0,\, x\in \Rp,\label{eq:WpoissonIntro}
\end{equation}
where $\eta\geq0$ denotes a constant sometimes referred to as \emph{dark current}. A peak-power constraint on the transmitter is given by the peak-input constraint $X\leq A$ with probability one, i.e., $\A=[0,A]$ and an average-power constraint on the transmitter is considered by $\E{X}\leq S$.

Up to now, no analytic expression for the capacity of a discrete-time Poisson channel is known. However, for different scenarios lower and upper bounds exist. Brady and Verd\'u derived a lower and upper bound in the presence of only an average-power constraint \cite{brady90}. Later, for $\eta=0$ and only an average-power constraint, Martinez introduced better upper and lower bounds \cite{martinez07}.
Lapidoth and Moser derived a lower bound and an asymptotic upper bound, which is valid only when the available peak and average power tend to infinity with their ratio held fixed, for the presence of a peak and average-power constraint \cite{lapidoth09}. Lapidoth \emph{et al.} computed the asymptotic capacity of the discrete-time Poisson channel when the allowed average-input power tends to zero with the allowed peak power --- if finite --- held fixed and the dark current is constant or tends to zero proportionally to the average power \cite{lapidoth11}. 

In \cite{ref:Chen-14-1} a numerical algorithm is presented, where the Blahut-Ariomoto algorithm is incorporated into the deterministic annealing method, that allows the computation of both the channel capacity under peak and average power constraints and its associated optimal input distribution. Furthermore, the works \cite{ref:Chen-14-1, ref:Chen-14-2} derive several fundamental properties of capacity achieving input distributions for the discrete-time Poisson channel. 

Here, we numerically approximate the capacity of a discrete-time Poisson channel using the proposed algorithm. For simplicity, we consider the case where only a peak power constraint is imposed; the case where an additional average power constraint is present can be treated similarly. It was shown in \cite{shamai90} that in the case of a peak power constraint (with or without average power constraint), the capacity achieving input distribution is discrete. This, in the limit as the number of iterations in the proposed approximation method goes to infinity, is consistent with the optimal input distribution given in Remark~\ref{rmk:stabilization:optimizer:cts}.

The following proposition provides an upper bound for the $k$-polynomial tail for the Poisson channel $W$ as defined in \eqref{eq:WpoissonIntro}.
	
	\begin{Prop}[Poisson tail]
	\label{prop:poisson_tail}
		The Poisson channel \eqref{eq:WpoissonIntro} having a bounded input alphabet $\X = [0,A]$ and \emph{dark current} parameter $\eta$ has a $k$-polynomial tail for any $k \in (0,1]$ in the sense of Definition \ref{def:tail}, which is upper bounded for all $M \ge A + \eta$ by
		\begin{align*}
			R_k(M) \le \Big( {\alpha\e^{(\alpha-1)(A+\eta)}\frac{(A+\eta)^{M}}{M!}} \Big)^k, \qquad \alpha \Let 2^{(k^{-1}-1)}.
		\end{align*}
	\end{Prop}
To prove Proposition \ref{prop:poisson_tail}, we need two lemmas. 
	
	\begin{Lem}\label{lem:ab}
		For any $k\in(0,1]$ and $a, b \ge 0$
		\begin{align*}
			a^k+b^k \le 2^{1-k}(a+b)^k.
		\end{align*}
	\end{Lem}
	 
	 \begin{proof}
		 Let $g(x) \Let 2^{1-k}(1+x)^k - x^k$. By setting $\frac{\diff}{\diff x}g(x^\star) = 0$, one can easily see that $x^\star = 1$ is the minimizer of function $g$ over the interval $[0,1]$, i.e., $g(x) \ge g(1) = 1$ for all $x \in [0,1]$. Suppose, without loss of generality, that $a \ge b$. By virtue of the preceding result of function $g$, we know that 
		 	$$ 1 \le g\left( \frac{b}{a} \right) =  2^{1-k}\left( 1+\frac{b}{a} \right)^k - \left(\frac{b}{a}\right)^k, $$
		 where by multiplying $a^k$ it readily leads to the desired assertion. 
	 \end{proof}
	 
	\begin{Lem}\label{lem:ai}
		Let $(a_i)_{i \in \mathbb{N}}$ be a non-negative sequence of real numbers. For any $k \in (0,1]$
		\begin{align*}
			\sum\limits_{i \in \mathbb{N}} a_i^k \le \Big(\sum\limits_{i \in \mathbb{N}} \alpha^{i}a_i\Big)^k, \qquad \alpha \Let 2^{(k^{-1}-1)}.
		\end{align*}
	\end{Lem}
	\begin{proof}
		For the proof we make use of an induction argument. Note that for any $a_1 \ge 0$ it trivially holds that $a_1^k \le 2^{1 - k}a_1^k$. We now assume that for any sequence $(a_i)_{i=1}^{N} \subset \R_{\ge 0}$ we have
		\begin{align}
		\label{eq:0}
			\sum\limits_{i=1}^{N} a_i^k \le \Big(\sum\limits_{i=1}^{N} 2^{(k^{-1}-1)i}a_i\Big)^k. 
		\end{align}
		Let $(a_i)_{i=1}^{N+1} \subset \R_{\ge 0}$. Then, 
		\begin{align}
			\label{eq:1} \sum\limits_{i=1}^{N+1} a_i^k & = a_1^k + \sum\limits_{i=2}^{N+1} a_i^k  \le a_1^k + \Big(\sum\limits_{i=2}^{N+1} 2^{(k^{-1}-1){(i-1)}}a_i\Big)^k \le 2^{1-k} \Big(a_1 + \sum\limits_{i=2}^{N+1} 2^{(k^{-1}-1){(i-1)}}a_i\Big)^k \\
			& \notag = \Big(2^{(k^{-1}-1)}a_1 + \sum\limits_{i=2}^{N+1} 2^{(k^{-1}-1){i}}a_i\Big)^k = \Big(\sum\limits_{i=1}^{N+1} 2^{(k^{-1}-1){i}}a_i\Big)^k,
		\end{align}
		where the first (resp.\ second) inequality in \eqref{eq:1} follows from \eqref{eq:0} (resp.\ Lemma \ref{lem:ab}). 
	\end{proof}

	\begin{proof}[Proof of Proposition \ref{prop:poisson_tail}]
		It is straightforward to see that
		\begin{align}
		\label{A,n}
			\max_{x \in [0,A]} \e^{-x} x^i = \e^{-\min\{A,i\}} \big(\min\{A,i\}\big)^{i}. 
		\end{align} 
		Moreover, based on a Taylor series expansion, it is well known that for all $M \in \mathbb{N}$ and $x \in \R_{\geq 0}$
		\begin{align}
		\label{taylor}
			\sum\limits_{i \ge M} \frac{x^i}{i!} \le \frac{\e^x}{M!} x^{M}. 
		\end{align}
		Therefore, it follows that
		\begin{subequations}
		\begin{align}
			\label{eq:3} R_k(M) & \Let \sum\limits_{i\ge M} \Big(\sup_{x \in [0,A]} \e^{-(x + \eta)} \frac{(x+\eta)^{i}}{i!} \Big)^k \le \sum\limits_{i\ge M} \Big( \e^{-(A + \eta)} \frac{(A+\eta)^{i}}{i!} \Big)^k \\
			\label{eq:4} & \le \e^{-k(A + \eta)}\Big( \sum\limits_{i \ge M} \alpha^{(i-M+1)} \frac{(A+\eta)^{i}}{i!} \Big)^k = \frac{\e^{-k(A + \eta)}}{\alpha^{k(M-1)}} \Big( \sum\limits_{i \ge M}  \frac{\big(\alpha(A+\eta)\big)^{i}}{i!} \Big)^k \\
			\label{eq:5} & \le \frac{\e^{-k(A + \eta)}}{\alpha^{k(M-1)}} \Big( \frac{\e^{\alpha(A+\eta)}}{M!} \alpha^{M} (A+\eta)^{M} \Big)^k = \Big( {\alpha\e^{(\alpha-1)(A+\eta)}\frac{(A+\eta)^{M}}{M!}} \Big)^k,
		\end{align}
		\end{subequations} 
		where \eqref{eq:3} results from \eqref{A,n} and the assumption $M \ge A + \eta$, and \eqref{eq:4} (resp.\ \eqref{eq:5}) follows from Lemma \ref{lem:ai} (resp.\ \eqref{taylor}). 
	\end{proof}

In the following we present an example to illustrate the theoretical results developed in the preceding sections and their performance. Note that for the discrete-time Poisson channel Assumption~\ref{ass:channel:Poisson} clearly holds.
\begin{Ex} \label{ex:poisson}
We consider a discrete-time Poisson channel $W$ as defined in \eqref{eq:WpoissonIntro} with a peak-power constraint $A$ and dark current $\eta = 1$. Up to now, the best known lower bound for the capacity is given by \cite[Theorem~4]{lapidoth09}
\begin{equation}
C(W) \geq \frac{1}{\ln 2}\left(\frac{1}{2} \ln A  + \left(\frac{A}{3}+1 \right)\ln \left(1 + \frac{3}{A} \right)-1 -\sqrt{\frac{\eta+\tfrac{1}{12}}{A}}\left(\frac{\pi}{4} + \frac{1}{2}\ln 2 \right) - \frac{1}{2}\ln \frac{\pi e}{2}\right)\!. \label{eq:MoserLB}
\end{equation}
To the best of our knowledge no upper bound for the capacity is known. In \cite{lapidoth09} an asymptotic upper bound is given which includes an unknown error term that is vanishing in the limit $A \to \infty$.
According to Theorems~\ref{thm:C} and \ref{thm:error:bound:capacity:continuous:channel}, the algorithm introduced in this article leads to an approximation error after $n$ iterations that is given by
\begin{align*}
\left|C_{\text{approx}}^{(n)}(W_{M})-C(W) \right| &\leq \left|C_{\text{approx}}^{(n)}(W_{M})-C(W_M) \right| + \left|C(W_M)-C(W) \right| \nonumber\\
&\leq F(\hat{\lambda}) + G(\hat{\lambda}) - \I{\hat{p}}{W} + \mathcal{E} , \label{eq:error:poisson:example}
\end{align*}
where $\mathcal{E} = \frac{2\log(\e)}{\e(1-k)} \Big[ M^{1-k}\big(R_1(M)\big)^k + R_k(M) \Big]$, $R_\ell(M) = \Big( {\alpha\e^{(\alpha-1)(A+\eta)}\frac{(A+\eta)^{M}}{M!}} \Big)^\ell$ and $\alpha \Let 2^{(\ell^{-1}-1)}$ for any $k\in(0,1)$ and $\ell\in(0,1]$. The truncation parameter $M$ was determined as described in Remark~\ref{rmk:optimal:truncation}. This finally leads to the following upper and lower bounds on $C(W)$
\begin{equation} \label{eq:upper:lower:Cp}
	2 \I{\hat{p}}{W}  -  \left( F(\hat{\lambda}) + G(\hat{\lambda}) \right) - \mathcal{E} \leq C(W) \leq 2  \left( F(\hat{\lambda}) + G(\hat{\lambda})  \right) - \I{\hat{p}}{W} +  \mathcal{E}. 
\end{equation}
Figure~\ref{fig:poissonPlott} compares the two bounds \eqref{eq:MoserLB} and \eqref{eq:upper:lower:Cp} for different values of $A$. Further details on the simulation can be found in Appendix~\ref{app:simulation:details}.

\begin{figure}[!htb]
	\centering
	\scalebox{0.9}{\input{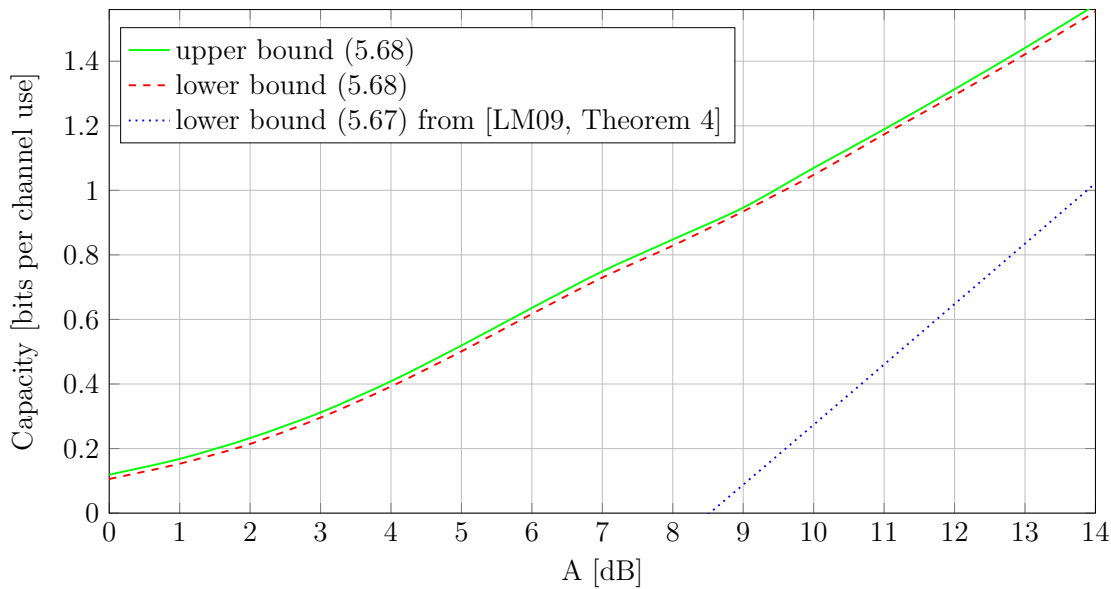}}
	\caption{This plot depicts the capacity of a discrete-time Poisson channel with dark current $\eta=1$ as a function of the peak-power constraint parameter $A$. The red (resp.\ green) line shows the lower (resp.\ upper) bound \eqref{eq:upper:lower:Cp} obtained for a moderate 	number of iterations, see Appendix~\ref{app:simulation:details}. As a comparison we plot the lower bound of \cite{lapidoth09}, which to the best of our knowledge is the tightest
	lower bound available to date (blue line). The parameter $A$ is given in decibels where $A[\textnormal{dB}]=10 \log_{10}(A)$.}	\label{fig:poissonPlott}
\end{figure}
\end{Ex}

\begin{Rem}[AWGN channel with a quantized output]
Another example of a channel that is well studied and can be treated by the proposed method is the discrete-time additive white Gaussian noise (AWGN) channel under output quantization. The output of the channel is described by
\begin{equation*}
Y = \mathsf{Q}(X+N),
\end{equation*}
where $X \in \R$ is the channel input, $N \sim \mathcal{N}(0,\sigma^2)$ for $\sigma^2 > 0$ is white Gaussian noise and $\mathsf{Q}(\cdot)$ is a quantizer that maps the real valued input $X + N$ to one of $M$ bins (where we assume $M < \infty$), which gives $Y \in \{y_1,\ldots,y_MÊ\}$. In addition an average and/or a peak power constraint at the input is considered. More information about this channel model and why it is of interest can be found in \cite{singh09,koch13}. By definition, the AWGN channel with a quantized output has a continuous input alphabet and a discrete output alphabet. Thus, the approximation method discussed in this section can be used to compute the capacity of such channels.  
\end{Rem}

\section{Quantum channel capacities}

Depending on the channel and allowed auxiliary resources, there is a variety of capacities for different communication tasks. An excellent overview can be found in \cite{wilde_book,holevo_book}. For a lot of these tasks, their corresponding capacity can be recast as an optimization problem. Some of them seem to be intrinsically more difficult than others and in general no closed form solution is available.

We consider the task of sending classical information over a classical-quantum (cq) channel which maps each element of a classical input alphabet to a finite-dimensional quantum state. We do not allow any additional resources such as entanglement shared between the sender and receiver nor feedback. The capacity for this task has been shown in \cite{holevo98,holevo_book,schumacher97} to be the maximization of a quantity called the \emph{Holevo information} over all possible input distributions. Unlike the classical channels where a specific efficient method --- the \emph{Blahut-Arimoto algorithm} \cite{blahut72,arimoto72} --- is known for numerical computation of the capacity with a provable rate of convergence, there is no counterpart for cq channels to date.

We have shown \cite{ref:D:Sutter-16} that the techniques developed in this chapter can be extended to numerically compute the capacity of such cq channels.



\section{Appendix}

\subsection{Simulation details}\label{app:simulation:details}	
This section provides some further details on the simulation in Example~\ref{ex:poisson}. The parameters considered are $k=\tfrac{1}{2}$, $L_{f}=0$ and $M$ is chosen according to Table~\ref{tab:poisson:details}. All the simulations in this section are performed on a 2.3 GHz Intel Core i7 processor with 8 GB RAM with Matlab.

\begin{table}[!htb]
\footnotesize{
\centering 
\caption{Simulation details to Example~\ref{ex:poisson}}
\label{tab:poisson:details}
\vspace{1mm}
  \begin{tabular}{c | c c c c c c c c}
 $A$ [dB] \hspace{1mm}  & \hspace{1mm}    0  &1 & 2 & 3 & 4 & 5 & 6 & 7 \\ 
 $M$ \hspace{1mm}  & \hspace{1mm}    16  &17 & 19 & 20 & 22 & 25 & 28 & 31 \\ 
 Iterations $n$ & \hspace{1mm} 4$\cdot 10^{4}$ & 4$\cdot 10^{4}$   & 4$\cdot 10^{4}$  & 5$\cdot 10^{4}$  & 6$\cdot 10^{4}$ & 7$\cdot 10^{4}$ & 9$\cdot 10^{4}$ &1.2$\cdot 10^{5}$ \\
 $\nu$ \hspace{1mm}  & \hspace{1mm} 0.0026 & 0.0029 & 0.0036 & 0.0029 & 0.0027 & 0.0029 & 0.0026 &  0.0022 \\ 
 $F(\hat{\lambda}) + G(\hat{\lambda})$ & \hspace{1mm} 0.1144 & 0.1626  & 0.2263 & 0.3063 & 0.4029 & 0.5129 & 0.6293 & 0.7423 \\
 $\I{\hat{p}}{W}$ & \hspace{1mm} 0.1105 & 0.1583  & 0.2206  & 0.3015  & 0.3979 & 0.5072 & 0.6234 &  0.7365 \\ 
 $\mathcal{E}$ & \hspace{1mm} 9.3$\cdot 10^{-4}$ &9.7$\cdot 10^{-4}$  & 4.8$\cdot 10^{-4}$  & 8.5$\cdot 10^{-4}$  & 8.2$\cdot 10^{-4}$  & 4.9 $\cdot 10^{-4}$ & 5.0$\cdot 10^{-4}$ & 9.5$\cdot 10^{-4}$ \\ 
  \end{tabular}
  
  \vspace{7mm}
  \phantom{bla}
  \centering 
    \begin{tabular}{c | c c c c c c c}
 $A$ [dB] \hspace{1mm}  & \hspace{1mm}  8  & 9   & 10  & 11  & 12 & 13 & 14   \\ 
 $M$ \hspace{1mm}  & \hspace{1mm}   36   & 42  & 49  & 59 &  71 & 85 & 104    \\ 
 Iterations $n$ & \hspace{1mm}   2$\cdot 10^{5}$ & 5$\cdot 10^{5}$   &2$\cdot 10^{6}$   & 3$\cdot 10^{6}$ & 4$\cdot 10^{6}$ & 9$\cdot 10^{6}$ & 1.5$\cdot 10^{7}$\\
 $\nu$ \hspace{1mm}  & \hspace{1mm} 0.0016 & 7.1 $\cdot 10^{-5}$ & 8.0$\cdot 10^{-4}$ &8.3$\cdot 10^{-4}$& 9.7$\cdot 10^{-4}$ &  6.2$\cdot 10^{-4}$ & 5.8$\cdot 10^{-4}$ \\
 $F(\hat{\lambda}) + G(\hat{\lambda})$ & \hspace{1mm}   0.8410 & 0.9422  & 1.0591  & 1.1835  & 1.3070 & 1.4343 & 1.5671  \\
 $\I{\hat{p}}{W}$ & \hspace{1mm} 0.8351 & 0.9388  & 1.0547  & 1.1788 & 1.3013 &1.4219 & 1.5605 \\ 
 $\mathcal{E}$ & \hspace{1mm}    7.5$\cdot 10^{-4}$ &  7.1$\cdot 10^{-4}$ & 8.0$\cdot 10^{-4}$  & 6.2$\cdot 10^{-4}$ & 5.2 $\cdot 10^{-4}$ & 9.0$\cdot 10^{-4}$ & 6.7$\cdot 10^{-4}$ \\ 
  \end{tabular} }
  
\end{table}	





\chapter{Generalized maximum entropy estimation} \label{chap:entropy:max}

In this chapter, we consider the problem of estimating a probability distribution that maximizes the entropy while satisfying a finite number of moment constraints, possibly corrupted by noise. Based on duality of convex programming, we present a novel approximation scheme using a smoothed fast gradient method that is equipped with explicit bounds on the approximation error. We further demonstrate how the presented scheme can be used for approximating the chemical master equation through the zero-information moment closure method.

The entropy maximization problem subject to moment constraints is a key tool in Chapter~\ref{chap:channel:cap} (see Equation~\eqref{opt:cover:cont} and Lemma~\ref{lem:cover:cont}). Moreover the problem also appears in Part~\ref{part:ADP} of this thesis (see Lemma~\ref{lem:entropy}), where so-called MaxEnt distributions act as a regularizer, leading to computationally more efficient optimization programs.

\section{Introduction}

Consider a one-dimensional moment problem formulated as follows: Given a set $\K\subset \R$ and a sequence $(y_i)_{i\in\N}\subset\R$ of moments, does there exist a measure $\mu$ supported on $\K$ such that
\begin{equation} \label{eq:moment:intro}
y_i = \int_\K x^i \mu(\drv x) \quad \text{for all } i\in\N \ ?
\end{equation}
For $\K=\R$ and $\K=[a,b]$ with $-\infty<a<b<\infty$ the above moment problem is known as the \emph{Hamburger moment problem} and \emph{Hausdorff moment problem}, respectively. If the moment sequence is finite, the problem is called a \emph{truncated moment problem}. In both full and truncated cases, a measure $\mu$ that satisfies \eqref{eq:moment:intro}, is called a \emph{representing measure} of the sequence $(y_i)_{i\in\N}$. If a representing measure is unique, it is said to be \emph{determined by its moments}.
From the Stone-Weierstrass theorem it followes directly that every non-truncated representing measure with compact support is determined by its moments.
In the Hamburger moment problem, given a representing measure $\mu$ for a moment sequence $(y_i)_{i\in\N}$, a sufficient condition for $\mu$ being determined by its moments is the so-called \emph{Carleman condition}, i.e.,  $\sum_{i=1}^{\infty}y_{2i}^{-\nicefrac{1}{2i}}=\infty$. Roughly speaking this says that the moments should not grow too fast, see~\cite{ref:Akhiezer-65} for further details. For the Hamburger and the Hausdorff moment problem, there are necessary and sufficient conditions for the existence of a representing measure for a given moment sequence $(y_i)_{i\in\N}$ in both the full as well as the truncated setting, see~\cite[Theorems~3.2, 3.3, 3.4]{ref:Lasserre-11}.

In this article, we investigate the problem of estimating an unknown probability distribution given a finite number of (possibly noisy) observed moments. Given that the observed moments are consistent (i.e., there exists a probability distribution that satisfies all the moment contstraints), the problem is underdetermined and has infinitely many solutions. This therefore raises the question of which solution to choose. A natural choice would be to pick the one with the highest entropy, called the \emph{MaxEnt distribution}. The main reason why the MaxEnt distribution is a natural choice is due to a concentration phenomenon described by Jaynes~\cite{ref:Jaynes-03}:

\vspace{2mm}
\begin{minipage}{35em}
``\emph{If the information incorporated into the maximum-entropy analysis includes all the constraints actually operating in the random experiment, then the distribution predicted by maximum entropy is overwhelmingly the most likely to be observed experimentally.}'' 
\end{minipage}
\vspace{2mm}

\noindent See~\cite{ref:Jaynes-03, ref:Grunwald-08} for a rigorous statement. This maximum entropy estimation problem, subject to moment constraints, also known as the \emph{principle of maximum entropy}, is applicable to large classes of problems in natural and social sciences --- in particular in economics, see~\cite{ref:Golan-08} for a comprehensive survey.
Also it has important applications in approximation methods to dynamical objects, such as in systems biology, where MaxEnt distributions are key objects in the so-called \emph{moment closure method} to approximate the chemical master equation~\cite{ref:Smadbeck-13}, or more recently in the context of approximating dynamic programming (see Part~\ref{part:ADP}) where MaxEnt distributions act as a regularizer, leading to computationally more efficient optimization programs.

 Their operational significance motivates the study of numerical methods to compute MaxEnt distributions, which are the solutions of an infinite-dimensional convex optimization problem and as such computationally intractable in general. Since it was shown that the MaxEnt distribution maximizing the entropy subject to a finite number of moment constraints (if it exists) belongs to the exponential family of distributions~\cite{ref:Csiszar-75}, its computation can be reduced to solving a system of nonlinear equations, whose dimension is equal to the number of moment constraints \cite{ref:Mead-84}. Furthermore, the system of nonlinear equations involves evaluating integrals over $\K$ that are computationally difficult in general. Even if the support set $\K$ is finite, finding the MaxEnt distribution is not straightforward, since solving a system of nonlinear equations can be computationally demanding. 

 In \cite[Section~12.3]{ref:Lasserre-11} it is shown that the maximum entropy subject to moment constraints can be approximated by using duality of convex programming. The problem can be reduced to an unconstrained finite-dimensional convex optimization problem and an approximation hierarchy of its gradient and Hessian in terms of two single semidefinite programs involving two linear matrix inequalities is presented. The desired accuracy is controlled by the size of the linear matrix inequalities constraints. The method seems to be powerful in practice, however a rate of convergence
has not been proven. Furthermore, it is not clear how the method extends to the case of uncertain moment constraints. 
In a finite dimensional setting, \cite{ref:Dudik-07} presents a treatment of the maximum entropy principle with generalized regularization measures, that as a special case contain the setting presented here. However, convergence rates of algorithms presented are not known and again it is not clear how the method extends to the case of uncertain moment constraints.

In this article, we present a new approximation scheme to minimize the relative entropy subject to noisy moment constraints. This is a generalization of the introduced maximum entropy problem and extends the principle of maximum entropy to the so-called \emph{principle of minimum discriminating information}~\cite{ref:Kullback-59}. We show that its dual problem exhibits a particular favourable structure that allows us to apply Nesterov's smoothing method~\cite{ref:Nest-05} and hence tackle the presented problem using a fast gradient method
obtaining process convergence properties, unlike \cite{ref:Lasserre-11}.

In many applications, it is important to efficiently compute the MaxEnt distribution. For example, the zero-information moment closure method~\cite{ref:Smadbeck-13}, as well as a recently developed method to approximate the channel capacity of a large class of memoryless channels~\cite{TobiasSutter15} deal with iterative algorithms that require the numerical computation of the MaxEnt distribution in each iteration step.

\vspace{2mm}
\prlsection{Structure} The layout of this paper is as follows: In Section~\ref{sec:problem:statement} we formally introduce the problem setting. Our results on an approximation scheme in a continuous setting are reported in Section~\ref{sec:rate:distortion}. In Section~\ref{sec:finite:dim:case}, we show how these results simplify in the finite-dimensional case. Section~\ref{sec:gradient:evalutaion} discusses the gradient approximation that is the dominant step of the proposed approximation method from a computational perspective. The theoretical results are applied in Section~\ref{sec:moment:closure} to the zero-information moment closure method.

\vspace{2mm}
\prlsection{Notation}
The logarithm with basis 2 and $\mathrm{e}$ is denoted by $\log(\cdot)$ and $\ln(\cdot)$, respectively. We define the standard $n-$simplex as $\Delta_{n}:=\{  x\in\R^{n} : x\geq 0, \sum_{i=1}^{n} x_{i}=1\}$. For a probability mass function $p \in \Delta_{n}$ we denote its entropy by $H(p):=\sum_{i=1}^n -p_i \log p_i$. 
Let $B(y,r):=\{x\in\R^n \ : \ \|x-y\|_2 \leq r \}$ denote the ball with radius $r$ centered at $y$. Throughout this article, measurability always refers to Borel measurability.
For a probability density $p$ supported on a measurable set $B\subset \R$ we denote the differential entropy by $h(p):=-\int_{B} p(x) \log p(x) \drv x$. For $\A\subset\R$ and $1\leq p \leq \infty$, let $\Lp{p}(\A)$ denote the space of $\Lp{p}$-functions on the measure space $(\A, \Borelsigalg{\A}\!, \drv x)$, where $\Borelsigalg{\A}$ denotes the Borel $\sigma$-algebra and $\drv x$ the Lebesgue measure.  
Let $\XX$ be a compact metric space, equipped with its  Borel $\sigma$-field $\mathcal{B}(\cdot)$. The space of all probability measures on $(\XX, \mathcal{B}(\XX))$ will be denoted by $\mathcal{P}(\XX)$. The \emph{relative entropy} (or
Kullback-Leibler divergence) between any two probability measures $\mu, \nu \in \mathcal{P}(\XX)$ is defined by
\begin{equation*}
\KL{\mu}{\nu} := \left\{ \begin{array}{ll}
\int_{\XX} \log\left( \frac{\drv\mu}{\drv \nu} \right) \drv\mu, &\text{if } \mu \ll \nu\\
+\infty, & \text{otherwise} \, ,
\end{array} \right.
\end{equation*}
where $\ll$ denotes absolute continuity of measures, and $\tfrac{\drv\mu}{\drv \nu}$ is the Radon-Nikodym derivative. The relative entropy is non-negative, and is equal to zero if and only if $\mu\equiv \nu$.
Let $\XX$ be restricted to a compact metric space and let us consider the pair of vector spaces
$(\M(\XX),\mathbb{B}(\XX))$ where $\M(\XX)$ denotes the space of finite signed measures on $\mathcal{B}(\XX)$ and $\mathbb{B}(\XX)$ is the Banach space of bounded measurable functions on $\XX$ with respect to the sup-norm and consider the bilinear form
\begin{align*}
\inprod{\mu}{f}:=\int_{\XX}  f(x)\mu(\drv x).
\end{align*}
 This induces the total variation norm as the dual norm on $\M(\XX)$, since by~\cite[p.2]{ref:Hernandez-99}
\begin{equation*}
\| \mu \|_* = \sup_{\|f\|_\infty \leq 1}\inprod{\mu}{f} = \|\mu\|_{\mathsf{TV}},
\end{equation*}
making $\M(\XX)$ a Banach space.
In the light of~\cite[p.~206]{ref:Hernandez-99} this is a dual pair of Banach spaces; we refer to~\cite[Section~3]{ref:Anderson-87} for the details of the definition of dual pairs.

\section{Problem statement}  \label{sec:problem:statement}

Let $\K \subset \R$ be compact and consider the scenario where a probability measure $\mu \in \mathcal{P}(\K)$ is unknown and only observed via the following measurement model
\begin{equation} \label{eq:measurement:model}
y_{i} = \inprod{\mu}{x^{i}} + u_{i}, \quad u_{i} \in \mathcal{U}_{i} \quad \text{for }i=1,\hdots,M \, ,
\end{equation}
where $u_{i}$ represents the uncertainty of the obtained data point $y_{i}$ and $\mathcal{U}_{i}\subset\R$ is compact, convex and $0\in\mathcal{U}_{i}$ for all $i=1,\hdots,M$.
Given the data $(y_{i})_{i=1}^M\subset \R$, the goal is to estimate a probability measure $\mu$ that is consistent with the measurement model \eqref{eq:measurement:model}. This problem (given that $M$ is finite) is underdetermined and has infinitely many solutions. Among all possible solutions for \eqref{eq:measurement:model}, we aim to find the solution that maximizes the entropy.
Define the set $T:=\times_{i=1}^{M} \{y_{i}-u  :  u\in \mathcal{U}_{i} \}\subset \R^{M}$ and the linear operator $\mathcal{A}:\MM(\K)\to \R^{M}$ by 
\begin{align*}
(\mathcal{A}\mu)_{i}:=\inprod{\mu}{x^{i}} = \int_{\K}x^{i}\mu(\drv x) \quad \text{for all} \quad i=1,\hdots, M\, .
\end{align*}
The operator norm is defined as $\norm{\mathcal{A}}:=\sup_{\norm{\mu}_{\mathsf{TV}}=1, \norm{y}_2=1} \inprod{\mathcal{A}\mu}{y}$. Note that due to the compactness of $\K$ the operator norm is bounded, see Lemma~\ref{lem:lipschitz:cts:gradient} for a formal statement.
The adjoint operator to $\mathcal{A}$ is given by $\mathcal{A}^{*}:\R^{M}\to\mathbb{B}(\K)$, where $\mathcal{A}^{*}z (x):=\sum_{i=1}^{M}z_{i}x^{i}$; note that the domain and image spaces of the adjoint operator are well defined as $(\mathbb{B}(\K),\MM(\K))$ is a topological dual pairs and the operator $\mathcal{A}$ is bounded \cite[Proposition 12.2.5]{ref:Hernandez-99}. 

Given a reference measure $\nu \in \mathcal{P}(\K)$, the problem of minimizing the relative entropy subject to moment constraints \eqref{eq:measurement:model} can be formally described by
 \begin{align} \label{eq:main:problem}
 	       \quad J^{\star}= \min\limits_{\mu\in\mathcal{P}(\K)} \left\{ \KL{\mu}{\nu} \ : \ \mathcal{A}\mu \in T \right\}.
 	\end{align}
	\begin{Prop}[Existence \& uniqueness of \eqref{eq:main:problem}] \label{prop:solvability}
	The optimization problem~\eqref{eq:main:problem} attains an optimal feasible solution that is unique.
	\end{Prop}
\begin{proof}
The variational representation of the relative entropy~\cite[Corollary~4.15]{ref:Bouch-13} states that the mapping $\mu \mapsto \KL{\mu}{\nu}$ is the Fenchel-Legendre dual of the mapping $X\mapsto \log \mathds{E}e^{X}$, where $X$ is a random variable with law $\nu$. As a basic property of the Legendre-Fenchel transform, the mapping $\mu \mapsto \KL{\mu}{\nu}$ therefore is lower-semicontinuous~\cite{ref:Luenberger-69}. Note also that the space of probability measures on $\K$ is compact~\cite[Theorem~15.11]{ref:aliprantis-07}. Moreover, since the linear operator $\mathcal{A}$ is bounded, it is continuous. As a result, the feasible set of problem \eqref{eq:main:problem} is compact and hence the optimization problem attains an optimal solution. Finally, the strict convexity of the relative entropy~\cite{ref:Csiszar-75} ensures uniqueness of the optimizer.
\end{proof}

Note that if $\mathcal{U}_i=\{0\}$ for all $i=1,\ldots,M$, i.e., there is no uncertainty in the measurement model~\eqref{eq:measurement:model}, Proposition~\ref{prop:solvability} reduces to a known result~\cite{ref:Csiszar-75}. Consider the special case where the reference measure $\nu$ is the uniform measure on $\K$ and let $p$ denote the Radon-Nikodym derivative $\tfrac{\drv \mu}{\drv \nu}$ (whose existence can be assumed without loss of generality). Since $\mathcal{A}$ is weakly continuous and the differential entropy is known to be weakly lower semi-continuous \cite{ref:Bouch-13}, we can restrict attention to a (weakly) dense subset of the feasible set and hence assume without los of generality that $p\in\Lp{1}(\K)$.
Problem \eqref{eq:main:problem} then reduces to
 \begin{align} \label{eq:main:problem:special case}
 				\max\limits_{p\in\Lp{1}(\K)}	\left \lbrace h(p) \, : \,  \mathcal{A}p(x) \drv x \in T, \,  \inprod{\drv x}{p} = 1 \right \rbrace .
 	\end{align}
Problem \eqref{eq:main:problem:special case} is a generalized maximum entropy estimation problem that, in case $\mathcal{U}_i=\{0\}$ for all $i=1,\hdots,M$, simplifies to the standard entropy maximization problem subject to $M$ moment constraints.
In this article, we present a new approach to solve \eqref{eq:main:problem} that is based on its dual formulation. It turns out that the dual problem of \eqref{eq:main:problem} has a particular structure that allows us to apply Nesterov's smoothing method~\cite{ref:Nest-05}. Furthermore, we will show how an $\varepsilon$-optimal primal solution can be reconstructed. This is done by solving the dual problem and comparing with the existing approach in this context, it additionally requires a second smoothing step that is motivated by~\cite{ref:devolder-12}. The problem of entropy maximization subject to uncertain moment constraints \eqref{eq:main:problem:special case} can be seen as a special case of \eqref{eq:main:problem}.

\begin{Rem}[Sanov's worst-case rate-function]
The problem \eqref{eq:main:problem} is closely related to Sanov's worst-case rate-function, that can be expressed as
\begin{equation*}
\inf\limits_{\mu\in\mathcal{P}(\K)} \left\{ \KL{\nu}{\mu} \ : \ \mathcal{A}\mu \in T \right\},
\end{equation*}
where $\mathcal{U}_i$ is the empty set for all $i$. We refer the interested reader to \cite{ref:Pandit-06} for a detailed treatment of that problem and to \cite{ref:Prandit-Meyn-04} for its applications to robust hypothesis testing.
\end{Rem}

\section{Relative entropy minimization}  \label{sec:rate:distortion}
We start by recalling that an unconstrained minimization of the relative entropy with an additional linear term in the cost admits a closed form solution. Let $c\in\mathbb{B}(\K)$, $\nu\in\mathcal{P}(\K)$ and consider the optimization problem
\begin{equation} \label{eq:modified:entropy:max}
\min\limits_{\mu\in\mathcal{P}(\K)}  \left \{ \KL{\mu}{\nu} - \inprod{\mu}{c} \right\}.
\end{equation}

\begin{Lem}[Gibbs distribution] \label{lem:entropy:max}
The unique optimizer to problem \eqref{eq:modified:entropy:max} is given by the Gibbs distribution, i.e.,
\begin{equation*}
\mu^{\star}(\drv x) = \frac{2^{c(x)}\nu(\drv x)}{\int_{\K}2^{c(x)}\nu(\drv x)} \quad\text{for } x \in \K  ,
\end{equation*}
which leads to the optimal value of $- \log  \int_{\K}2^{c(x)}\nu(\drv x) $.
\end{Lem}
\begin{proof}
The result is standard and follows from~\cite{ref:Csiszar-75} or alternatively by~\cite[Lemma~3.10]{TobiasSutter15}.
\end{proof}

Let $\R^{M}\ni z\mapsto \sigma_{T}(z):=\max_{x\in T}\inprod{x}{z}\in\R$ denote the support function of $T$, which is continuous since $T$ is compact \cite[Corollary~13.2.2]{rockafellar70}. The primal-dual pair of problem \eqref{eq:main:problem} can be stated as
\begin{align}
\text{(primal program)}: \quad J^{\star} &= \min\limits_{\mu\in\mathcal{P}(\K)} \Big \{  \KL{\mu}{\nu} + \sup_{z\in\R^{M}}\left\{ \inprod{\mathcal{A}\mu}{z} - \sigma_{T}(z)\right\} \Big \} \label{eq:primal:problem}\\
\text{(dual program)}: \quad J_{\mathsf{D}}^{\star} &= \sup_{z\in\R^{M}} \Big \{ - \sigma_{T}(z)  +   \min\limits_{\mu\in\mathcal{P}(\K)} \left\{ \KL{\mu}{\nu}  + \inprod{\mathcal{A}\mu}{z}\right\} \Big \} \, ,  \label{eq:dual:problem}
\end{align}
where the dual function is given by 
\begin{equation} \label{eq:dual:function}
F(z)= - \sigma_{T}(z)  +   \min\limits_{\mu\in\mathcal{P}(\K)} \left\{ \KL{\mu}{\nu}  + \inprod{\mathcal{A}\mu}{z}\right\}.
\end{equation}
Note that the primal program \eqref{eq:primal:problem} is an infinite-dimensional convex optimization problem. The key idea of our analysis is driven by Lemma~\ref{lem:entropy:max} indicating that the dual function, that involves a minimization running over an infinite-dimensional space, is analytically available. As such, the dual problem becomes an unconstrained finite-dimensional convex optimization problem, which is amenable to first-order methods.
\begin{Lem}[Zero duality gap] \label{lem:zero:duality:gap}
There is no duality gap between the primal program \eqref{eq:primal:problem} and its dual \eqref{eq:dual:problem}, i.e., $J^{\star}=J_{\mathsf{D}}^{\star}$. Moreover, if there exists $\bar{\mu}\in\mathcal{P}(\K)$ such that $\mathcal{A}\bar{\mu}\in\mathsf{int}(T)$, then the set of optimal dual variables in \eqref{eq:dual:problem} is compact.
\end{Lem}
\begin{proof}
Recall that the relative entropy is known to be lower semicontinuous and convex in the first argument, which can be seen as a direct consequence of the duality relation for the relative entropy~\cite[Corollary~4.15]{ref:boucheron-13}. Hence, the desired zero duality gap follows by Sion's minimax theorem~\cite[Theorem~4.2]{ref:Sion-58}. The compactness of the set of dual optimizers is due to~\cite[Proposition~5.3.1]{ref:Bertsekas-09}. 
\end{proof}

Because the dual function \eqref{eq:dual:function} turns out to be non-smooth, in the absence of any additional structure, the efficiency estimate of a black-box first-order method is of order $O( \nicefrac{1}{\varepsilon^{2}})$, where $\varepsilon$ is the desired absolute additive accuracy of the approximate solution in function value~\cite{ref:nesterov-book-04}. 
We show, however, that the generalized entropy maximization problem \eqref{eq:primal:problem} has a certain structure that allows us to deploy the recent developments in \cite{ref:Nest-05} for approximating non-smooth problems by smooth ones, leading to an efficiency estimate of  order $O( \nicefrac{1}{\varepsilon})$.  This, together with the low complexity of each iteration step in the approximation scheme, offers a numerical method that has an attractive computational complexity. 
In the spirit of~\cite{ref:Nest-05,ref:devolder-12}, we introduce a smoothing parameter $\eta:=(\eta_{1},\eta_{2})\in\Rsp^{2}$ and consider a smooth approximation of the dual function 
\begin{equation} \label{eq:dual:function:smoothed}
F_{\eta}(z):=  -\max_{x\in T}\left\{ \inprod{x}{z} - \frac{\eta_{1}}{2}\norm{x}_{2}^{2} \right\} +  \min\limits_{\mu\in\mathcal{P}(\K)} \left\{ \KL{\mu}{\nu} + \inprod{\mathcal{A}\mu}{z}\right\}-\frac{\eta_{2}}{2}\norm{z}_{2}^{2} \, ,
\end{equation}
with respective optimizers denoted by $x^{\star}_{z}$ and $\mu^{\star}_{z}$. It is straightforward to see that the optimizer $x^{\star}_{z}$ is given by
\begin{align*}
x^{\star}_{z} = \arg\min_{x\in T} \| x - \eta_1^{-1}z\|_2^2 = \pi_T\left(\eta_1^{-1}z\right). 
\end{align*}
Hence, the complexity of computing $x^\star_z$ is determined by the projection operator onto $T$; for simple enough cases (e.g., 2-norm balls, hybercubes) the solution is analytically available, while for more general cases (e.g., simplex, 1-norm balls) it can be computed at relatively low computational effort, see \cite[Section~5.4]{richter_phd} for a comprehensive survey. The optimizer $\mu^{\star}_{z}$ according to Lemma~\ref{lem:entropy:max} is given by
\begin{equation*}
\mu^{\star}_{z}(B)  = \frac{\int_B 2^{-\mathcal{A}^*z(x)}\nu(\drv x)}{\int_\K 2^{-\mathcal{A}^*z(x)}\nu(\drv x)}, \quad \text{for all } B\in\mathcal{B}(\K).
\end{equation*}
\begin{Lem}[Lipschitz gradient] \label{lem:lipschitz:cts:gradient}
The dual function $F_{\eta}$ defined in \eqref{eq:dual:function:smoothed} is $\eta_{2}$-strongly concave and differentiable. Its gradient $\nabla F_{\eta}(z) =-x^{\star}_{z} + \mathcal{A}\mu^{\star}_{z} - \eta_{2} z$ is Lipschitz continuous with Lipschitz constant $\tfrac{1}{\eta_{1}}+\left( \sum_{i=1}^M B^i \right)^2+\eta_{2}$ and $B:=\max\{|x| \ : \ x\in \K\}$.
\end{Lem}
\begin{proof}
The proof follows along the lines of~\cite[Theorem~1]{ref:Nest-05} and in particular by recalling that the relative entropy (in the first argument) is strongly convex with convexity parameter one and Pinsker's inequality, that says that for any $\mu\in\mathcal{P}(\K)$ we have
$\KL{\mu}{\nu} \geq \frac{1}{2}\|\mu-\nu\|_{\textsf{TV}}.$ Moreover, we use the bound
\begin{align}
\norm{\AAA} &=		\sup\limits_{\lambda\in\R^{M}\!, \, \mu\in\mathcal{P}(\K)} \left\{ \inprod{\AAA \mu}{\lambda} \ : \ \norm{\lambda}_{2}=1, \ \norm{\mu}_{\mathsf{TV}}=1 \right\} \nonumber \\
			  &\leq 	\sup\limits_{\lambda\in\R^{M}\!, \, \mu\in\mathcal{P}(\K)} \left\{ \norm{\AAA \mu}_{2} \norm{\lambda}_{2} \ : \ \norm{\lambda}_{2}=1, \ \norm{\mu}_{\mathsf{TV}}=1\right\}\label{eq:norm:W:proof:step:CS} \\
			&\leq 	   \sup\limits_{ \mu\in\mathcal{P}(\K)} \left\{ \norm{\AAA \mu}_{1} \ : \ \norm{\mu}_{\mathsf{TV}}=1\right\}\nonumber \\
			&= 		\sup\limits_{ \mu\in\mathcal{P}(\K)} \left\{ \sum_{i=1}^{M} \left| \int_{\mathcal{\K}} x^i \mu(\drv x) \right|   \ : \ \norm{\mu}_{\mathsf{TV}}=1\right\}\nonumber \\
			&\leq 	\sum_{i=1}^M B^i \, , \nonumber
\end{align}
where \eqref{eq:norm:W:proof:step:CS} is due to the Cauchy-Schwarz inequality.
\end{proof}

Note that $F_{\eta}$ is $\eta_{2}$-strongly concave and according to Lemma~\ref{lem:lipschitz:cts:gradient} its gradient is Lipschitz continuous with constant $L(\eta):=\tfrac{1}{\eta_{1}}+\norm{\mathcal{A}}^{2}+\eta_{2}$.
We finally consider the approximate dual program given by
\begin{align} 
\text{(smoothed dual program)}: \quad J_{\eta}^{\star} &= \sup_{z\in\R^{M}} F_{\eta}(z) \, .  \label{eq:dual:problem:approx:double}
\end{align}
It turns out that \eqref{eq:dual:problem:approx:double} belongs to a favorable class of smooth and strongly convex optimization problems that can be solved by a fast gradient method given in Algorithm~\hyperlink{algo:1}{1} (see~\cite{ref:nesterov-book-04}) with an efficiency estimate of the order $O(\nicefrac{1}{\sqrt{\varepsilon}})$.

 \begin{table}[!htb]
\centering 
\begin{tabular}{c}
  \Xhline{3\arrayrulewidth}  \hspace{1mm} \vspace{-3mm}\\ 
\hspace{12.2mm}{\bf{\hypertarget{algo:1}{Algorithm 1: } }} Optimal scheme for smooth $\&$ strongly convex optimization \hspace{1.2mm} \\ \vspace{-3mm} \\ \hline \vspace{-0.5mm}
\end{tabular} \\
\vspace{-5mm}
 \begin{flushleft}
  {\hspace{3mm}Choose $w_0=y_{0} \in \R^{M}$ and $\eta\in\Rsp^2$}
 \end{flushleft}
 \vspace{-8mm}
 \begin{flushleft}
  {\hspace{3mm}\bf{For $k\geq 0$ do$^{*}$}}
 \end{flushleft}
 \vspace{-2mm}

  \begin{tabular}{l l}
{\bf Step 1: } & Set $y_{k+1}=w_{k}+\frac{1}{L(\eta)}\nabla F_{\eta}(w_{k})$ \\
{\bf Step 2: } & Compute $w_{k+1}=y_{k+1} + \frac{\sqrt{L(\eta)}-\sqrt{\eta_{2}}}{\sqrt{L(\eta)}+\sqrt{\eta_{2}}}(y_{k+1}-y_{k})$\\
  \end{tabular}
   \begin{flushleft}
  {\hspace{3mm}[*The stopping criterion is explained in Remark~\ref{remark:stopping}]}
  \vspace{-10mm}
 \end{flushleft}  
\begin{tabular}{c}
\hspace{12.2mm} \phantom{ {\bf{Algorithm:}} Optimal Scheme for Smooth $\&$ Strongly Convex Optimization}\hspace{1.5mm} \\ \vspace{-1.0mm} \\\Xhline{3\arrayrulewidth}
\end{tabular}
\end{table}

Under an additional regularity assumption, solving the smoothed dual problem~\eqref{eq:dual:problem:approx:double} provides an estimate of the primal and dual variables of the original non-smooth problems \eqref{eq:primal:problem} and \eqref{eq:dual:problem}, respectively, as summarized in the next theorem (Theorem~\ref{thm:main:result:inf:dim}).
The main computational difficulty of the presented method lies in the gradient evaluation $\nabla F_{\eta}$. We refer to Section~\ref{sec:gradient:evalutaion}, for a detailed discussion on this subject.

\begin{As}[Slater point] \label{ass:slater}
There exits a strictly feasible solution to \eqref{eq:main:problem}, i.e., $\mu_0\in\mathcal{P}(\K)$ such that $\mathcal{A}\mu_0\in T$ and $\delta:=\min_{y\in  T^c} \| \mathcal{A}\mu_0 - y\|_2 >0$.
\end{As}
Note that finding a Slater point $\mu_0$ such that Assumption~\ref{ass:slater} holds, in general can be difficult. In Remark~\ref{rem:slater:point} we present a constructive way of finding such an interior point. Given Assumption~\ref{ass:slater}, for $\varepsilon>0$ define
\begin{align}
 &C:=\KL{\mu_0}{\nu}, \qquad D :={1 \over 2} \max_{x \in T} \|x\|_2, \qquad \eta_{1}(\varepsilon) :=\frac{\varepsilon}{4D},  \qquad  \eta_{2}(\varepsilon) :=\frac{\varepsilon \delta^2}{2C^2} \nonumber \\
 &N_1(\varepsilon):=2 \left( \sqrt{\frac{8DC^2}{\varepsilon^2 \delta^2}+\frac{2\|\mathcal{A}\|^2 C^2}{\varepsilon \delta^2}+1}\right) \ln\left(\frac{10(\varepsilon +2C)}{\varepsilon}\right) \label{eq:definitions:algo:cont} \\
&N_2(\varepsilon):=2 \left( \sqrt{\frac{8DC^2}{\varepsilon^2 \delta^2}+\frac{2\|\mathcal{A}\|^2 C^2}{\varepsilon \delta^2}+1}\right) \ln\left( \frac{C}{\varepsilon \delta(2-\sqrt{3})}\sqrt{4\left( \frac{4D}{\varepsilon}+\|\mathcal{A}\|^2 + \frac{\varepsilon \delta^2}{2C^2} \right)\left( C +\frac{\varepsilon}{2} \right)} \right). \nonumber
\end{align}

\begin{Thm}[Almost linear convergence rate]\label{thm:main:result:inf:dim}
Given Assumption~\ref{ass:slater} and the definitions~\eqref{eq:definitions:algo:cont}, let $\varepsilon>0$ and $N(\varepsilon) := \left \lceil \max\{ N_1(\varepsilon), N_2(\varepsilon)\} \right \rceil$. Then, $N(\varepsilon)$ iterations of Algorithm~\hyperlink{algo:1}{1} produce approximate solutions to the problems
\eqref{eq:dual:problem} and \eqref{eq:primal:problem} given by
\begin{equation} \label{eq:estimates:primal:dual}
\hat{z}_{k,\eta}:=y_{k} \quad \text{and} \quad \hat{\mu}_{k,\eta}(B) := \frac{\int_B 2^{-\mathcal{A}^*\hat{z}_{k,\eta}(x)}\nu(\drv x)}{\int_\K 2^{-\mathcal{A}^*\hat{z}_{k,\eta}(x)}\nu(\drv x)}\, , \quad \text{for all } B\in\mathcal{B}(\K) \, ,
\end{equation}
which satisfy
\begin{subequations}
\begin{alignat}{3}
&\text{dual $\varepsilon$-optimality:}\hspace{20mm} &&0\leq J^{\star} - F(\hat{z}_{k(\varepsilon)})\leq \varepsilon \label{eq:thm:dual:optimality:cts} \\
&\text{primal $\varepsilon$-optimality:}\hspace{25mm} &&|\KL{\hat{\mu}_{k(\varepsilon)}}{\nu}- J^{\star} | \leq  2(1+2\sqrt{3})\varepsilon \label{eq:thm:primal:optimality:cts} \\
&\text{primal $\varepsilon$-feasibility:}\hspace{20mm} &&d(\mathcal{A}\hat{\mu}_{k(\varepsilon)},T)\leq \frac{2\varepsilon\delta}{C} \, , \label{eq:thm:primal:feasibility:cts}
\end{alignat}
\end{subequations}
where $d(\cdot,T)$ denotes the distance to the set $T$, i.e., $d(x,T):=\min_{y\in T}\|x-y\|_2$.
\end{Thm}

In some applications, Assumption~\ref{ass:slater} does not hold, as for example in the classical case where $\mathcal{U}_i=\{0\}$ for all $i=1,\ldots,M$, i.e., there is no uncertainty in the measurement model~\eqref{eq:measurement:model}. Moreover, in other cases satisfying Assumption~\ref{ass:slater} using the construction described in Remark~\ref{rem:slater:point} might be computationally expensive. 
Interestingly, Algorithm~\hyperlink{algo:1}{1} can be run irrespective of whether Assumption~\ref{ass:slater} holds or not, i.e. for any choice of $C$ and $\delta$. While explicit error bounds of Theorem~\ref{thm:main:result:inf:dim} as well as the a-posteriori error bound discussed below do not hold anymore, the asymptotic convergence is not affected.

\begin{proof}
Using Assumption~\ref{ass:slater}, note that the constant defined as
\begin{equation*}
L:= \frac{\KL{\mu_0}{\nu}-\min_{\mu\in\mathcal{P}(\K)}\KL{\mu}{\nu}}{\min_{y\in  T^c} \| \mathcal{A}\mu_0 - y\|_2} = \frac{C}{\delta}
\end{equation*}
can be shown to be an upper bound for the optimal dual multiplier \cite[Lemma 1]{ref:Nedic-08}, i.e., $\| z^\star \|_2 \leq L$.
The dual function can be bounded from above by $C$, since weak duality ensures $ F(z) \leq J^\star \leq \KL{\mu_0}{\nu}=C$ for all $z \in \R^M$.
Moreover, if we recall the preparatory Lemmas~\ref{lem:zero:duality:gap} and \ref{lem:lipschitz:cts:gradient}, we are finally in the setting such that the presented error bounds can be derived from \cite{ref:devolder-12}.
\end{proof}

Theorem~\ref{thm:main:result:inf:dim} directly implies that we need at most $O(\frac{1}{\varepsilon} \log \frac{1}{\varepsilon})$ iterations of Algorithm~\hyperlink{algo:1}{1} to achieve $\varepsilon$-optimality of primal and dual solutions as well as $\varepsilon$-feasible primal variable. 
Note that Theorem~\ref{thm:main:result:inf:dim} provides an explicit bound on the so-called \emph{a-priori errors}, together with approximate optimizer of the primal \eqref{eq:primal:problem} and dual \eqref{eq:dual:problem} problem. The latter allows us to derive an \emph{a-posteriori error} depending on the approximate optimizers, which is often significantly smaller than the a-priori error.
\begin{Cor}[Posterior error estimation] \label{Cor:aposteriori:infinite:dim}
Given Assumption~\ref{ass:slater}, let $z^{\star}$ denote the dual optimizer to \eqref{eq:dual:problem}. The approximate primal and dual variables $\hat{\mu}$ and $\hat{z}$ given by~\eqref{eq:estimates:primal:dual}, satisfy the following a-posteriori error bound 
\begin{align*}
F(\hat{z})\leq J^{\star}\leq \KL{\hat{\mu}}{\nu} + \frac{C}{\delta} d(\mathcal{A}\hat{\mu},T) \, ,
\end{align*}
where $d(\cdot,T)$ denotes the distance to the set $T$, i.e., $d(x,T):=\inf_{y\in T}\|x-y\|_2$.
\end{Cor}
\begin{proof}
The two key ingredients of the proof are Theorem~\ref{thm:main:result:inf:dim} and the Lipschitz continuity of the so-called perturbation function of convex programming. We introduce the perturbed program as
\begin{subequations}
 \begin{align} 
 	       \quad J^{\star}(\varepsilon) &= \min\limits_{\mu\in\mathcal{P}(\K)} \{ \KL{\mu}{\nu}  \, : \, d(\mathcal{A}\mu,T)\leq \varepsilon  \} \nonumber\\
		&= \min\limits_{\mu\in\mathcal{P}(\K)}  \KL{\mu}{\nu} + \sup\limits_{\lambda\geq 0} \inf_{y\in T} \lambda \| \mathcal{A}\mu-y \| - \lambda \varepsilon \nonumber\\
		&= \sup\limits_{\lambda\geq 0} - \lambda \, \varepsilon + \inf_{\substack{\mu\in\mathcal{P}(\K)\\  y\in T}} \sup_{\| z \|_{2}\leq \lambda} \inprod{\mathcal{A}\mu-y}{z} + \KL{\mu}{\nu} \label{eq:cor:step:slater:cts} \\
		&=  \sup_{\substack{\lambda\geq 0 \\ \| z \|_{2}\leq \lambda}} - \lambda\, \varepsilon + \inf_{\substack{\mu\in\mathcal{P}(\K)\\  y\in T}} \inprod{\mathcal{A}\mu-y}{z} + \KL{\mu}{\nu}\label{eq:cor:step:sion:cts}\\
		&\geq  -\| z^{\star} \|_{2}\, \varepsilon + \inf_{\substack{\mu\in\mathcal{P}(\K)\\  y\in T}} \inprod{\mathcal{A}\mu-y}{z^{\star}} + \KL{\mu}{\nu} \nonumber \\
		& = -\| z^{\star} \|_{2} \, \varepsilon + J^{\star}.  \nonumber
 	\end{align}
	\end{subequations}
Equation~\eqref{eq:cor:step:slater:cts} uses the strong duality property that follows by the existence of a Slater point that is due to the definition of the set $T$, see Section~\ref{sec:problem:statement}. Step \eqref{eq:cor:step:sion:cts} follows by Sion's minimax theorem~\cite[Theorem~4.2]{ref:Sion-58}. Hence, we have shown that the perturbation function is Lipschitz continuous with constant $\| z^{\star} \|_{2}$. Finally, recalling $\| z^\star \|_2 \leq \frac{C}{\delta}$, established in the proof of Theorem~\ref{thm:main:result:inf:dim} completes the proof.
\end{proof}

\begin{Rem}[Stopping criterion of Algorithm~\hyperlink{algo:1}{1}] \label{remark:stopping}
There are two alternatives for defining a stopping criterion for Algorithm~\hyperlink{algo:1}{1}. Choose desired accuracy $\varepsilon>0$.
\begin{enumerate}[label=(\roman*), itemsep = 1mm, topsep = -1mm]
\item \emph{a-priori stopping criterion}: Theorem~\ref{thm:main:result:inf:dim} provides the required number of iterations $N(\varepsilon)$ to ensure an $\varepsilon$-close solution.
\item \emph{a-posteriori stopping criterion}: Choose the smoothing parameter $\eta$ as in \eqref{eq:definitions:algo:cont}. Fix a (small) number of iterations $\ell$ that are run using Algorithm~\hyperlink{algo:1}{1}. Compute the a-posteriori error $\KL{\hat{\mu}}{\nu} + \frac{C}{\delta} d(\mathcal{A}\hat{\mu},T)  - F(\hat{z})$ according to Corollary~\ref{Cor:aposteriori:infinite:dim} and if it is smaller than If $\varepsilon$ terminate the algorithm. Otherwise continue with another $\ell$ iterations.
\end{enumerate}
\end{Rem}

\begin{Rem}[Slater point computation] \label{rem:slater:point}
To compute the respective constants in Assumption~\ref{ass:slater}, we need to construct a strictly feasible point for \eqref{eq:main:problem}. For this purpose, we consider a polynomial density of degree $r$ defined as $p_r(\alpha,x) := \sum_{i=0}^{r-1} \alpha_i x^i$. For notational simplicity we assume that the support set is the unit interval ($\K = [0,1]$), such that the moments induced by the polynomial density are given by
\begin{align*}
\inprod{p_r(\alpha,x)}{x^i} = \int_0^1 \sum_{j=0}^{r-1} \alpha_j x^{j+i} \drv x= \sum_{j=0}^{r-1} \frac{\alpha_j}{j+i+1},
\end{align*}
for $i=0,\hdots, M$. Consider $\beta\in\R^{M+1}$, where $\beta_1 = 1$ and $\beta_i=y_{i-1}$ for $i=2,\hdots,M+1$. Hence, the feasibility requirement of \eqref{eq:main:problem} can be expressed as the linear constraint $A \alpha = \beta$,
where $A\in\R^{(M+1)\times r}$, $\alpha\in\R^r$, $\beta\in\R^{M+1}$ and $A_{i,j} = \frac{1}{i+j-1}$
and finding a strictly feasible solution reduces to the following feasibility problem
\begin{align} \label{eq:SOS:entropy:max} 
 \left\{ \begin{array}{ll}
		\max\limits_{\alpha \in \R^r} & \text{const}   \\
		\subjectto & A\alpha = \beta \\
		& p_r(\alpha,x) \geq 0 \quad \forall x\in[0,1], 
\end{array} \right.
\end{align}
where $p_r$ is a polynomial function in $x$ of degree $r$ with coefficients $\alpha$. 
The second constraint of the program \eqref{eq:SOS:entropy:max} (i.e., $p_r(\alpha,x) \geq 0 \ \forall x\in[0,1]$)\footnote{In a multi-dimensional setting one has to consider a tightening (i.e., $p_r(\alpha,x) >0 \ \forall x\in[0,1]^n$).} can be equivalently reformulated as linear matrix inequalities of dimension $\ceil{\frac{r}{2}}$, using a standard result in polynomial optimization, see \cite[Chapter~2]{ref:Lasserre-11} for details.
We note that for small degree $r$, the set of feasible solutions to problem \eqref{eq:SOS:entropy:max} may be empty, however, by choosing $r$ large enough and assuming that the moments can be induced by a continuous density, problem \eqref{eq:SOS:entropy:max} becomes feasible. Moreover, if $0 \in \mathsf{int}(T)$ the Slater point leads to a $\delta>0$ in Assumption~\ref{ass:slater}.
\end{Rem}

\begin{Ex}[Density estimation] \label{ex:Example1}
We are given the first $3$ moments of an unknown probability measure defined on $\K=[0,1]$ as\footnote{The considered moments are actually induced by the probability density $f(x):=(\ln 2 \ (1+x))^{-1}$. We, however, do not use this information at any point of this example.}
\begin{align*}
y:=\left( \frac{1-\ln 2}{\ln 2}, \frac{\ln 4 - 1}{\ln 4}, \frac{5-\ln 64}{\ln 64} \right) \approx (0.44,\ 0.28,\ 0.20).
\end{align*}
The uncertainty set in the measurement model \eqref{eq:measurement:model} is assumed to be $\mathcal{U}_{i}=[-u,u]$ for all $i=1,\hdots,3$. A Slater point is constructed using the method described in Remark~\ref{rem:slater:point}, where $r=5$ is enough for the problem~\eqref{eq:SOS:entropy:max} to be feasible, leading to the constant $C=0.0288$. The Slater point is depicted in Figure~\ref{fig:plotExperimentMeanVar1} and its differential entropy can be numerically computed as $-0.0288$.

We consider two  simulations for two different uncertainty sets (namely, $u=0.01$ and $u=0.005$). 
The underlying maximum entropy problem \eqref{eq:main:problem:special case} is solved using Algorithm~\hyperlink{algo:1}{1}. The respective features of the a-priori guarantees by Theorem~\ref{thm:main:result:inf:dim} as well as the a-posteriori guarantees (upper and lower bounds) by Corollary~\ref{Cor:aposteriori:infinite:dim} are reported in Table~\ref{tab:ex1}. 
Recall that $\hat{\mu}_{k(\varepsilon)}$ denotes the approximate primal variable after $k$-iterations of Algorithm~\hyperlink{algo:1}{1} as defined in Theorem~\ref{thm:main:result:inf:dim} and that $d(\mathcal{A}\hat{\mu}_{k(\varepsilon)},T)$ (resp.~$\frac{2\varepsilon \delta}{C}$) represent the a-posteriori (resp.~a-priori) feasibility guarantees.
It can be seen in Table~\ref{tab:ex1} that increasing the uncertainty set $\mathcal{U}$ leads to a higher entropy, where the uniform density clearly has the highest entropy. This is also intuitively expected since enlarging the uncertainty set is equivalent to relaxing the moment constraints in the respective maximum entropy problem. The corresponding densities are graphically visualized in Figure~\ref{fig:plotExperimentMeanVar1}.

 \begin{table}[!htb]  {\small{
\centering 
\caption{Some specific simulation points of Example~\ref{ex:Example1}. }
\label{tab:ex1}

\hspace{37mm} $\mathcal{U}=[-0.01,0.01]$ \hspace{41mm} $\mathcal{U}=[-0.005,0.005]$
\vspace{3mm} \phantom{..}
  \begin{tabular}{c@{\hskip 3mm} | c@{\hskip 2mm} c@{\hskip 2mm} c@{\hskip 2mm} c  | c@{\hskip 2mm} c@{\hskip 3mm} c@{\hskip 3mm} c  }
 a-priori error $\varepsilon$ \hspace{1mm}  & \hspace{1mm}    1  &$0.1$ & $0.01$ & $0.001$ \hspace{1mm}   &\hspace{1mm}  1  &$0.1$ & $0.01$ & $0.001$   \\ 
 $J_{\textnormal{UB}}$ & \hspace{1mm}-0.0174 & -0.0189 & -0.0194  & -0.0194 \hspace{1mm}   &\hspace{1mm}  -0.0223 & -0.0236 & -0.0237 & -0.0238   \\
 $J_{\textnormal{LB}}$ & \hspace{1mm} -0.0220 & -0.0279 & -0.0204  & -0.0195 \hspace{1mm}   & \hspace{1mm}   -0.0263 & -0.0298 & -0.0244 & -0.0238 \\
 iterations $k(\varepsilon)$ & \hspace{1mm} 99 & 551 & 5606  & 74423 \hspace{1mm}  & \hspace{1mm}   232 & 1241 & 12170 & 157865 \\
 $d(\mathcal{A}\hat{\mu}_{k(\varepsilon)},T)$& \hspace{1mm} 0.0008  &0.0036 &0.0005  & 0 \hspace{1mm} & \hspace{1mm}  0 & 0.001 & 0.0001 & 0  \\
  $\frac{2\varepsilon \delta}{C}$ & \hspace{1mm} 0.69  &0.069 &0.0069  & 0.00069 \hspace{1mm} & \hspace{1mm}  0.35 & 0.035 & 0.0035 & 0.00035
  \end{tabular} }}
\end{table}

\begin{figure}[!htb]   
\center{    {\input{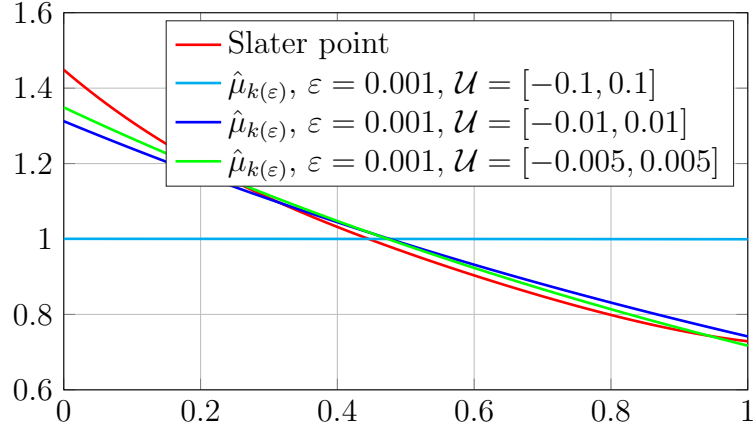} \label{fig:firstExperiment}} }
    \caption[]{Maximum entropy densities obtained by Algorithm~\hyperlink{algo:1}{1} for two different uncertainty sets. As a reference, the Slater point density, that was computed as described in Remark~\ref{rem:slater:point} is depicted in red.}
    \label{fig:plotExperimentMeanVar1}
\end{figure}

\end{Ex}

\section{Finite-dimensional case} \label{sec:finite:dim:case}
We consider the finite-dimensional case where $\K=\{1,\hdots,N\}$ and hence we optimize in \eqref{eq:main:problem} over the probability simplex $\mathcal{P}(\K)=\Delta_{N}$. One substantial simplification, when restricting to the finite-dimensional setting, is that the Shannon entropy is non-negative and bounded from above (by $\log N$). 
Therefore, we can substantially weaken Assumption~\ref{ass:slater} to the following assumption.

\begin{As} \label{ass:finite:dim} \
\begin{enumerate}[label=(\roman*), itemsep = 1mm, topsep = -1mm]
\item There exists $\delta>0$ such that $B(0,\delta)\subset\{x-\mathcal{A}\mu \ : \ \mu\in\Delta_{N}, x\in T\}$.
\item The reference measure $\nu\in\Delta_{N}$ has full support, i.e., $\min\limits_{1\leq i \leq N} \nu_{i}>0$.
\end{enumerate}
\end{As}

Consider the the definitions given in \eqref{eq:definitions:algo:cont} with $C:=\max\limits_{1\leq i \leq N} \log \frac{1}{ \nu_{i}}$, then the following finite-dimensional equivalent to Theorem~\ref{thm:main:result:inf:dim} holds.

\begin{Cor}[a-priori error]\label{cor:main:result:finite:dim}
Given Assumption~\ref{ass:finite:dim}, $C:=\max\limits_{1\leq i \leq N} \log \frac{1}{ \nu_{i}}$ and the definitions~\eqref{eq:definitions:algo:cont}, let $\varepsilon>0$ and $N(\varepsilon) := \left \lceil \{ N_1(\varepsilon) , N_2(\varepsilon) \} \right \rceil$. Then, $N(\varepsilon)$ iterations of Algorithm~\hyperlink{algo:1}{1} produce the approximate solutions to the problems
\eqref{eq:dual:problem} and \eqref{eq:primal:problem}, given by
\begin{equation} \label{eq_varAlgo}
\hat{z}_{k(\varepsilon)}:=y_{k(\varepsilon)} \quad \text{and} \quad \hat{\mu}_{k(\varepsilon)}(B) := \frac{\sum_{i\in B}2^{-\left(\mathcal{A}^*\hat{z}_{k(\varepsilon)}\right)_i}\nu_i}{\sum_{i=1}^N 2^{-\left(\mathcal{A}^*\hat{z}_{k(\varepsilon)}\right)_i}\nu_i} \quad \text{for all }B\subset \{1,2,\hdots,N\} \, ,
\end{equation}
which satisfy
\begin{subequations}
\begin{alignat}{3}
&\text{dual $\varepsilon$-optimality:}\hspace{20mm} &&0\leq F(\hat{z}_{k(\varepsilon)})-J^{\star}\leq \varepsilon \label{eq:thm:dual:optimality} \\
&\text{primal $\varepsilon$-optimality:}\hspace{25mm} &&|\KL{\hat{\mu}_{k(\varepsilon)}}{\nu}- J^{\star} | \leq  2(1+2\sqrt{3})\varepsilon \label{eq:thm:primal:optimality} \\
&\text{primal $\varepsilon$-feasibility:}\hspace{20mm} &&d(\mathcal{A}\hat{\mu}_{k(\varepsilon)},T)\leq \frac{2\varepsilon\delta}{C} \, , \label{eq:thm:primal:feasibility}
\end{alignat}
\end{subequations}
where $d(\cdot,T)$ denotes the distance to the set $T$, i.e., $d(x,T):=\min_{y\in T}\|x-y\|_2$.
\end{Cor}
\begin{proof}
Under Assumption~\ref{ass:finite:dim} the dual optimal solutions in \eqref{eq:dual:problem} are bounded by 
\begin{align} \label{eq:dual:var:bound}
\| z^{\star} \| \leq \frac{1}{r}\max\limits_{1\leq i \leq N} \log \frac{1}{ \nu_{i}} \, .
\end{align}
This bound on the dual optimizer follows along the lines of~\cite[Theorem~6.1]{ref:devolder-12}. The presented error bounds can then be derived along the lines of Theorem~\ref{thm:main:result:inf:dim}.
\end{proof}
In addition to the explicit error bound provided by Corollary~\ref{cor:main:result:finite:dim}, the a-posteriori upper and lower bounds presented in Corollary~\ref{Cor:aposteriori:infinite:dim} directly apply to the finite-dimensional setting as well.

\section{Gradient approximation} \label{sec:gradient:evalutaion}
The computationally demanding element for Algorithm~\hyperlink{algo:1}{1} is the evaluation of the gradient $\nabla F_{\eta}(\cdot)$ given in Lemma~\ref{lem:lipschitz:cts:gradient}. In particular, Theorem~\ref{thm:main:result:inf:dim} and Corollary~\ref{cor:main:result:finite:dim} assume that this gradient is known exactly. While this is not restrictive if, for example, $\K$ is a finite set, in general, $\nabla F_{\eta}(\cdot)$ involves an integration that can only be computed approximately. In particular if we consider a multi-dimensional setting (i.e., $\K\subset \R^d$), the evaluation of the gradient $\nabla F_{\eta}(\cdot)$ represents a multi-dimensional integration problem. This gives rise to the question of how the fast gradient method (and also Theorem~\ref{thm:main:result:inf:dim}) behaves in a case of inexact first-order information. We refer the interested readers to \cite{ref:Devolver-13} for further details in this regard.

In this section we discuss two numerical methods to approximate this gradient. Note that in Lemma~\ref{lem:lipschitz:cts:gradient}, given that $T$ is simple enough the optimizer $x^{\star}_{z}$ is analytically available, so what remains is to compute $\mathcal{A}\mu^{\star}_{z}$, that according to Lemma~\ref{lem:entropy:max} is given by
\begin{align} \label{eq:gradient:approximation}
(\mathcal{A}\mu^{\star}_{z})_{i} &= \frac{\int_{\K} x^{i}2^{-\mathcal{A}^{*}z(x)}\nu(\drv x)}{\int_{\K} 2^{-\mathcal{A}^{*}z(x)}\nu(\drv x)} \quad\text{ for all }i=1,\hdots, M \, .
\end{align}

\textbf{Semidefinite programming.} Due to the specific structure of the considered problem, \eqref{eq:gradient:approximation} represents an integration of exponentials of polynomials for which an efficient approximation in terms of two single semidefinite programs involving two linear matrix inequalities has been derived, where the desired accuracy is controlled by the size of the linear matrix inequalities constraints, see~\cite{ref:Bertsimas-08, ref:Lasserre-11} for a comprehensive study. \vspace{2mm}

\textbf{Quasi-Monte Carlo.}
The most popular methods for integration problems of the from \eqref{eq:gradient:approximation} are Monte Caro (MC) schemes, see \cite{ref:robert-04} for a comprehensive summary. The main advantage of MC methods is that the root-mean-square error of the approximation converges to $0$ with a rate of $O(N^{-1/2})$ that is independent of the dimension, where $N$ are the number of samples used.
In practise, this convergence often is too slow. Under mild assumptions on the integrand, the MC methods can be significantly improved with a more recent technique known as Quasi-Monte Carlo (QMC) methods. QMC methods can reach a convergence rate arbitrarily close to $O(N^{-1})$ with a constant not depending on the dimension of the problem. We would like to refer the reader to \cite{ref:Dick-13, Sloan-98, Sloan-05, niederreiter2010quasi} for a detailed discussion about the theory of QMC methods.

\begin{Rem}[Computational stability] \label{rem:computational:stability}
The evaluation of the gradient in Lemma~\ref{lem:lipschitz:cts:gradient} involves the term $\mathcal{A}\mu^{\star}_{z}$, where $\mu^{\star}_{z}$ is the optimizer of the second term in \eqref{eq:dual:function:smoothed}. By invoking Lemma~\ref{lem:entropy:max} and the definition of the operator $\mathcal{A}$, the gradient evaluation reduces to
\begin{equation} \label{eq:rem:stable}
\left(\mathcal{A}\mu^{\star}_{z} \right)_i = \frac{\int_\K x^i 2^{-\sum_{j=1}^M z_j x^j}\drv x}{\int_\K 2^{-\sum_{j=1}^M z_j x^j}\drv x} \quad \text{for }i=1,\hdots,M \, .
\end{equation}
Note that a straightforward computation of the gradient via \eqref{eq:rem:stable} is numerically difficult. To alleviate this difficulty, we follow the suggestion of \cite[p.~148]{ref:Nest-05} which we briefly elaborate here. Consider the functions $f(z,x):= -\sum_{j=1}^M z_j x^j $, $\bar{f}(z):=\max_{x\in\K} f(z,x)$ and $g(z,x):= f(z,x)- \bar{f}(z)$. Note, that  $g(z,x)\geq 0$ for all $(z,x)\in \R^M\times\R$. One can show that
\begin{equation*} 
\left(\mathcal{A}\mu^{\star}_{z} \right)_i = \frac{\int_\K  2^{g(z,x)} \frac{\partial}{\partial z_i}g(z,x)\drv x}{\int_\K 2^{g(z,x)}\drv x} + \frac{\partial}{\partial z_i}\bar{f}(z) \quad \text{for }i=1,\hdots,M \, ,
\end{equation*}
which can be computed with significantly smaller numerical error than \eqref{eq:rem:stable} as the numerical exponent are always negative, leading to values always being smaller than $1$. 
\end{Rem}

\section{Zero-information moment closure method} \label{sec:moment:closure}

In chemistry, physics, systems biology and related fields, stochastic chemical reactions are described by the \emph{chemical master equation} (CME), that is a special case of the Chapman-Kolmogorov equation as applied to Markov processes~\cite{ref:vanKampen-81, ref:Wilkinson-06}.
These equations are usually infinite-dimensional and analytical solutions are generally impossible. Hence, effort has been directed toward developing of a variety of numerical schemes for efficient approximation of the CME, such as stochastic simulation techniques (SSA)~\cite{ref:Gillespie-76}. In practical cases, one is often interested in the first few moments of the number of molecules involved in the chemical reactions. This motivated the development of approximation methods to those low-order moments without having to solve the underlying infinite-dimensional CME. One such approximation method is the so-called \emph{moment closure method}~\cite{ref:Gillespie-09}, that briefly described works as follows: First the CME is recast in terms of moments as a linear ODE of the form
\begin{equation} \label{eq:moment:CME}
\frac{\drv}{\drv t}\mu(t) =  A \mu(t) + B \zeta(t)\, ,
\end{equation}
where $\mu(t)$ denotes the moments up to order $M$ at time $t$ and $\zeta(t)$ is an infinite vector describing the contains moments of order $M+1$ or higher. 
In general $\zeta$ can be an infinite vector, but for most of the standard chemical reactions considered in, e.g., systems biology it turns out that only a finite number of higher order moments affect the evolution of the first $M$ moments. Indeed, if the chemical system involves only the so-called zeroth and first order reactions the vector $\zeta$ has dimension zero (reduces to a constant affine term), whereas if the system also involves second order reactions then $\zeta$ also contains some moments of order $M+1$ only. It is widely speculated that reactions up to second order are sufficient to realistically model most systems of interest in chemistry and biology \cite{ref:Gillespie-07, ref:Gillespie-13}.
The matrix $A$ and the linear operator $B$ (that may potentially be infinite-dimensional) can be found analytically from the CME. The ODE \eqref{eq:moment:CME}, however, is intractable due to its higher order moments dependence. The approximation step is introduced by a so-called closure function
\begin{equation*}
\zeta=\varphi(\mu)\, ,
\end{equation*}
where the higher-order moments are approximated as a function of the lower-oder moments, see \cite{ref:Singh-06,ref:Singh-07}. A closure function that has recently attracted interest is known as the \emph{zero-information} closure function (of order $M$)~\cite{ref:Smadbeck-13}, and is given by
\begin{equation} \label{eq:closure:function}
(\varphi(\mu))_i = \inprod{p^\star_\mu}{x^{M+i}}, \text{ for } i=1,2,\hdots\, ,
\end{equation}
where $p^\star_\mu$ denotes the maximizer to \eqref{eq:main:problem:special case}, where $T=\times_{i=1}^M \left\{ \inprod{\mu}{x^i}-u \ : \ u\in\mathcal{U}_i \right\}$, where $\mathcal{U}_i=[-\kappa, \kappa]\subset\R$ for all $i$ and for a given $\kappa>0$, that acts as a regularizer, in the sense of Assumption~\ref{ass:slater}.
This approximation reduces the infinite-dimensional ODE \eqref{eq:moment:CME} to a finite-dimensional ODE
\begin{equation} \label{eq:finite:ODE}
\frac{\drv}{\drv t}\mu(t) =  A \mu(t) + B \varphi(\mu(t))\, .
\end{equation}
To numerically solve \eqref{eq:finite:ODE} it is crucial to have an efficient evaluation of the closure function $\varphi$. In the zero-information closure scheme this is given by to the entropy maximization problem \eqref{eq:main:problem:special case} and as such can be addressed using Algorithm~\hyperlink{algo:1}{1}.

To illustrate this point, we consider a reversible dimerisation reaction where two monomers ($\mathcal{M}$) combine in a second-order monomolecular reaction to form a dimer ($\mathcal{D}$); the reverse reaction is first order and involves the decomposition of the dimer into the two monomers. This gives rise to the chemical reaction system
\begin{equation} \label{eq:reversible:dimerization}
\begin{aligned}
2\mathcal{M} & \overset{k_1}{\longrightarrow} \mathcal{D} \\
\mathcal{D}&\overset{k_2}{\longrightarrow} 2\mathcal{M} \, ,
\end{aligned}
\end{equation}
with reaction rate constants $k_1,k_2>0$.
Note that the system as described has a single degree of freedom since $\mathcal{M} = 2\mathcal{D}_0 - 2\mathcal{D} + \mathcal{M}_0$, Where $\mathcal{M}$ denotes the count of the monomers, $\mathcal{D}$ the count of dimers, and $\mathcal{M}_0$ and $\mathcal{D}_0$ the corresponding initial conditions.
Therefore, the matrices can be reduced to include only the moments of one component as a simplification and as such the zero-information closure function \eqref{eq:closure:function} consists of solving a one-dimensional entropy maximization problem such as given by \eqref{eq:main:problem:special case}, where the support are the natural numbers (upper bounded by $\mathcal{M}_0+\mathcal{D}_0$ and hence compact). 
For illustrative purposes, let us look at a second order closure scheme, where the corresponding moment vectors are defined as $\mu=(1, \langle \mathcal{M}\rangle, \langle \mathcal{M}^2\rangle)\transp\in\R^3$ and $\zeta = \langle \mathcal{M}^3 \rangle \in\R$ and the corresponding matrices are given by
\begin{align*}
A = \begin{pmatrix*}[c]
0 & 0 & 0 \\
k_2 S_0 & 2k_1-k_2 & -2k_1 \\
2k_2 S_0 & \quad 2k_2(S_0 -1) - 4k_1 \quad\quad  & 8k_1 -2k_2
\end{pmatrix*}, \quad
B =  \begin{pmatrix*}[c]
0  \\
0  \\
-4k_1
\end{pmatrix*},
\end{align*}
where $S_0 = \mathcal{M}_0 + 2\mathcal{D}_0$. The simulation results, Figure~\ref{fig:plotExperimentMeanVar}, show the time trajectory for the average and the second moment of the number of $\mathcal{M}$ molecules in the reversible dimerization model~\eqref{eq:reversible:dimerization}, as calculated for the zero information closure~\eqref{eq:finite:ODE} using Algorithm~\hyperlink{algo:1}{1}, for a second-order closure as well as a third-order closure. To solve the ODE~\eqref{eq:finite:ODE} we use an explicit Runge-Kutta (4,5) formula (ode45) built into MATLAB.
The results are compared to the average of $10^6$ SSA \cite{ref:Wilkinson-06} trajectories. It can be seen how increasing the order of the closure method improves the approximation accuracy.

\begin{figure}[!htb]
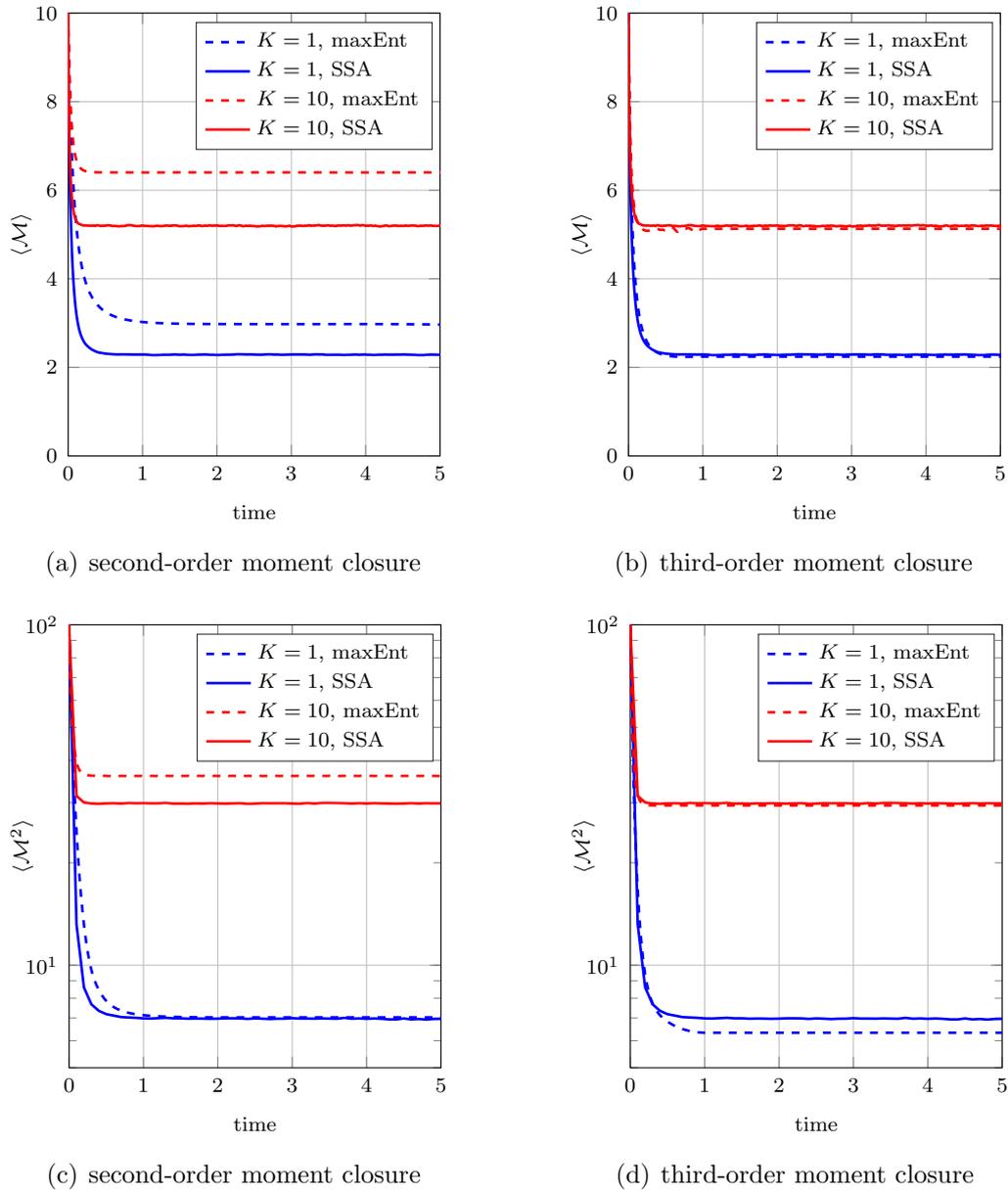
 
    \subfigure[second-order moment closure]{\input{secondOrder.tex} \label{fig:secondOrder} } \hspace{10mm}
    \subfigure[third-order moment closure]{\input{thirdOrder.tex} \label{fig:thirdOrder} } \\
     \subfigure[second-order moment closure]{\input{secondOrder_secondMoment.tex} \label{fig:secondOrder:secondMoment} } \hspace{10mm}
    \subfigure[third-order moment closure]{
%
%
\begin{tikzpicture}

\begin{axis}[%
height=6.0cm,
width=5cm,
at={(1.011111in,0.641667in)},
scale only axis,
grid=major,
style={font=\scriptsize},
y label style={at={(axis description cs:0.12,.5)}},
xmin=0,
xmax=5,
xlabel={time},
ylabel={$\langle \mathcal{M}^2\rangle$},
ymin=5,
ymax=100,
ymode=log,
xmode=linear,
legend style={legend cell align=left,align=left,draw=white!15!black}
]


\addplot [color=blue,dashed,line width=1.0pt]
  table[row sep=crcr]{%
0	100\\
0.00247348597181315	92.4303194890925\\
0.00735048582743701	79.7756597419882\\
0.0123278178962352	69.3763592898415\\
0.0174980692934942	60.6271824979529\\
0.0229267736098641	53.1635322315972\\
0.0286566963140782	46.7526157518448\\
0.0347325413593649	41.2158769059201\\
0.0412042415823096	36.4123321323829\\
0.0481269628495207	32.2302481477056\\
0.0555607717361404	28.5807325173315\\
0.0635701996395642	25.3925783149999\\
0.0722242536485956	22.6079613716804\\
0.0815976352908185	20.1788557170896\\
0.0917725906976913	18.0644495400862\\
0.102841854737132	16.2292809556737\\
0.114911882161327	14.642086532445\\
0.128107235703439	13.2749865111902\\
0.142576111877736	12.1029436598612\\
0.158497230331308	11.1033923303349\\
0.176088858004435	10.2559389107605\\
0.195620873305622	9.5420950034733\\
0.217431807256255	8.9450074848086\\
0.241955222520164	8.44916197827403\\
0.26976155435028	8.04010814873858\\
0.301621452972396	7.70427721753663\\
0.338620939864165	7.42876226411267\\
0.382348154600584	7.2005243259817\\
0.428873912780483	7.02327974189987\\
0.500593103370028	6.82173365550059\\
0.600593103370028	6.62647062575438\\
0.700593103370028	6.49392463738132\\
0.800593103370028	6.40525390121874\\
0.900593103370028	6.35267801701014\\
1	6.33463914589916\\
5       6.33463914589916\\
};
\addlegendentry{$K=1$, maxEnt};

\addplot [color=blue,solid,line width=1.0pt]
  table[row sep=crcr]{%
0	100\\
0.1	13.1846762110879\\
0.2	8.60518069868365\\
0.3	7.67705910115536\\
0.4	7.34545603269304\\
0.5	7.18353210621548\\
0.6	7.1114579395795\\
0.7	7.04235168650224\\
0.8	7.0229051606718\\
0.9	7.00530140166039\\
1	6.97863980223658\\
1.1	6.97242185713404\\
1.2	6.98942429059647\\
1.3	7.00643603998493\\
1.4	6.98334311510182\\
1.5	6.9831067546736\\
1.6	6.96556439090907\\
1.7	6.97227305231403\\
1.8	6.96918606625087\\
1.9	6.96020496539274\\
2	6.96809164733651\\
2.1	6.98224668882705\\
2.2	6.98342464797445\\
2.3	6.97864265030472\\
2.4	6.95555347531089\\
2.5	6.97917716800895\\
2.6	6.96680259550642\\
2.7	6.94904421478716\\
2.8	6.96503539423505\\
2.9	6.96706152631044\\
3	6.98958444613472\\
3.1	6.97143117396451\\
3.2	6.97142143465491\\
3.3	6.96786581964804\\
3.4	6.97759320887489\\
3.5	6.97023878683558\\
3.6	6.96366759608914\\
3.7	7.00368375292945\\
3.8	6.9650680764096\\
3.9	6.97689139280527\\
4	6.96028149587453\\
4.1	6.97621009089989\\
4.2	6.96055792892175\\
4.3	6.93364512458152\\
4.4	6.93009790531627\\
4.5	6.97740664964208\\
4.6	6.93255333039614\\
4.7	6.95181937901597\\
4.8	6.92442013912225\\
4.9	6.95202987137574\\
5	6.95736862670111\\
};
\addlegendentry{$K=1$, SSA};


\addplot [color=red,dashed,line width=1.0pt, each nth point={1}]
  table[row sep=crcr]{%
0	100\\
0.0011438505922504	96.4104102370032\\
0.0022877011845008	93.0423386769572\\
0.0034315517767512	89.8795154069735\\
0.0045754023690016	86.9068855146559\\
0.00744852627070453	80.1839856789375\\
0.0103216501724075	74.3946946742465\\
0.0131947740741104	69.3956201377125\\
0.0160678979758133	65.0518823930898\\
0.0190932302350526	61.0645677574669\\
0.0221185624942919	57.5974037885953\\
0.0251438947535313	54.5747021757256\\
0.0281692270127706	51.9260084977718\\
0.0316646866544118	49.255049453673\\
0.0351601462960529	46.9447987254587\\
0.0386556059376941	44.941976506034\\
0.0421510655793353	43.1964153713399\\
0.046197509507085	41.4426925458071\\
0.0502439534348347	39.9364590486322\\
0.0542903973625844	38.6405954313525\\
0.0583368412903341	37.5196705256296\\
0.0631029909079174	36.384216430406\\
0.0678691405255006	35.4210358146173\\
0.0726352901430838	34.6033958239671\\
0.077401439760667	33.9052810012997\\
0.0831080943944591	33.1968366625455\\
0.0888147490282511	32.6074292233886\\
0.0945214036620432	32.1174968954863\\
0.100228058295835	31.7074910374806\\
0.107181820724837	31.2929955383042\\
0.11413558315384	30.9580372353421\\
0.121089345582842	30.6884311524954\\
0.128043108011844	30.469408154333\\
0.136665685200106	30.250488960823\\
0.145288262388369	30.080930328584\\
0.153910839576631	29.9510699148845\\
0.162533416764893	29.8500049502439\\
0.173388410595362	29.7507238450427\\
0.184243404425831	29.6783746274986\\
0.1950983982563	29.6273947477368\\
0.205953392086768	29.5901018363626\\
0.219750911047443	29.5534235192299\\
0.233548430008118	29.5289773619351\\
0.247345948968793	29.514723606201\\
0.261143467929468	29.5052536185119\\
0.278644351428597	29.4934069119628\\
0.296145234927726	29.4868140478167\\
0.313646118426856	29.4865358204288\\
0.331147001925985	29.486484846078\\
0.351557031744073	29.4716457773141\\
0.371967061562161	29.4672731742283\\
0.392377091380249	29.4862527798329\\
0.412787121198337	29.4958719712177\\
0.424241845468213	29.4858839181489\\
0.435696569738089	29.4810328676897\\
0.447151294007965	29.4830917618974\\
0.458606018277841	29.4849538792234\\
0.470060742547717	29.4820492764798\\
0.481515466817593	29.4806441808781\\
0.492970191087469	29.4812645321659\\
0.504424915357345	29.4818204746261\\
0.519939489865311	29.4795093665395\\
0.535454064373278	29.4787093408972\\
0.550968638881244	29.4810387088347\\
0.566483213389211	29.4824864487988\\
0.58534750253132	29.4727475764936\\
0.604211791673428	29.4700747938641\\
0.623076080815537	29.4838215674651\\
0.641940369957646	29.4911751254512\\
0.654115550667265	29.4828262442801\\
0.666290731376883	29.4790595585703\\
0.678465912086502	29.4821714623455\\
0.690641092796121	29.4846203614801\\
0.702816273505739	29.4813957191312\\
0.714991454215358	29.479944494842\\
0.727166634924977	29.4811623568749\\
0.739341815634595	29.4821176520533\\
0.754780591940258	29.4792154281429\\
0.77021936824592	29.4782141238194\\
0.785658144551582	29.4811466934853\\
0.801096920857244	29.4829763028976\\
0.81902212980442	29.4737082937901\\
0.836947338751596	29.471011859682\\
0.854872547698772	29.4833243280892\\
0.872797756645947	29.4901874653486\\
0.885280210502631	29.4820795618251\\
0.897762664359314	29.4785149525111\\
0.910245118215998	29.4820542862879\\
0.922727572072681	29.4847436395811\\
0.935210025929365	29.4811887241214\\
0.947692479786048	29.4796294255005\\
0.960174933642731	29.4811977025947\\
0.972657387499415	29.4823862489473\\
0.979493040624561	29.4817665223727\\
0.986328693749707	29.4813516523387\\
0.993164346874854	29.4810984095674\\
1	29.480931170018\\
5	29.480931170018\\
};  
\addlegendentry{$K=10$, maxEnt};

\addplot [color=red,solid,line width=1.0pt]
  table[row sep=crcr]{%
0	100\\
0.1	31.4360922253912\\
0.2	30.0725916160037\\
0.3	29.867174510949\\
0.4	29.8085674534831\\
0.5	29.8875101161383\\
0.6	29.9171280835894\\
0.7	29.8554049546653\\
0.8	29.961318657568\\
0.9	29.9114373067155\\
1	29.9440136551355\\
1.1	29.8994439277092\\
1.2	29.9031887629185\\
1.3	29.9905792250483\\
1.4	29.8424324944621\\
1.5	29.8536442469341\\
1.6	29.8269212135024\\
1.7	29.8860254628935\\
1.8	29.8550227496348\\
1.9	29.9187223869966\\
2	29.8732687651085\\
2.1	29.8429214658471\\
2.2	29.8852924894355\\
2.3	29.9234429176317\\
2.4	29.7987714454821\\
2.5	29.8427059226136\\
2.6	29.9268353907351\\
2.7	29.9082619796268\\
2.8	29.9157084143231\\
2.9	29.917644225447\\
3	29.8531848649264\\
3.1	29.8216189532455\\
3.2	29.954456293972\\
3.3	29.8539947979522\\
3.4	29.9717836447014\\
3.5	29.9793352215058\\
3.6	29.9574445232326\\
3.7	29.9221668949465\\
3.8	29.8529642202597\\
3.9	29.9253814587457\\
4	29.8585158294713\\
4.1	29.9312268157462\\
4.2	29.8867818345868\\
4.3	29.8548832964391\\
4.4	29.8327445121527\\
4.5	29.864094089698\\
4.6	29.9294031702885\\
4.7	29.8555522103354\\
4.8	29.920183596465\\
4.9	29.8910287356865\\
5	29.8993821031991\\
};
\addlegendentry{$K=10$, SSA};

\end{axis}
\end{tikzpicture}%
 \label{fig:thirdOrder:secondMoment} } 
    \caption[]{Reversible dimerization system~\eqref{eq:reversible:dimerization} with reaction constants $K = \nicefrac{k_2}{k_1}$: Comparison of the zero-information moment closure method \eqref{eq:finite:ODE}, solved using Algorithm~\hyperlink{algo:1}{1} and the average of $10^6$ SSA trajectories. The initial conditions are $\mathcal{M}_0 = 10$ and $\mathcal{D}_0=0$ and the regularization term $\kappa=0.01$.}
    \label{fig:plotExperimentMeanVar}
\end{figure}

\mypart{1}{Conclusions}





\chapter{Conclusions and future directions}\label{chap:ConclusionsOutlook}

This thesis addresses classes of optimization problems from two disciplines: optimal control and information theory. In Part~\ref{part:ADP} we studied the optimal control of discrete-time continuous space MDPs, whereas Part~\ref{part:IT} treats the channel capacity problem as well as the entropy maximization subject to moment constraints.

\section{Approximate Dynamic Programming}

In Chapters~\ref{chap:infLP} to \ref{chap:approx:MDP}, we presented a two-step approximation scheme for a general class of linear programming problems in infinite-dimensional spaces, including the LP formulation of discrete-time MDPs with Borel state and action spaces under the long-run average (or discounted) cost optimality criterion. The first approximation step bridges the infinite-dimensional LP~\ref{primal-inf} to a semi-infinite relaxed LP~\ref{primal-n}. The second step transforms the latter problem to a finite convex program through two different approaches: a randomized method based on the so-called scenario approach, and a regularized, accelerated first-oder method. Both methods lead to explicit bounds on the approximation error of the proposed scheme.

For potential research directions on the content of this part, one may think of the following possibilities:
\begin{enumerate}
\item Given such an approximating scheme, we aim to study how to find $\varepsilon$-approximating policies, i.e., policies whose corresponding cost is $\varepsilon$ away from the optimal value.
\item In light of Theorem~\ref{thm:semi-fin:rand}, one may look more closely into the derivation of TBs as introduced in Definition~\ref{def:tail}. Meaningful TBs may depend highly on the individual structure of the optimization problems, in particular the uncertainty set and the constraint functions. In fact, the probability measure $\PP$ can also be viewed as a decision variable to propose the most \emph{informative} scenario program to tackle the robust formulation.
\end{enumerate}

\section{Information theoretic problems}
\subsection{Channel capacity approximation}
In Chapter~\ref{chap:channel:cap}, we introduced a new approach to approximate the capacity of DMCs possibly having constraints on the input distribution. The dual problem of Shannon's capacity formula turns out to have a particular structure such that the Lagrange dual function admits a closed form solution. Applying smoothing techniques to the non-smooth dual function enables us to solve the dual problem efficiently. This new approach, in the case of no constraints on the input distribution, has a computational complexity per iteration step of $O(MN)$, where $N$ is the input alphabet size and $M$ the size of the output alphabet. In comparison, the Blahut-Arimoto algorithm has a computational cost of $O(MN^2)$ per iteration step. More precisely for no input power constraint, the total computational cost to find an $\varepsilon$-close solution is $O(\tfrac{M^2 N \sqrt{\log(N)}}{\varepsilon})$ for the algorithm developed in this chapter, whereas the Blahut-Arimoto algorithm requires $O(\tfrac{MN^2 \log(N)}{\varepsilon})$. We would like to emphasize that the computational cost of the smallest unit, i.e.\ the cost of one iteration is strictly better for the algorithm introduced in this chapter. As highlighted by Example~\ref{ex:one}, this can make a crucial difference especially for large input alphabets. Another strength of the new approach is that it provides an a posteriori error, i.e., after having run a certain number of iterations we can precisely estimate the actual error in the current approximation. This is computationally appealing as explicit (or a priori) error bounds often are conservative in practice. By exploiting this a posteriori bound we can stop the computation once the desired accuracy has been reached.

As a second contribution, we have shown how similar ideas can be used to approximate the capacity of memoryless channels with continuous bounded input alphabets and countable output alphabets under a mild assumption on the channels tail. This assumption holds, for example, for discrete-time Poisson channels, allowing us to efficiently approximate their capacity. As an example we derived upper and lower bounds for a discrete-time Poisson channel having a peak-power constraint at the input. 

The presented optimization method highly depends on the Lipschitz constant estimate of the objective's gradient. The worse this estimate the more steps the method requires for an a priori $\varepsilon$-precision. For future work, we aim to study the derivation of local Lipschitz constants of the gradient. This technique has recently been shown to be very efficient in practice (up to three orders of magnitude reduction of computation time), while preserving the worst-case complexity \cite{ref:Baes-14}.

In the case of a continuous input alphabet, Section~\ref{sec:cont:Channels}, the proposed method requires to evaluate the gradient $\nabla G_{\nu}(\cdot)$ in every step of Algorithm~1, that requires solving an integral over $\mathbb{A}$. As such the method used to compute those integrals has to be included to the complexity of the proposed algorithm. Therefore, it would be interesting to investigate under which structural properties on the channel the gradient $\nabla G_{\nu}(\cdot)$ can be evaluated efficiently.

The approach introduced in this chapter can be used to efficiently approximate the capacity of classical-quantum channels, i.e., channels that have classical input and quantum mechanical output, with a discrete or bounded continuous input alphabet. Using the idea of a universal encoder allows us to compute close upper and lower bounds for the Holevo capacity \cite{ref:D:Sutter-16}.

\subsection{Generalized maximum entropy estimation}
In Chapter~\ref{chap:entropy:max}, we presented an approximation scheme to a generalization of the classical problem of estimating a density via a maximum entropy criterion, given some moment constraints. The key idea used is to apply smoothing techniques to the non-smooth dual function of the entropy maximization problem, that enables us to solve the dual problem efficiently with fast gradient methods. Due to the favourable structure of the considered entropy maximization problem, we provide explicit error bounds on the approximation error as well as a-posteriori error estimates.

The proposed method requires one to evaluate the gradient \eqref{eq:gradient:approximation} in every iteration step, which, as highlighted in Section~\ref{sec:gradient:evalutaion}, in the infinite-dimensional setting involves an integral. As such the method used to compute those integrals has to be included to the complexity of the proposed algorithm and, in higher dimensions, may become is the dominant factor.~Therefore, it would be interesting to investigate this integration step in more detail.~Two approaches, one based on semidefinite programming and another invoking Quasi-Monte Carlo integration techniques, are briefly sketched. 
Another potential direction, would be to test the proposed numerical method in the context of approximating the channel capacity of a large class of memoryless channels~\cite{TobiasSutter15}, as mentioned in the introduction.
 
 Finally it should be mentioned that the approximation scheme proposed in this chapter can be further generalized to quantum mechanical entropies. In this setup probability mass functions are replaced by density matrices (i.e., positive semidefinite matrices, whose trace is equal to one). The von Neumann entropy of such a density matrix $\rho$ is defined by $H(\rho):=-\mathrm{tr}(\rho \log \rho)$, which reduces to the (Shannon) entropy in case the density matrix $\rho$ is diagonal. Also the relative entropy can be generalized to the quantum setup~\cite{umegaki62} and general treatment of our approximation scheme, its analysis can be lifted to the this (strictly) more general framework. As demonstrated in~\cite{ref:D:Sutter-16}, (quantum) entropy maximization problems can be used to efficiently approximate the classical capacity of quantum channels.



\fancyhead[LO, RE]{Bibliography}

	\bibliographystyle{alpha} 



	\addcontentsline{toc}{part}{Bibliography}

	\small{\bibliography{library_GXZ}}





\chapter*{Curriculum Vitae}
\addcontentsline{toc}{chapter}{Curriculum Vitae}

{\bf Tobias Sutter\\
Born on March $14^\text{th}$, 1987 in St.~Gallen, Switzerland

\vspace{1.0cm}

\begin{table}[htb]
\begin{tabular}{r ll}

2012 -- 2017 & {\bf Doctoral Studies} at Automatic Control Laboratory, \\
	& Department of Information Technology and Electrical Engineering, \\
			& Swiss Federal Institute of Technology Zurich (ETH Zurich), Switzerland\\[2ex]

2010 -- 2012 & {\bf Master Studies} \\
	& Department of Mechanical and Process Engineering, \\
			&  Swiss Federal Institute of Technology Zurich (ETH Zurich), Switzerland\\[2ex]

2007 -- 2010 & {\bf Bachelor Studies} \\
	& Department of Mechanical and Process Engineering, \\
			&  Swiss Federal Institute of Technology Zurich (ETH Zurich), Switzerland\\[2ex]

2006 -- 2007 & {\bf Military Service} \\
		& Swiss Military Orchestra (clarinetist),\\
			& Swiss Army, Bern, Switzerland \\ [2ex]

2002 -- 2006 & {\bf High School, Swiss Matura}\\
		 & Kantonsschule am Burggraben, St.~Gallen, Switzerland \\
		 
\end{tabular}
\label{default}
\end{table}%

\end{document}